\documentclass[12pt,reqno]{amsart}

\usepackage{amsmath}
\usepackage[T1]{fontenc}
\usepackage{mathrsfs}
\usepackage{latexsym}
\usepackage{amssymb}
\usepackage{graphicx}
\usepackage{float}
\usepackage{extarrows}
\usepackage{epstopdf}
\usepackage{caption}
\usepackage{subcaption}
\usepackage{amsmath}
\usepackage{tikz}
\usepackage[colorlinks]{hyperref}
\usepackage[title]{appendix}
\usepackage{color} 
\usepackage{algorithm}
\usepackage[noend]{algpseudocode}
\usepackage{geometry}
\usepackage[utf8]{inputenc}

\def\DataDonutA{\texttt{donut1d}}
\def\DataDonutB{\texttt{donut2d}}
\def\DataSquares{\texttt{squares}}
\def\DataSpiral{\texttt{spiral}}
\def\mnist{\texttt{MNIST100}}

\def\CellWidth{3.55cm}
\def\RowHeight{3.1cm}
\def\PicWidth{3.5cm}

\newcommand\PlotLabels[4]{
\node at (0*\CellWidth,0.5*\RowHeight + 2mm) {#1};
\node at (1*\CellWidth,0.5*\RowHeight + 2mm) {#2};
\node at (2*\CellWidth,0.5*\RowHeight + 2mm) {#3};
\node at (3*\CellWidth,0.5*\RowHeight + 2mm) {#4};}

\newcommand\PlotRow[6]{%
\node at (-0.5*\CellWidth - 3mm,-#1*\RowHeight) {\rotatebox{90}{#2}};
\node at (0*\CellWidth,-#1*\RowHeight) {\PlotIm{#3}};%
\node at (1*\CellWidth,-#1*\RowHeight) {\PlotIm{#4}};%
\node at (2*\CellWidth,-#1*\RowHeight) {\PlotIm{#5}};%
\node at (3*\CellWidth,-#1*\RowHeight) {\PlotIm{#6}};%
}%

\newcommand{\PlotIm}[1]{\includegraphics[width=\PicWidth, clip, trim=3cm 2cm 4cm 2cm]{#1}}%

\author[Benning \and Celledoni \and Ehrhardt \and Owren \and Sch\"onlieb]{Martin Benning \and Elena Celledoni \and Matthias J. Ehrhardt \and Brynjulf Owren \and Carola-Bibiane Sch\"onlieb}

\title[Deep learning as optimal control problems]{Deep learning as optimal control problems: models and numerical methods}

\newtheorem{proposition}[subsection]{Proposition}
\newtheorem{remark} [subsection]{Remark}

\begin{document}
\maketitle

\begin{abstract}
We consider recent work of \cite{haber2017stable} and \cite{chang2018reversible}, where deep learning neural networks have been interpreted as discretisations of an optimal control problem subject to an ordinary differential equation constraint. We review the first order conditions for optimality, and the conditions ensuring optimality after discretisation. This leads to a class of algorithms for solving the discrete optimal control problem which guarantee that the corresponding discrete necessary conditions for optimality are fulfilled. The differential equation setting lends itself to learning additional parameters such as the time discretisation. We explore this extension alongside natural constraints (e.g. time steps lie in a simplex). We compare these deep learning algorithms numerically in terms of induced flow and generalisation ability.
\end{abstract}

\section{Introduction}
Deep learning has had a transformative impact on a wide range of tasks related to Artificial Intelligence, ranging from computer vision and speech recognition to playing games \cite{Igami:2017aa,LeCun:2015aa}. 

Despite impressive results in applications, the mechanisms behind deep learning remain rather mysterious, resulting in deep neural networks mostly acting as black-box algorithms. Consequently, also theoretical guarantees for deep learning are scarce, with major open problems residing in the mathematical sciences. An example are questions around the stability of training as well as the design of stable architectures. These questions are fed by results on the possible instabilities of the training (due to the high-dimensional nature of the problem in combination with its non-convexity) \cite{sutskever2013importance,dauphin2014identifying} which are connected to the lack of generalisability of the learned architecture, and adversarial vulnerability of trained networks \cite{Szegedy2013IntriguingPO} that can result in instabilities in the solution and gives rise to systematic attacks with which networks can be fooled \cite{explainingharnessing,kurakin2016adversarial,moosavi2016deepfool}. In this work we want to shed light on these issues by interpreting deep learning as an optimal control problem in the context of binary classification problems. Our work is mostly inspired by a very early paper by LeCun \cite{lecun1988theoretical}, and a series of recent works by Haber, Ruthotto et al. \cite{haber2017stable,chang2018reversible}. 

{\bf Classification in machine learning:} \emph{Classification} is a key task in machine learning; the goal is to learn functions, also known as \emph{classifiers}, that map their input arguments onto a discrete set of labels that are associated with a particular class. A simple example is image classification, where the input arguments are images that depict certain objects, and the classifier aims to identify the class to which the object depicted in the image belongs to. We can model such a classifier as a function $g:\mathbb{R}^n \rightarrow \{ c^0, c^1, \ldots, c^{K - 1}\}$ that takes $n$-dimensional real-valued vectors and maps them onto a discrete set of $K$ class labels. Note that despite using numerical values, there is no particular ordering of the class labels. The special case of $K = 2$ classes (and class labels) is known as \emph{binary classification}; for simplicity, we strictly focus on binary classification for the remainder of this paper. The extension to multi-class classification is straightforward, see e.g. \cite{Bishop2006ml}.

In supervised machine learning, the key idea is to find a classifier by estimating optimal parameters of a parametric function given pairs of data samples $\{ (x_i, c_i) \}_{i = 1}^m$, for $c_i \in \{ c^0, c^1\}$, and subsequently defining a suitable classifier that is parameterised with these parameters. The process of finding suitable parameters is usually formulated as a generalised regression problem, i.e. we estimate parameters $u, W, \mu$ by minimising a cost function of the form\footnote{One can of course use other cost functions such as the cross-entropy \cite{Bishop2006ml}. Our theory includes all smooth cost functions.}
\begin{align}
    \frac12 \sum_{i = 1}^m \left| \, \mathcal{C}\left( W h(x_i, u) + \mu \right) - c_i \, \right|^2 + \mathcal R(u) \, , \label{eq:reg_classification}
\end{align}
with respect to $u, W$ and $\mu$. Here $h$ is a model function parameterised by parameters $u$ that transforms inputs $x_i \in \mathbb{R}^n$ onto $n$-dimensional outputs. The vector $W \in \mathbb{R}^{1 \times n}$ is a \emph{weight} vector that weights this $n$-dimensional model output, whereas $\mu \in \mathbb{R}$ is a scalar that allows a \emph{bias} of the weighted model output, and $\mathcal{C}:\mathbb{R}^n \rightarrow \mathbb{R}$ is the so-called \emph{hypothesis-function} (cf. \cite{he2016identity, haber2017stable}) that maps this weighted and biased model output to a scalar value that can be compared to the class label $c_i$. The function $\mathcal{R}$ is a regularisation function that is chosen to ensure some form of regularity of the parameters $u$ and existence of parameters that minimise \eqref{eq:reg_classification}. Typical regularisation functions include the composition of the squared 2-norm with a linear operator (Tikhonov--Phillips regularisation \cite{tikhonov1963solution,phillips1962technique}; in statistics this technique is called ridge regression, while in machine learning it is known as weight decay \cite{plaut1986experiments}) or the 1-norm to induce sparsity of the weights \cite{santosa1986linear,tibshirani1996regression}. However, depending on the application, many different choices of regularisation functions are possible.

Note that minimising \eqref{eq:reg_classification} yields parameters that minimise the deviation of the output of the hypothesis function and the given labels. If we denote those parameters that minimise \eqref{eq:reg_classification} by $\hat W$, $\hat \mu$ and $\hat u$, and if the hypothesis function $\mathcal C$ maps directly onto the discrete set $\{ c_0, c_1\}$, then a suitable classifier can simply be defined via
\begin{align*}
    g(x) := \mathcal{C}\left( \hat W h(x, \hat u) + \hat \mu \right) \, .
\end{align*}
However, in practice the hypothesis function is often rather continuous and does not map directly on the discrete values $\{ c_0, c_1\}$. In this scenario, a classifier can be defined by subsequent thresholding. Let $c_0$ and $c_1$ be real numbers and w.l.o.g. $c_0 < c_1$, then a suitable classifier can for instance be defined via
\begin{align*}
g(x) := \begin{cases} c_0 & \mathcal{C}\left( \hat W h(x, \hat u) + \hat \mu \right) \leq \frac{c_0 + c_1}{2} \\
c_1 &  \mathcal{C}\left( \hat W h(x, \hat u) + \hat \mu \right) > \frac{c_0 + c_1}{2} \end{cases} \, .
\end{align*}

{\bf Deep learning as an optimal control problem:} One recent proposal towards the design of deep neural network architectures is \cite{weinan2017proposal,haber2017stable,chang2018reversible,Han2018,Li2018}. There, the authors propose an interpretation of deep learning by the popular Residual neural Network (ResNet) architecture \cite{He2016resnet} as discrete optimal control problems. Let
$$u^{[j]}:= \left(K^{[j]}, \beta^{[j]}\right),\qquad j=0,\dots , N-1, \qquad u = \left(u^{[0]}, \ldots, u^{[N-1]}\right) \, ,$$
where $K^{[j]}$ is a $n\times n$ matrix of weights, $\beta^{[j]}$ represents the biases, and $N$ is the number of layers.

In order to use ResNet for binary classification we can define the output of the model function $h$ in \eqref{eq:reg_classification} as the output of the ResNet. With the ResNet state variable denoted by $y = (y^{[0]}, \ldots, y^{[N]})$, $y^{[j]} = (y_1^{[j]}, \ldots, y_n^{[j]})$, this implies that the classification problem \eqref{eq:reg_classification} can be written as a constraint minimisation problem of the form
\begin{equation}
\label{Doptimisation}
\min_{y, u, W, \mu}\, \sum_{i = 1}^m \left| \, \mathcal{C}\left( W y^{[N]}_i + \mu \right) - c_i \, \right|^2 + \mathcal R(u),
\end{equation}
subject to the constraint
\begin{equation}
\label{Dconstraint}
y^{[j+1]}_i = y^{[j]}_i + \Delta t\,f(y^{[j]}_i,u^{[j]}), \qquad j=0,\dots , N-1, \qquad y^{[0]}_i = x_i \, .
\end{equation}
Here $\Delta t$ is a parameter which for simplicity at this stage can be chosen to be equal to $1$ and whose role will become clear in what follows. The constraint \eqref{Dconstraint} is the ResNet parametrisation of a neural network \cite{He2016resnet}. In contrast, the widely used feed-forward network that we will also investigate later is given by
\begin{equation}
\label{eq:feedforward}
y^{[j+1]}_i = f(y^{[j]}_i,u^{[j]}), \qquad j=0,\dots , N-1, \qquad y^{[0]}_i = x_i \, .
\end{equation}
For deep learning algorithms, one often has
\begin{equation}
\label{eq:deeplearning:f}
f(y^{[j]}_i, u^{[j]}):= \sigma \left( K^{[j]} \,y^{[j]}_i + \beta^{[j]} \right),
\end{equation}
where $\sigma$ is a suitable activation function acting component-wise on its arguments. For a more extensive mathematical introduction to deep learning we recommend~\cite{Higham2018deeplearning}.

Suppose, in what follows, that $y_i=y_i(t)$ and $u=u(t)=(K(t),\beta(t))$, $t \in[0,T]$, are functions of time  and $y_i^{[j]}\approx y_i(t_j)$.  To view \eqref{Doptimisation} and \eqref{Dconstraint} as a discretisation of an optimal control problem \cite{chang2018reversible}, one observes that the constraint equation \eqref{Dconstraint} is the discretisation of the ordinary differential equation (ODE) $\dot{y_i}=f(y_i,u),$ $y_i(0)=x_i$, on $[0,T]$, with step-size $\Delta t$ and with the forward Euler method. In the continuum, the following optimal control problem is obtained \cite{chang2018reversible},
\begin{equation}
\label{optimisation}
\min_{y, u, W, \mu}\,\sum_{i = 1}^m \left|\mathcal C\left(W\,y_i(T) + \mu\right) - c_i\right|^2 + \mathcal R(u)
\end{equation}
subject to the ODE constraint
\begin{equation}
\label{constraint}
\dot{y_i}=f(y_i,u),\quad t\in [0,T],\quad y_i(0)= x_i.
\end{equation}
Assuming that problem  \eqref{optimisation}-\eqref{constraint} satisfies necessary conditions for optimality \cite[ch. 9]{sontag2013mathematical}, and that a suitable activation function and cost function have been chosen,  a number of new deep learning algorithms can be generated. For example, the authors of \cite{chen2018neural,pmlr-v80-lu18d} propose to use accurate approximations of \eqref{constraint} obtained by black-box ODE solvers. Alternatively, some of these new strategies are obtained by considering constraint ODEs \eqref{constraint} with different structural properties, e.g. taking $f$ to be a Hamiltonian vector field, and by choosing accordingly the numerical integration methods to approximate \eqref{constraint}, \cite{chang2018reversible, haber2017stable}. This entails augmenting the dimension, e.g. by doubling the number of variables in the ODE, a strategy also studied in \cite{gholami19aua}. 
Stability is perhaps not important in networks with a fixed and modest number of layers. However, in designing and understanding deep neural networks, it is of importance to analyse its behaviour when the depth grows towards infinity. The stability of neural networks has been an important issue in many papers in this area. The underlying continuous dynamical system offers a common framework for analysing the behaviour of different architectures, for instance through backward error analysis, see e.g. \cite{hairer2006geometric}. The optimality conditions are useful for ensuring consistency between the discrete and continuous optima, and possibly the adjoint variables can be used to analyse the sensitivity of the network to perturbations in initial data.
We also want to point out that the continuous limit of neural networks is not only relevant for the study of optimal control problems, but also for optimal transport  \cite{sonoda2017double} or data assimilation \cite{abarbanel2018machine} problems.

{\bf Our contribution:} The main purpose of this paper is the investigation of different discretisations of the underlying continuous deep learning problem  \eqref{optimisation}-\eqref{constraint}. In \cite{haber2017stable, chang2018reversible} the authors investigate different ODE-discretisations for the neural network \eqref{constraint}, with a focus on deriving a neural network that describes a `stable' flow, i.e. the solution $y(T)$ should be bounded by the initial condition $y(0)$.

Our  point of departure from the state-of-the-art will be to outline the well established theory on optimal control,  numerical ODE problems based on  \cite{hager2000runge, ross2005roadmap, sanzserna2015sympletic,li2017maximum}, where we investigate the complete optimal control problem  \eqref{optimisation}-\eqref{constraint} under the assumption of necessary conditions for optimality \cite[ch. 9]{sontag2013mathematical}. 

The formulation of the deep learning problem \eqref{Doptimisation}-\eqref{Dconstraint} is a first-discretise-then-optimise approach to optimal control, where ODE \eqref{constraint} is first discretised with a forward Euler method to yield an optimisation problem which is then solved with gradient descent (direct method). In this setting the forward Euler method could be easily replaced by a different and more accurate integration method, but the back-propagation for computing the gradients of the discretised objective function will typically become more complicated to analyse.

Here, we propose a first-optimise-then-discretise approach for deriving new deep learning algorithms. There is a two-point boundary value Hamiltonian problem associated to \eqref{optimisation}-\eqref{constraint} expressing first  order optimality conditions of the optimal control problem \cite{Pontryagin}. This boundary value problem consists of \eqref{constraint}, with $y_i(0)=x$, together with its adjoint equation with boundary value at final time $T$, and in addition an algebraic constraint. In the first-optimise-then-discretise approach, this boundary value problem is solved by a numerical integration method. It is natural to solve equation \eqref{constraint} forward in time with a Runge--Kutta method (with non vanishing weights $b_i$, $i=1,\dots, s$), while the adjoint equation must be solved backward in time and with a matching  Runge--Kutta  method (with weights satisfying \eqref{quadInvPRK}) and imposing the constraints at each time step. If the combination of the forward integration method and its counterpart used backwards in time form a symplectic partitioned Runge--Kutta method then the overall discretisation is equivalent to a first-discretise-then-optimise approach, but with an efficient and automatic computation of the gradients \cite{hager2000runge, sanzserna2015sympletic}, see Proposition \ref{necdiscreteoptimality}.

We implement discretisation strategies based on different Runge--Kutta methods for \eqref{optimisation}-\eqref{constraint}. To make the various methods  comparable to each other, we use the same learned parameters for every Runge--Kutta stage, in this way the total the number of parameters will not depend on how many stages each method has. The discretisations are adaptive in time, and learning the step-sizes the number of layers is determined automatically by the algorithms. From the optimal control formulation we derive different instances of deep learning algorithms \eqref{Doptimisation}-\eqref{Dconstraint} by numerical discretisations of the first-order optimality conditions  using a partitioned Runge--Kutta method. 

{\bf Outline of the paper:} In Section \ref{sec:optcontrol} we derive the optimal control formulation of  \eqref{optimisation}-\eqref{constraint} and discuss its main properties. In particular, we derive the variational equation, the adjoint equation, the associated Hamiltonian and first-order optimality conditions. Different instances of the deep learning algorithm \eqref{Doptimisation}-\eqref{Dconstraint} are derived in Section \ref{sec:numdiscrete} by using symplectic partitioned Runge--Kutta methods applied to the constraint equation \eqref{constraint_noi}, and the adjoint equations \eqref{adjoint} and \eqref{constraintAd}. Using partitioned Runge--Kutta discretisation guarantees equivalence of the resulting optimality system to the one derived from a first-discretise-then-optimise approach using gradient descent, cf. Proposition \ref{necdiscreteoptimality}. In Section \ref{sec:PRKnetworks} we derive several new deep learning algorithms from such optimal control discretisations, and investigate by numerical experiments their dynamics in Section \ref{sec:numericaltests} on a selection of toy problems for binary classification in two dimensions.

\section{Properties of the optimal control problem}\label{sec:optcontrol}
In this section we review established literature on optimal control which justifies the use of the numerical methods of the next section.

\subsection{Variational equation}
In this section, we consider a slightly simplified formulation of \eqref{optimisation}-\eqref{constraint}. In particular, for simplicity we discard the term $\mathcal R(u)$ in \eqref{optimisation}, and remove the index "$i$" in \eqref{optimisation}-\eqref{constraint} and the summation over the number of data points. Moreover, as we here focus on the ODE \eqref{constraint}, we also remove the dependency on the classification parameters $W$ and $\mu$ for now. We rewrite the optimal control problem in the simpler form
\begin{equation}
\label{optimisation_noi}
\min_{y, u}\,\mathcal{J}(y(T)),
\end{equation}
subject to the ODE constraint
\begin{equation}
\label{constraint_noi}
\dot{y}=f(y,u),\quad y(0)= x.
\end{equation}

\noindent Then, the variational equation for \eqref{optimisation_noi}-\eqref{constraint_noi} reads
\begin{equation}
\label{variational}
\frac{d}{dt}v =\partial_{y} f(y(t),u(t))\,v +\partial_u f(y(t),u(t))\,w
\end{equation}
where $\partial_y f$ is the Jacobian of $f$ with respect to $y$, $\partial_u f$ is the Jacobian of $f$ with respect to $u$, and $v$ is the variation in $y$, while $w$ is the variation in $u$\footnote{$\tilde{y}(t)=y(t)+\xi v(t)$ for $|\xi|\rightarrow 0$ and similarly for $\tilde{u}(t)=u(t)+\xi w(t)$ $|\xi|\rightarrow 0$.}. 
Since $y(0)=x$ is fixed, $v(0)=0$. 

\subsection{Adjoint equation}
The adjoint of \eqref{variational} is a system of ODEs for a variable $p(t)$, obtained assuming
\begin{equation}
\label{quadraticInv}
\langle p(t),v(t)\rangle=\langle p(0),v(0)\rangle,\qquad \forall t\in[0,T].
\end{equation}
Then \eqref{quadraticInv} implies
$$\langle p(t),\dot{v}(t)\rangle=-\langle \dot{p}(t),v(t)\rangle,$$
an integration-by-parts formula which together with \eqref{variational} leads to the following equation for $p$:
\begin{equation}
\label{adjoint}
\frac{d}{dt}p=-\left(\partial_y f(y(t),u(t))\right)^T\,\left(p\right) \, ,
\end{equation}
with constraint
\begin{equation}
\label{constraintAd}
\left(\partial_u f(y(t),u(t))\right)^T p=0 \, ,
\end{equation}
see \cite{sanzserna2015sympletic}.
Here we have denoted by $\left(\partial_y f\right)^T$ the transpose of $\partial_y f$ with respect to the Euclidean inner product $\langle \cdot ,\cdot \rangle$, and similarly $\left(\partial_u f\right)^T$ is the transpose of $\partial_u f$.

\subsection{Associated Hamiltonian system}\label{sec:Hamiltonian}
For such an optimal control problem, there is an associated Hamiltonian system with Hamiltonian 
$$\mathcal H(y,p,u):=\langle p, f(y,u)\rangle$$
with
\begin{equation}
\label{HamiltonianSys}
\dot{y}=\partial_{p} \mathcal H,\quad \dot{p}=-\partial_y \mathcal H,\quad \partial_u \mathcal H = 0,
\end{equation}
where we recognise that the first equation $\dot{y}=\partial_{p} \mathcal{H}$ coincides with \eqref{constraint_noi}, the second $\dot{p}=-\partial_y \mathcal{H}$ with \eqref{adjoint} and the third $\partial_u\mathcal{H}=0$ with \eqref{constraintAd}.

The constraint Hamiltonian system is a differential algebraic equation of index one if the Hessian $\partial_{u,u} \mathcal{H}$ is invertible.
In this case, by the implicit function theorem there exists $\varphi$ such that
$$u=\varphi(y,p),\quad \mathrm{and}\quad \bar{\mathcal{H}}(y,p)=\mathcal{H}(y,p,\varphi(y,p)),$$
where the differential algebraic Hamiltonian system is transformed into a canonical Hamiltonian system of ODEs with Hamiltonian $\bar{\mathcal{H}}$. Notice that it is important to know that $\varphi$ exists, but it is not necessary to compute $\varphi$ explicitly for discretising the problem.

\subsection{First order necessary conditions for optimality}\label{sec:firstordcond}
The solution of the two point boundary value problem \eqref{constraint_noi} and \eqref{adjoint},\eqref{constraintAd} with $y(0)=x$ and 
$$p(T)=\left. \partial_y \mathcal{J}(y)\right|_{y=y(T)},$$
has the following property 
$$\langle \left. \partial_y \mathcal{J}(y)\right|_{y=y(T)},v(T)\rangle=\langle p(T), v(T)\rangle=\langle p(0), v(0)\rangle=0,$$
so the variation $v(T)$ is orthogonal to the gradient of the cost function $\left. \partial_y \mathcal{J}(y)\right|_{y=y(T)}$.  This means that the solution $(y(t),v(t),p(t))$ satisfies the first order necessary conditions for extrema of $\mathcal{J}$ (Pontryagin maximum principle) \cite{Pontryagin}, see also \cite[ch. 9.2]{sontag2013mathematical}.

\section{Numerical discretisation of the optimal control problem}\label{sec:numdiscrete}
We consider a time discrete setting $y^{[0]},y^{[1]},\dots ,y^{[N]}$, $u^{[0]},\dots,$ $u^{[N-1]}$ and a cost function $\mathcal{J}(y^{[N]})$, assuming to apply a numerical time discretisation $y^{[j+1]}=\Phi_{\Delta t}(y^{[j]},u^{[j]})$, $j=0,\dots, N-1$ of \eqref{constraint_noi}, the discrete optimal control problem becomes
\begin{equation*}
\label{Doptimisation1}
\min_{(y^{[j]},u^{[j]})} \,\mathcal{J}(y^{[N]}),
\end{equation*}
subject to
\begin{equation}
\label{Dconstraint1}
y^{[j]}=\Phi_{\Delta t}(y^{[j-1]},u^{[j-1]}),\quad y^{[0]}=x.
\end{equation}
Here the subscript $\Delta t$ denotes the discretisation step-size of the time interval $[0,T]$. This discrete optimal control problem corresponds  to a deep learning algorithm with the outlined choices for $f$ and $\mathcal{J}$, see for example \cite{haber2017stable}.

We assume that $\Phi_{\Delta t}$ is a Runge--Kutta method with non vanishing weights for the discretisation of \eqref{constraint_noi}. Applying a  Runge--Kutta method to \eqref{constraint_noi}, for example the forward Euler method, we obtain
$$y^{[j+1]}=y^{[j]}+\Delta t \,f(y^{[j]},u^{[j]}),$$
and taking variations $y^{[j+1]}+\xi v^{[j+1]}$, $y^{[j]}+\xi v^{[j]}$, , $u^{[j]}+\xi w^{[j]}$, for $\xi\rightarrow 0$ one readily obtains the same Runge--Kutta method applied to the variational equation
$$v^{[j]}=v^{[j+1]}+\Delta t\,  \left[ \partial_y f(y^{[j]},u^{[j]})\,v^{[j]} + \partial_u f(y^{[j]},u^{[j]})\,w^{[j]}\right].$$
This means that taking variations is an operation that commutes with applying Runge--Kutta methods. This is a well known property of Runge--Kutta methods and also of the larger class of so called B-series methods, see for example \cite[ch. VI.4, p. 191]{hairer2006geometric} for details.

In order to ensure that the first order necessary conditions for optimality from Section \ref{sec:firstordcond} are satisfied also after discretisation, fixing a certain Runge--Kutta method $\Phi_{\Delta t}$ for \eqref{constraint_noi}, we need to discretise the adjoint equations \eqref{adjoint} and \eqref{constraintAd} such that the overall method is a symplectic, partitioned Runge--Kutta method for the system spanned by \eqref{constraint_noi}, \eqref{adjoint} and \eqref{constraintAd}. This will in particular guarantee the preservation of the quadratic invariant  \eqref{quadraticInv}, as emphasised in \cite{sanzserna2015sympletic}. The general format of a partitioned Runge--Kutta method as applied to \eqref{constraint_noi}, \eqref{adjoint} and \eqref{constraintAd} is
for $j=0,\dots , N-1$
\begin{subequations}
\begin{eqnarray}
y^{[j+1]}&=& y^{[j]} +\Delta t\,\sum_{i=1}^s{b_i}\, f^{[j]}_i \label{PRK1}\\
f^{[j]}_i &=&f(y^{[j]}_i, u^{[j]}_i),\quad i=1,\dots, s, \label{PRK2}\\
y^{[j]}_{i} &=&y^{[j]} + \Delta t\,\sum_{l=1}^s a_{i,l}\,f^{[j]}_{l},\quad i=1,\dots, s \label{PRK3}
\end{eqnarray} \label{PRK:y}
\end{subequations}%
\begin{subequations}%
\begin{eqnarray}
p^{[j+1]} &=&p^{[j]}+\Delta t\,\sum_{i=1}^s\tilde{b}_i\,\ell^{[j]}_{i}\label{PRK4}\\
\ell^{[j]}_{i} &=& -\partial_y f(y^{[j]}_{i}, u^{[j]}_{i})^T p^{[j]}_{i}, \quad i=1, \dots, s,\label{PRK5}\\
p^{[j]}_{i} &=& p^{[j]} + \Delta t\,\sum_{l=1}^s\tilde{a}_{i,l}\ell^{[j]}_{l},\quad i=1,\dots, s \label{PRK6}
\end{eqnarray} \label{PRK:p}
\end{subequations}%
\begin{eqnarray}%
&&\left(\partial_u f(y^{[j]}_{i}, u^{[j]}_{i})\right)^T p^{[j]}_{i}=0, \quad i=1,\dots, s \label{PRK:opt}
\end{eqnarray} 
and boundary conditions  $y^{[0]}=x$, $p^{[N]}:=\partial\mathcal{J}(y^{[N]}).$ We will assume $b_i\neq 0$, 
$i=1,\dots, s.$ \footnote{Generic $b_i$ in the context of optimal control is discussed in  \cite{sanzserna2015sympletic}.} 

\noindent It is well known \cite{hairer2006geometric} that if the coefficients of a partitioned Runge--Kutta satisfy
{\small
\begin{equation}
\label{quadInvPRK}
b_i=\tilde{b}_i,\quad b_i\tilde{a}_{i,j}+\tilde{b}_j a_{i,j}-b_i\tilde{b}_j=0,\quad c_i=\tilde{c}_i,\quad i,j=1,\dots , s,
\end{equation}
}
then the partitioned Runge--Kutta preserves invariants of the form
$$S:\mathbb{R}^d\times \mathbb{R}^{d}\rightarrow \mathbb{R},$$
where $S$ is bilinear. As a consequence the invariant $S(v(t),p(t)):=\langle p(t),v(t)\rangle$ \eqref{quadraticInv} will be preserved by such method. These partitioned Runge--Kutta methods are called symplectic. 

The simplest symplectic partitioned Runge--Kutta method is the symplectic Euler method, which is a combination of the explicit Euler method $b_1=1$, $a_{1,1}=0$, $c_1=0$  and the implicit Euler method $\tilde{b}_1=1$, $\tilde{a}_{1,1}=1$  $\tilde{c}_1=1$. This method applied to \eqref{constraint_noi}, \eqref{adjoint} and \eqref{constraintAd} gives
\begin{eqnarray*}
y^{[j+1]}&=&y^{[j]}+\Delta t\,f(y^{[j]},u^{[j]}),\\
p^{[j+1]}&=&p^{[j]}-\Delta t\,\left(\partial_y f(y^{[j]},u^{[j]})\right)^T p^{[j+1]},\\
0&=&\left(\partial_u f(y^{[j]},u^{[j]})\right)^T p^{[j+1]},
\end{eqnarray*} 
for $j=0,\dots , N-1$ and with the boundary conditions $y^{[0]}=x$, and $p^{[N]}:=\partial_y \mathcal{J}(y^{[N]})$.
 
\begin{proposition}\label{necdiscreteoptimality}
If \eqref{constraint_noi}, \eqref{variational} and \eqref{adjoint} are discretised with \eqref{PRK:y}-\eqref{PRK:opt} with $b_i\neq 0$, $i=1,\dots, s$, then the first order necessary conditions for optimality for the discrete optimal control problem
\begin{equation}
\min_{\begin{array}{l}
\{u_i^{[j]}\}_{j=0}^{N-1},\\
\{y^{[j]}\}_{j=1}^{N},\{y_i^{[j]}\}_{j=1}^{N},
\end{array}}\mathcal{J}(y^{[N]}),
\end{equation}
subject to
\begin{eqnarray}
\label{RK1}
y^{[j+1]}&=& y^{[j]} +\Delta t\,\sum_{i=1}^s{b_i}\, f^{[j]}_i\\
\label{RK2}
f^{[j]}_i &=&f(y^{[j]}_i, u^{[j]}_i),\quad i=1,\dots, s,\\
\label{RK3}
y^{[j]}_{i}&=&y^{[j]}+\Delta t\,\sum_{m=1}^s a_{i,m}\,f^{[j]}_{m},\quad i=1,\dots, s,
\end{eqnarray}
are satisfied.
\end{proposition}
\begin{proof} See appendix~\ref{DNOC}.\end{proof}

In \eqref{PRK:y}--\eqref{PRK:opt} it is assumed that there is a parameter set $u^{[j]}_i$ for each of the $s$ stages in each layer. This may simplified by considering only one parameter set $u^{[j]}$ per layer. Discretisation of the Hamiltonian boundary value problem with a symplectic partitioned Runge--Kutta method yields in this case the following expressions for the derivative of the cost function with respect to the controls.

\begin{proposition} \label{prop:grad}
Let $y^{[j]}$ and $p^{[j]}$ be given by \eqref{PRK:y} and \eqref{PRK:p} respectively. Then the gradient of the cost function $\mathcal J$ with respect to the controls is given by 
\begin{subequations}
\begin{align}
    \ell_i^{[j]}& = - \partial_y f(y_i^{[j]},u^{[j]})^T \left(p^{[j+1]}-\Delta t\sum_{k=1}^s \frac{a_{k,i}b_k}{b_i} \ell_k^{[j]}\right) \quad i = 1, \ldots, s\label{elldefgrad}\\
    \partial_{u^{[j]}} \mathcal J(y^{[N]}) &= \Delta t \sum_{i=1}^s b_i \partial_{u^{[j]}} f(y_i^{[j]},u^{[j]})^T
    \left(p^{[j+1]}- \Delta t\sum_{k=1}^s \frac{a_{k,i}b_k}{b_i} \ell_k^{[j]}\right).
\end{align}\label{gradu}
\end{subequations}
\end{proposition}

\begin{remark}
In the case that the Runge--Kutta method is explicit we have $a_{k,i}=0$ for $i\geq k$. In this case the stages $\ell_s^{[j]},\ell_{s-1}^{[j]},\ldots,\ell_{1}^{[j]}$ can be computed explicitly
from~\eqref{elldefgrad}.
\end{remark}

\begin{remark}
For the explicit Euler method, these formulas greatly simplify and the derivative of the cost function with respect to the controls can be computed as
\begin{align}
y^{[j+1]}&= y^{[j]} +\Delta t \, f(y^{[j]}, u^{[j]}) \\
p^{[j+1]} &=p^{[j]}-\Delta t\,\partial_y f(y^{[j]}, u^{[j]})^T p^{[j+1]}\\
\partial_{u^{[j]}} \mathcal J(y^{[N]}) &= \Delta t \,  \partial_{u^{[j]}} f(y^{[j]},u^{[j]})^T p^{[j+1]}.
\end{align}
\end{remark}

\begin{remark}
The convergence of the outlined Runge--Kutta discretisations to the continuous optimal control problem has been addressed in \cite{hager2000runge} see also \cite{dontchev2000tea} and recently also in the context of deep learning in \cite{thorpe2018dlo}.
\end{remark}

\section{Optimal Control motivated Neural Network Architectures}\label{sec:PRKnetworks}
An ODE-inspired neural network architecture is uniquely defined by choosing $f$ and specifying a time discretisation of \eqref{constraint_noi}. Here we will focus on the common choice $f(u, y) = \sigma(K y + \beta), u = (K, \beta)$ (e.g. ResNet) and also discuss a novel option $f(u, y) = \alpha \, \sigma(K y + \beta), u = (K, \beta, \alpha)$ which we will refer to as ODENet.

\subsection{Runge--Kutta networks, e.g. ResNet}
Here we choose $f(u, y) = \sigma(K y + \beta), u = (K, \beta)$. For simplicity we focus on the simplest Runge--Kutta method---the explicit Euler. This corresponds to the ResNet in the machine learning literature.%

\noindent In this case the network relation (forward propagation) is given by
\begin{align}
y^{[j+1]}&= y^{[j]} +\Delta t \, \sigma \left(K^{[j]} y^{[j]} + \beta^{[j]}\right)
\end{align}
and gradients with respect to the controls can be computed by first solving for the adjoint variable (backpropagation)
\begin{align}
\gamma^{[j]} &= \sigma^\prime\left(K^{[j]} y^{[j]} + \beta^{[j]}\right) \odot p^{[j+1]} \\
p^{[j+1]} &= p^{[j]} - \Delta t \, K^{[j],T} \gamma^{[j]}
\end{align}
and then computing
\begin{align}
\partial_{K^{[j]}} \mathcal J(y^{[N]}) &= \Delta t \,  \gamma^{[j]} \, y^{[j], T}\\
\partial_{\beta^{[j]}} \mathcal J(y^{[N]}) &= \Delta t \, \gamma^{[j]}.
\end{align}

\subsection{ODENet}
In contrast to the models we discussed so far, we can also enlarge the set of controls to model varying time steps. Let $u = (K, \beta, \alpha)$ and define
\begin{equation}
f(u, y) = \alpha \, \sigma(K y + \beta) \, .  \label{eq:odenet:f} 
\end{equation}
The function $\alpha$ can be interpreted as varying time steps. Then the network relation (forward propagation) is given by
\begin{align}
y^{[j+1]}&= y^{[j]} + \Delta t \, \alpha^{[j]} \, \sigma\left(K^{[j]} y^{[j]} + \beta^{[j]}\right)
\end{align}
and gradients with respect to the controls can be computed by first solving for the adjoint variable (backpropagation)
\begin{align}
\gamma^{[j]} &= \alpha^{[j]} \, \sigma^\prime \left(K^{[j]} y^{[j]} + \beta^{[j]}\right) \odot p^{[j+1]} \\
p^{[j+1]} &= p^{[j]} - \Delta t \, K^{[j],T} \gamma^{[j]}
\end{align}
and then computing
\begin{align}
\partial_{K^{[j]}} \mathcal J(y^{[N]}) &= \Delta t \,  \gamma^{[j]} \, y^{[j], T}\\
\partial_{\beta^{[j]}} \mathcal J(y^{[N]}) &= \Delta t \, \gamma^{[j]} \\
\partial_{\alpha^{[j]}} \mathcal J(y^{[N]}) &= \Delta t \, \left\langle p^{[j+1]}, \sigma \left(K^{[j]} y^{[j]} + \beta^{[j]} \right) \right\rangle \, .
\end{align}

\noindent It is natural to assume the learned time steps $\alpha$ should lie in the set of probability distributions
$$
S = \left \{ \alpha \; \middle| \;  \alpha \geq 0, \; \int \alpha = 1 \right \} \, ,
$$
or discretised in the simplex
\begin{equation}
S = \left \{ \alpha \in \mathbb R^N \; \middle| \;  \alpha^{[j]} \geq 0, \; \sum_j \alpha^{[j]} = 1 \right \} \, . \label{eq:simplex}
\end{equation}
This discretised constraint can easily be incorporated into the learning process by projecting the gradient descent iterates onto the constraint set $S$. Efficient finite-time algorithms are readily available~\cite{Duchi2008}.

\section{Numerical results}\label{sec:numericaltests}
\begin{algorithm}[t]
\caption{Training ODE-inspired neural networks with gradient descent.} \label{ALG:GD1}
\raggedright
\textbf{Input:} initial guess for the controls $u$, step-size $\tau$
\begin{algorithmic}[1]
    \For{$k = 1, \ldots$}
    \State forward propagation: compute $y$ via  \eqref{PRK:y}
    \State backpropagation: compute $p$ via \eqref{PRK:p}
    \State compute gradient $g$ via \eqref{gradu} and \eqref{gradientclassifier}
    \State update controls: $u = u - \tau g$
    \EndFor
\end{algorithmic}%
\end{algorithm}

\begin{algorithm}[t]
\caption{Training ODE-inspired neural networks with gradient descent and backtracking.} \label{ALG:GD2}
\raggedright
\textbf{Input:} initial guess for  controls $u$ and parameter $L$, \\ hyperparameters $\overline{\rho} > 1$ and $\underline{\rho} < 1$.
\begin{algorithmic}[1]
    \State forward propagation: compute $y$ with controls $u$ via \eqref{PRK:y} and $\phi = \mathcal J(y^{[N]})$
    \For{$k = 1, \ldots$}
    \State backpropagation: compute $p$ via \eqref{PRK:p}
    \State compute gradient $g$ via \eqref{gradu} and \eqref{gradientclassifier}
    \For{$t = 1, \ldots$}
    \State update controls: $\tilde u = u - \frac1L g$
    \State forward propagation: compute $\tilde y$  with controls $\tilde u$ and $\tilde \phi = \mathcal J(\tilde y^{[N]})$
    \If{$\tilde \phi \leq \phi + \langle g, \tilde u - u\rangle + \frac L2 \|\tilde u - u\|^2$}
    \State accept: $u = \tilde u, y = \tilde y, \phi = \tilde \phi$, $L = \underline{\rho} L$ 
    \State \textbf{break} inner loop
    \Else{} reject: $L = \overline{\rho} L$
    \EndIf
    \EndFor
    \EndFor
\end{algorithmic}%
\end{algorithm}

\begin{figure}
\renewcommand{\PlotIm}[1]{\includegraphics[width=\PicWidth, clip, trim=4cm 2cm 4cm 2cm]{#1}}%
\begin{tikzpicture}
\PlotRow{0}{}
{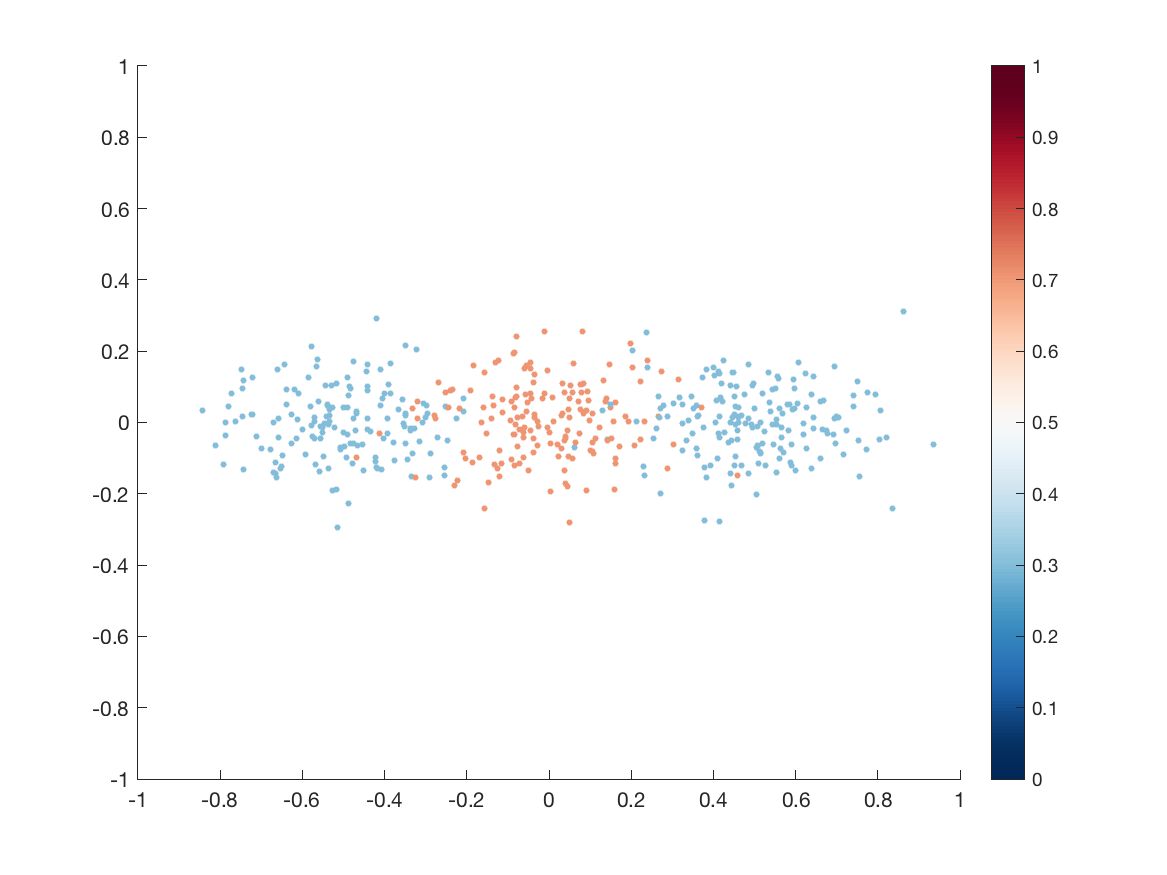}
{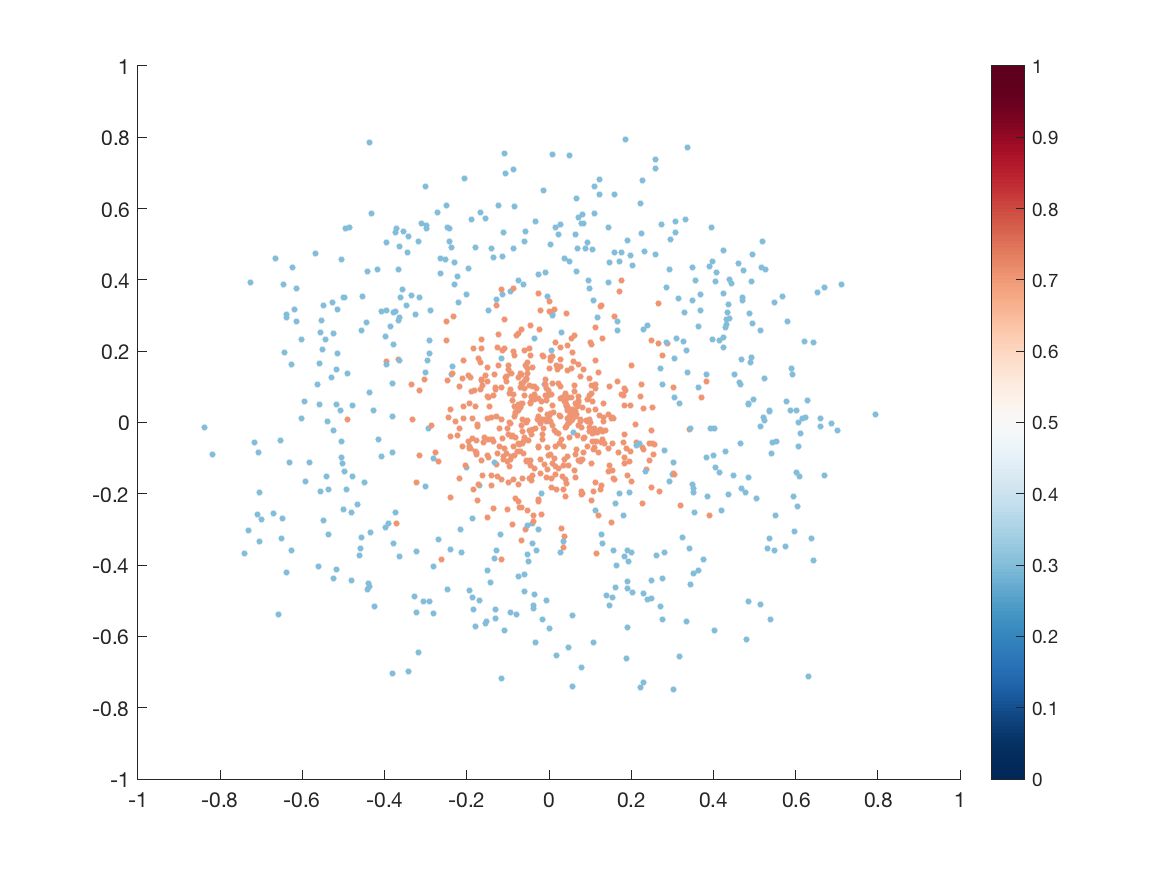}
{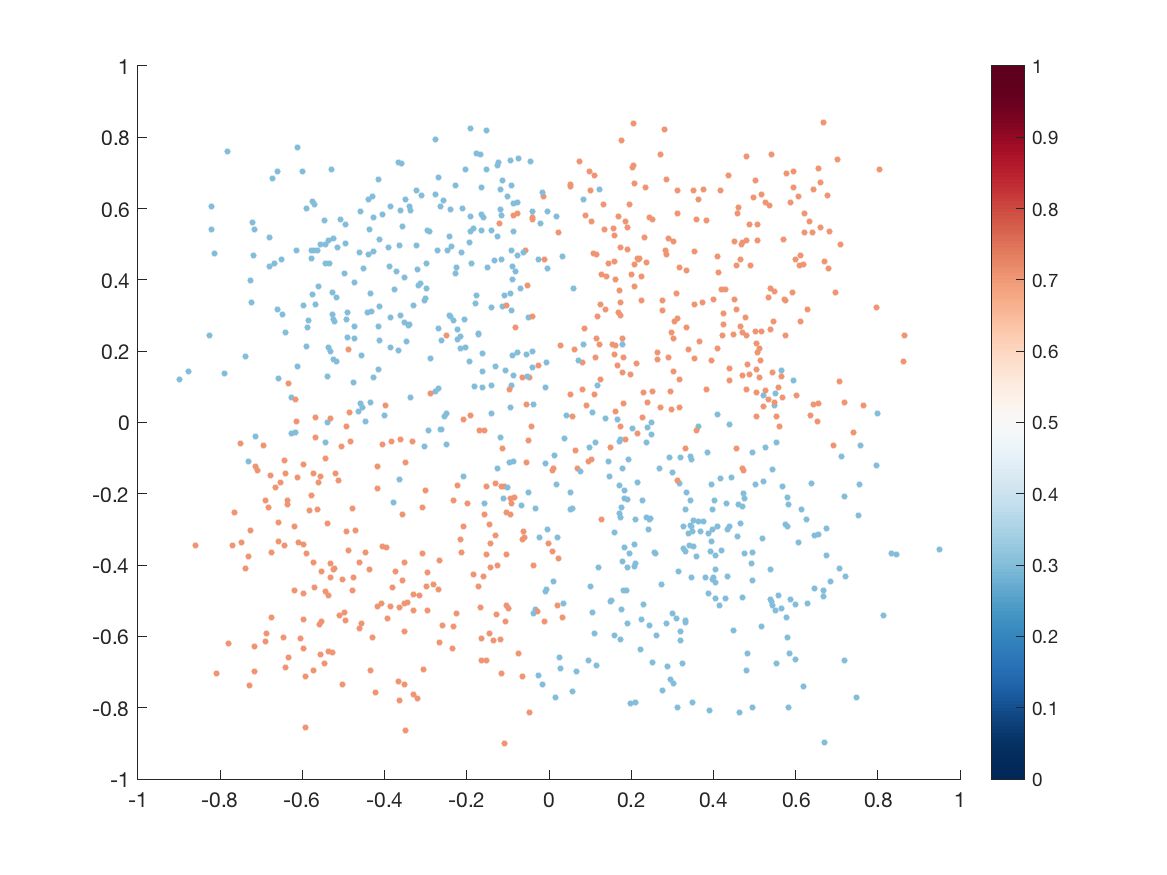}
{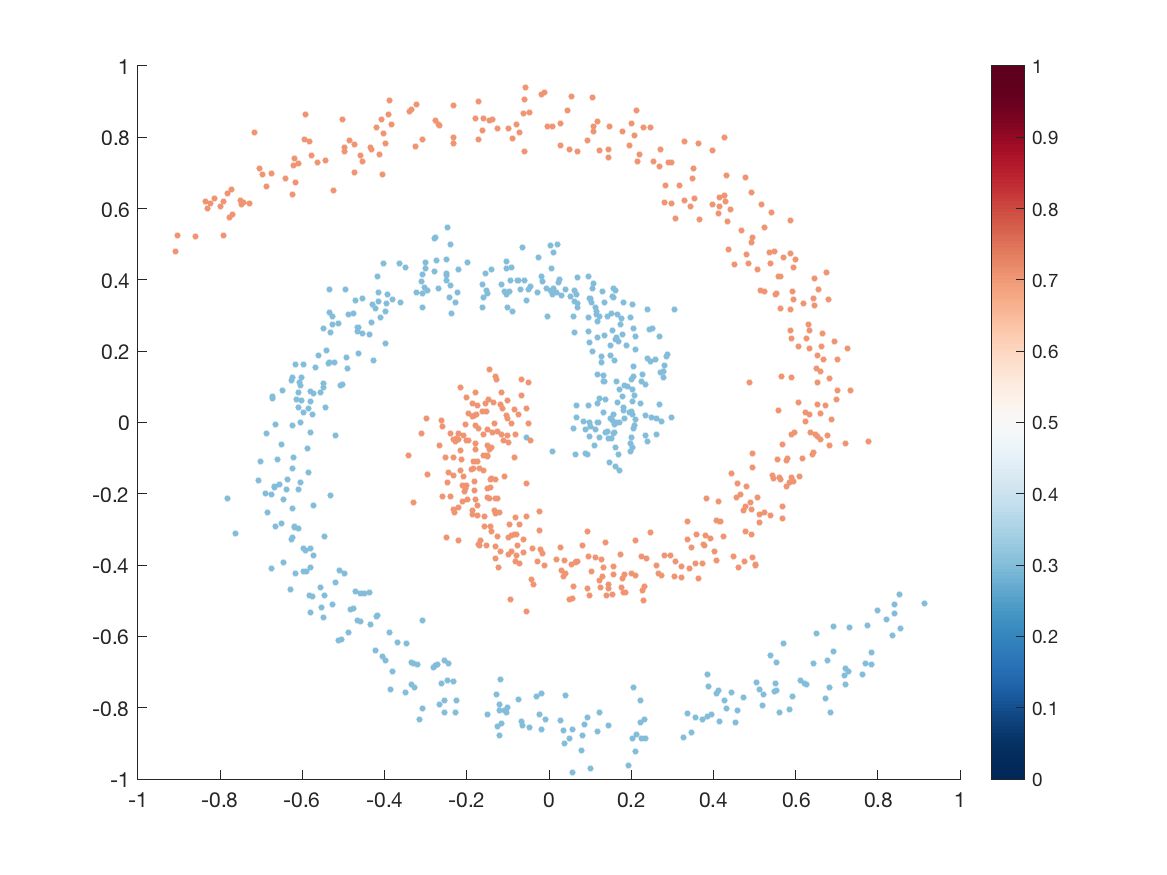}
\PlotLabels{\DataDonutA}{\DataDonutB}{\DataSquares}{\DataSpiral}%
\end{tikzpicture}%
\caption{The four data sets used in the numerical study.}\label{FIG:DATA}
\end{figure}

\begin{figure}
\begin{tikzpicture}
\PlotRow{0}{prediction}
{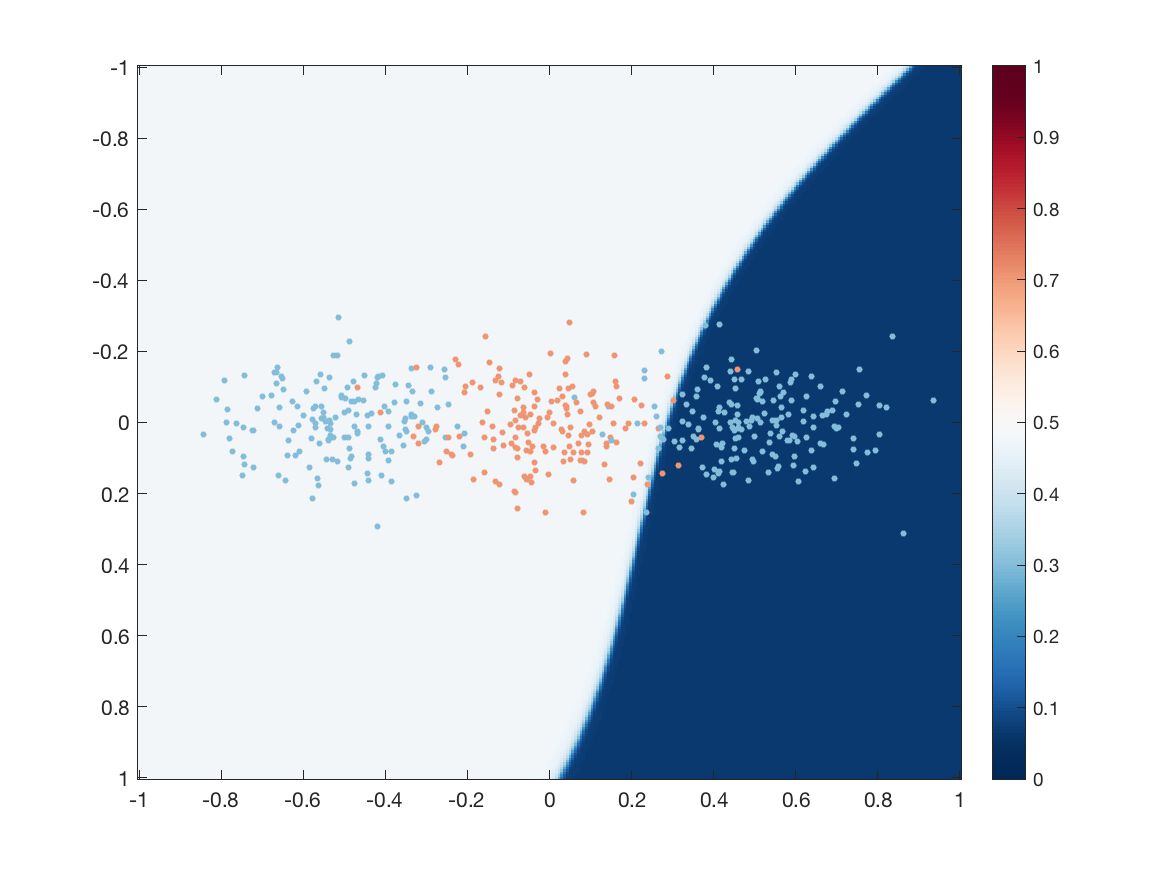}
{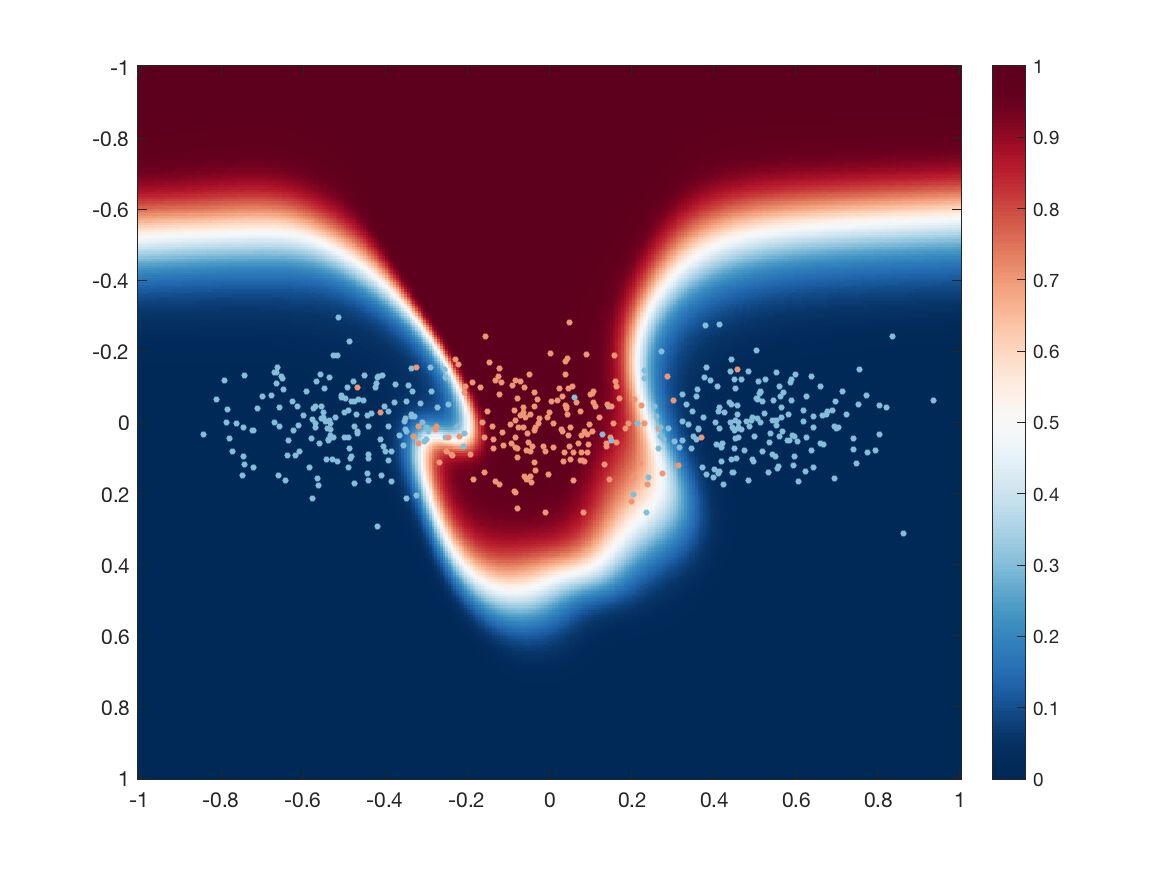}
{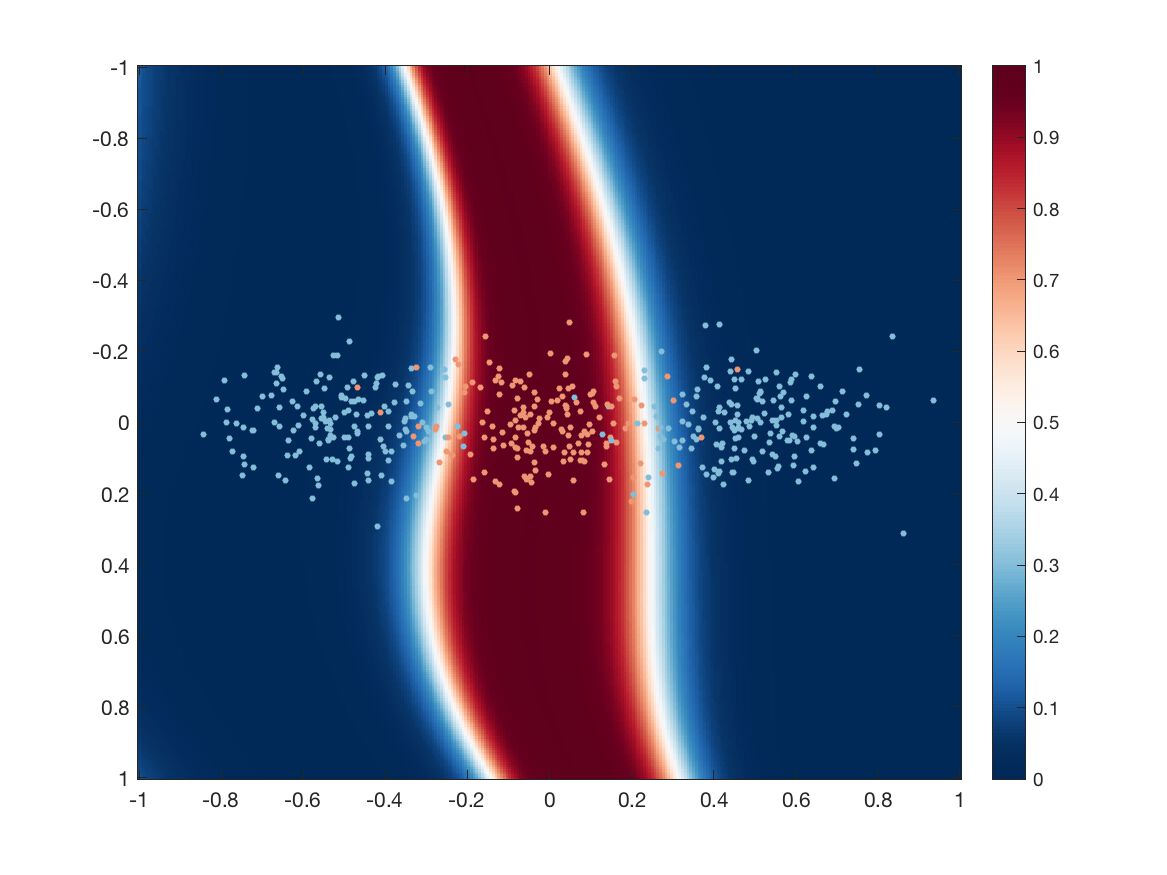}
{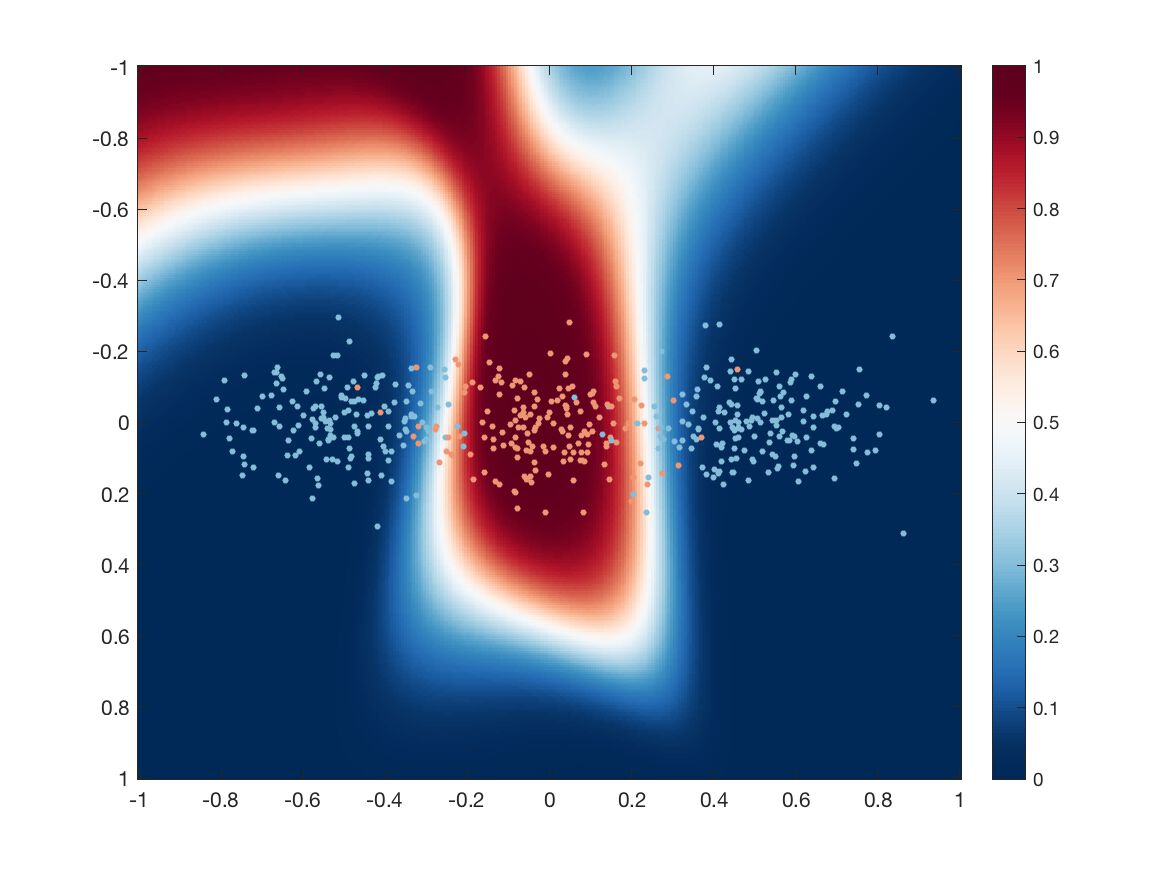}
\PlotRow{1}{transformation}
{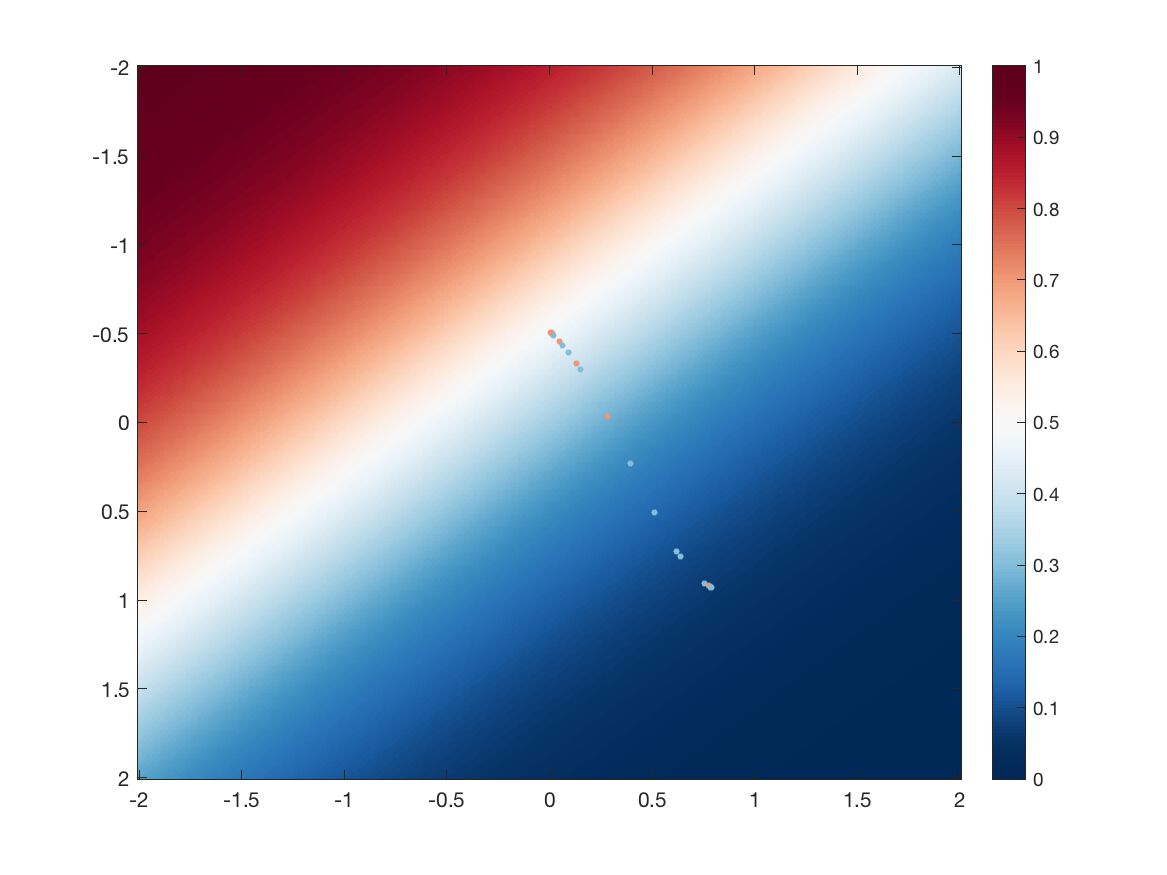}
{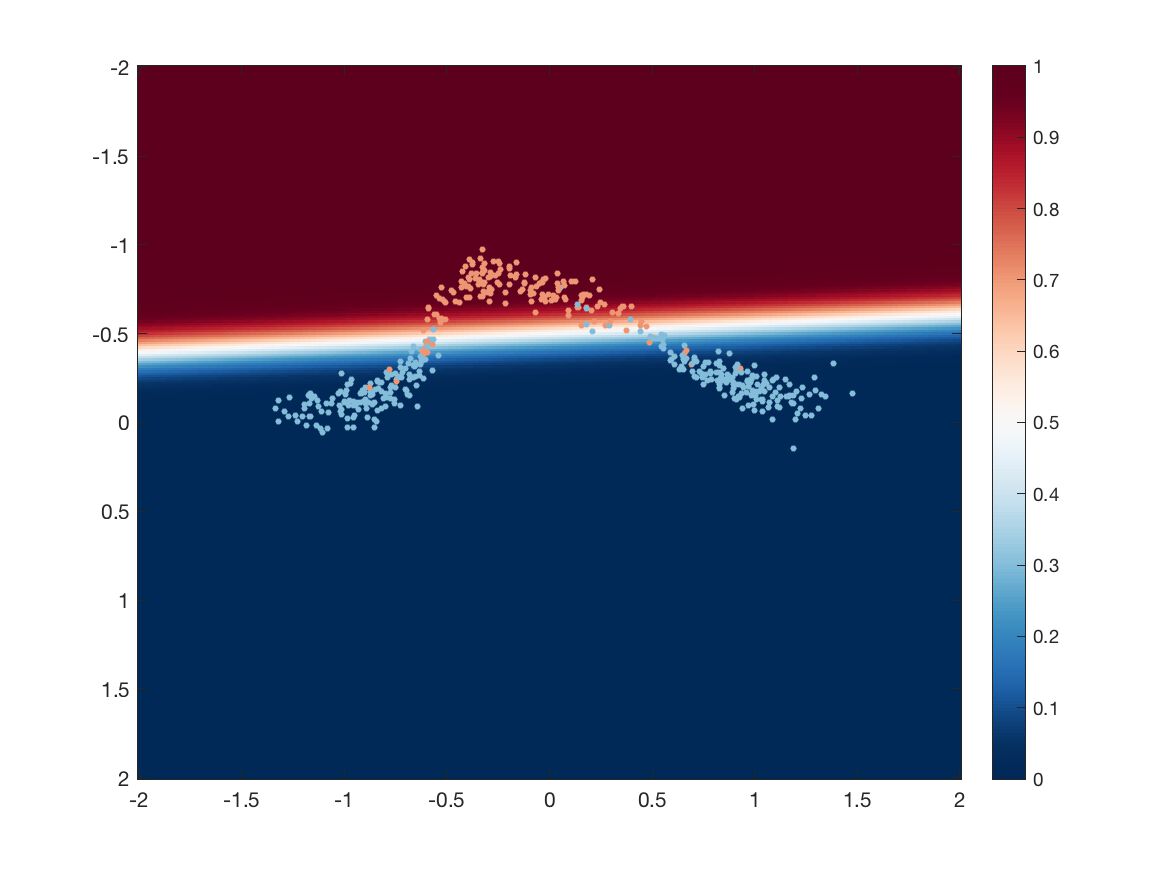}
{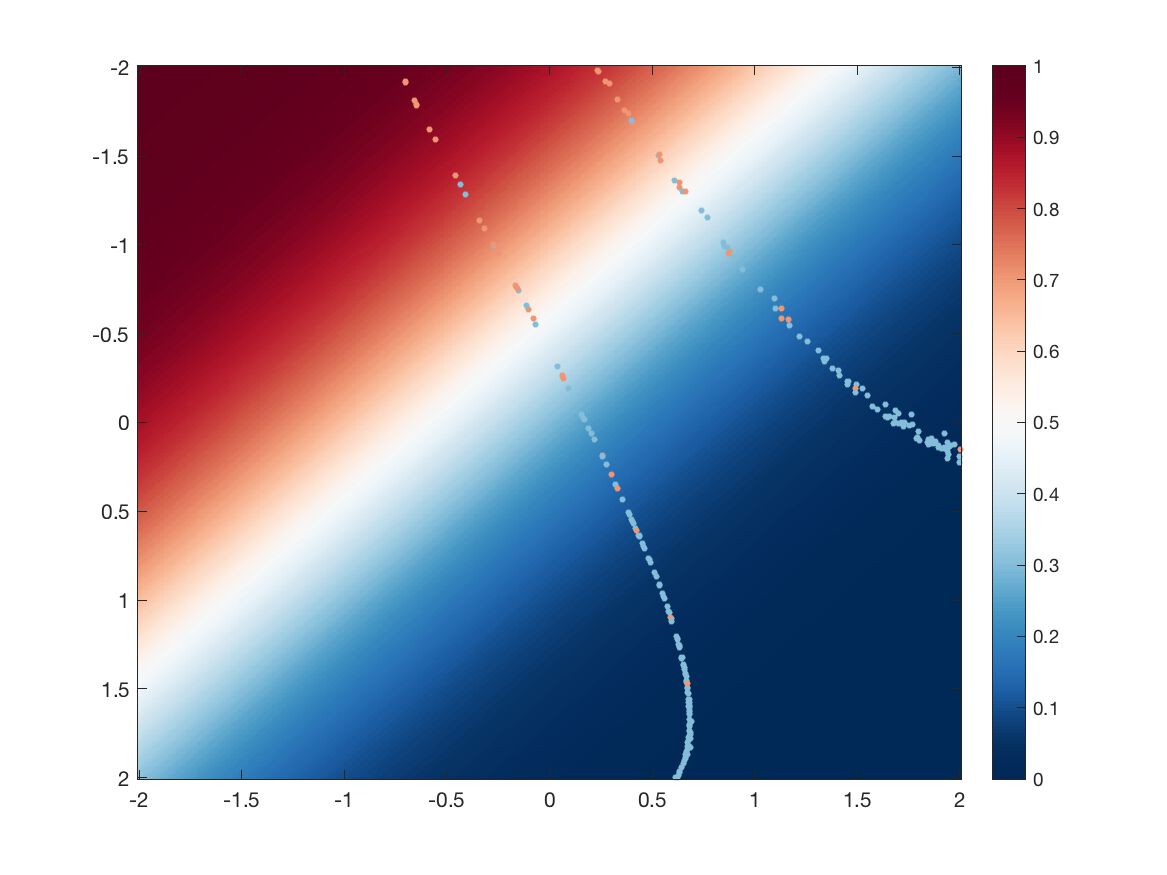}
{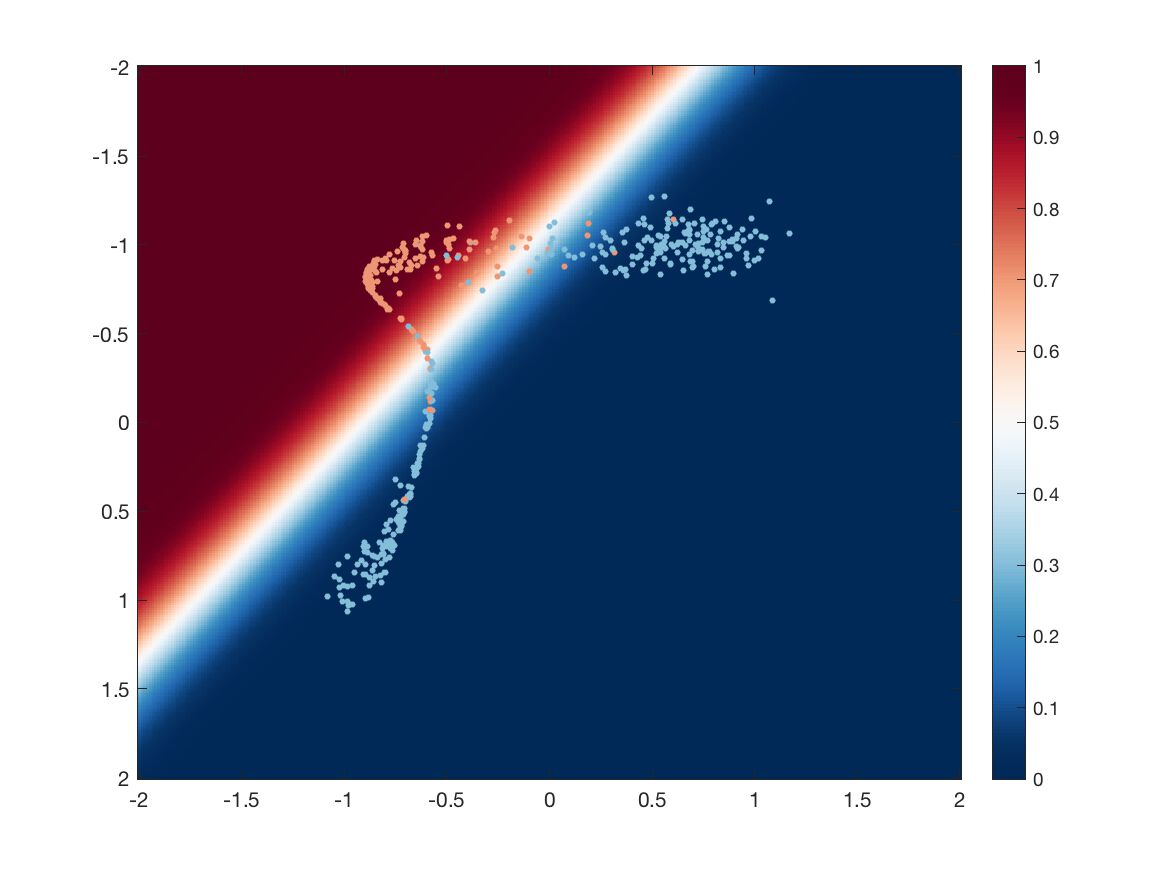}
\PlotLabels{Net}{ResNet}{ODENet}{ODENet+Simplex}
\end{tikzpicture}

\vspace*{2mm}

\begin{tikzpicture}
\PlotRow{0}{prediction}
{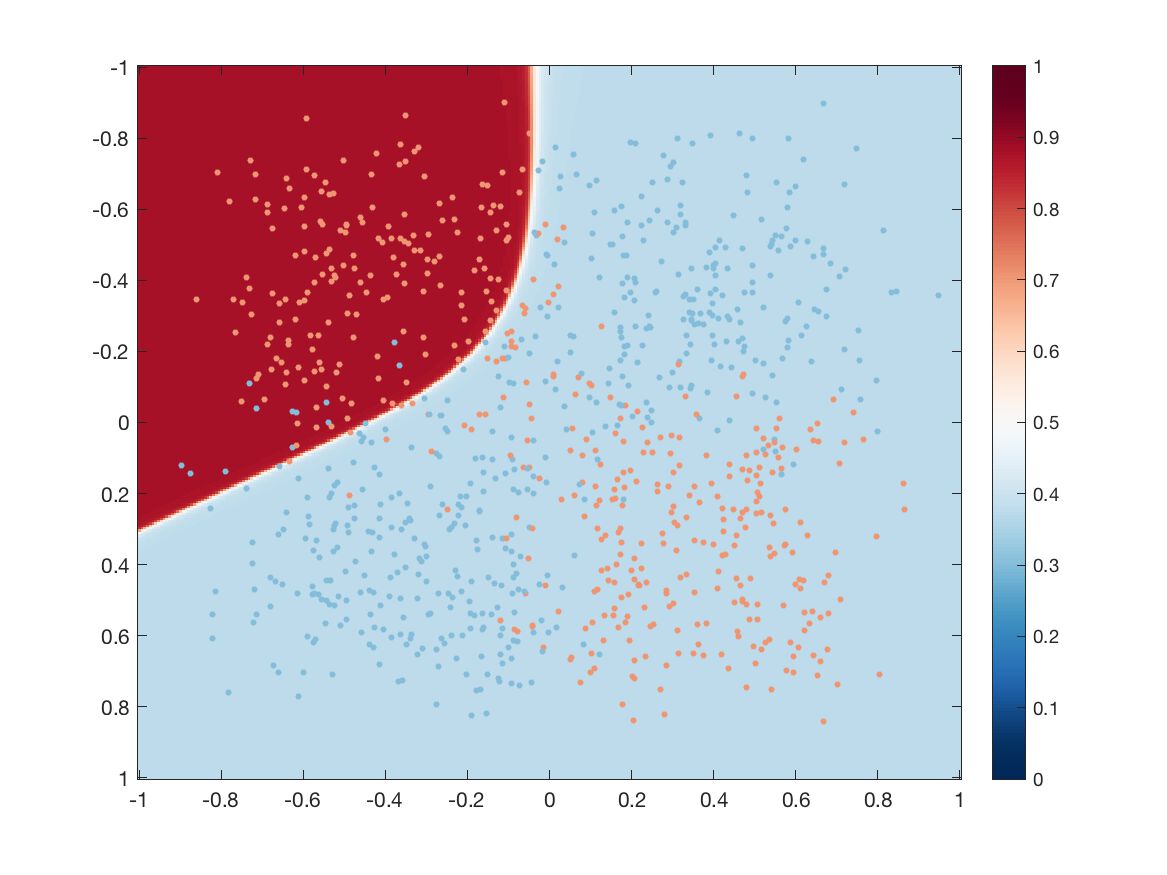}
{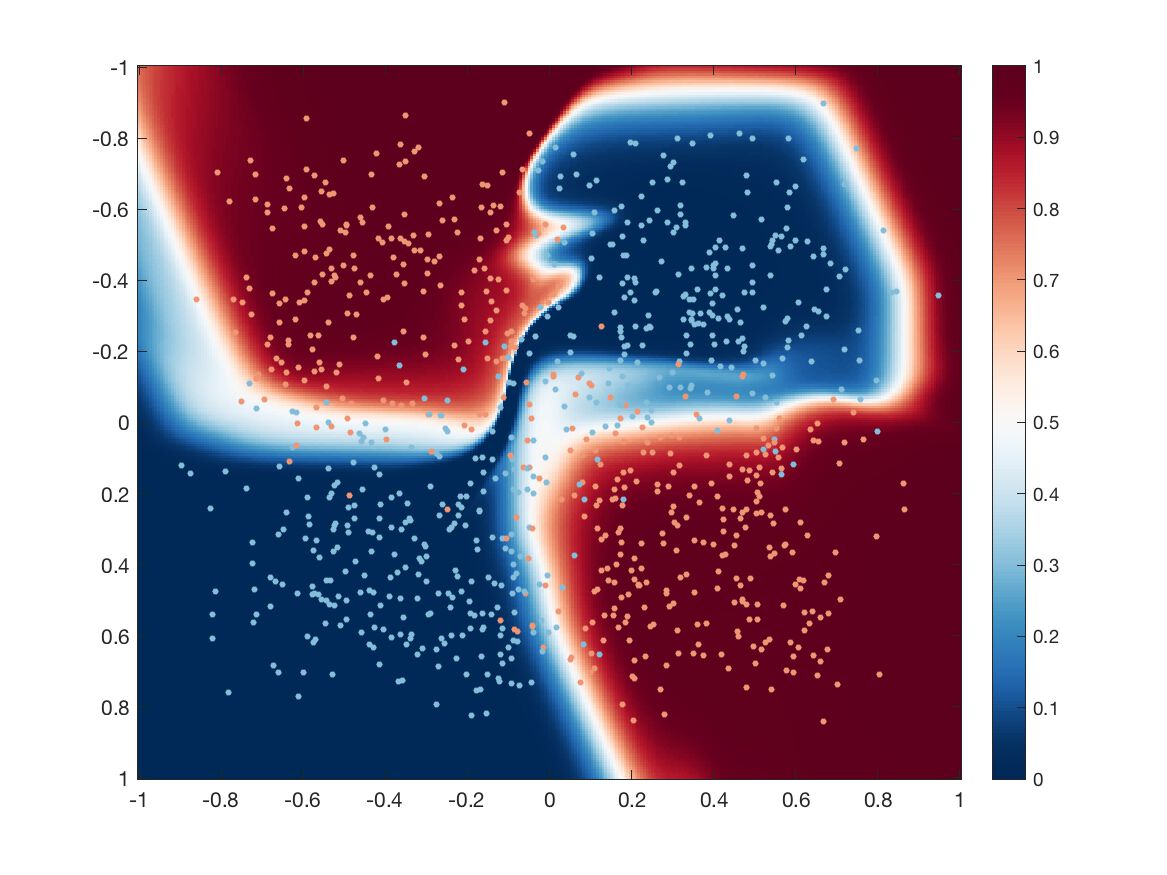}
{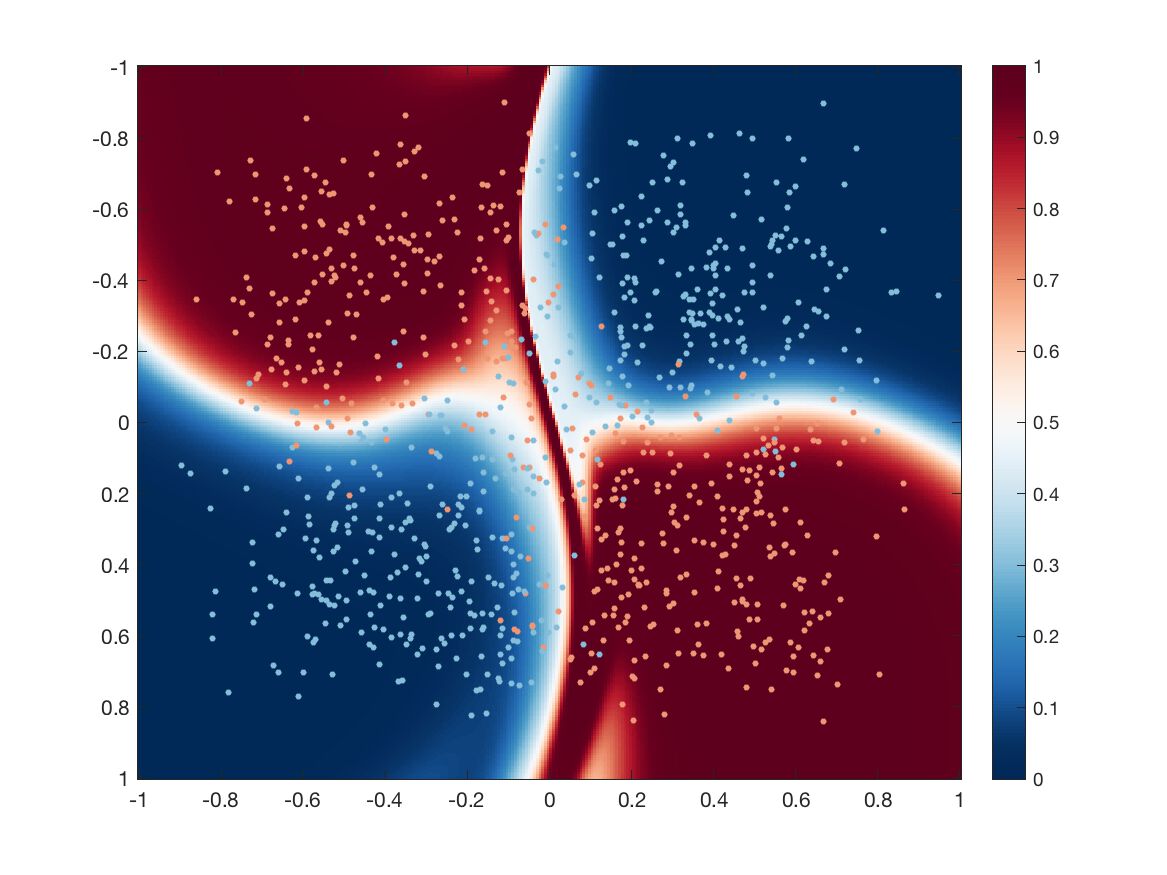}
{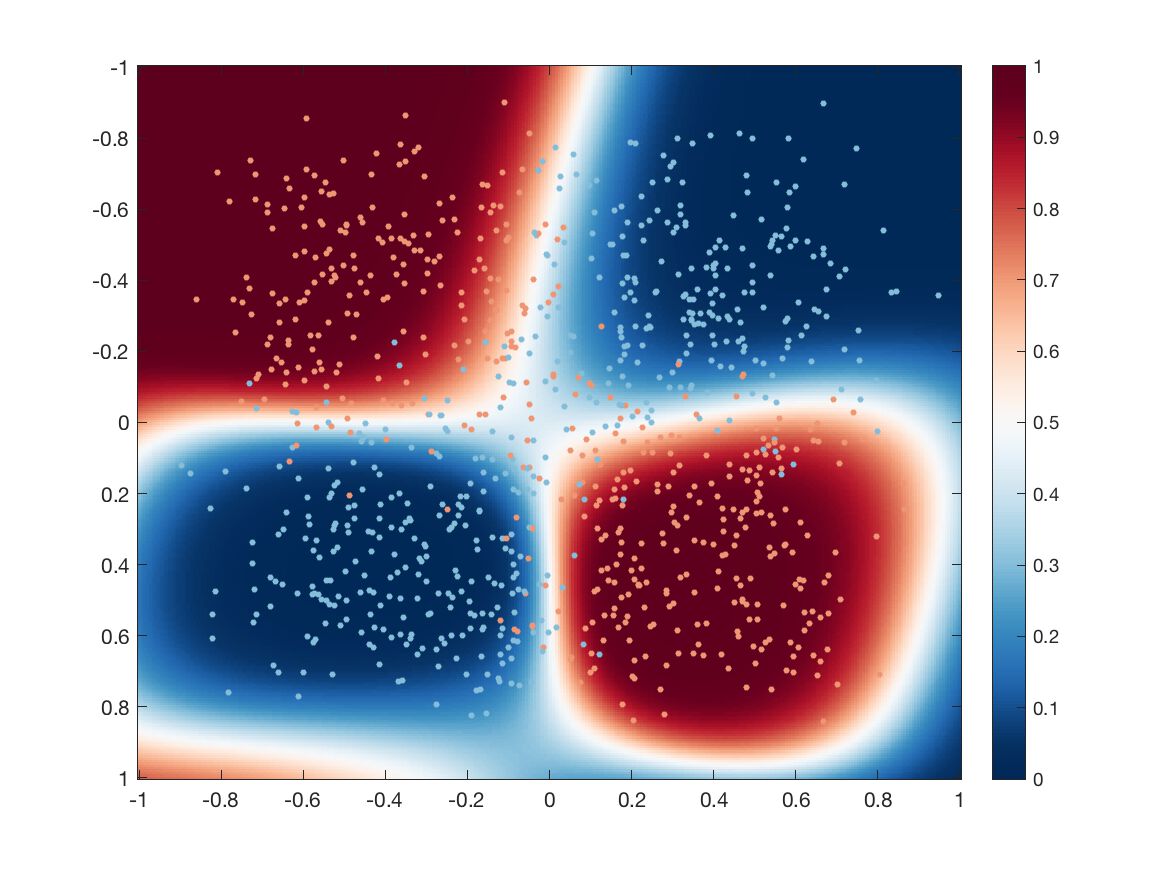}
\PlotRow{1}{transformed}
{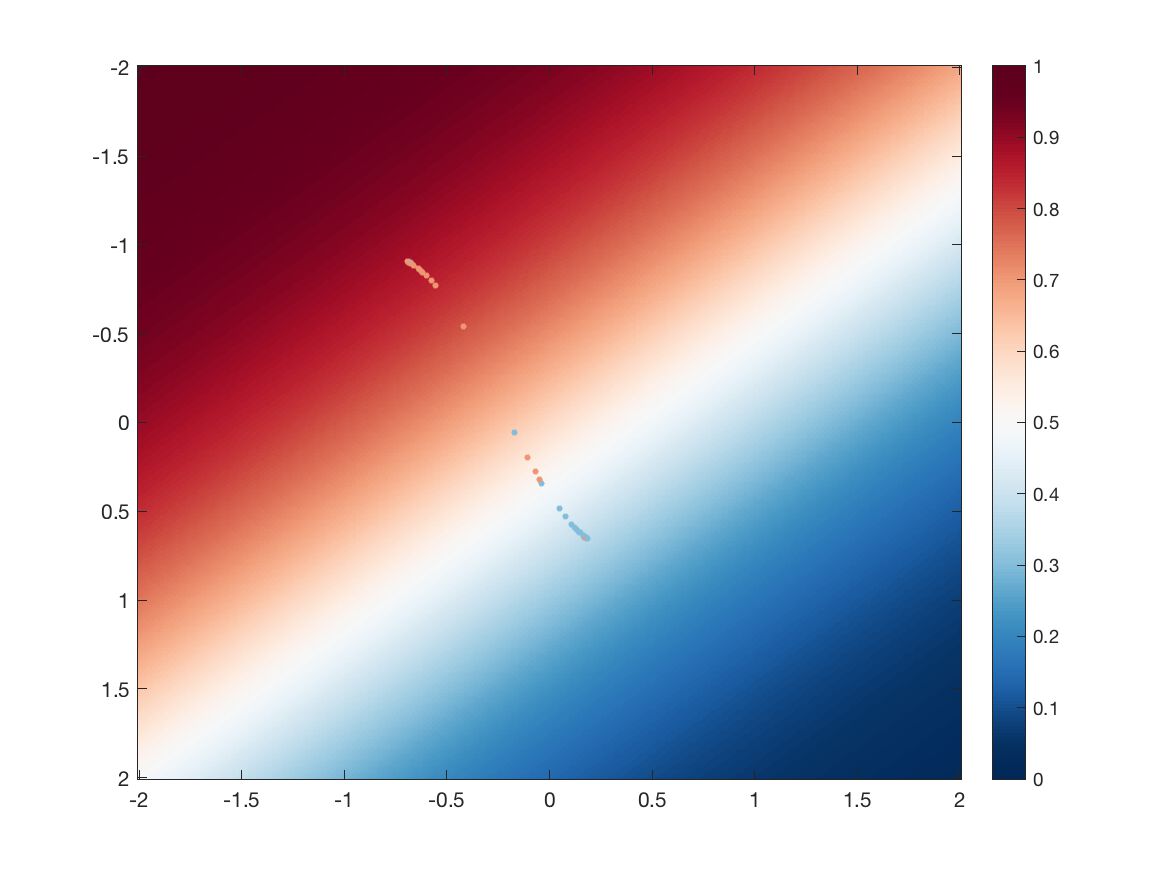}
{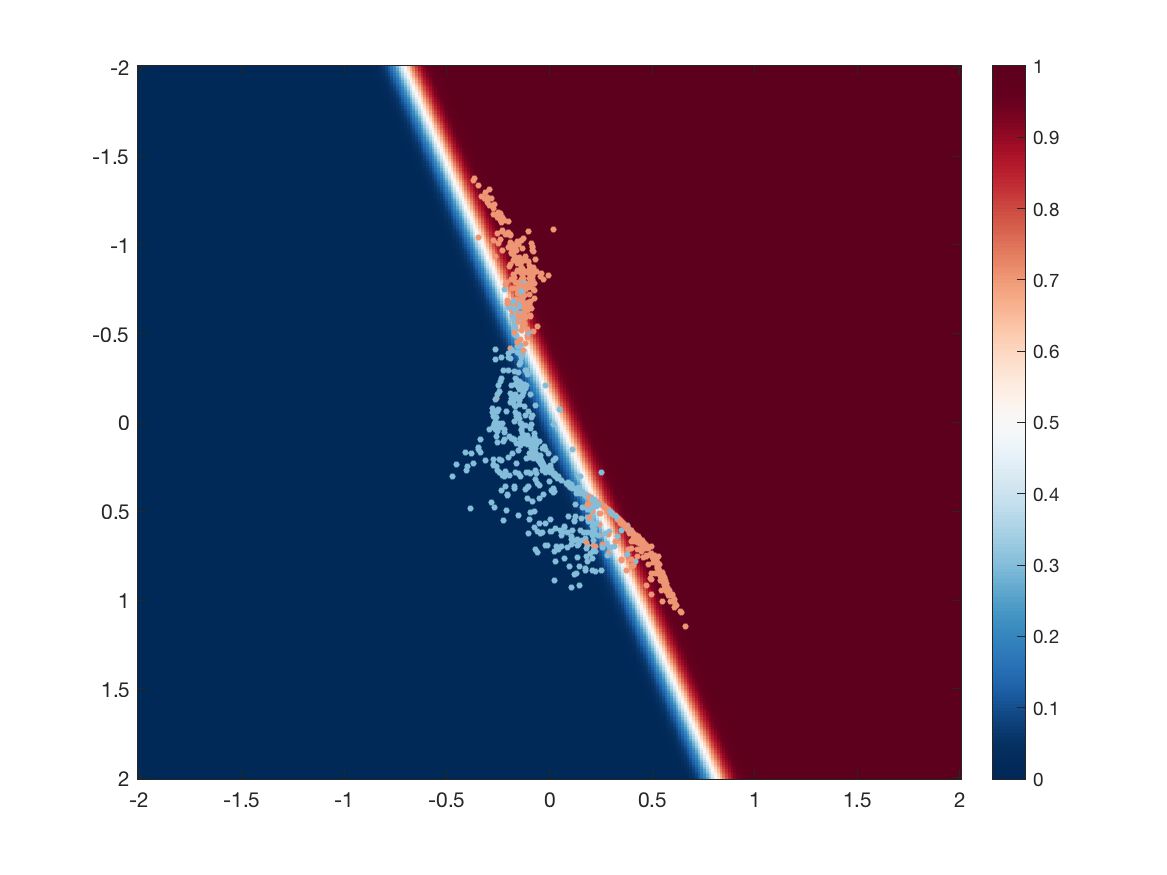}
{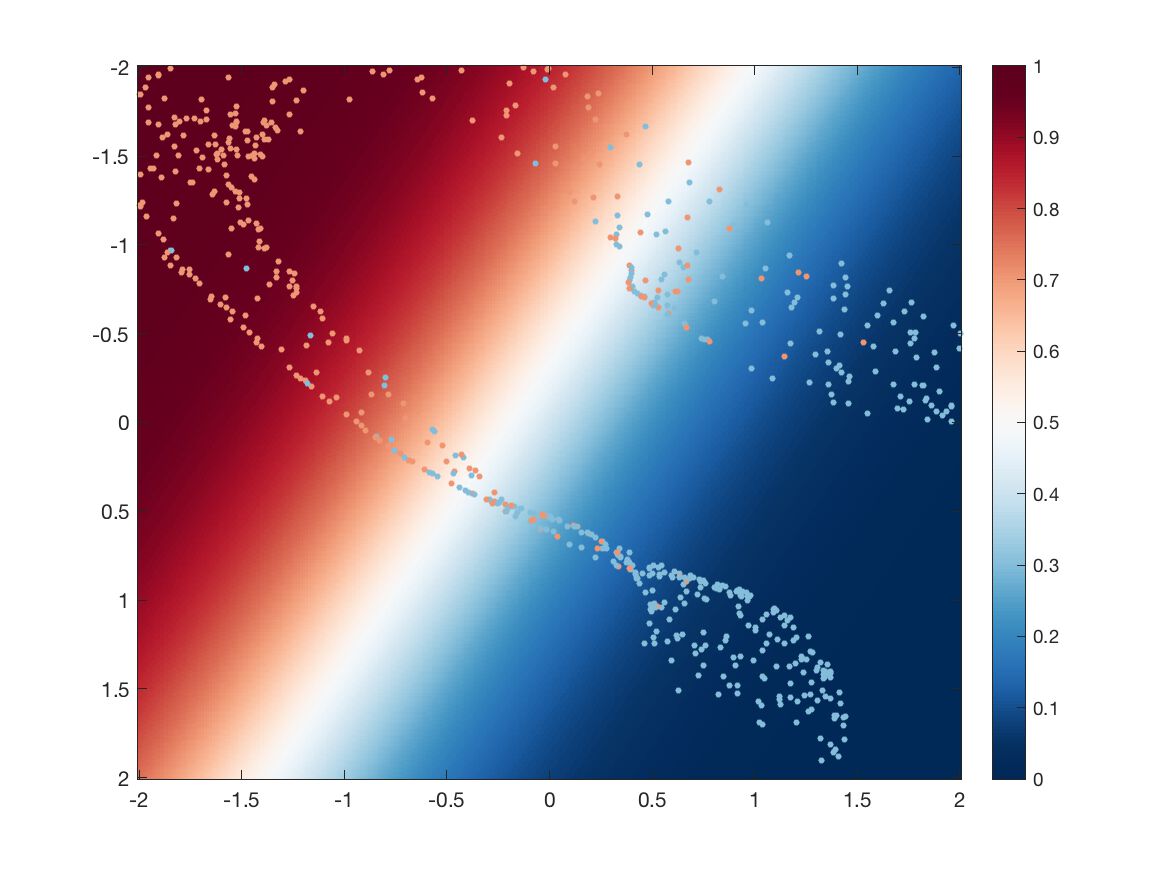}
{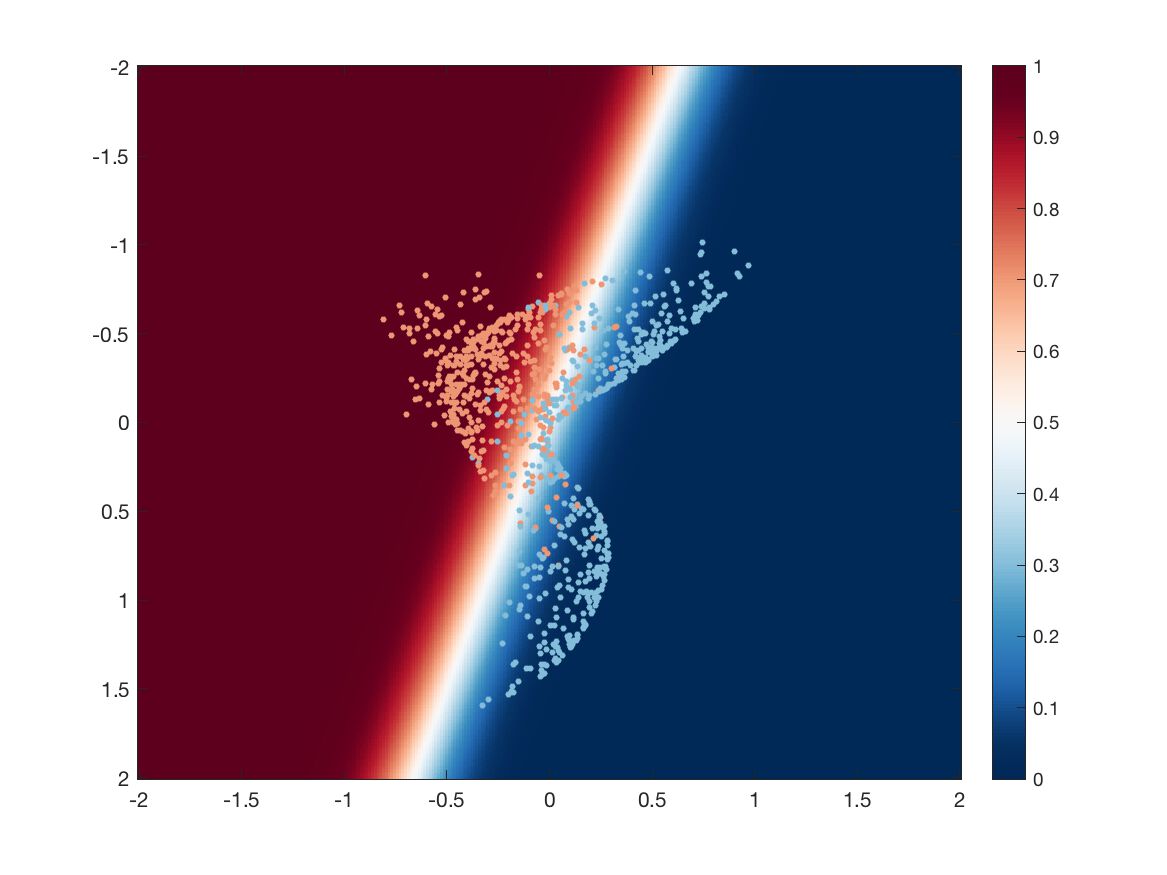}
\PlotLabels{Net}{ResNet}{ODENet}{ODENet+Simplex}
\end{tikzpicture}
\caption{Learned transformation and classifier for data set \DataDonutA~(top) and \DataSquares~(bottom).} \label{FIG:M:QUAL}
\end{figure}

\subsection{Setting, Training and Data sets}%
Throughout the numerical experiments we use labels $c_i \in \{0, 1\}$ and make use of the link function $\sigma(x) = \tanh(x)$ and hypothesis function $\mathcal H(x) = 1/(1 + \exp(-x))$. For all experiments we use two channels ($n=2$) but vary the number of layers $N$\footnote{In this paper we make the deliberate choice of keeping the number of dimensions equal to the dimension of the original data samples rather than augmenting or doubling the number of dimensions as proposed in \cite{gholami19aua} or \cite{haber2017stable}. Numerical experiments after augmenting the dimension of the ODE (not reported here) led to improved performance for all the considered architectures.}.

In all numerical experiments we use gradient descent with backtracking, see Algorithm \ref{ALG:GD2}, to train the network (estimate the controls). The algorithm requires the derivatives with respect to the controls which we derived in the previous section. Finally, the gradients with respect to $W$ and $\mu$ of the discrete cost function are required in order to update these parameters with gradient descent, which read as
\begin{subequations}
\begin{eqnarray}
\label{gradientW}
\gamma_i &=& \left(\mathcal{C}(W y_i^{[N]}+\mu)-c_i\right)\odot \mathcal{C}'(Wy_i^{[N]}+\mu) \\
\partial_W\, \mathcal{J}(y_i^{[N]},W,\mu)&=& \gamma_i y_i^{[N], T},\\
\label{gradientmu}
\partial_{\mu}\, \mathcal{J}(y_i^{[N]},W,\mu)&=& \gamma_i. 
\end{eqnarray} \label{gradientclassifier}
\end{subequations}

We consider 4 different data sets (\DataDonutA, \DataDonutB, \DataSquares, \DataSpiral) that have different topological properties, which are illustrated in Figure~\ref{FIG:DATA}. These are samples from a random variable with prescribed probability density functions. We use 500 samples for data set \DataDonutA~and each 1,000 for the other three data sets. For simplicity we chose not to use explicit regularisation, i.e. $\mathcal R = 0$, in all numerical examples. Code to reproduce the numerical experiments is available on the University of Cambridge repository under \url{https://doi.org/10.17863/CAM.43231}.

\subsection{Comparison of Optimal Control Inspired Methods}%
We start by comparing qualitative and quantitative properties of four different methods. These are: 1) the standard neural network approach (\eqref{eq:feedforward} with \eqref{eq:deeplearning:f}), 2) the ResNet (\eqref{Dconstraint} with \eqref{eq:deeplearning:f}), 3) the ODENet (\eqref{Dconstraint} with \eqref{eq:odenet:f}) and 4) the ODENet with simplex constraint \eqref{eq:simplex} on the varying time steps. Throughout this subsection we consider networks with 20 layers.

\begin{figure}[t]
\renewcommand{\PlotIm}[1]{\includegraphics[width=\PicWidth, clip, trim=6cm 4cm 8cm 4cm]{results_spiral2d_20layers/#1}}%

\begin{tikzpicture}
\PlotRow{0}{Net}
{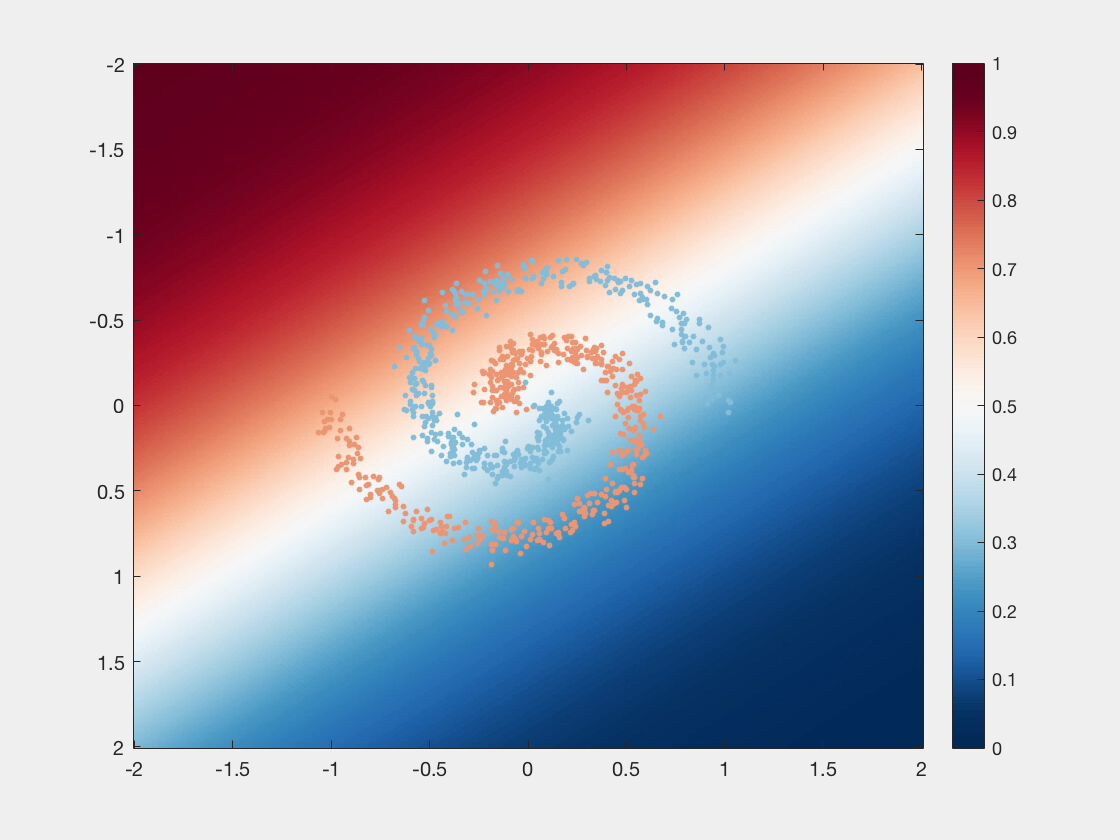}
{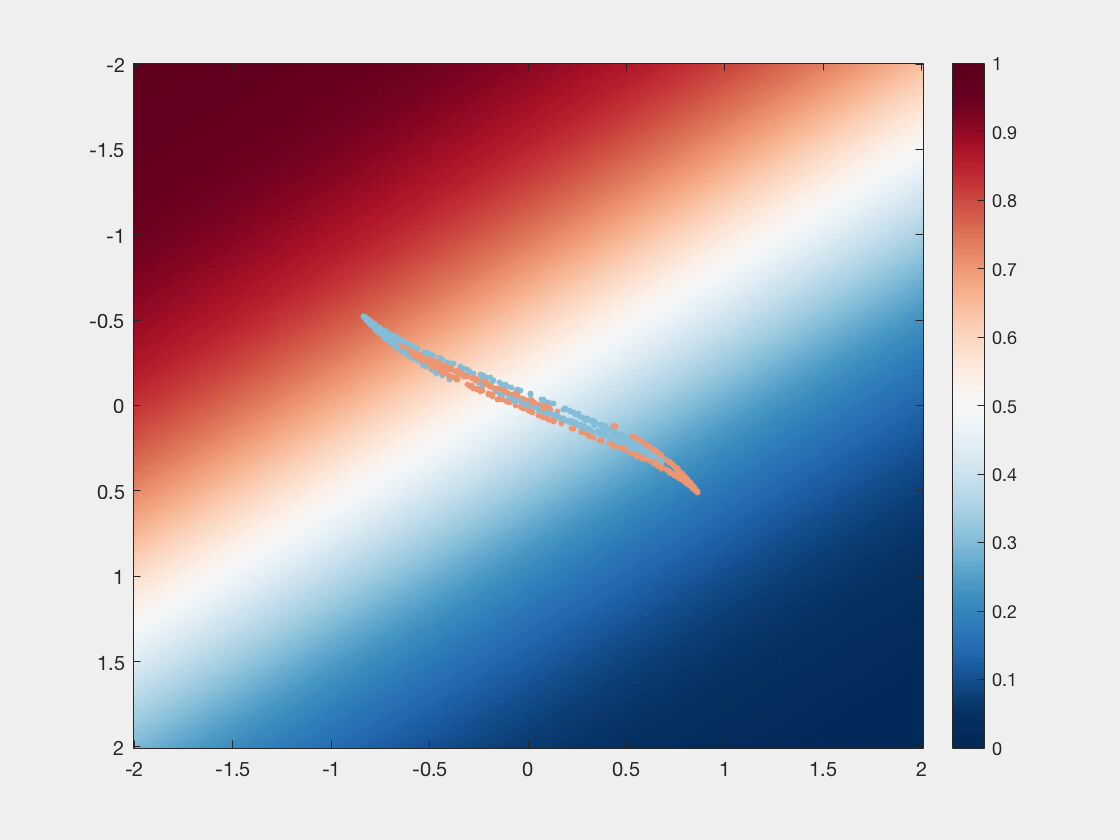}
{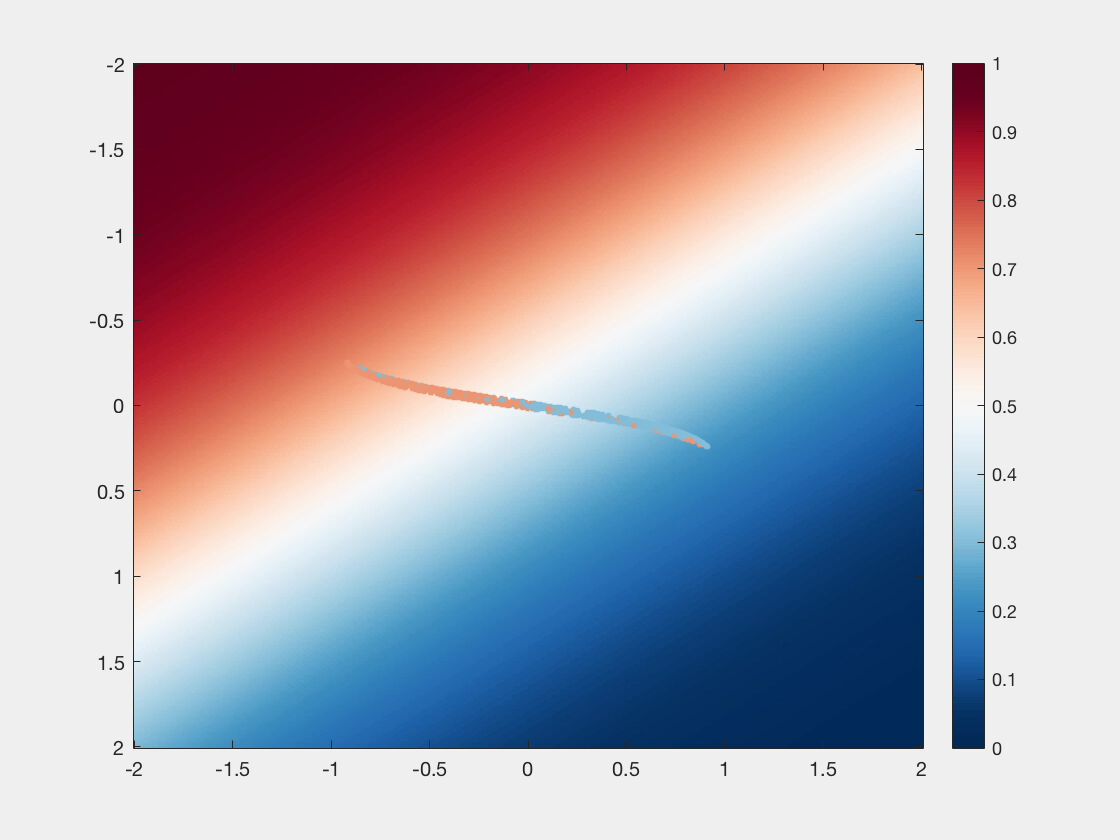}
{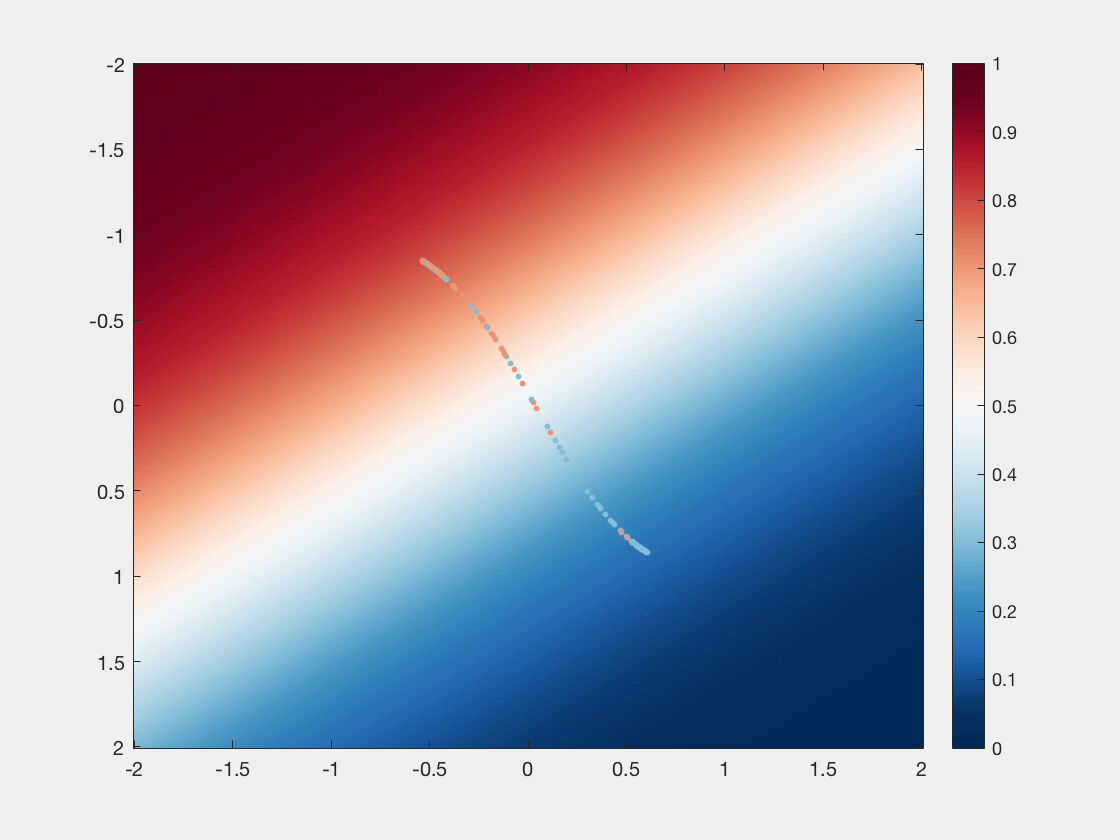}
\PlotRow{1}{ResNet}
{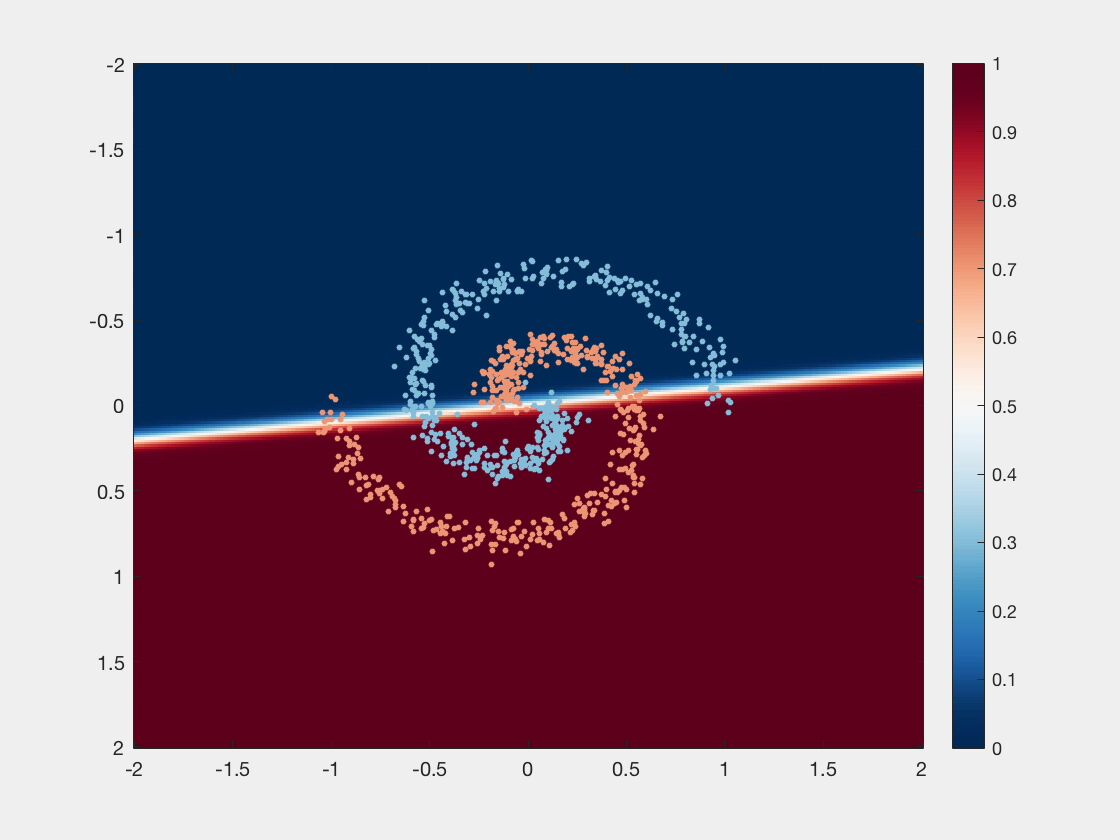}
{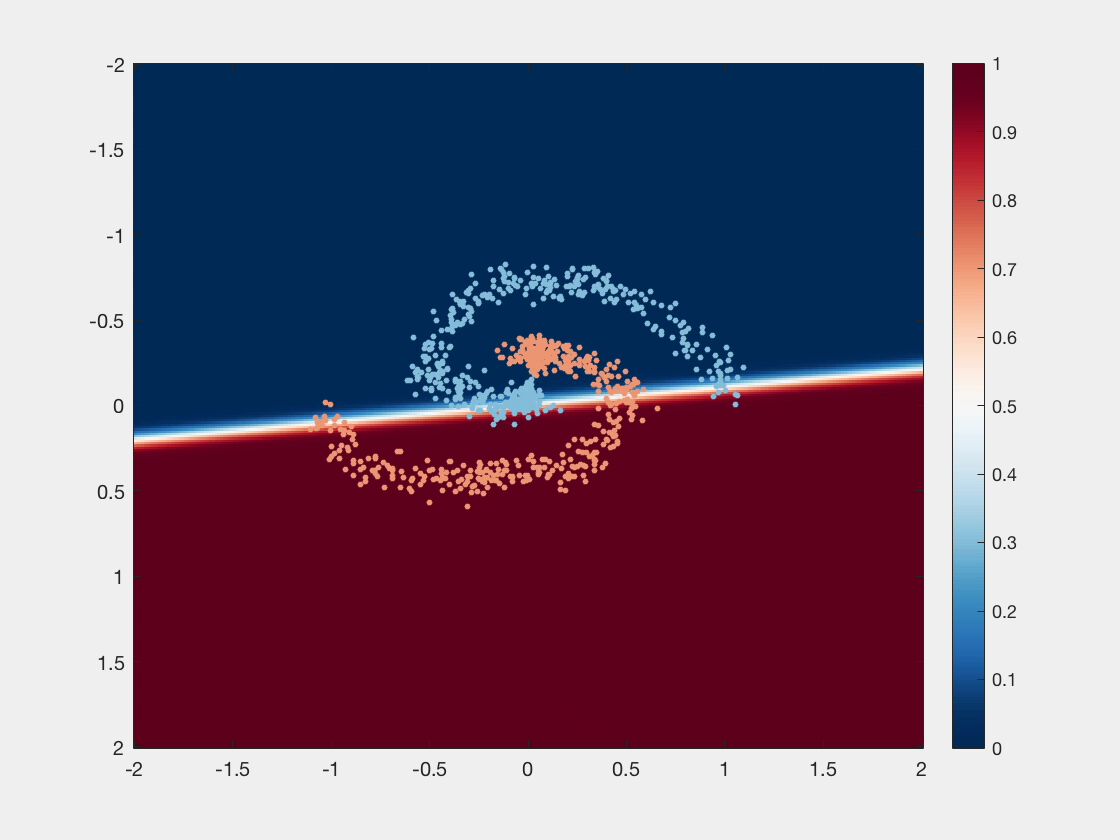}
{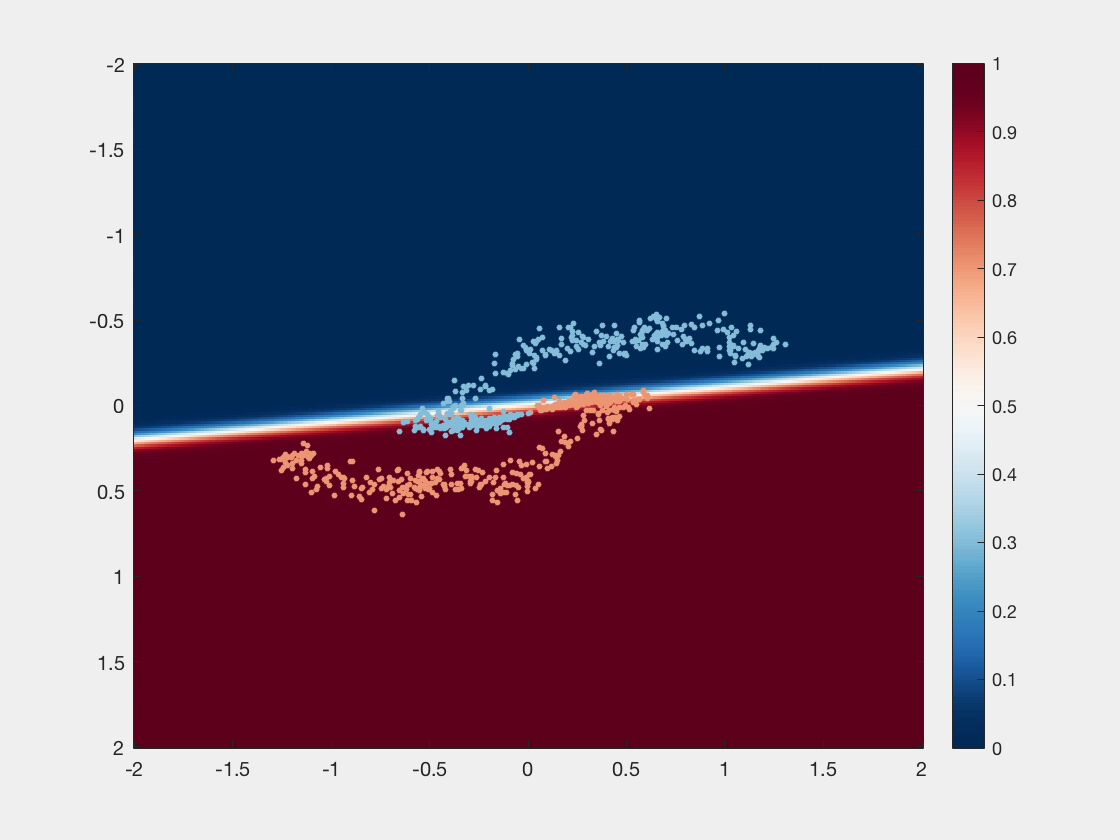}
{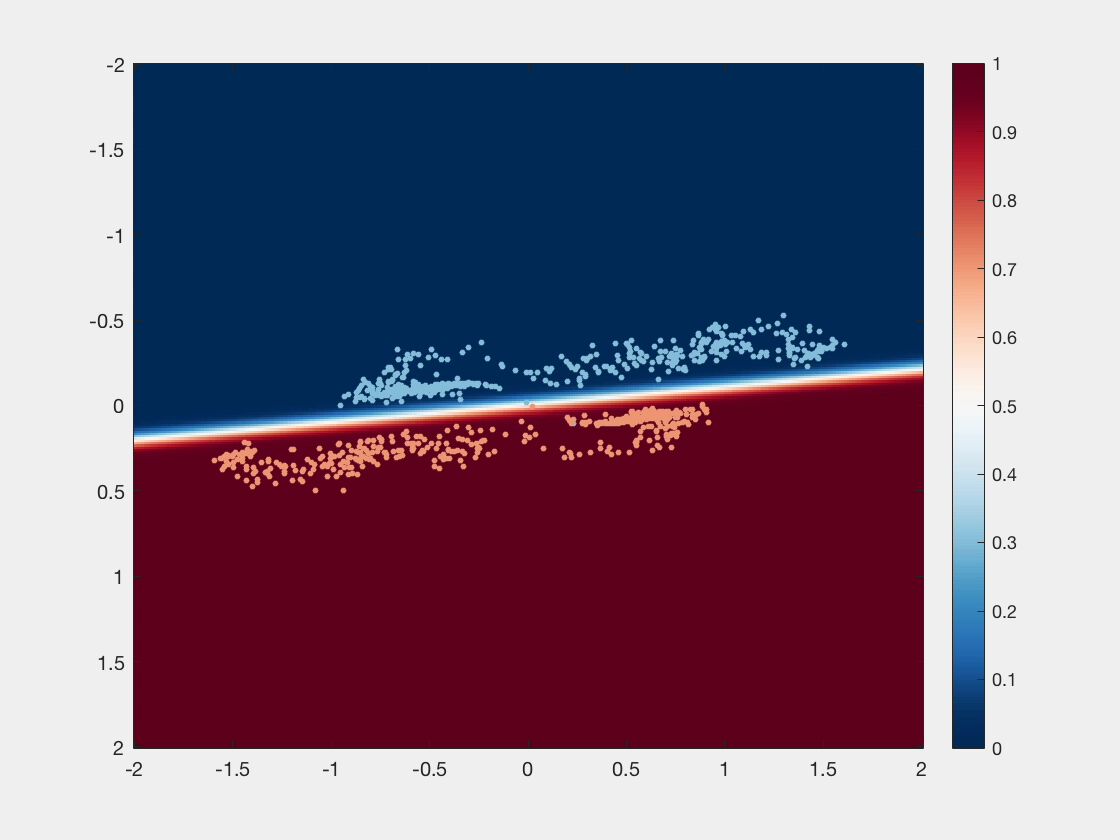}
\PlotRow{2}{ODENet}
{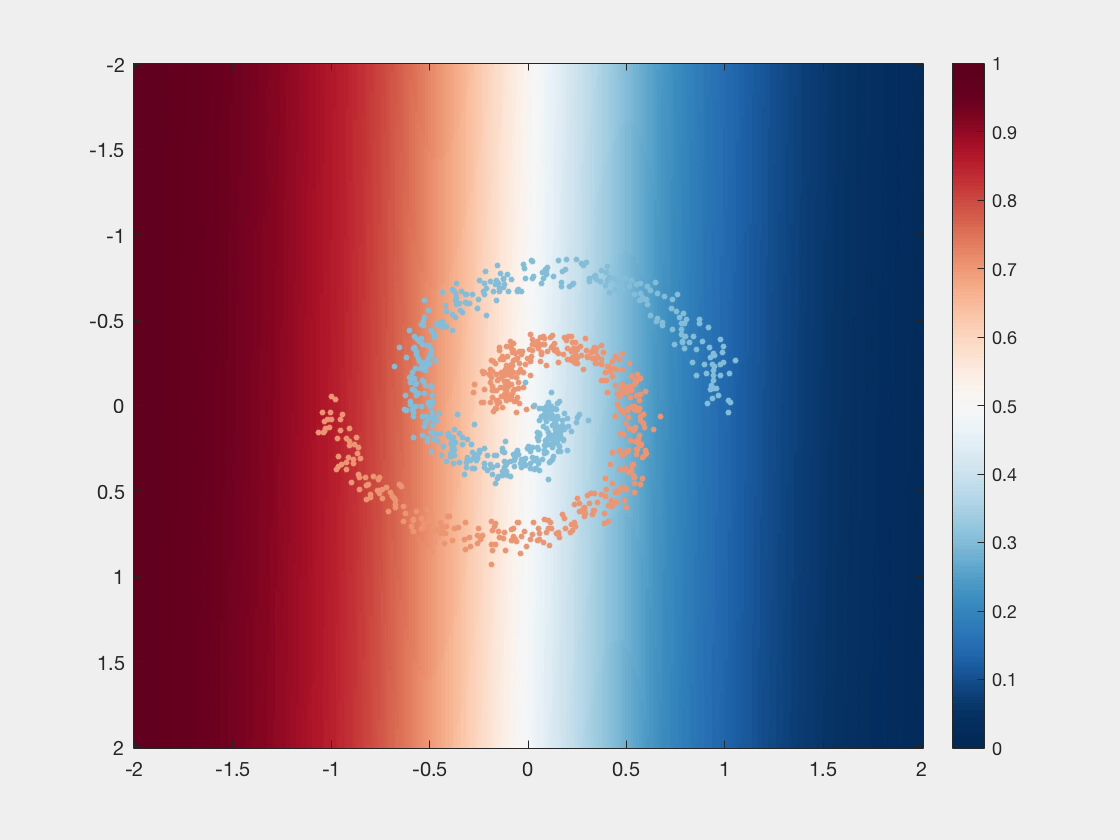}
{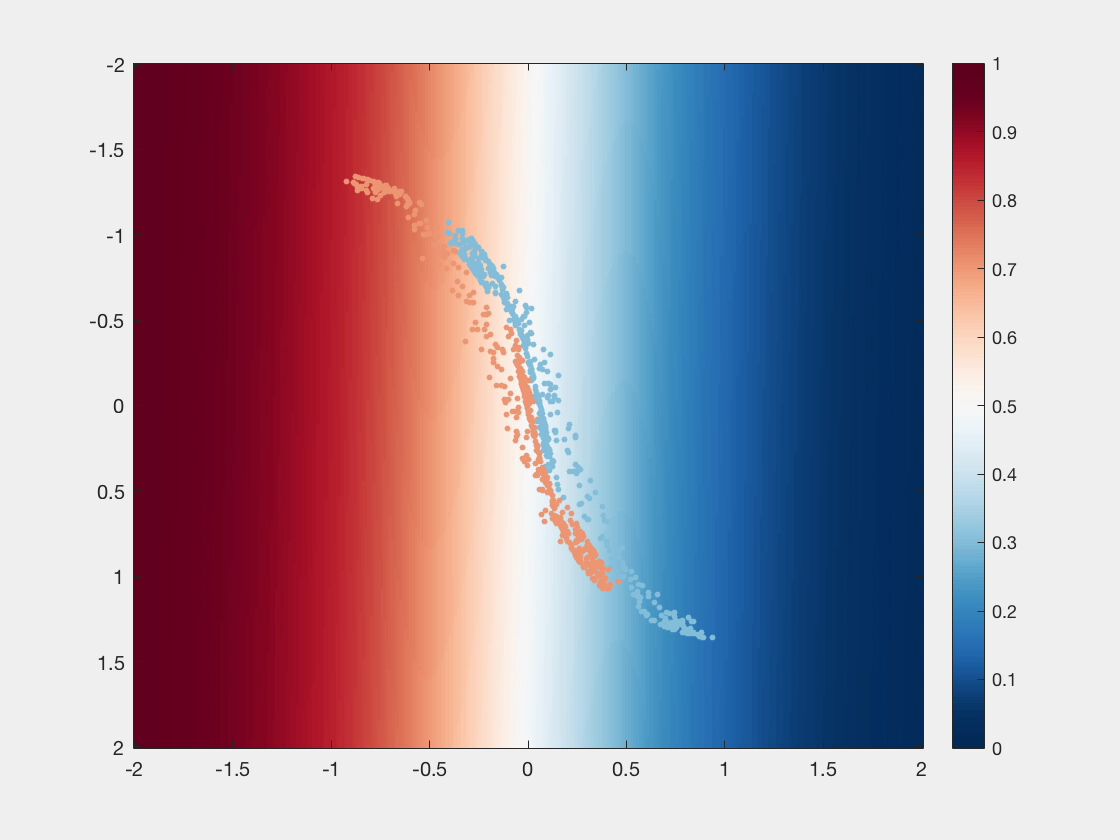}
{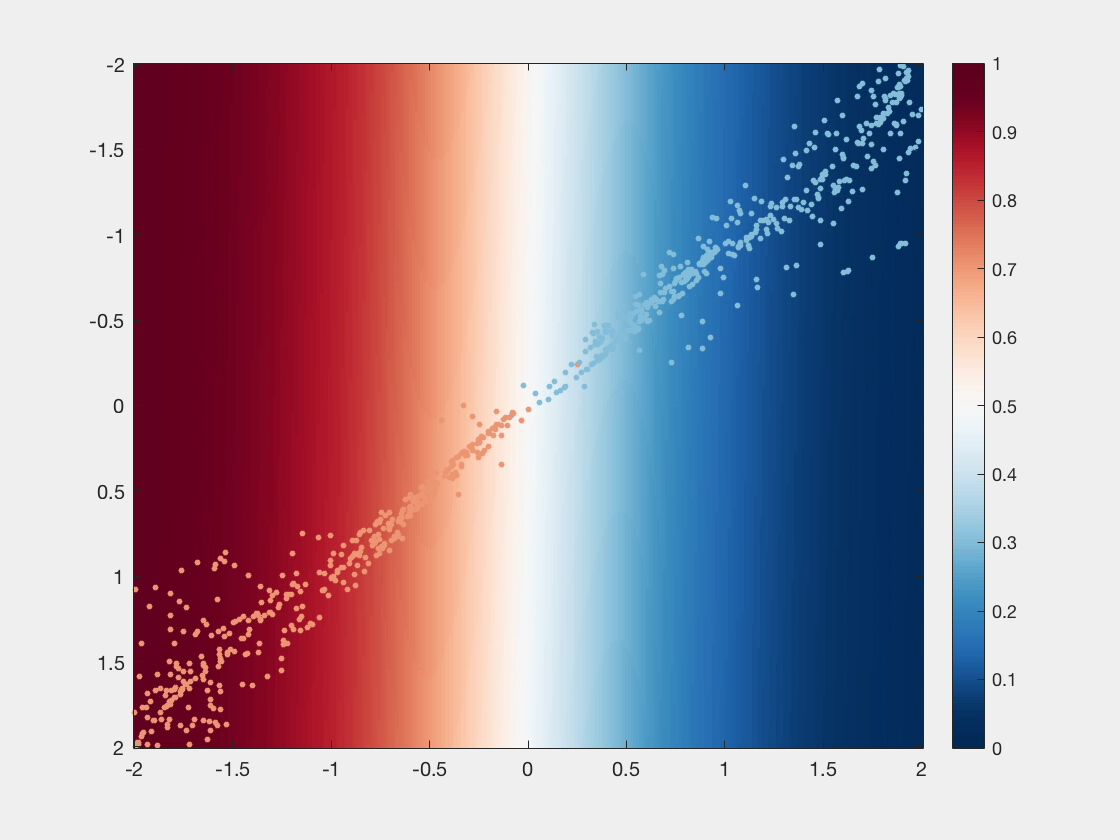}
{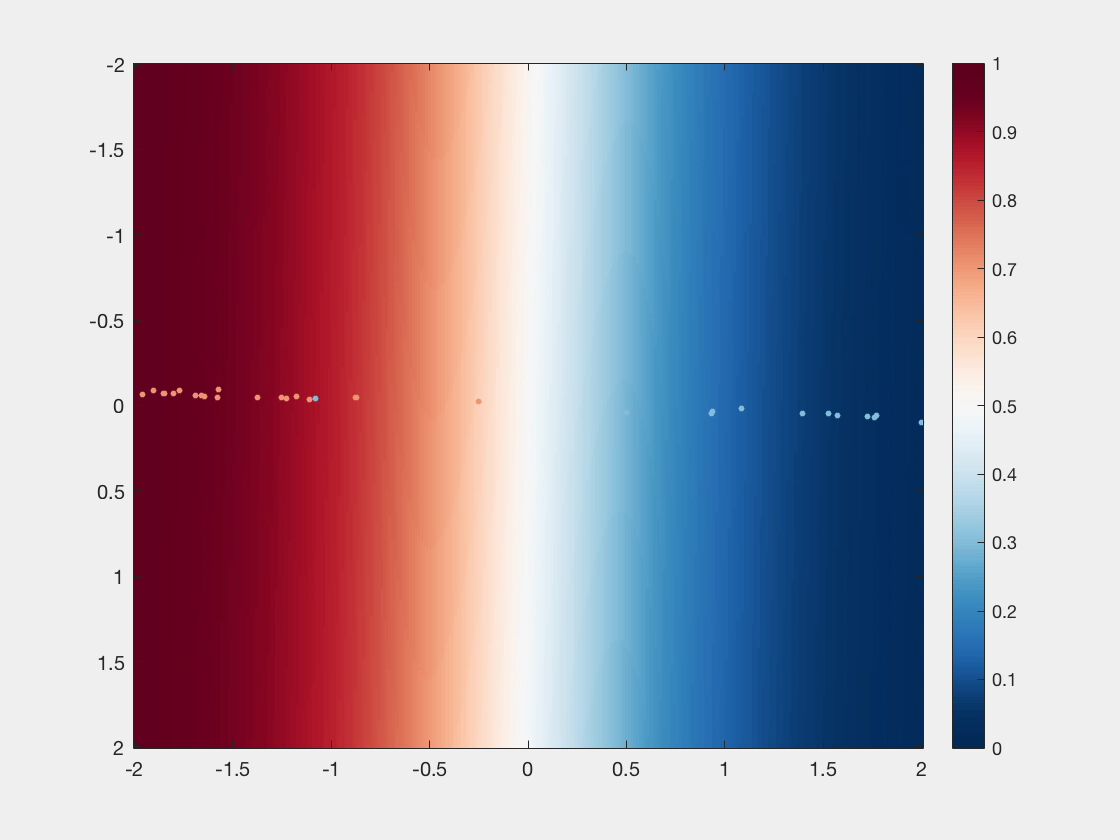}
\PlotRow{3}{ODENet+Simplex}
{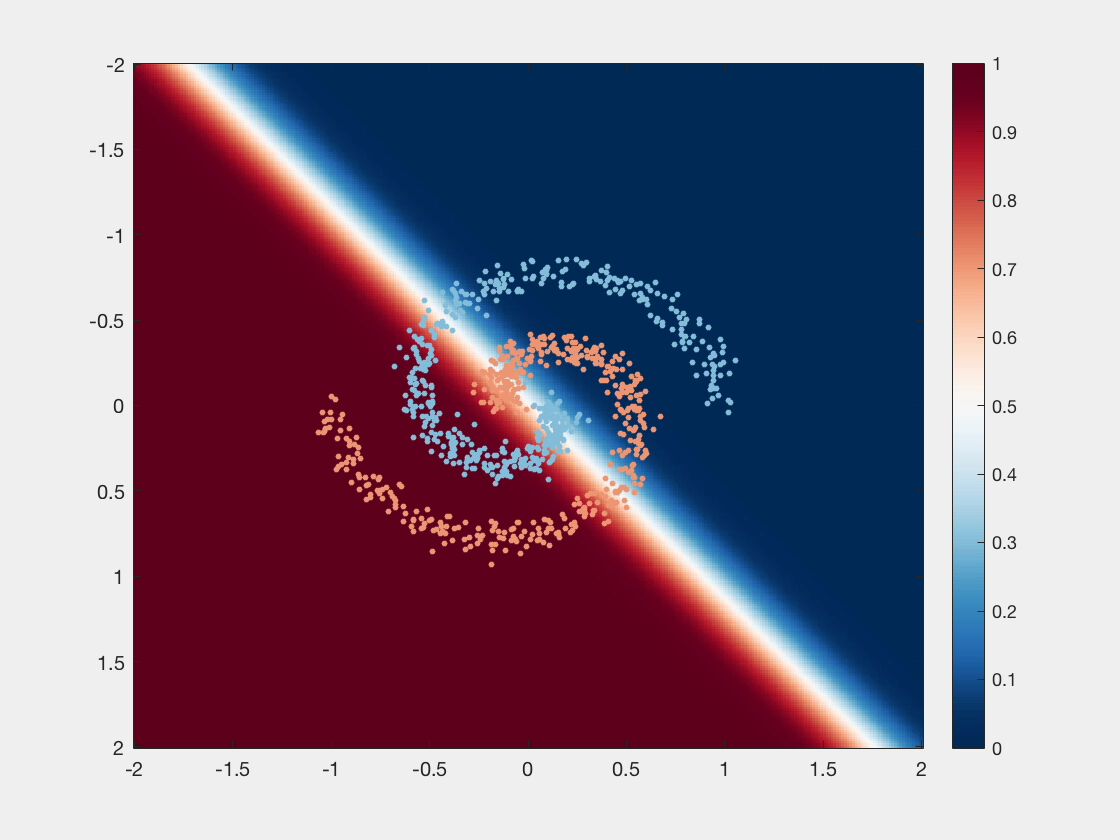}
{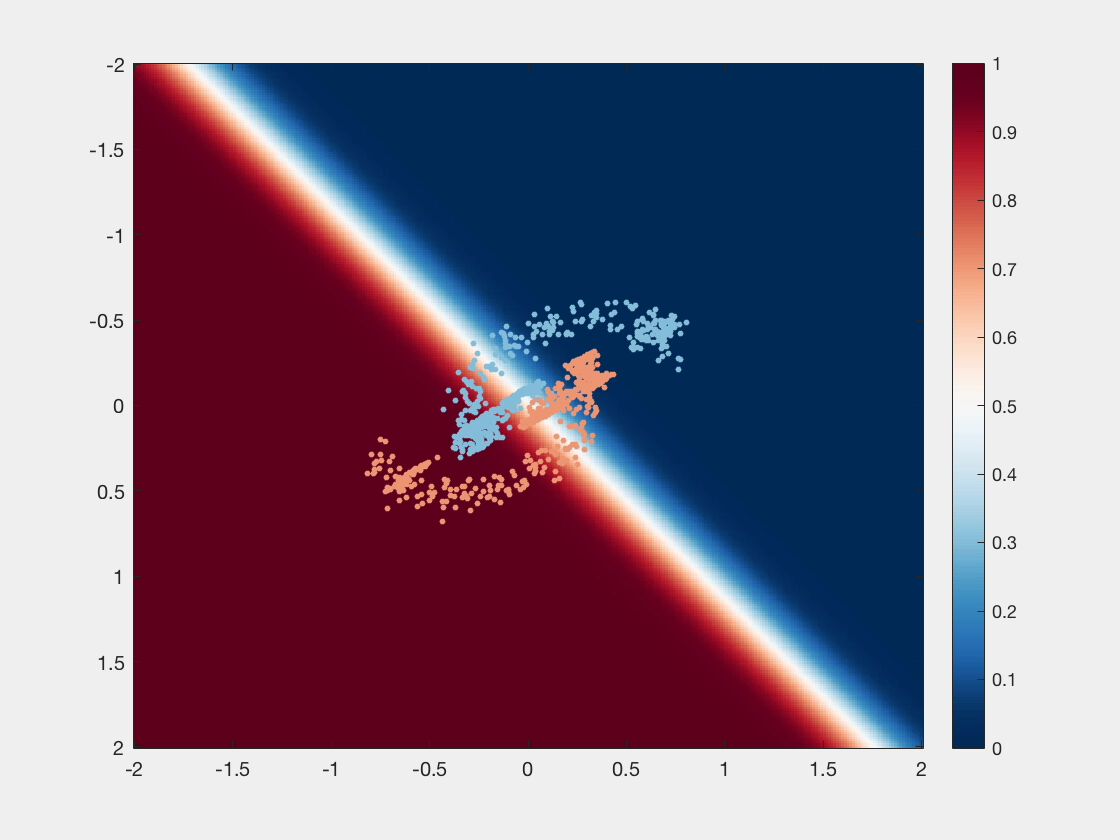}
{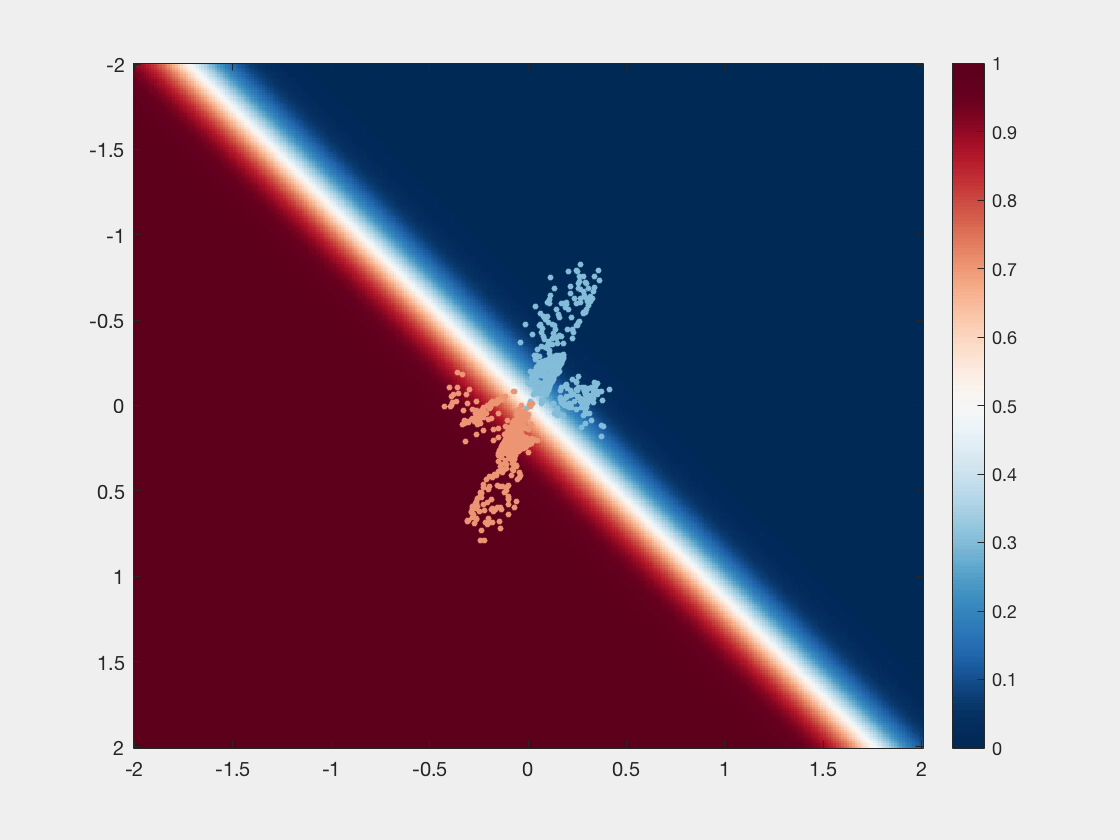}
{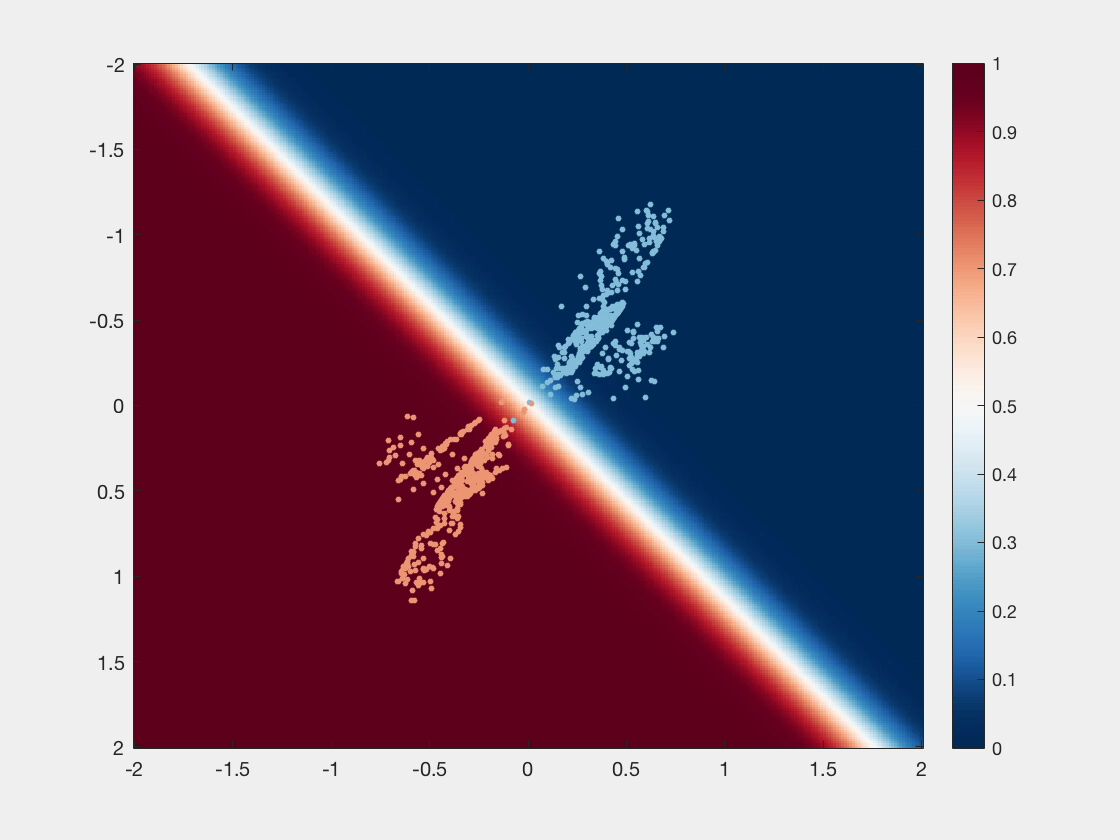}
\PlotLabels{time 0}{}{}{time $T$}
\end{tikzpicture}
\caption{Snap shots of transformation of features for data set \DataSpiral.} \label{FIG:M:TRANSFORMATION}
\end{figure}

\subsubsection{Qualitative Comparison}%
We start with a qualitative comparison of the prediction performance of the four methods on \DataDonutA~and~\DataSpiral, see Figure~\ref{FIG:M:QUAL}. The top rows of both figures show the prediction performance of the learned parameters. The data is plotted as dots in the foreground and the learned classification in the background. A good classification has the blue dots in the dark blue areas and similarly for red. We can see that for both data sets Net classifies only a selection of the points correctly whereas the other three methods do rather well on almost all points. Note that the shape of the learned classifier is still rather different despite them being very similar in the area of the training data. 

For the bottom rows of both figures we split the classification into the transformation and a linear classification. The transformation is the evolution of the ODE for ResNet, ODENet and ODENet+simplex. For Net this is the recursive formula~\eqref{eq:feedforward}. Note that the learned transformations are very different for the four different methods.

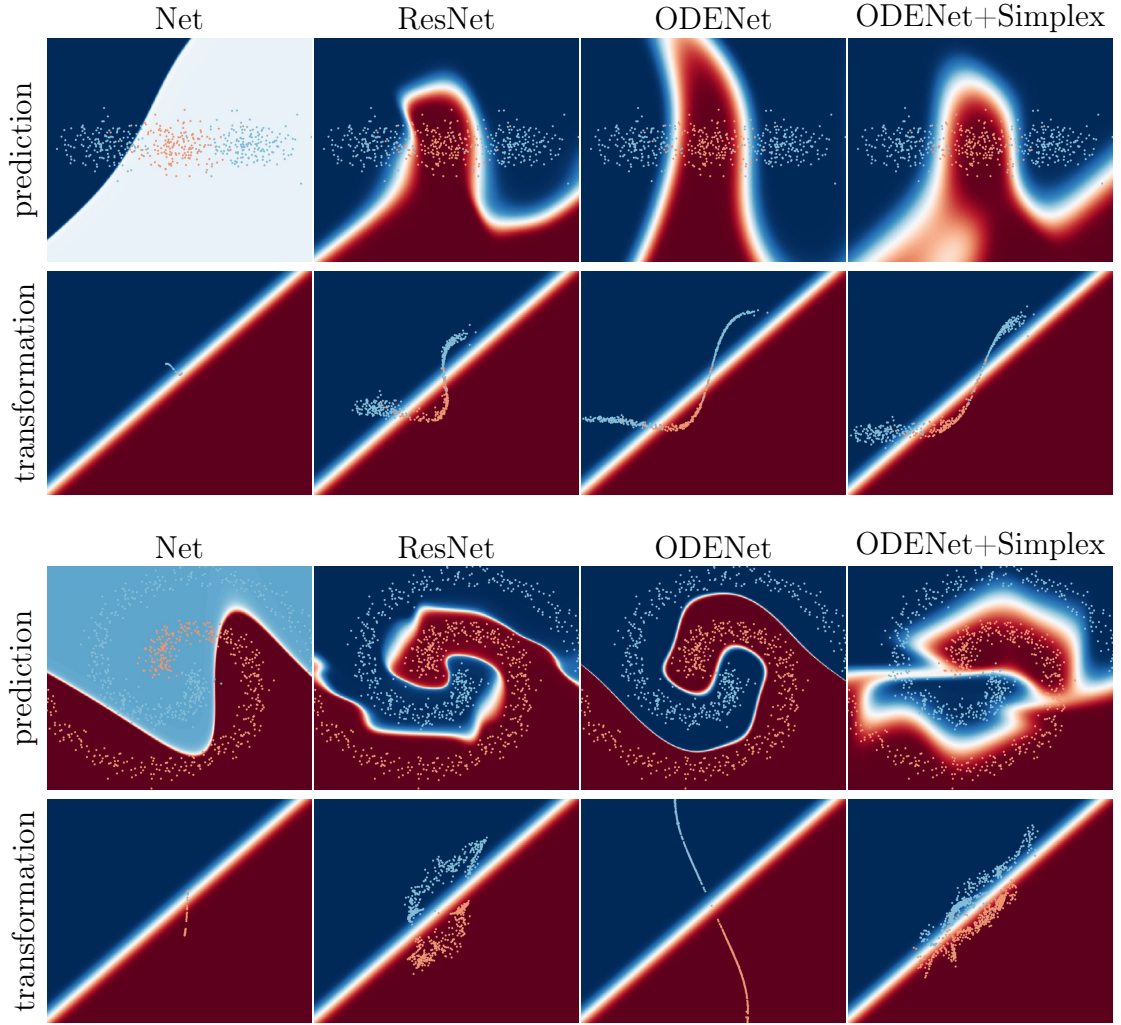
\begin{figure}[t]
\begin{tikzpicture}
\PlotRow{0}{prediction}
{results_fixedW_donut1d_20layers/Net_squared_seed1_prediction}
{results_fixedW_donut1d_20layers/ResNet_squared_seed1_prediction}
{results_fixedW_donut1d_20layers/ODENet_squared_seed1_prediction}
{results_fixedW_donut1d_20layers/ODENetSimplex_squared_seed1_prediction}
\PlotRow{1}{transformation}
{results_fixedW_donut1d_20layers/Net_squared_seed1_transformed}
{results_fixedW_donut1d_20layers/ResNet_squared_seed1_transformed}
{results_fixedW_donut1d_20layers/ODENet_squared_seed1_transformed}
{results_fixedW_donut1d_20layers/ODENetSimplex_squared_seed1_transformed}
\PlotLabels{Net}{ResNet}{ODENet}{ODENet+Simplex}
\end{tikzpicture}

\vspace*{2mm}

\begin{tikzpicture}
\PlotRow{0}{prediction}
{results_fixedW_spiral2d_20layers/Net_squared_seed1_prediction}
{results_fixedW_spiral2d_20layers/ResNet_squared_seed1_prediction}
{results_fixedW_spiral2d_20layers/ODENet_squared_seed1_prediction}
{results_fixedW_spiral2d_20layers/ODENetSimplex_squared_seed1_prediction}
\PlotRow{1}{transformation}
{results_fixedW_spiral2d_20layers/Net_squared_seed1_transformed}
{results_fixedW_spiral2d_20layers/ResNet_squared_seed1_transformed}
{results_fixedW_spiral2d_20layers/ODENet_squared_seed1_transformed}
{results_fixedW_spiral2d_20layers/ODENetSimplex_squared_seed1_transformed}
\PlotLabels{Net}{ResNet}{ODENet}{ODENet+Simplex}
\end{tikzpicture}
\caption{Learned transformation with fixed classifier for data set \DataDonutA~(top) and \DataSpiral~(bottom).} \label{FIG:M:TRANSFORMATION:CLASSFIXED}
\end{figure}

\subsubsection{Evolution of Features}%
Figure \ref{FIG:M:TRANSFORMATION} shows the evolution of the features by the learned parameters for the data set \DataSpiral. It can be seen that all four methods result in different dynamics, Net and ODENet reduce the two dimensional point cloud to a one-dimensional string whereas ResNet and ODENet+simplex preserve their two-dimensional character. This observations seem to be characteristic as we qualitatively observed similar dynamics for other data sets and random initialisation (not shown). 

Note that the dynamics of ODENet transform the points outside the field-of-view and the decision boundary (fuzzy bright line in the background) is wider than for ResNet and ODENet+simplex.

Intuitively, a scaling of the points and a fuzzier classification is equivalent to leaving the points where they are and a sharper classification. We tested the aforementioned effect by keeping a fixed classification throughout the learning process and only learning the transformation. The results in Figure~\ref{FIG:M:TRANSFORMATION:CLASSFIXED} show that this is indeed the case.

\subsubsection{Dependence on Randomness}%
We tested the dependence of our results on different random initialisations. For conciseness we only highlight one result in Figure~\ref{FIG:SEED}. Indeed, the two rows which correspond to two different random initialisations show very similar topological behaviour.

\begin{figure}
\begin{tikzpicture}
\PlotRow{0}{random seed 1}
{results_donut2d_20layers/Net_squared_seed1_transformed}
{results_donut2d_20layers/ResNet_squared_seed1_transformed}
{results_donut2d_20layers/ODENet_squared_seed1_transformed}
{results_donut2d_20layers/ODENetSimplex_squared_seed1_transformed}
\PlotRow{1}{random seed 2}
{results_donut2d_20layers/Net_squared_seed2_transformed}
{results_donut2d_20layers/ResNet_squared_seed2_transformed}
{results_donut2d_20layers/ODENet_squared_seed2_transformed}
{results_donut2d_20layers/ODENetSimplex_squared_seed2_transformed}
\PlotLabels{Net}{ResNet}{ODENet}{ODENet+Simplex}
\end{tikzpicture}%
\caption{Robustness on random initialisation for transformed data \DataDonutB~and linear classifier for two different initialisations.} \label{FIG:SEED}
\end{figure}
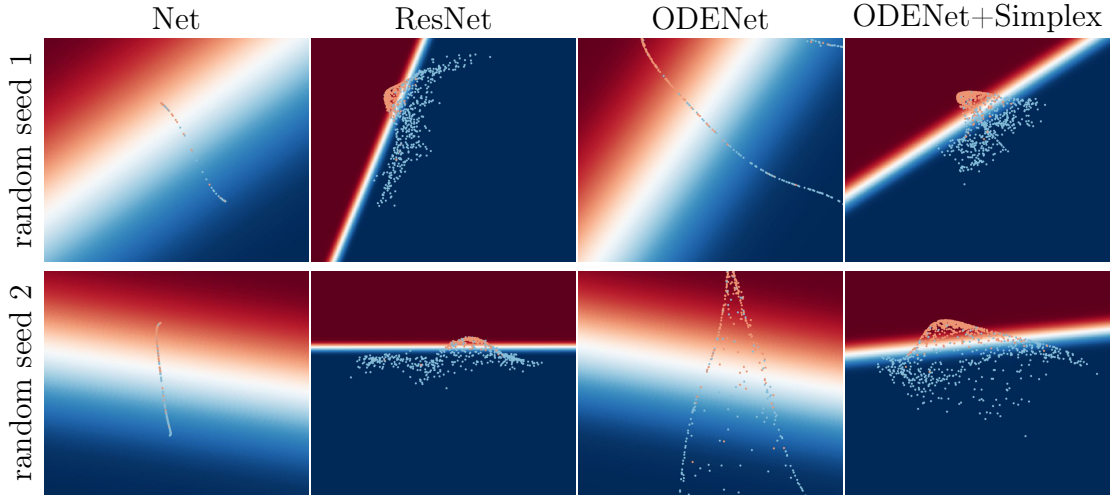

\subsubsection{Quantitative Results}
Quantitative results are presented in Figures~\ref{FIG:M:FUNCTION} and~\ref{FIG:M:ACCURACY} which show the evolution of function values and the classification accuracy over the course of the gradient descent iterations. The solid lines are for the training data and dashed for the test data, which is an independent draw from the same distribution and of the same size as the training data.

We can see that Net does not perform as well for any of the data sets than the other three methods. Consistently, ODENet is initially the fastest but at later stages ResNet overtakes it. All three methods seem to converge to a similar function value. As the dashed line follows essentially the solid line we can observe that there is not much overfitting taking place.

\begin{figure}
\centering
\includegraphics[width=7cm]{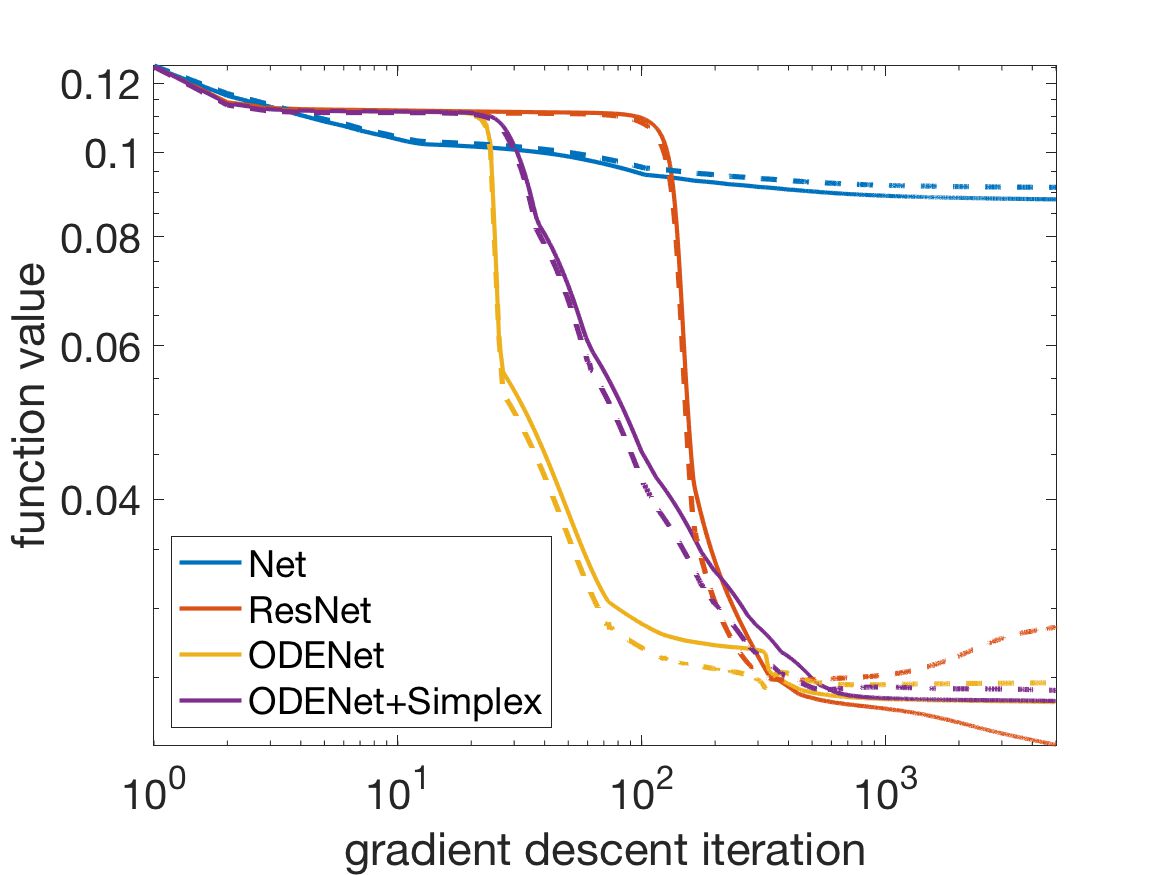}%
\includegraphics[width=7cm]{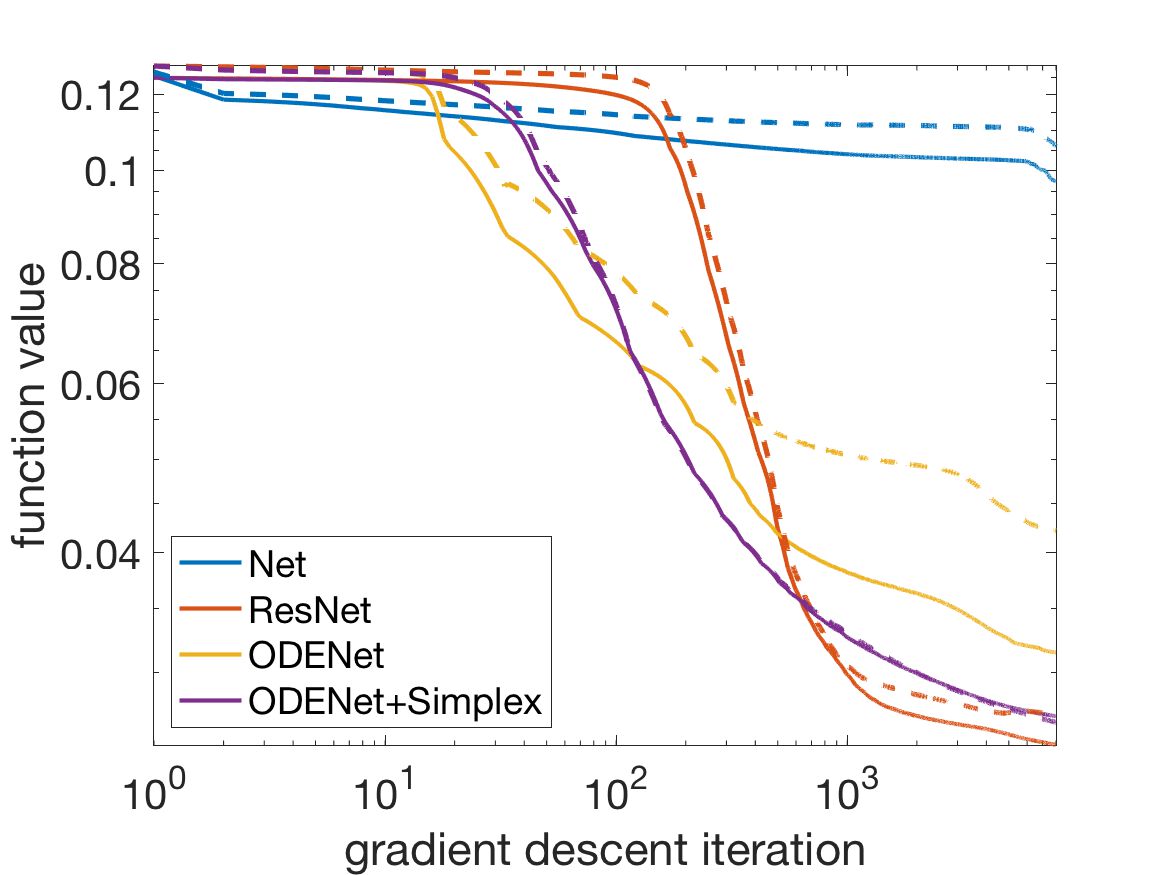}\\
\includegraphics[width=7cm]{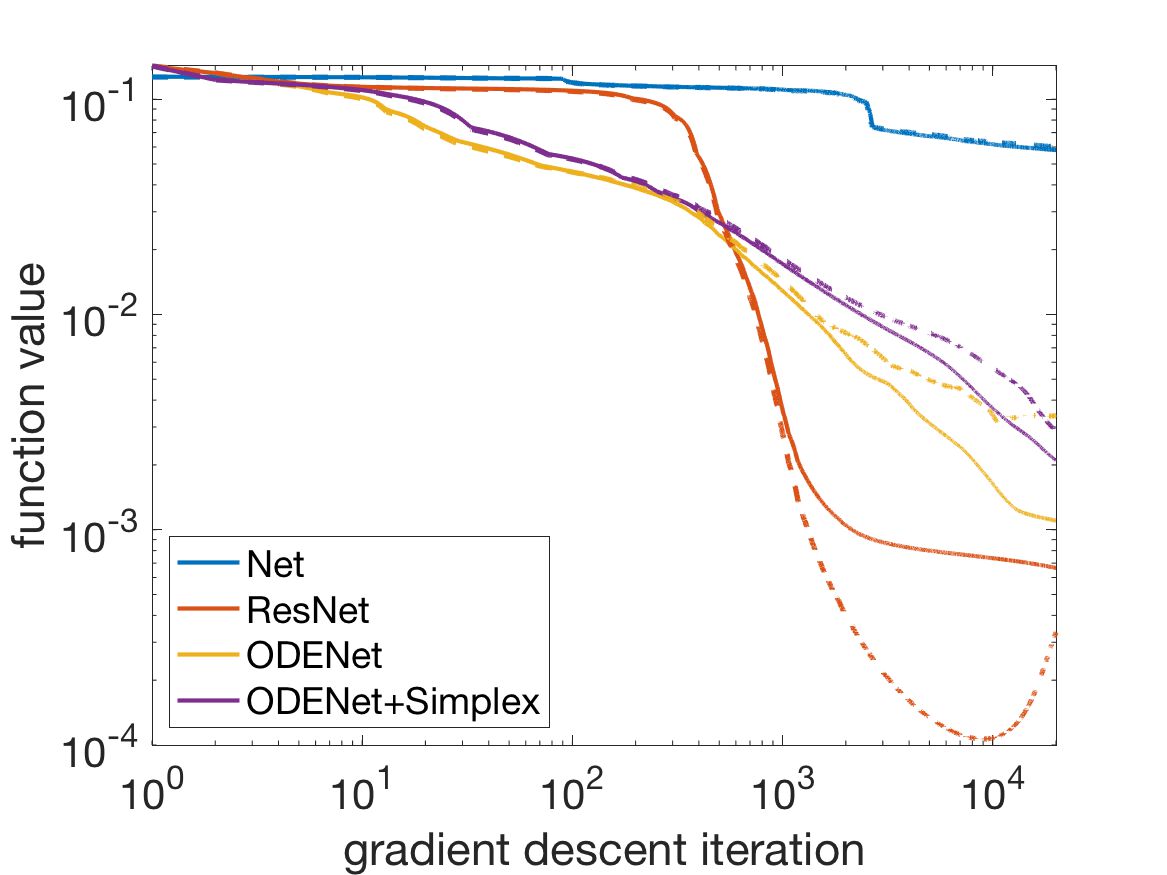}%
\includegraphics[width=7cm]{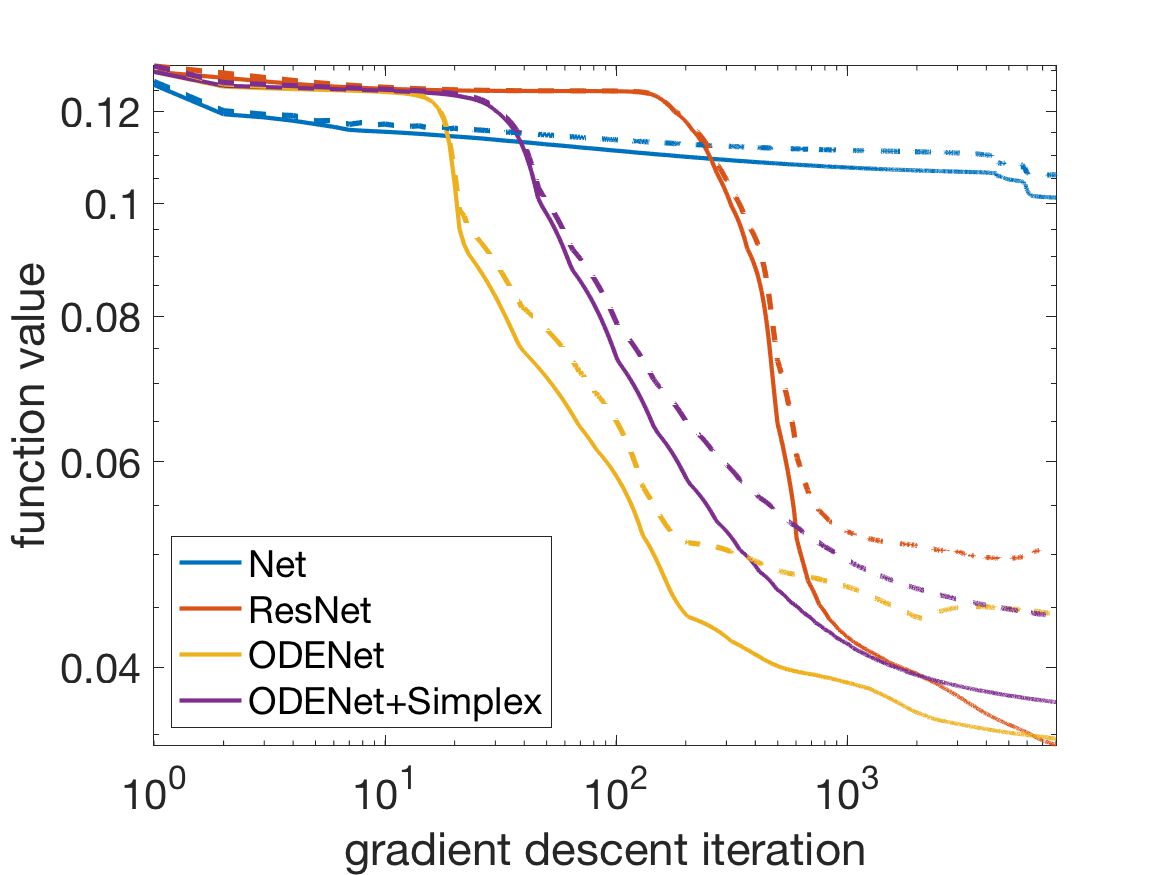}
\caption{Function values over the course of the gradient descent iterations for data sets \DataDonutA, \DataDonutB, \DataSpiral, \DataSquares~(left to right and top to bottom). The solid line represents training and the dashed line test data.}
\label{FIG:M:FUNCTION}
\end{figure}

\begin{figure}
\centering
\includegraphics[width=7cm]{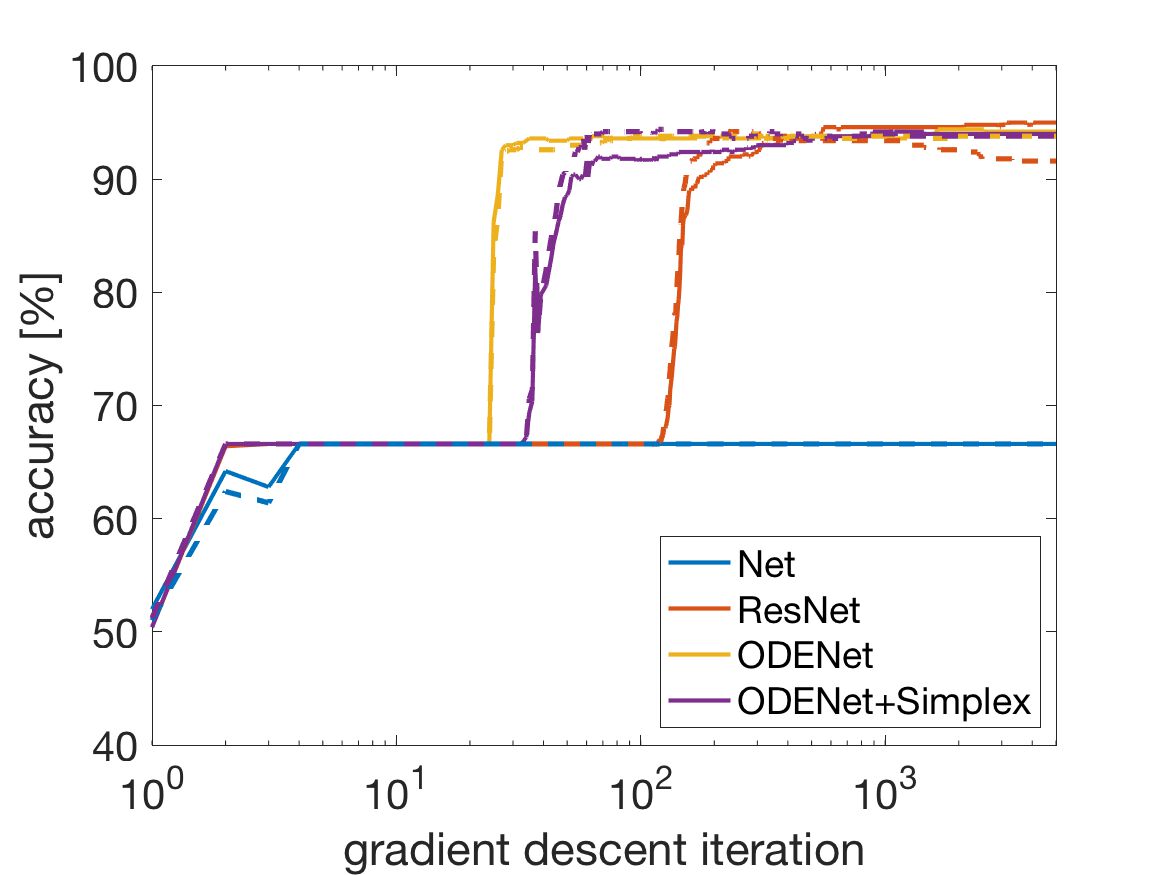}%
\includegraphics[width=7cm]{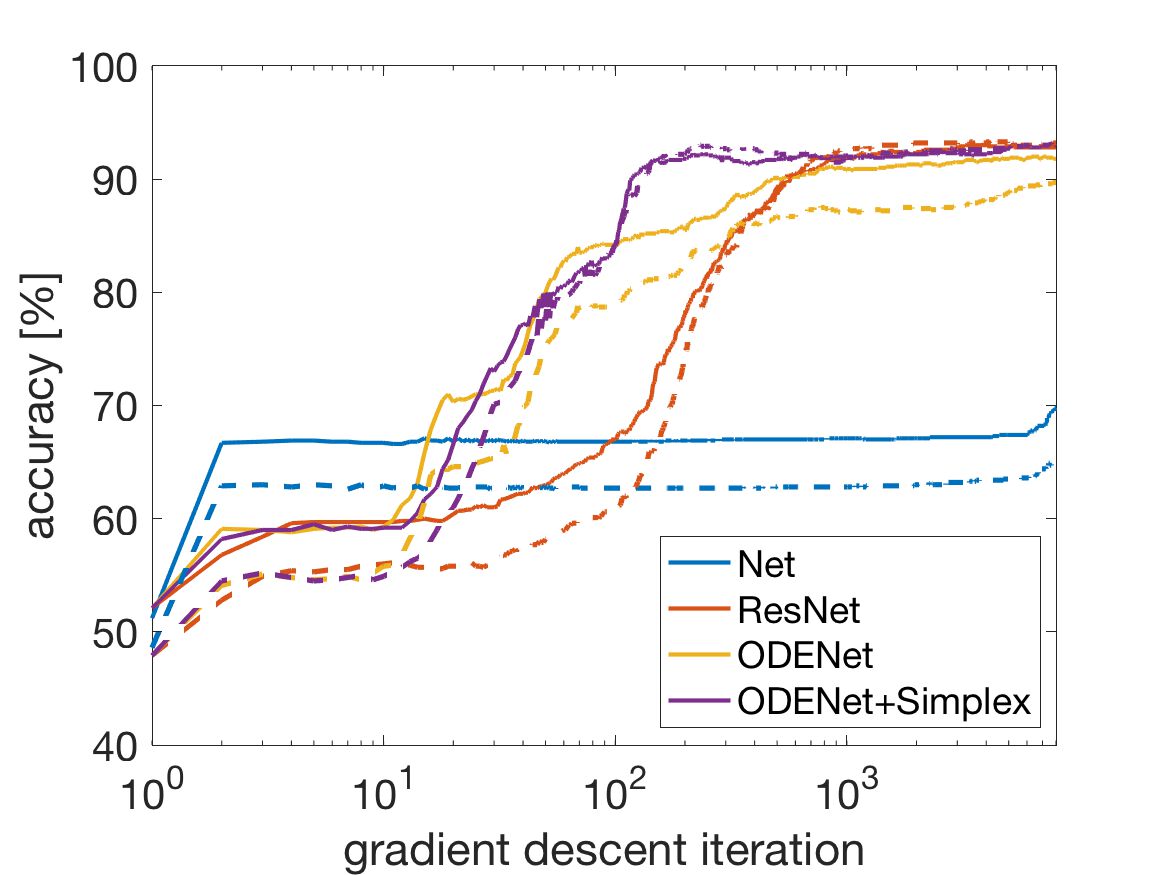}\\
\includegraphics[width=7cm]{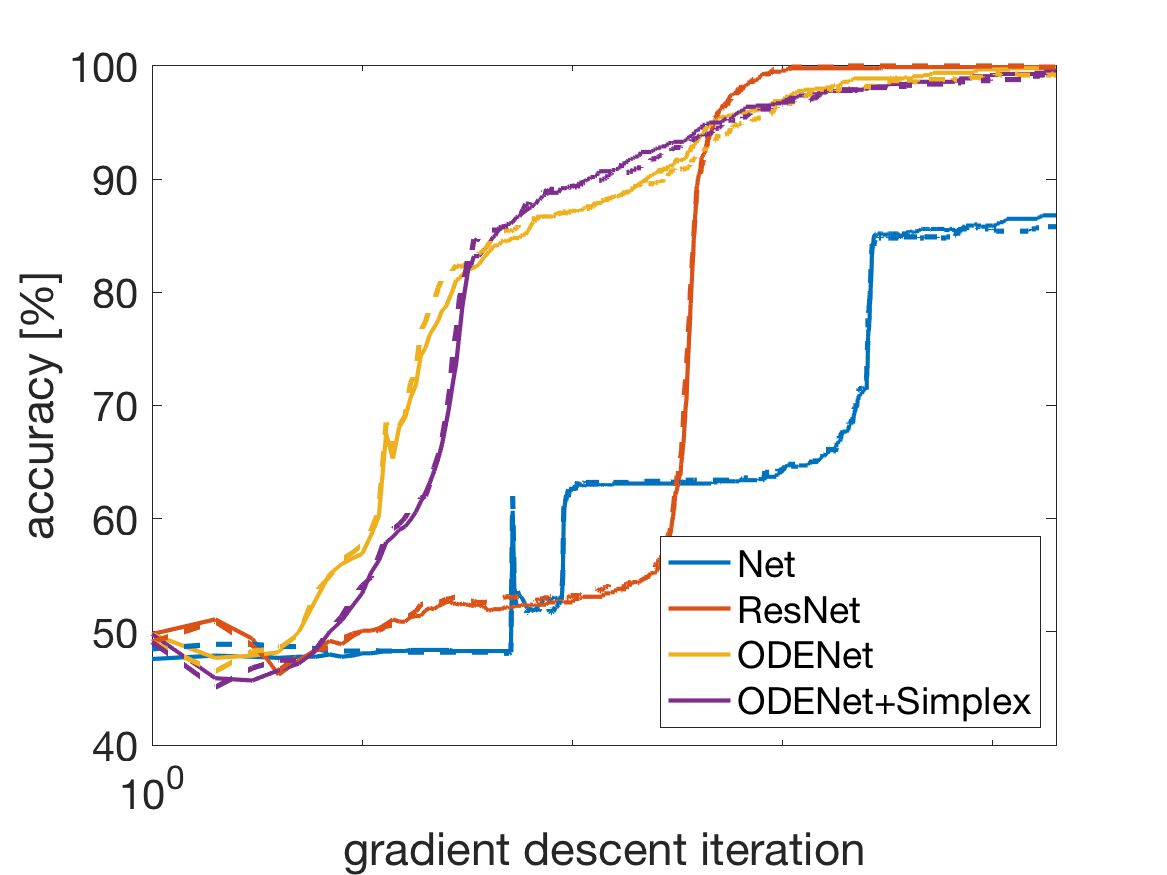}%
\includegraphics[width=7cm]{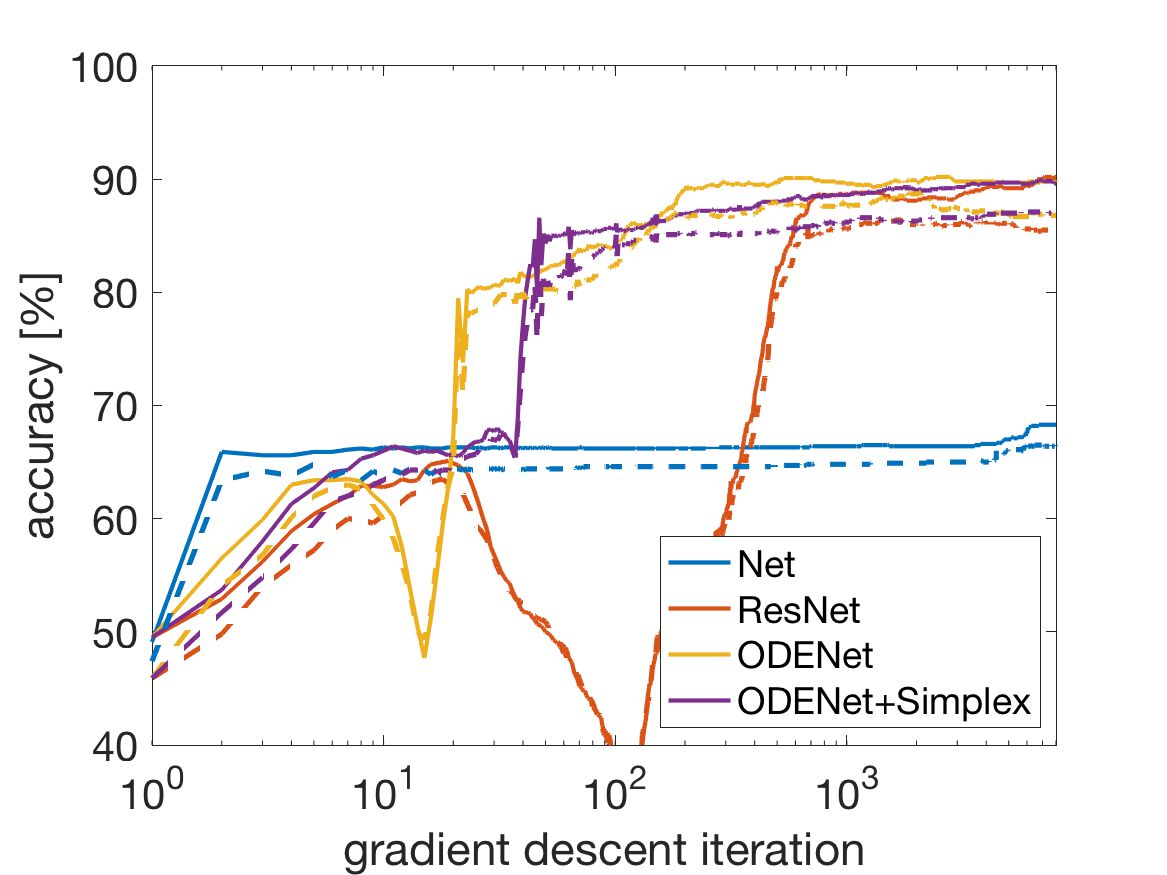}
\caption{Classification accuracy over the course of the gradient descent iterations for data sets \DataDonutA, \DataDonutB, \DataSpiral, \DataSquares~(left to right and top to bottom). The solid line represents training and the dashed line test data.}\label{FIG:M:ACCURACY}
\end{figure}

\subsubsection{Estimation of Varying Time Steps}
Figure \ref{FIG:M:TIME} shows the (estimated) time steps for ResNet/Euler, ODENet and ODENet+simplex. While ResNet uses equidistant time discretisation, ODENet and ODENet+simplex learn these as part of the training. In addition, ODENet+simplex use a simplex constraint on these values which allow the interpretation as varying time steps. It can be seen consistently for all four data sets that ODENet chooses both negative and positive time steps and these are generally of larger magnitude than the other two methods. Moreover, these are all non-zero. In contrast, ODENet+Simplex picks a few time steps (two or three) and sets the rest to zero. Sparse time steps have the advantage that less memory is needed to store this network and that less computation is needed for classification at test time.

Although it might seem unnatural to allow for negative time steps in this setting, a benefit is that this adds to the flexibility of the approach. It should also be noted that negative steps are rather common in the design of e.g. splitting and composition methods from the ODE literature, \cite{blanes05otn}.

\begin{figure}
\centering
\includegraphics[width=7cm]{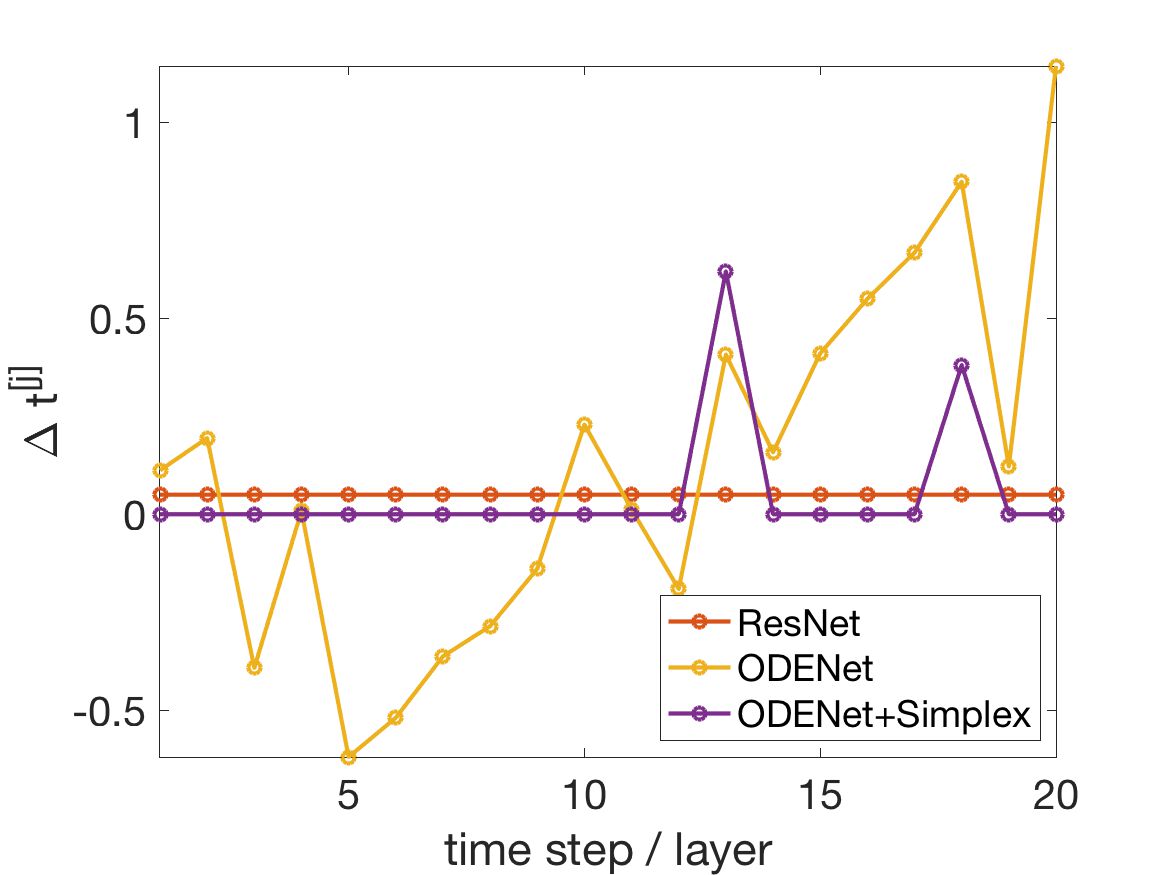}%
\includegraphics[width=7cm]{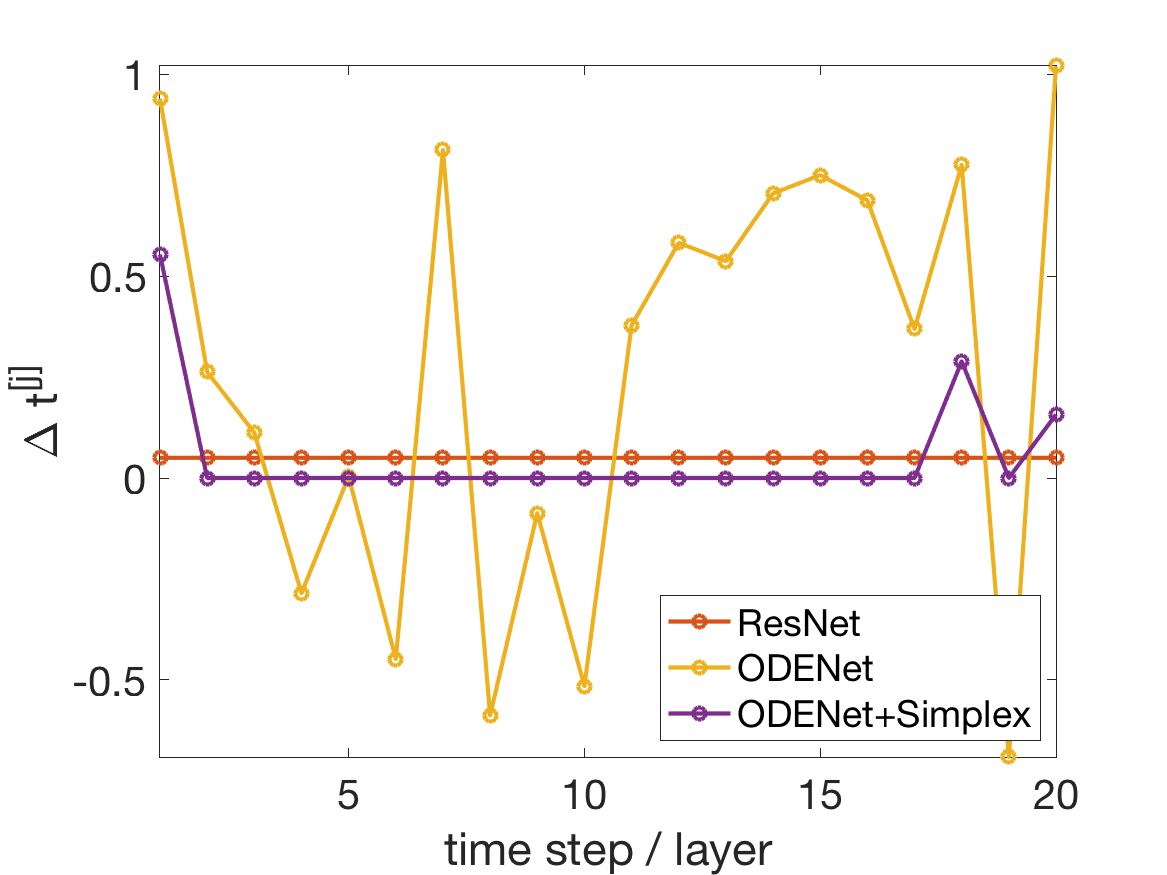}\\
\includegraphics[width=7cm]{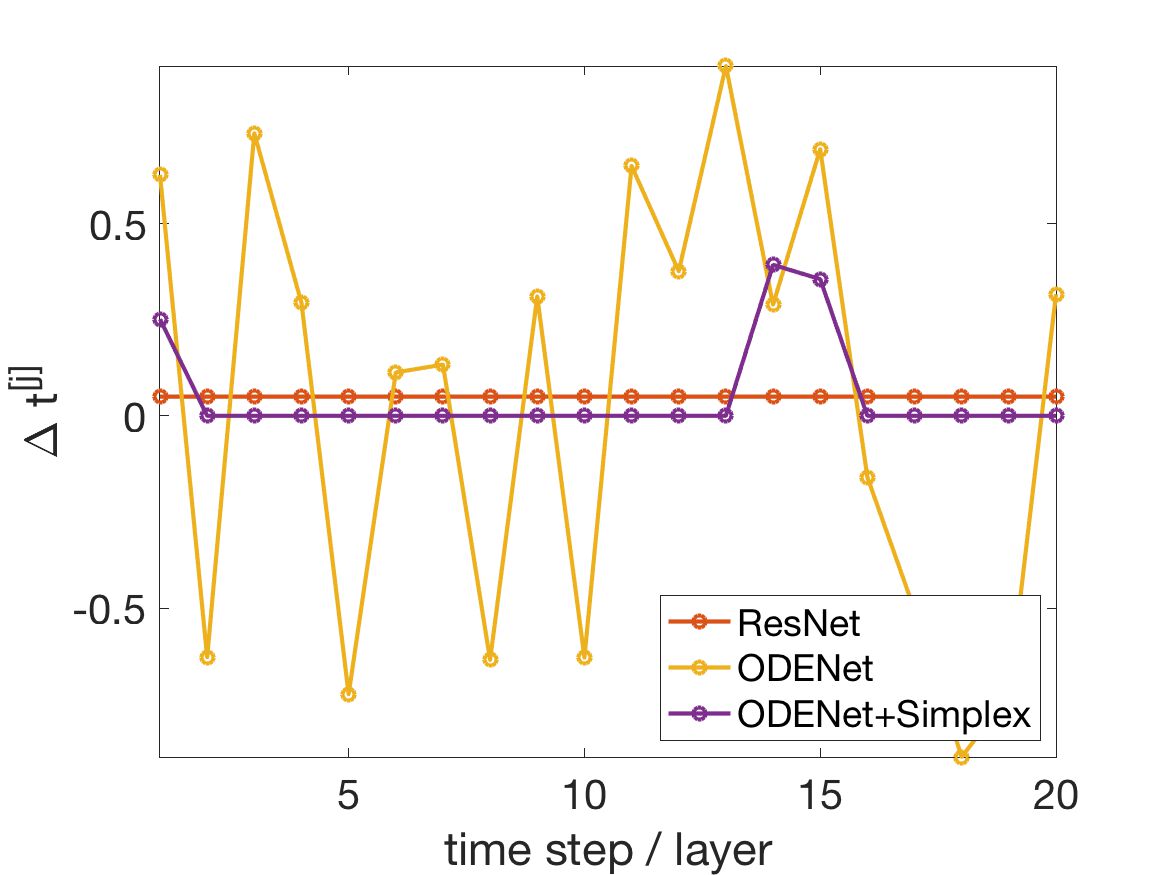}%
\includegraphics[width=7cm]{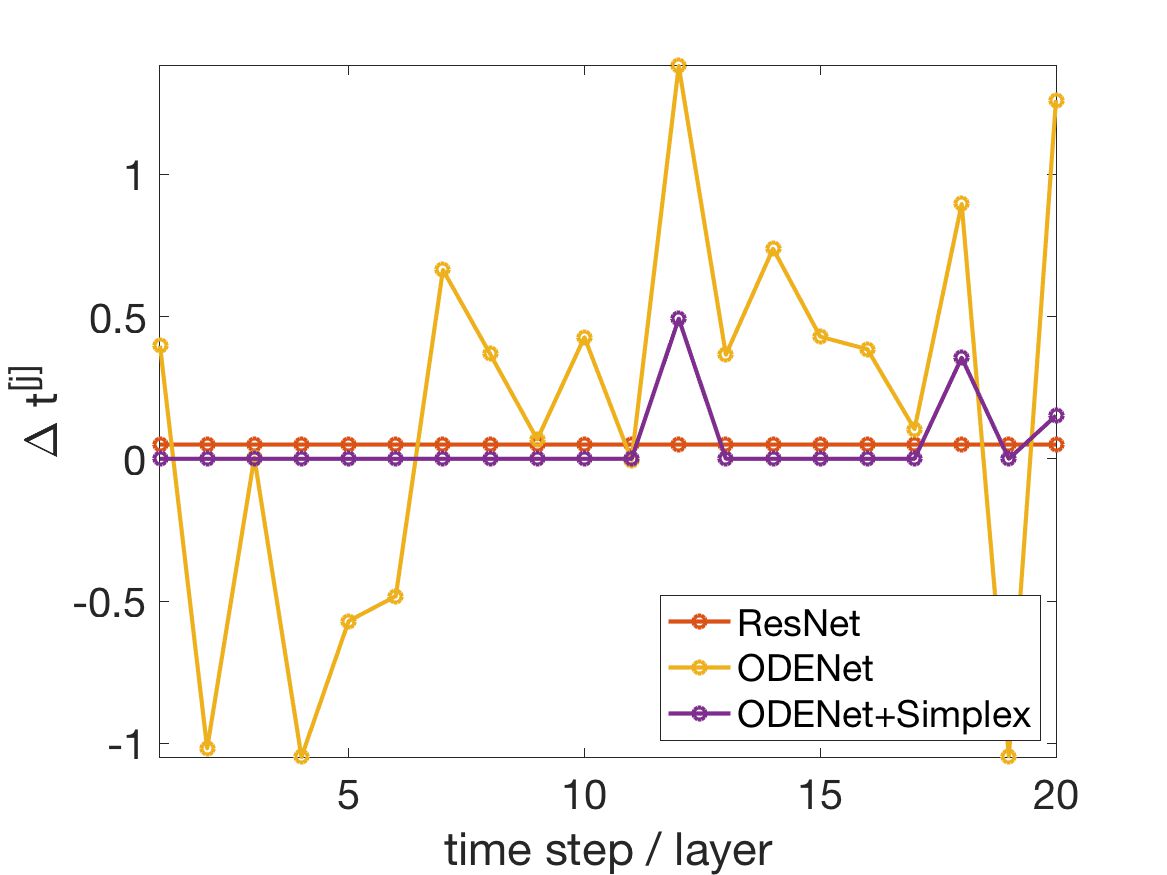}
\caption{Estimated time steps by ResNet/Euler, ODENet and ODENet+simplex for for data sets \DataDonutA, \DataDonutB, \DataSpiral, \DataSquares~(left to right and top to bottom). ODENet+simplex consistently picks two to three time steps and set the rest to zero.}\label{FIG:M:TIME}
\end{figure}

\subsection{Comparing different explicit Runge--Kutta architectures}
We are here showing results for 4 different explicit Runge--Kutta schemes of orders 1--4, their Butcher tableaux are given in Table~\ref{RKMethods}.

\begin{table}
\newcommand{\mipawi}{0.23}
\newcommand{\morespacearray}{\renewcommand{\arraystretch}{1.3}}
\begin{minipage}[htp]{\mipawi\textwidth}
$$
\morespacearray
\begin{array}{c|c}
0 &   \\ \hline
   & 1
\end{array}
$$
\end{minipage}
\begin{minipage}{\mipawi\textwidth}
$$
\morespacearray
\begin{array}{c|cc}
0 &   \\ 
 1  & 1 \\ \hline
 &\frac12 & \frac12
\end{array}
$$
\end{minipage}
\begin{minipage}{\mipawi\textwidth}
$$
\morespacearray
\begin{array}{r|rrr}
0            &   \\ 
 \frac12  & \frac12 \\ 
 1           & -1         & 2 \\ \hline
 &\frac16&\frac23  & \frac16
\end{array}
$$
\end{minipage}
\begin{minipage}{\mipawi\textwidth}
$$
\morespacearray
\begin{array}{r|rrrr}
0            &  \\ 
 \frac12  & \frac12 \\ 
 \frac12  &     0      & \frac12 \\
 1           &     0       &     0      & 1  \\ \hline
 &\frac16&\frac13   & \frac13 &  \frac16
\end{array}
$$
\end{minipage}
\caption{Four explicit Runge--Kutta methods: ResNet/Euler, Improved Euler, Kutta(3) and Kutta(4). \label{RKMethods}}
\end{table}
The first two methods are the Euler and Improved Euler methods over orders one and two respectively. The other two are due to Kutta \cite{Kutta1901} and have convergence orders three and four. The presented results are obtained with the data sets \DataDonutA,\ \DataDonutB,\ \DataSpiral,\ and \DataSquares. In the results reported here we have taken the number of layers to be 15.
In Figure \ref{pred-trans} we illustrate the initial and final configurations of the data points for the learned parameters. The blue and red background colours can be thought of as test data results in the upper row of plots. For instance, any point which was originally in a red area will be classified as red with high probability. Similarly, the background colours in the bottom row of plots show the classification of points which have been transformed to a given location. In the transition between red and blue the classification will have less certainty.

In Figures \ref{anim-spiral2d}--\ref{anim-squares2d}, more details of the transition are shown. The leftmost and rightmost plot show the initial and final states respectively, whereas the two in the middle show the transformation in layers 5 and 10. The background colours always show the same and correspond to the final state.

Finally, in Figure~\ref{FIG:RK:FUNCTION} we show the progress of the gradient descent method over 10,000 iterations for each of the four data sets.

\begin{figure}
\def\RowHeight{2.9cm}
\renewcommand{\PlotIm}[1]{{\includegraphics[width=\PicWidth, clip, trim=3cm 1.5cm 2cm 1cm]{RKfigs/#1}}}%

\begin{tikzpicture}
\PlotRow{0}{prediction}
{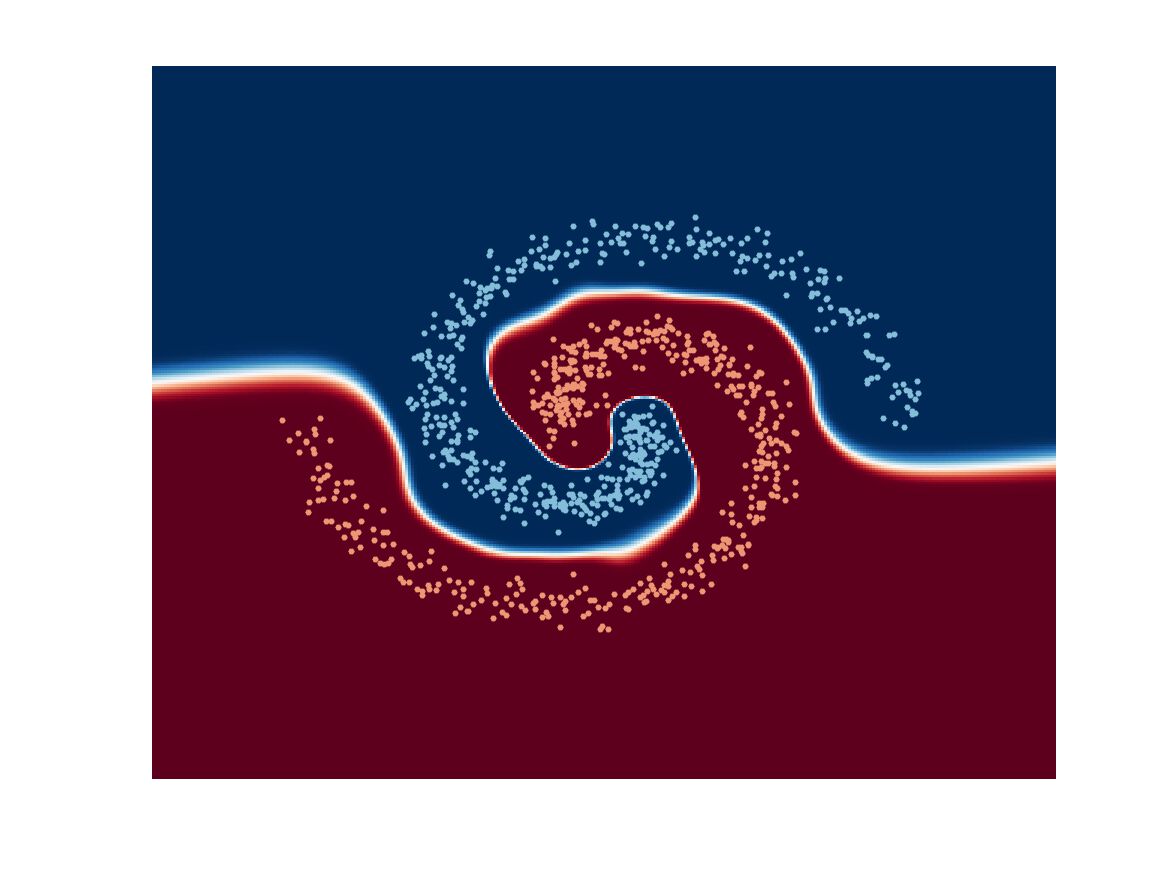}
{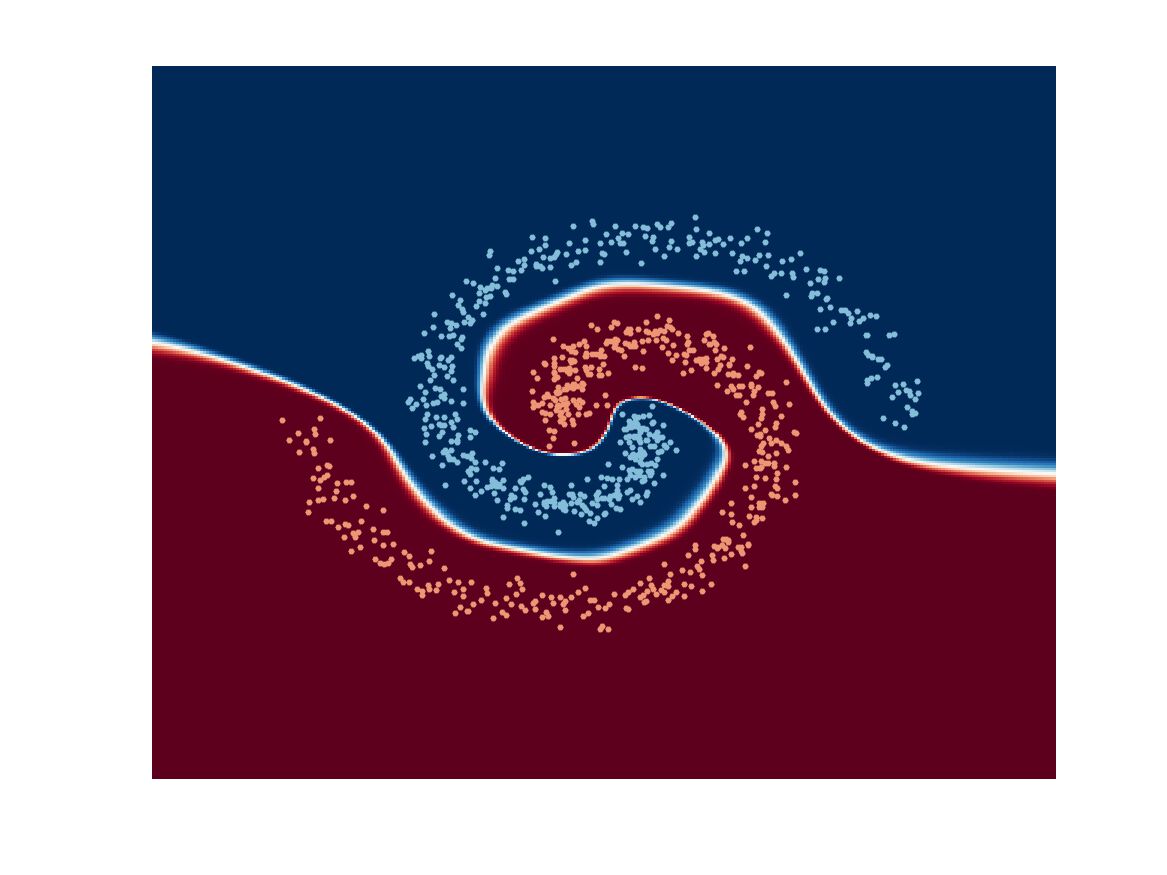}
{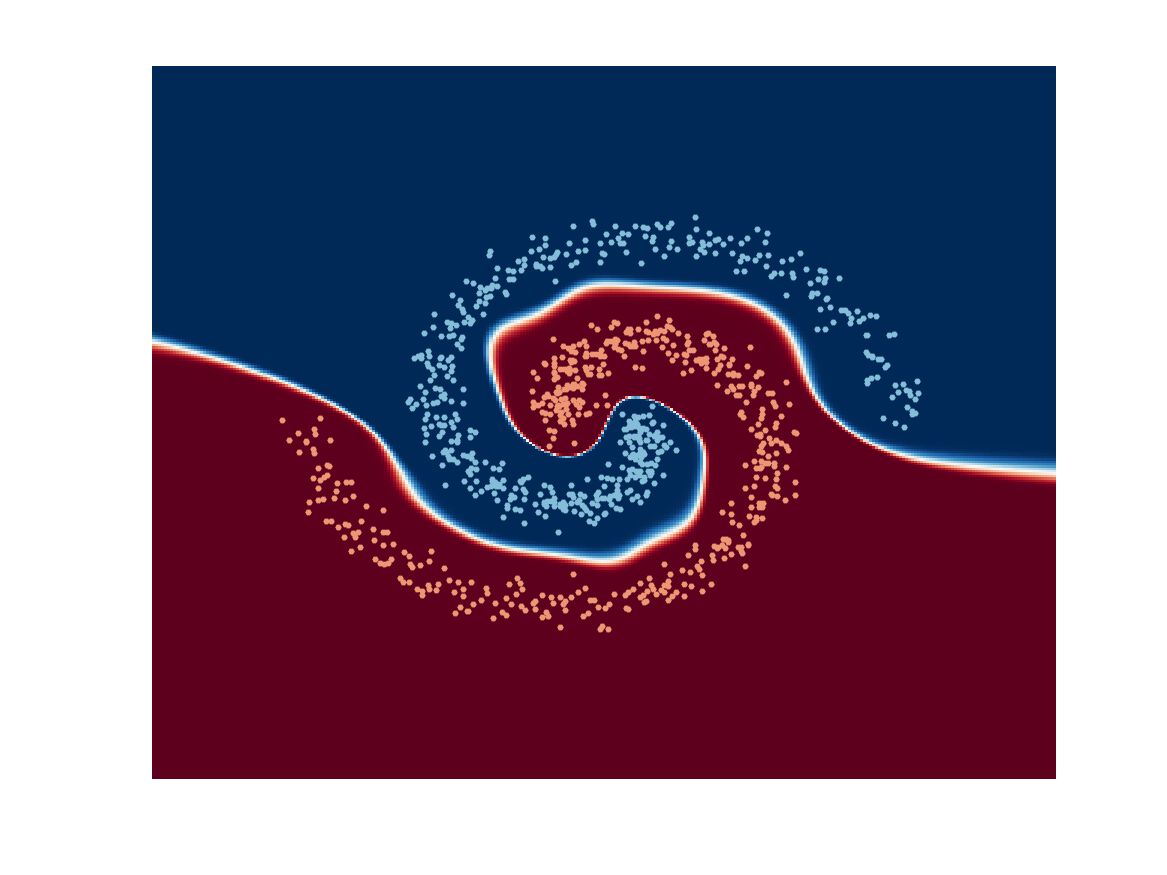}
{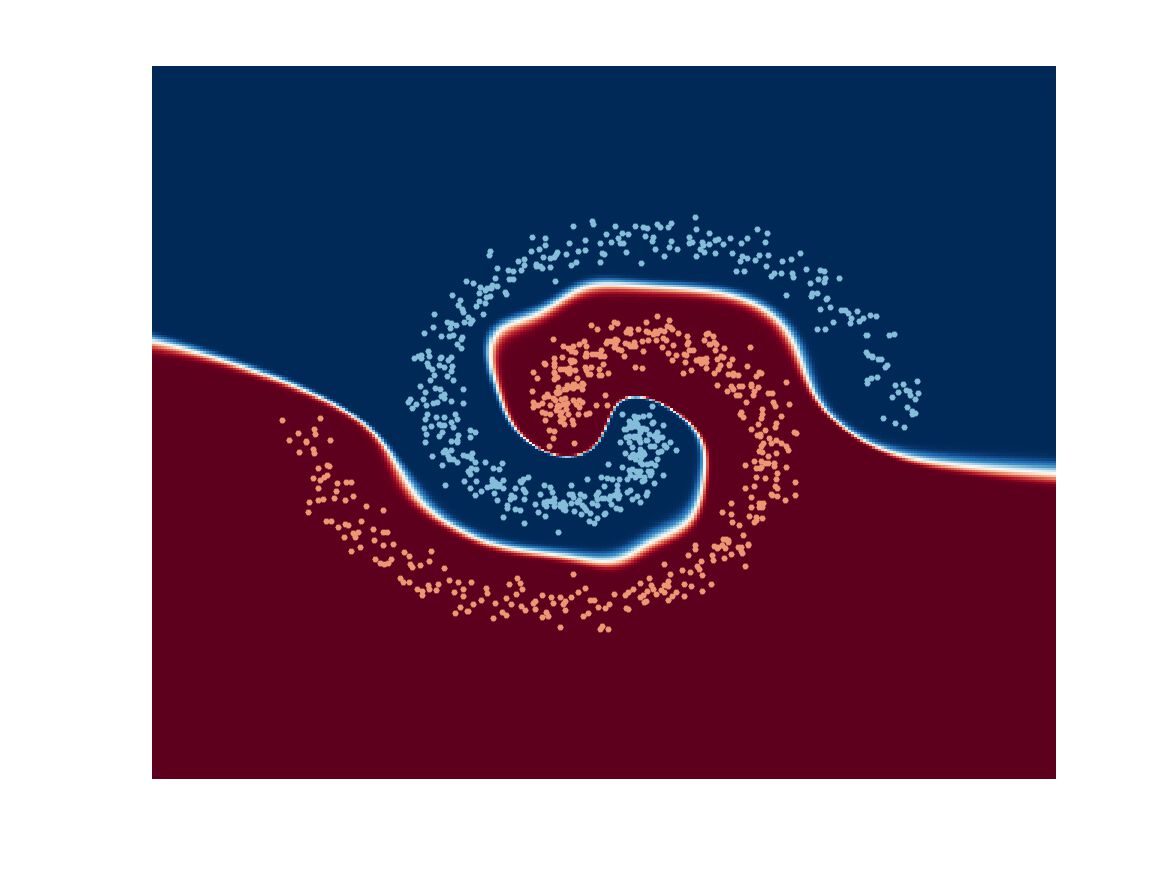}
\PlotRow{1}{transformation}
{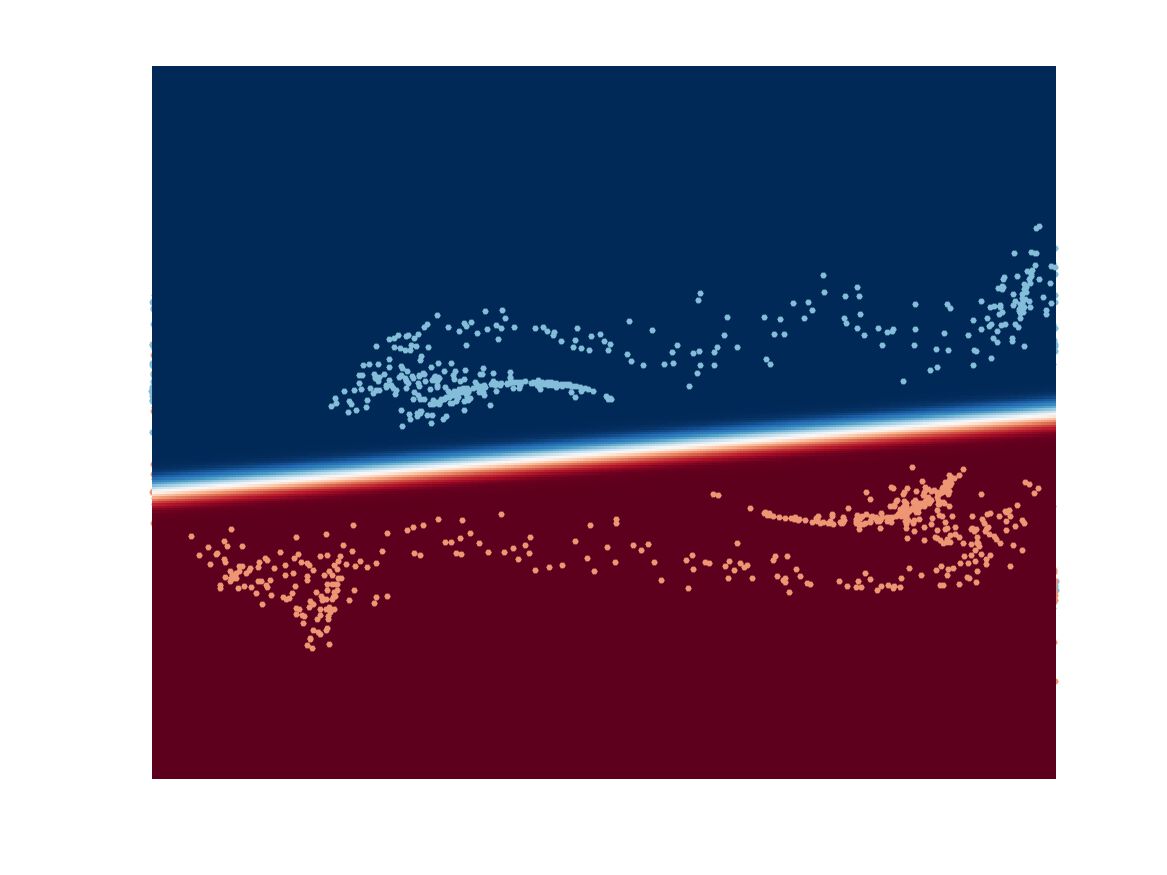}
{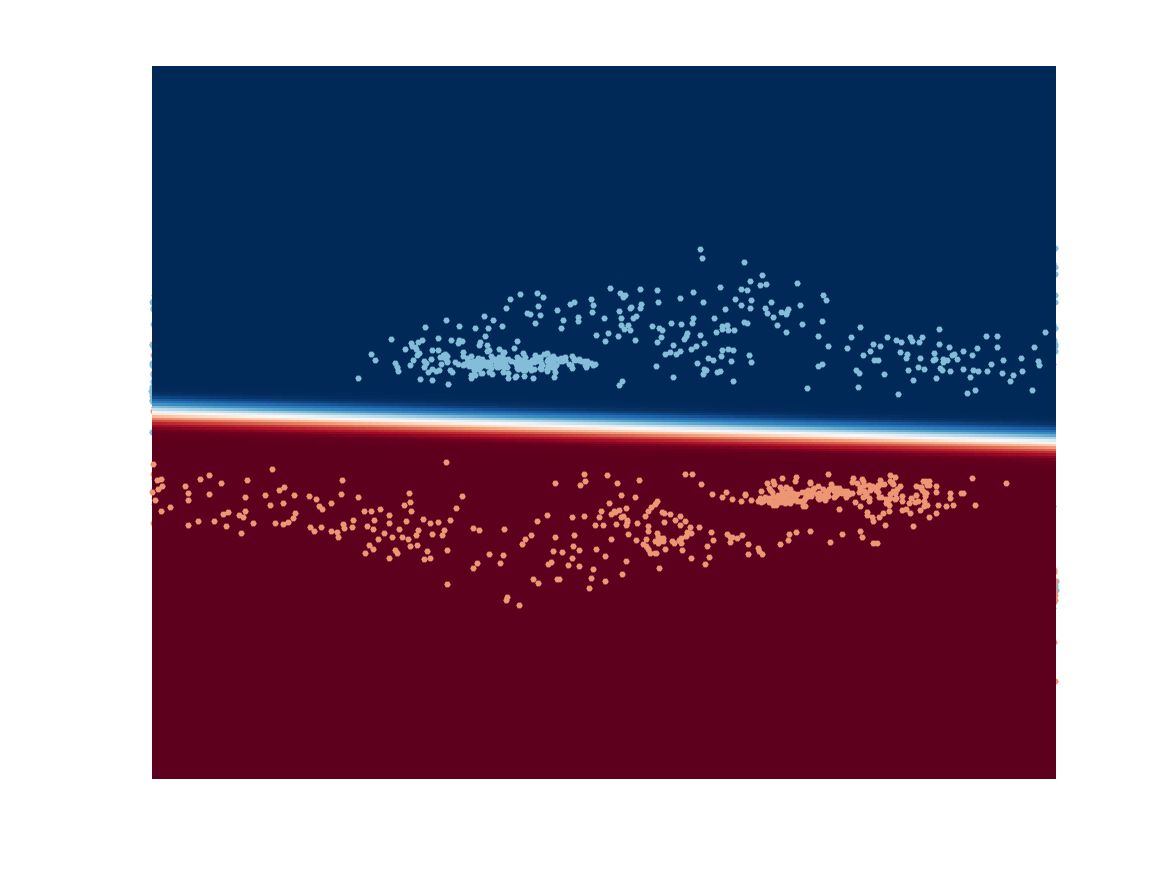}
{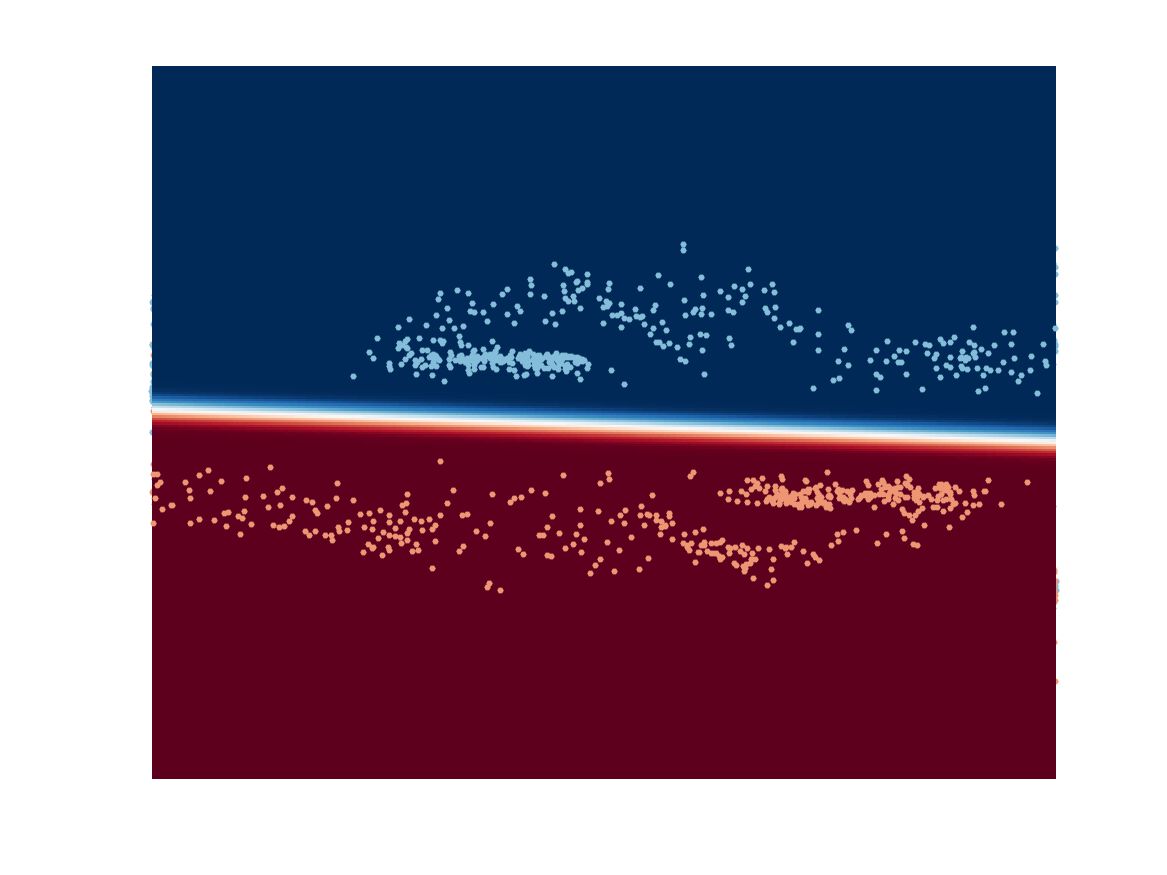}
{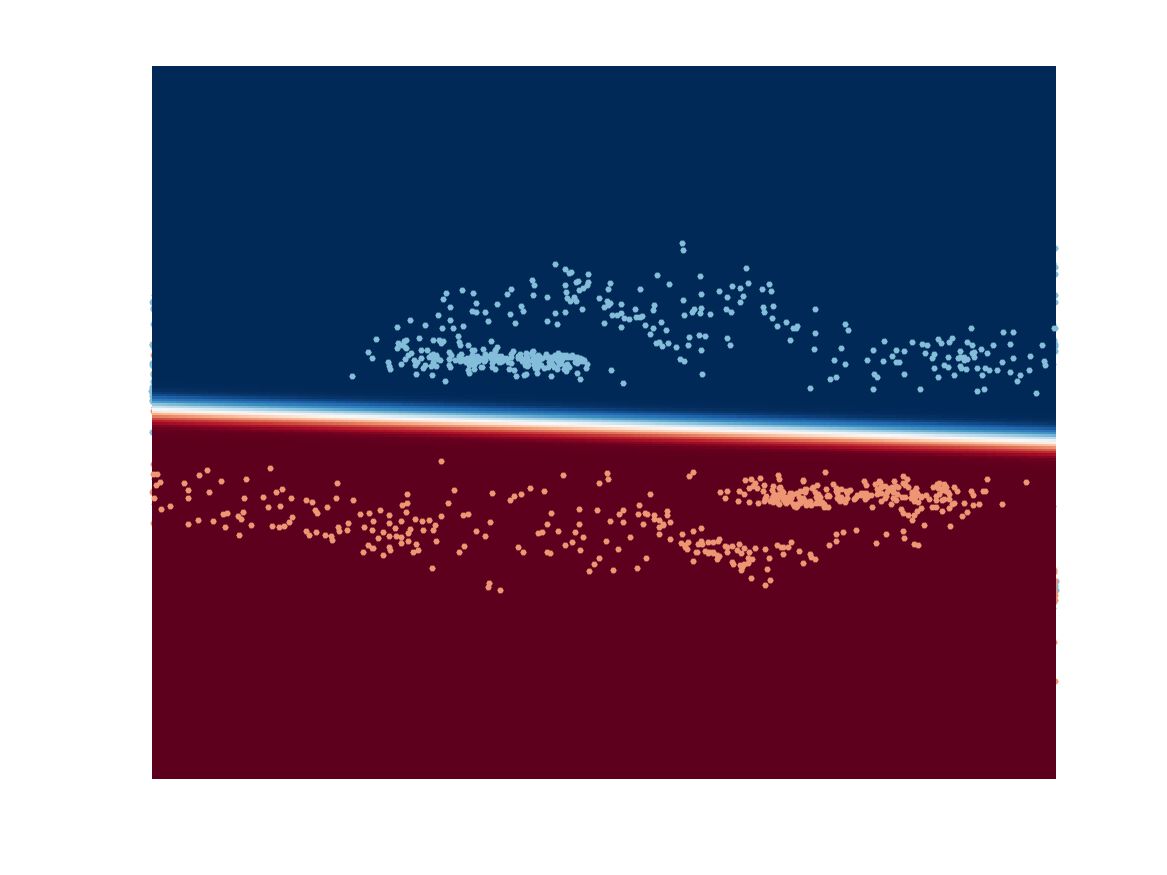}
\PlotLabels{ResNet/Euler}{Improved Euler}{Kutta(3)}{Kutta(4)}
\end{tikzpicture}

\begin{tikzpicture}
\PlotRow{0}{prediction}
{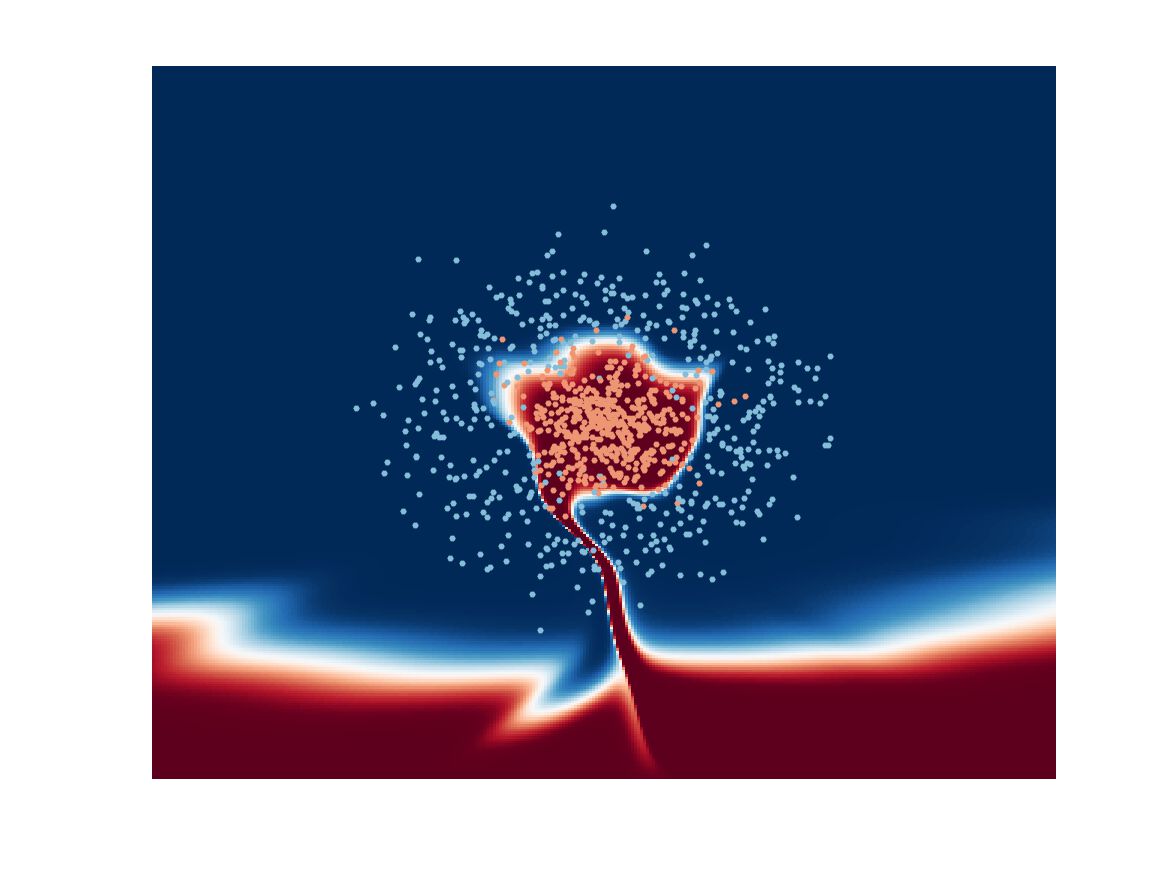}
{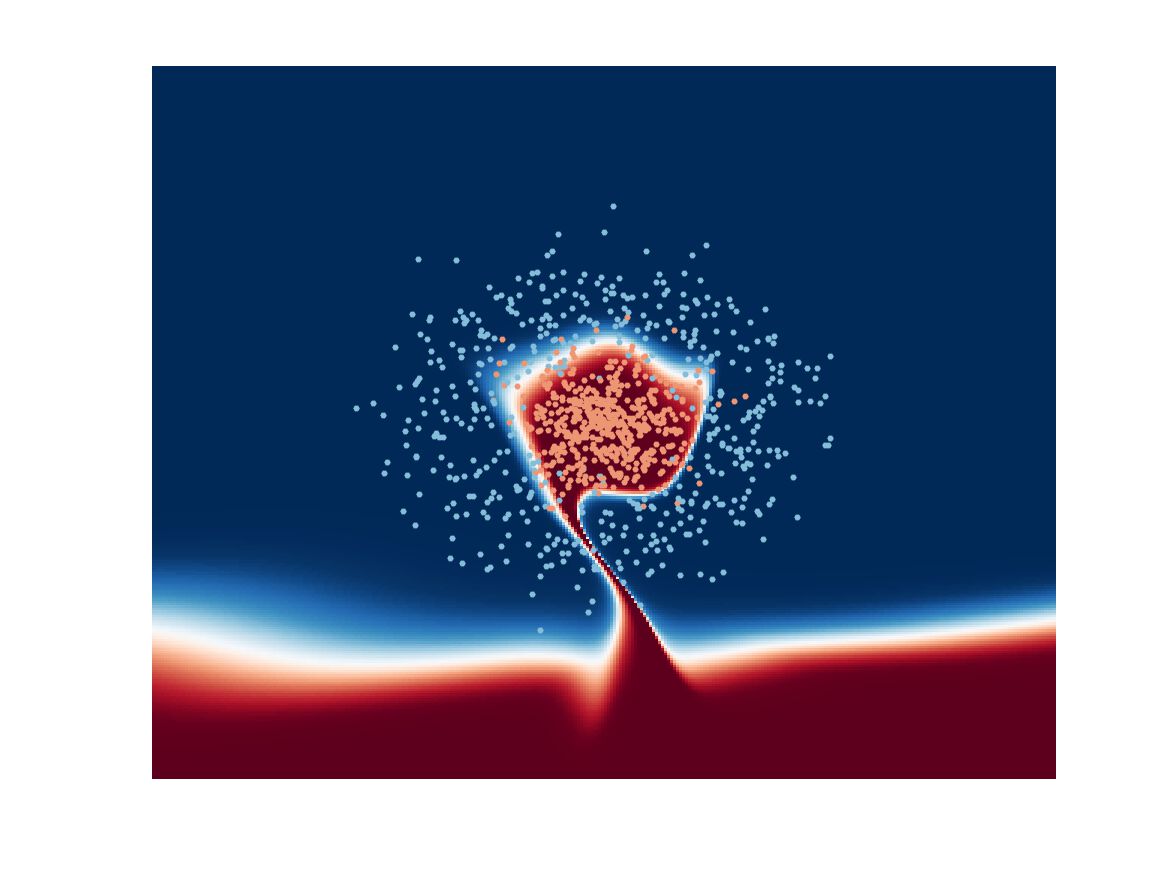}
{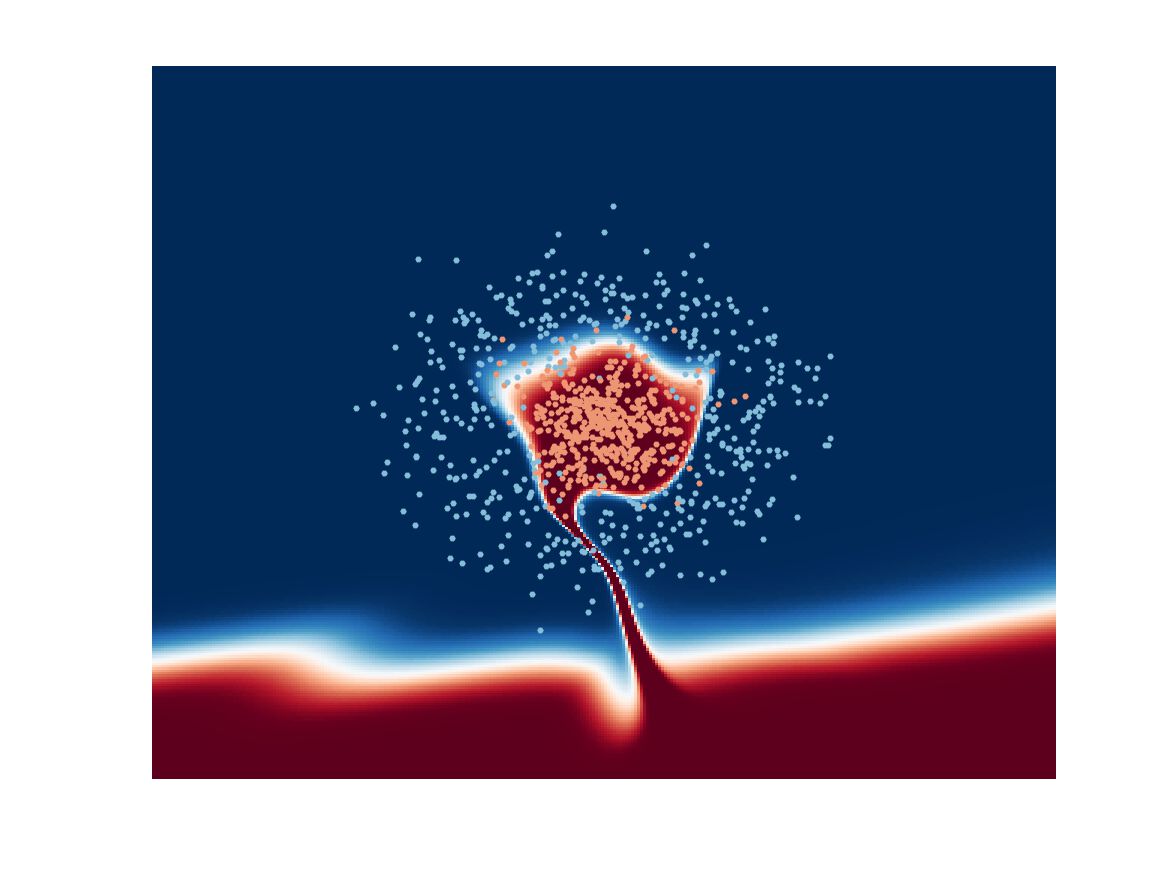}
{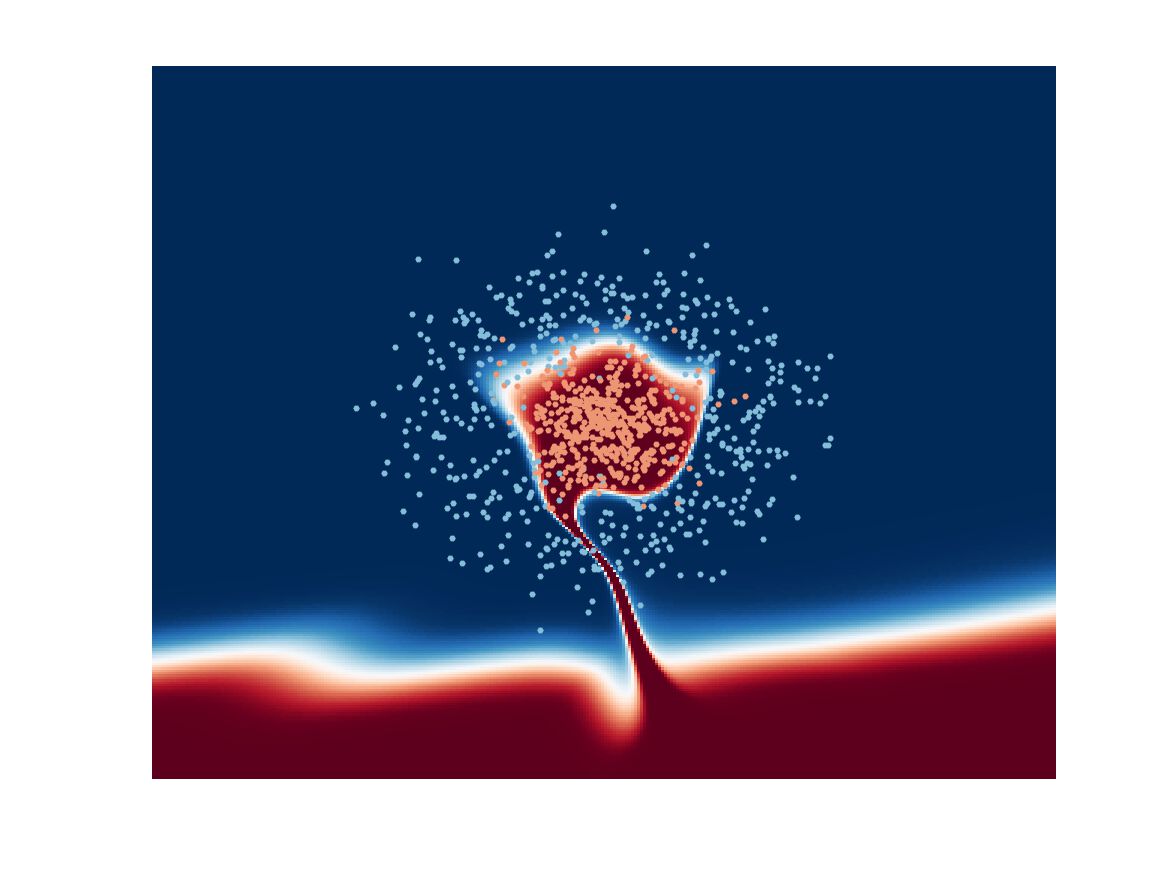}
\PlotRow{1}{transformation}
{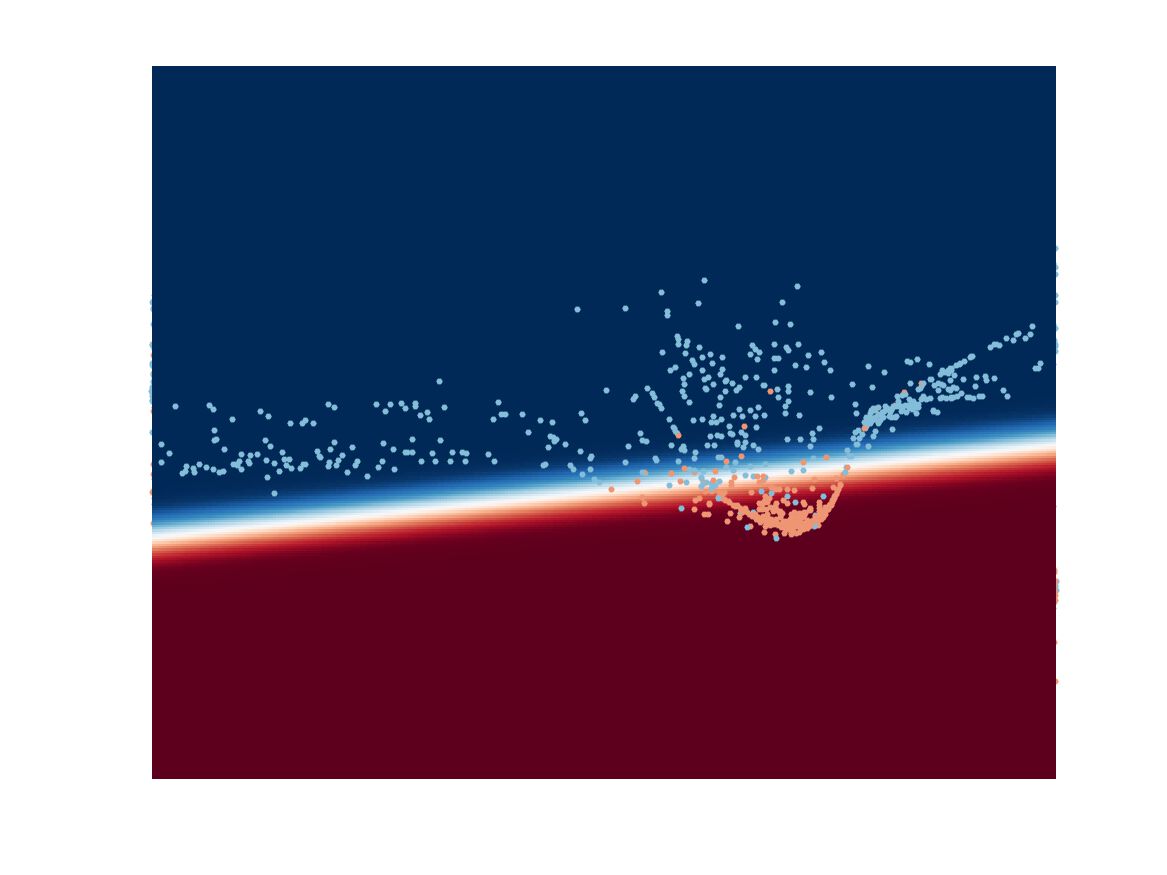}
{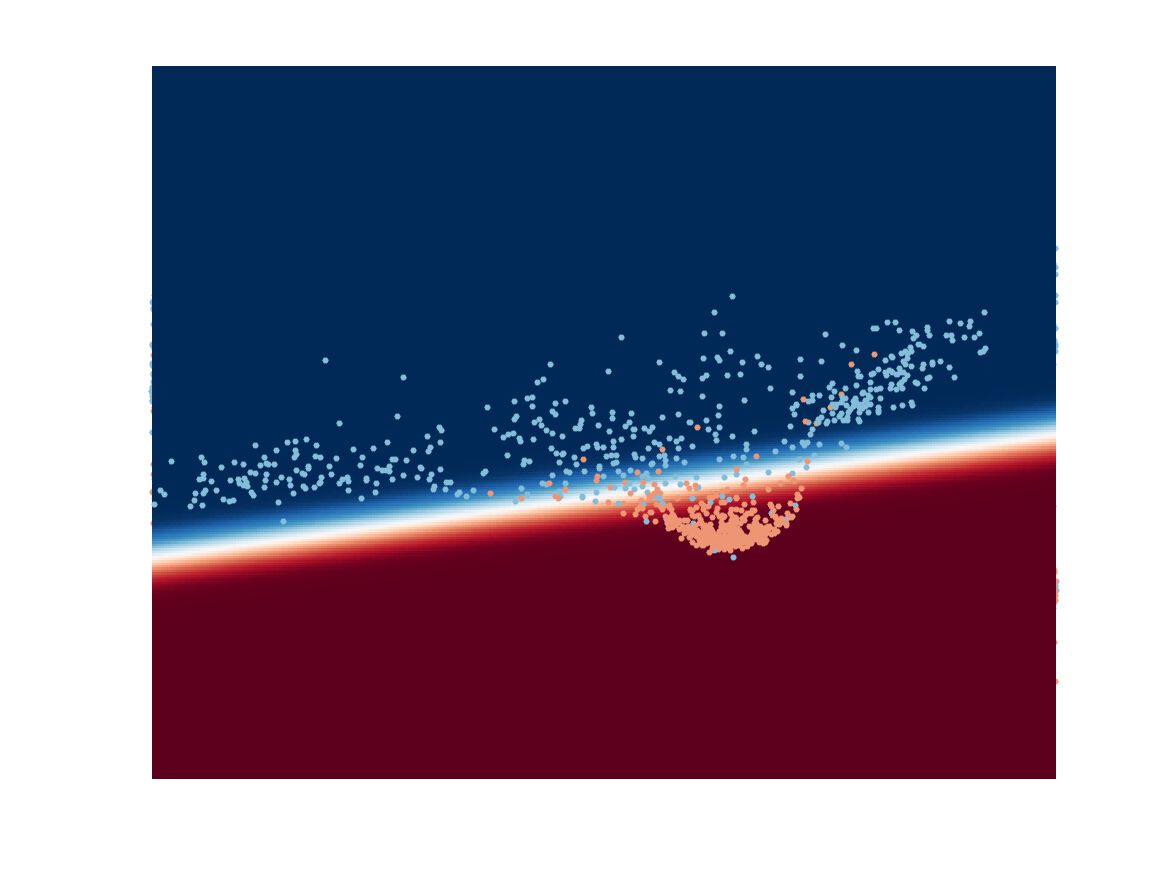}
{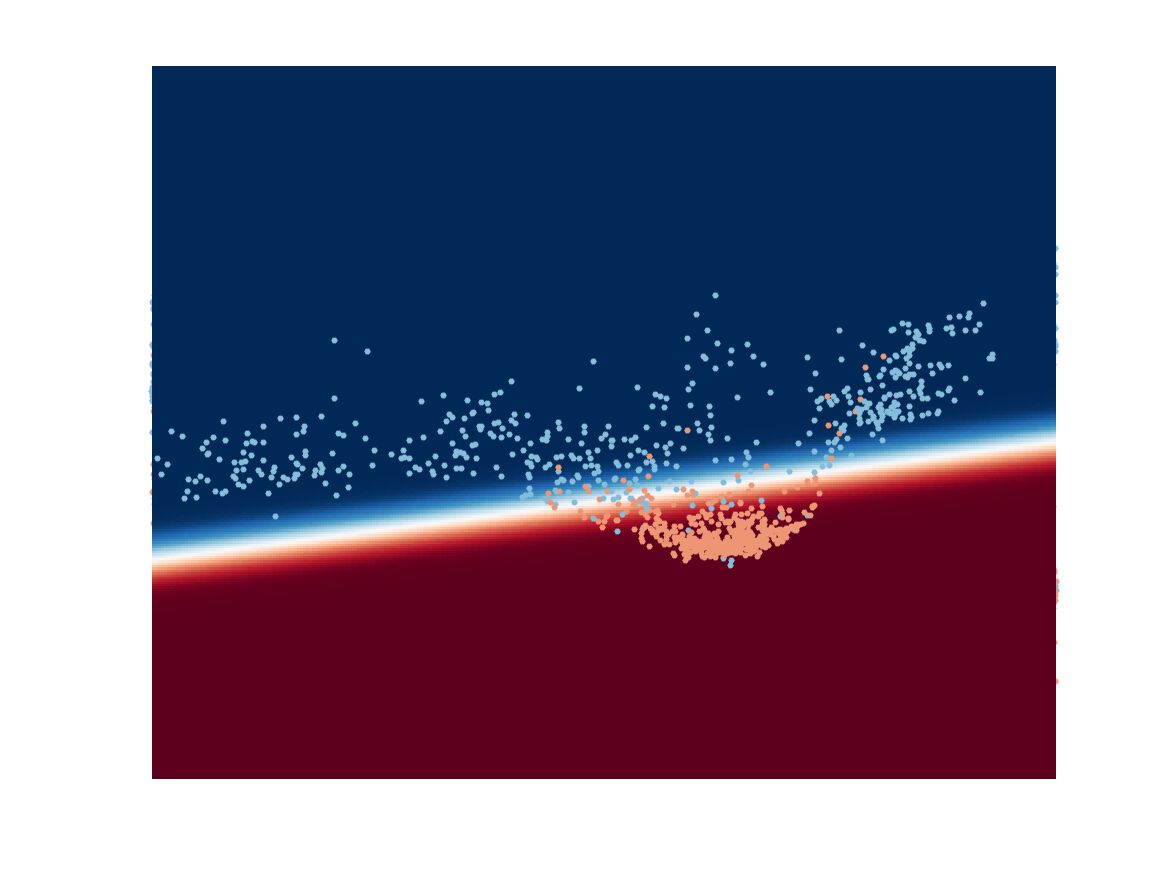}
{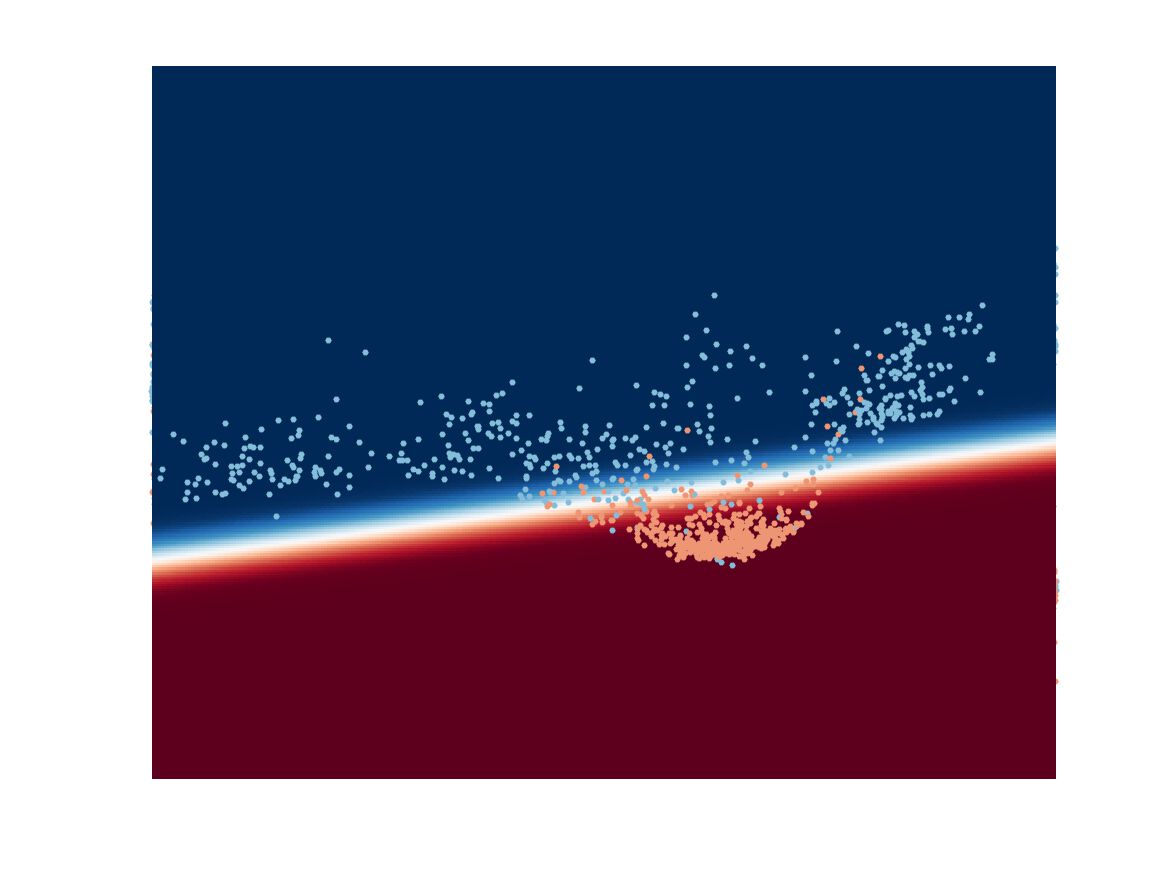}
\PlotLabels{ResNet/Euler}{Improved Euler}{Kutta(3)}{Kutta(4)}
\end{tikzpicture}

\begin{tikzpicture}
\PlotRow{0}{prediction}
{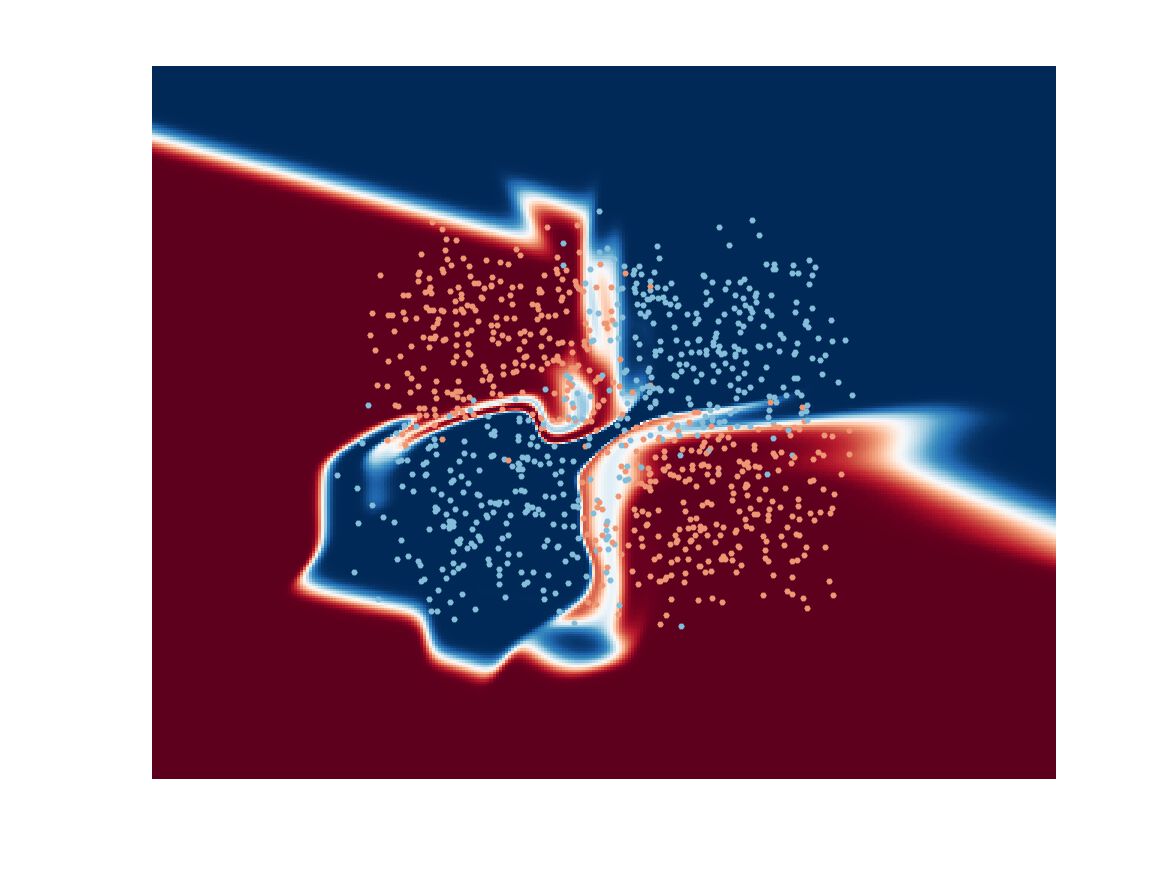}
{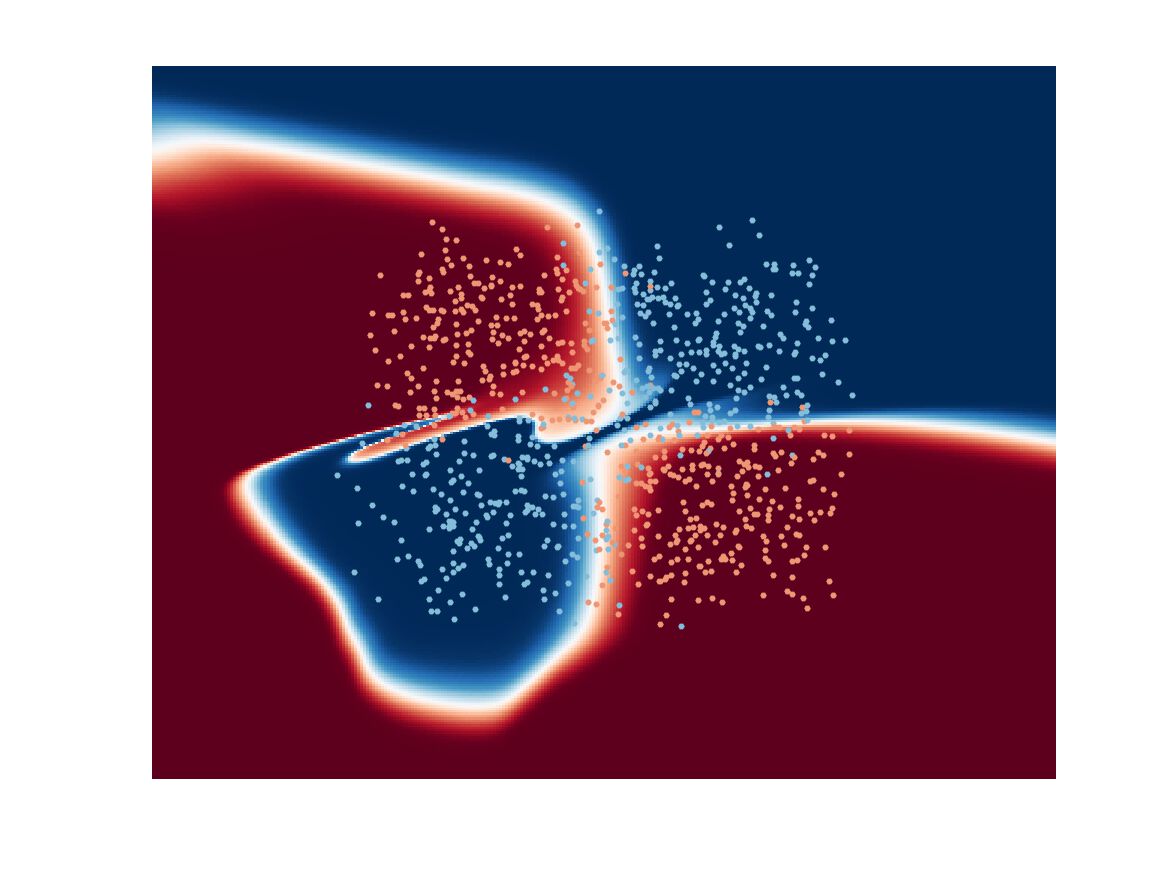}
{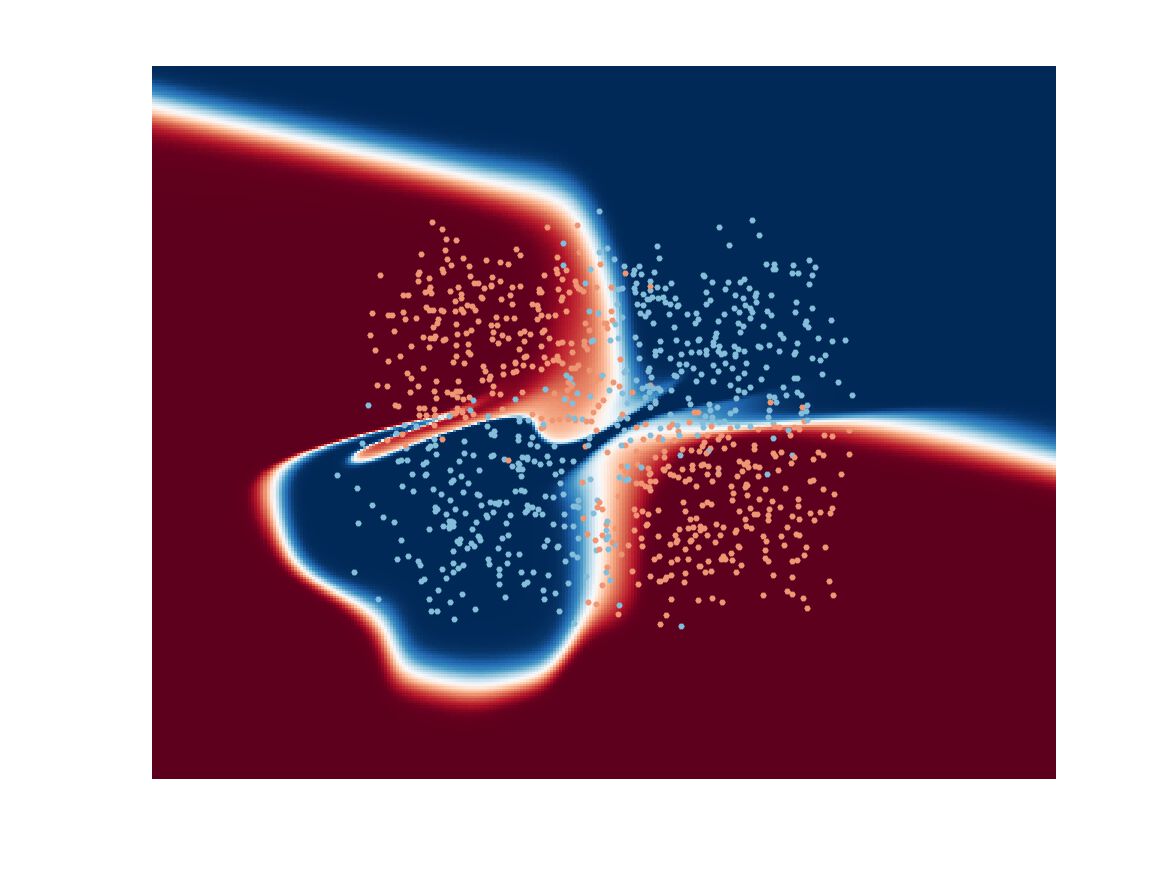}
{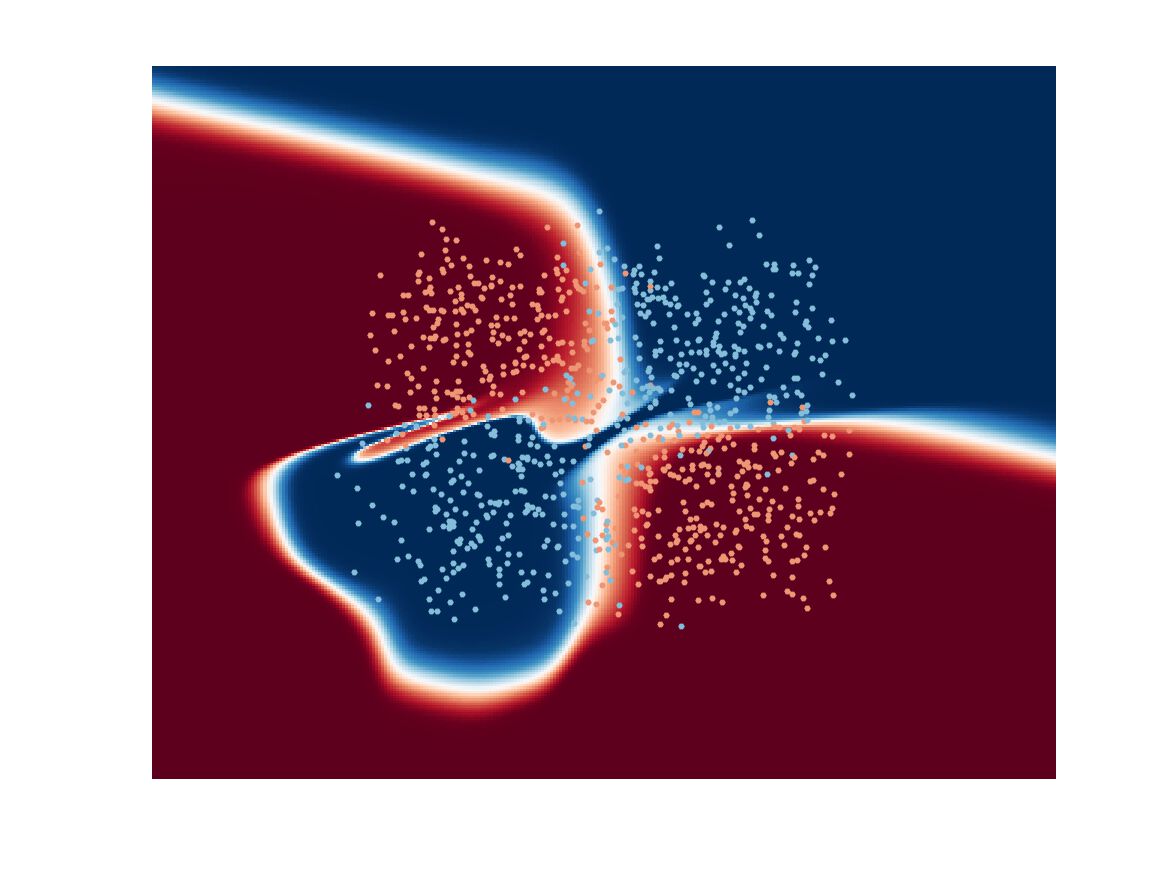}
\PlotRow{1}{transformation}
{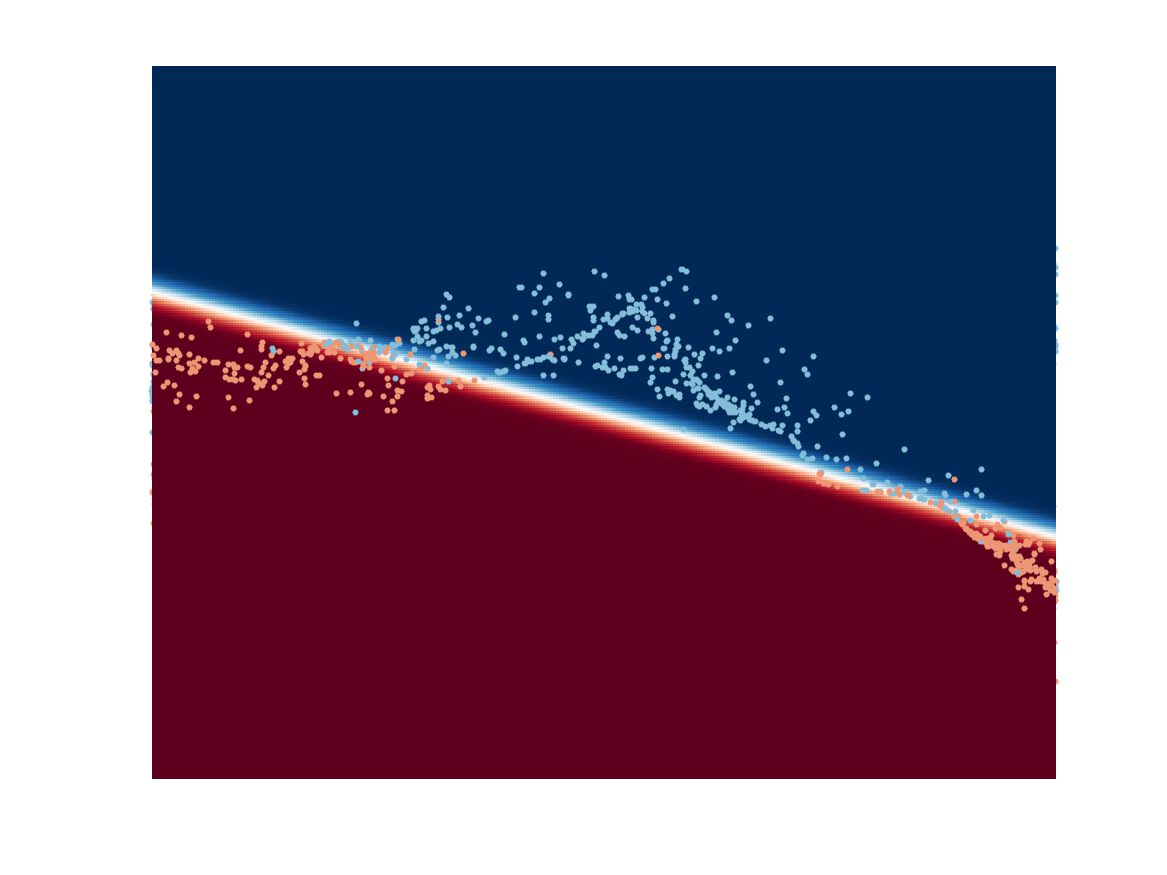}
{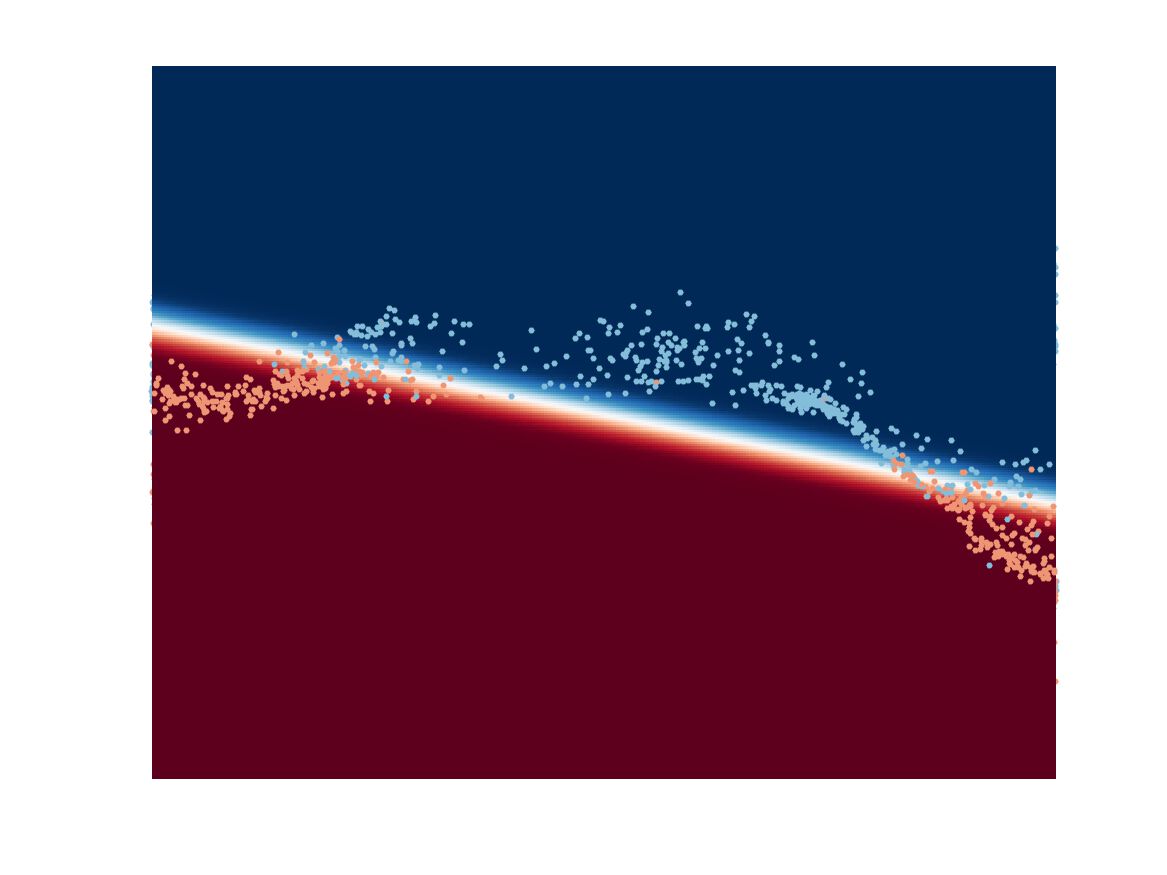}
{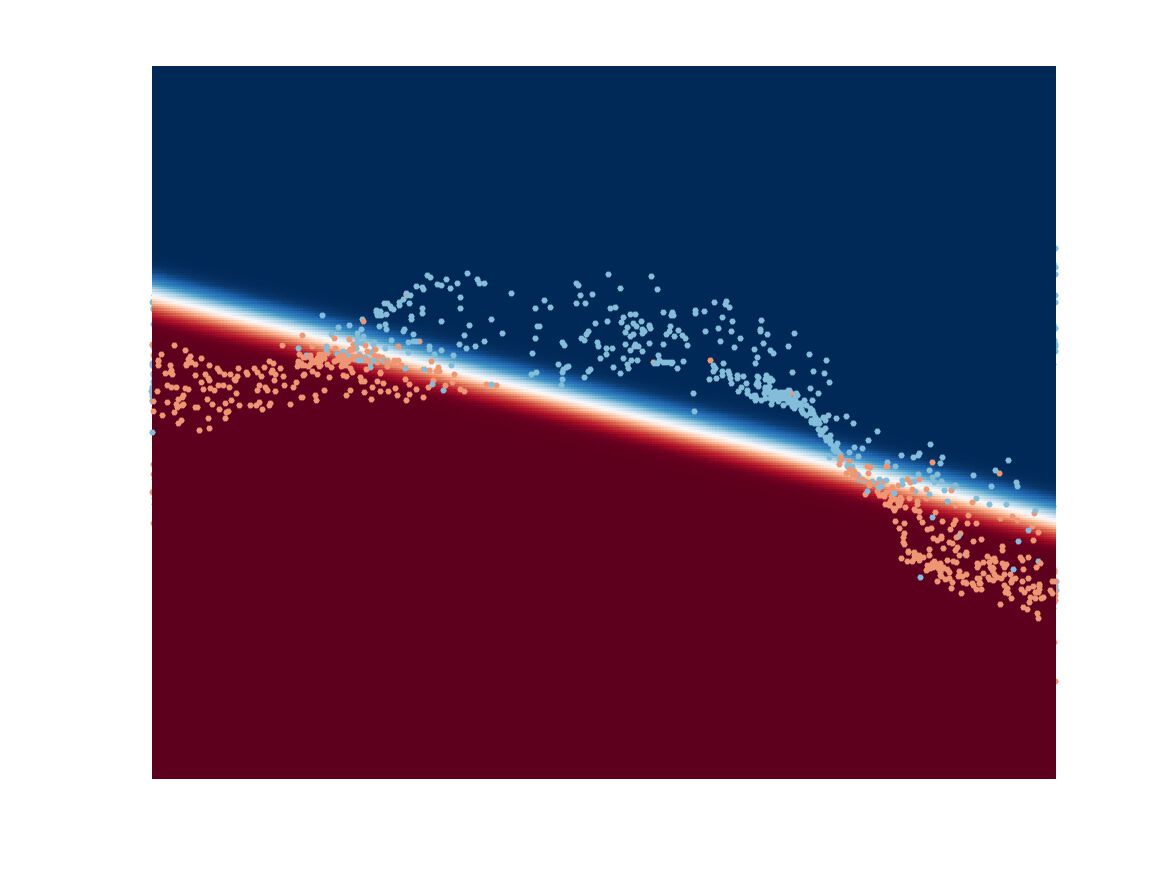}
{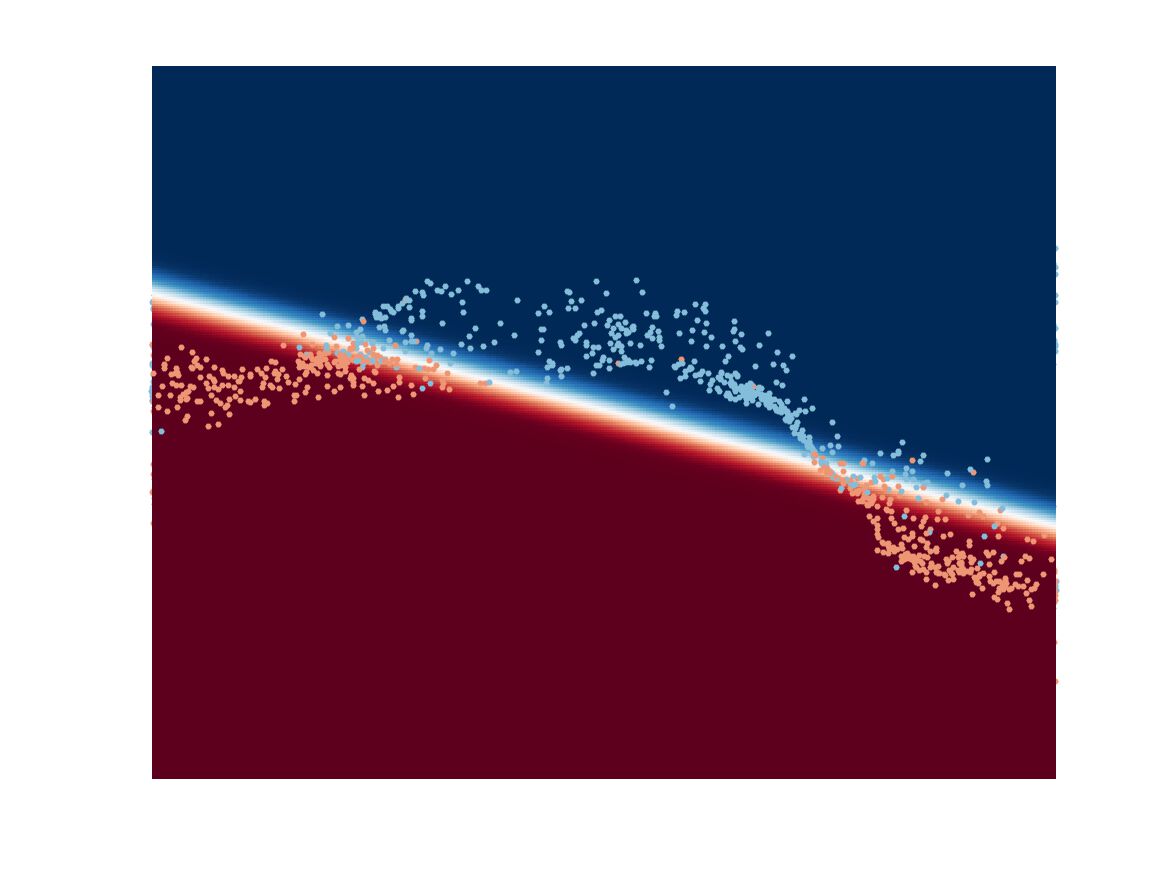}
\PlotLabels{ResNet/Euler}{Improved Euler}{Kutta(3)}{Kutta(4)}
\end{tikzpicture}
\caption{Learned prediction and transformation for different Runge--Kutta methods and data sets \DataSpiral~(top), \DataDonutB~(centre) and \DataSquares~(bottom). All results are for 15 layers.} \label{pred-trans}
\end{figure}

\begin{figure}
\def\RowHeight{3cm}
\renewcommand{\PlotIm}[1]{{\includegraphics[width=\PicWidth, clip, trim=3cm 1.5cm 2cm 1cm]{RKSnaps/#1}}}%

\begin{tikzpicture}
\PlotRow{0}{ResNet/Euler}
{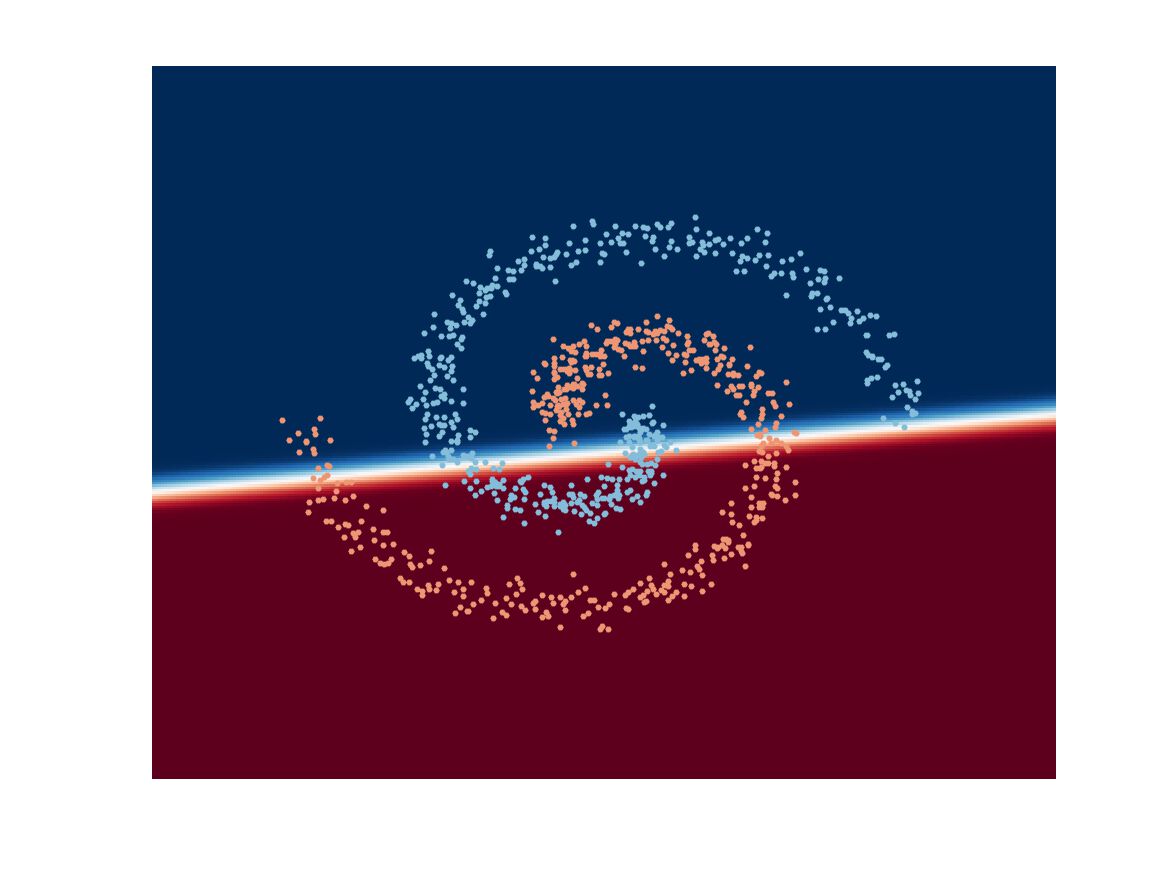}
{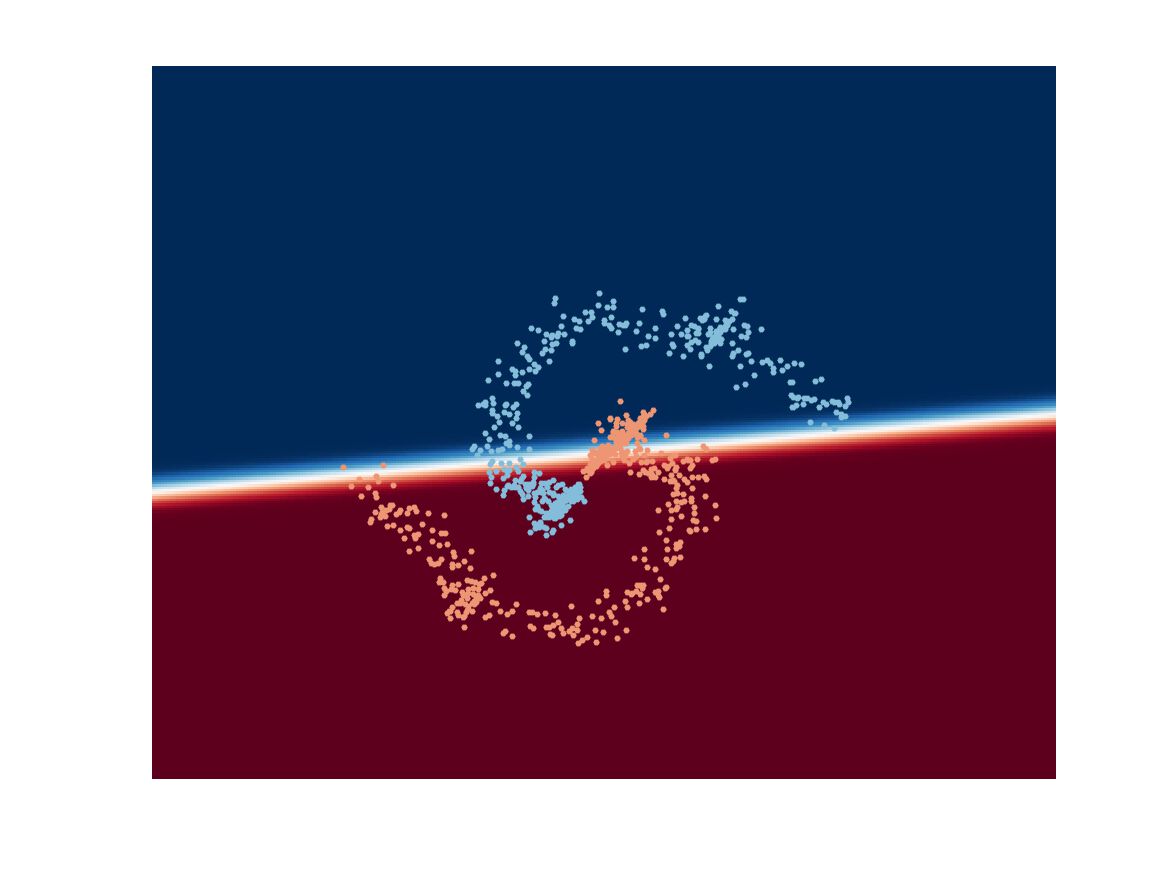}
{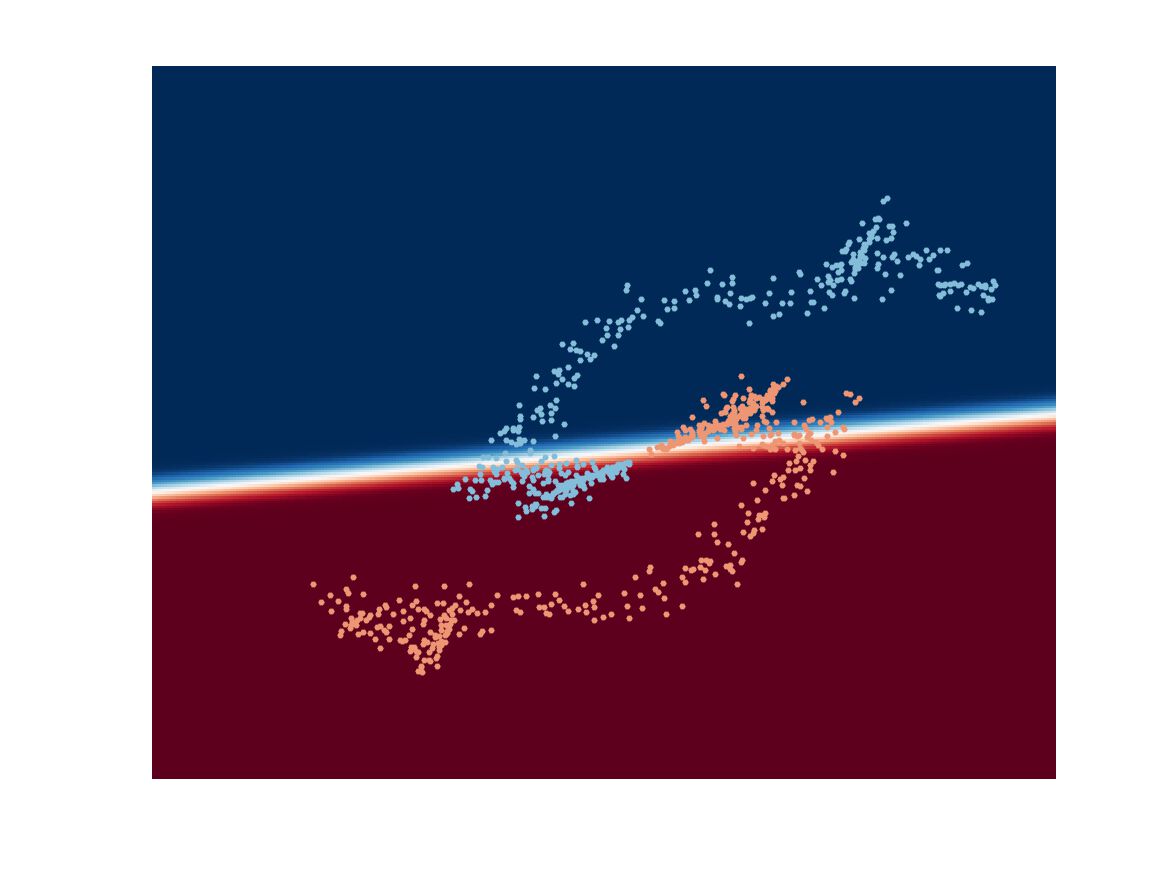}
{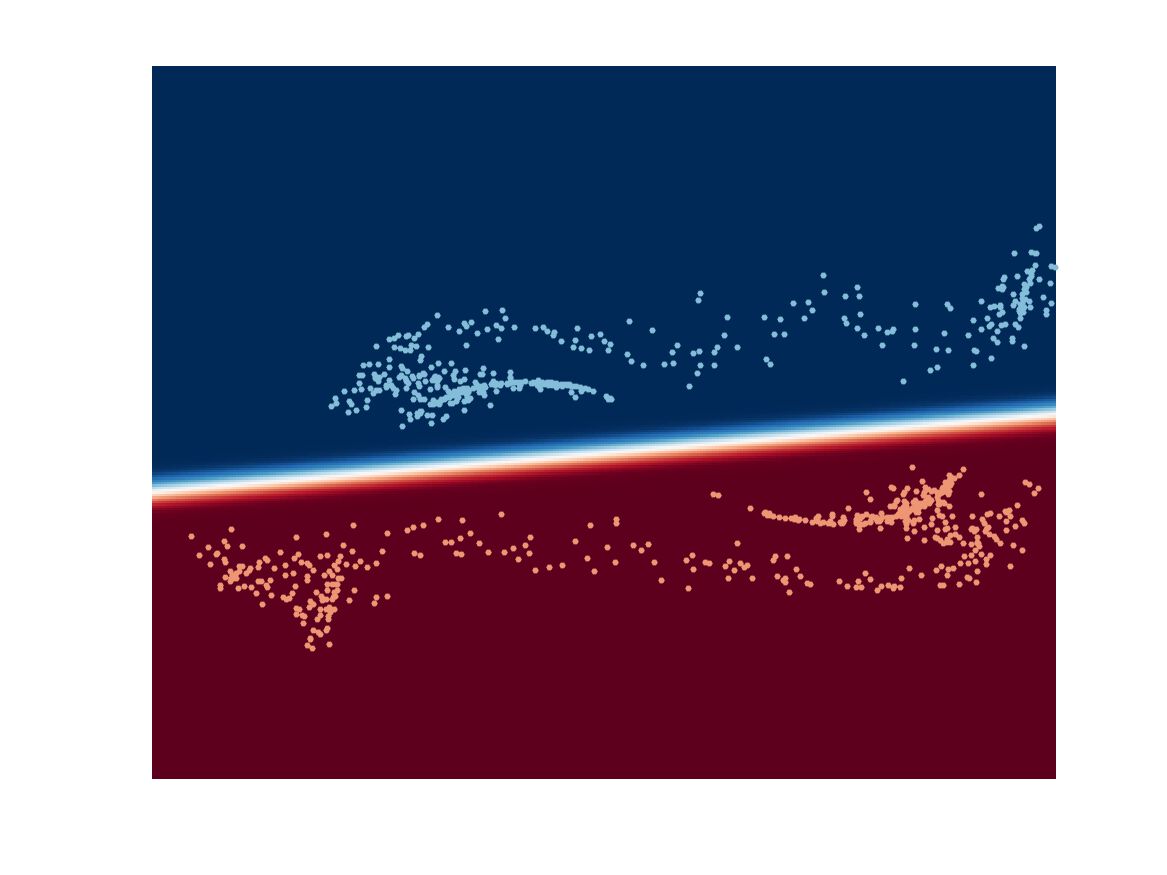}
\PlotRow{1}{Improved Euler}
{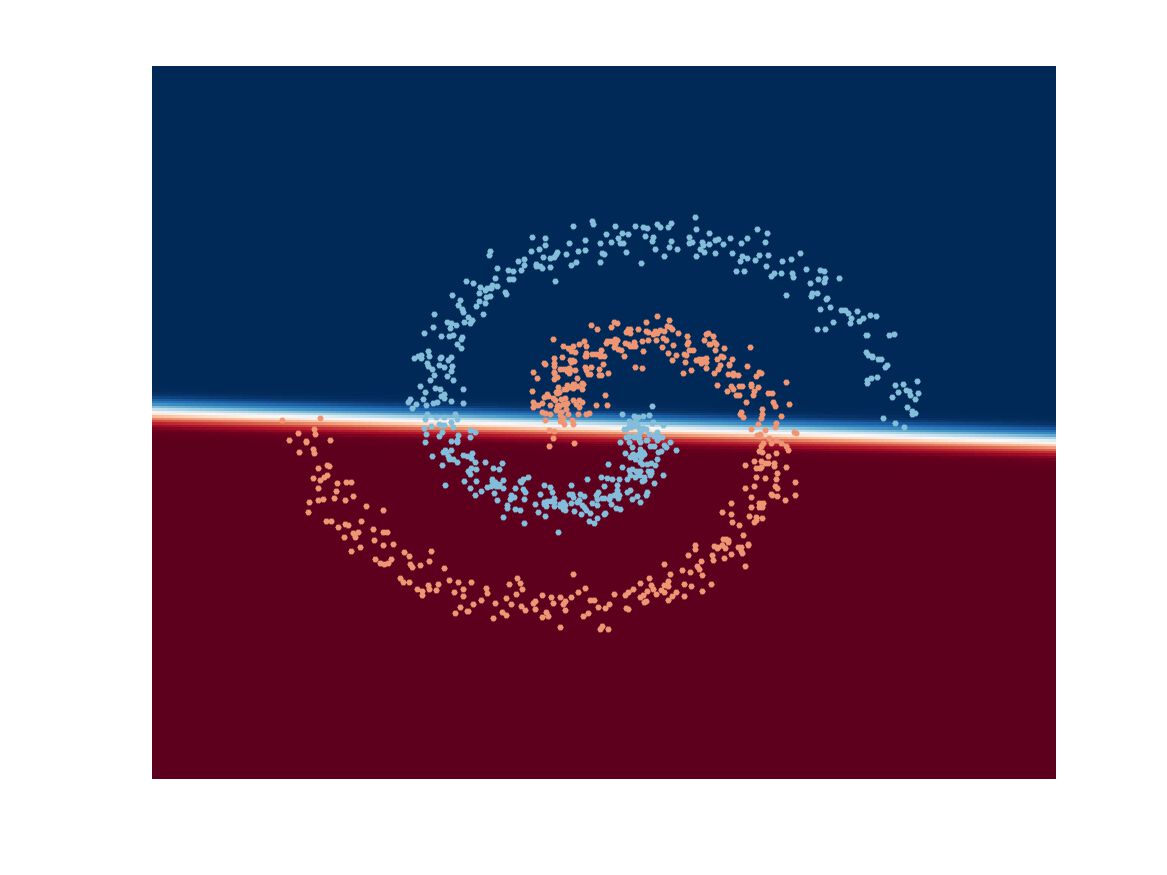}
{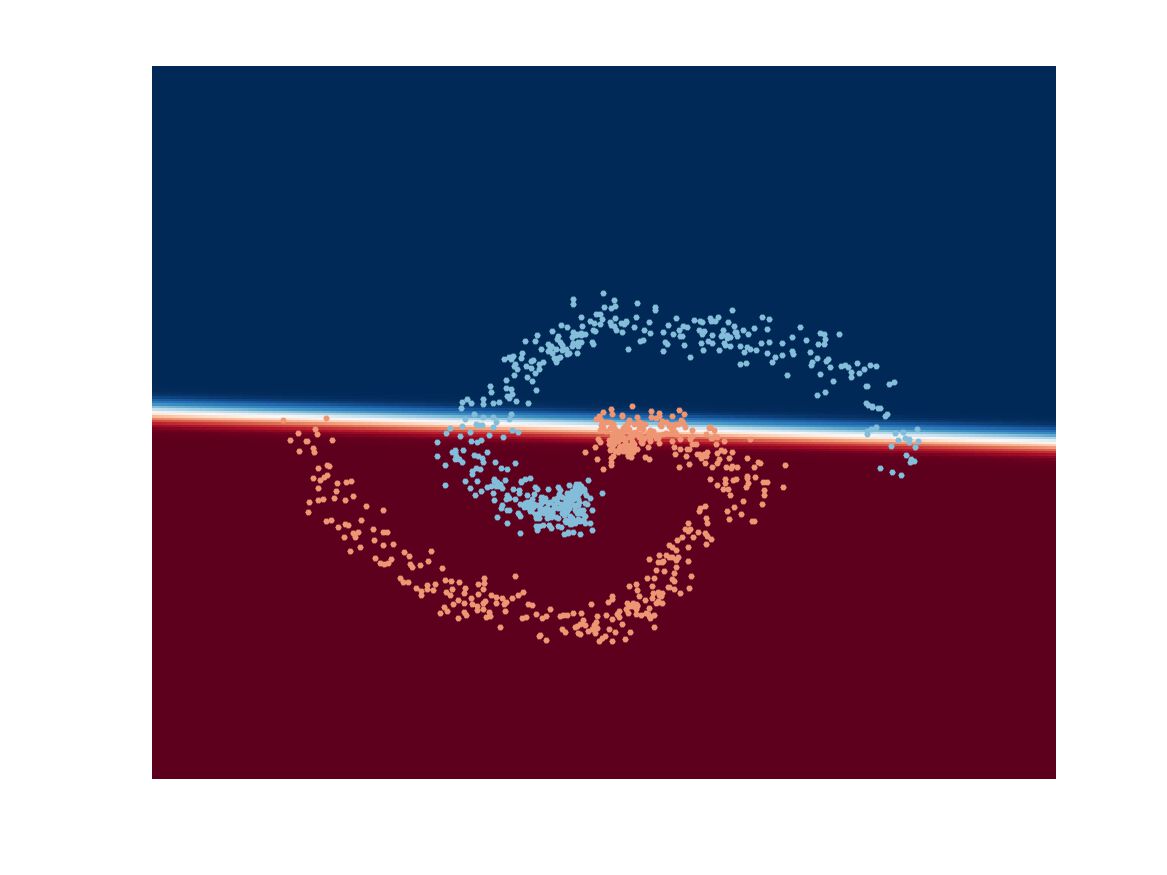}
{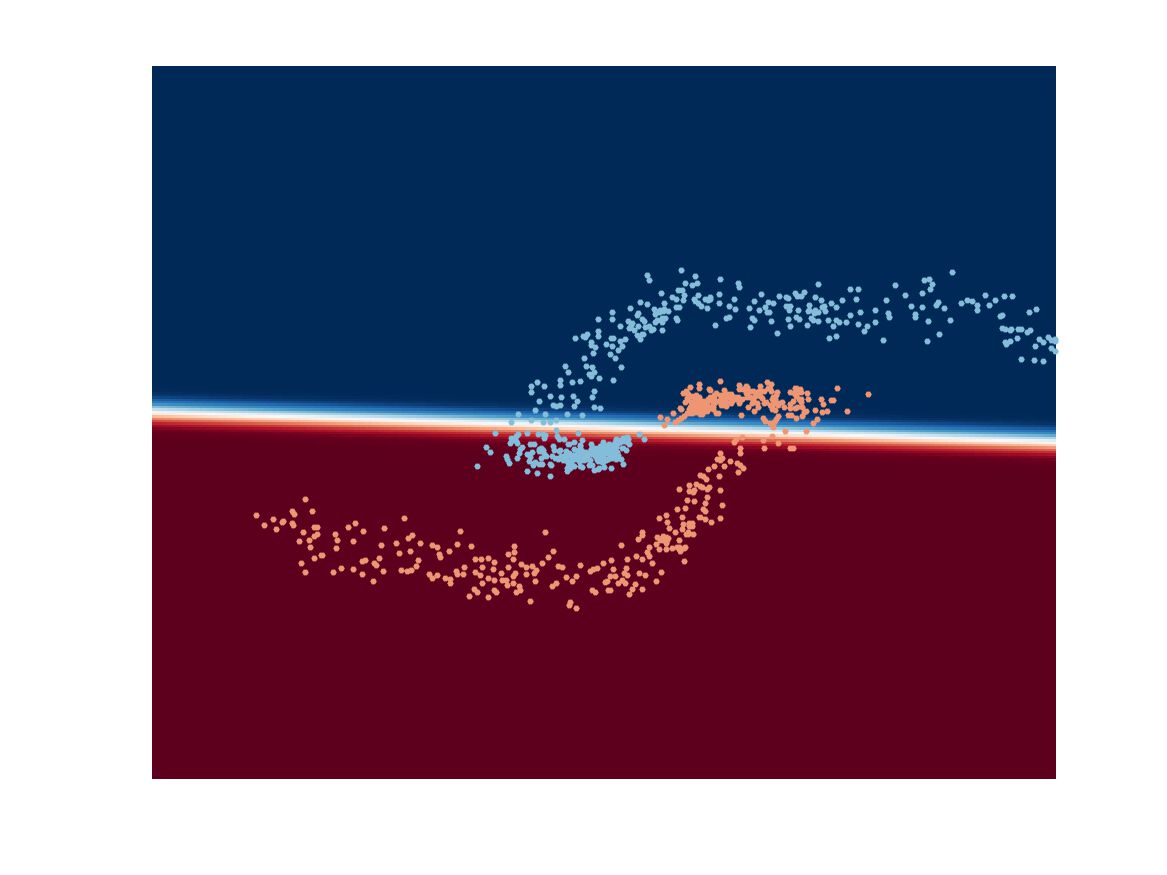}
{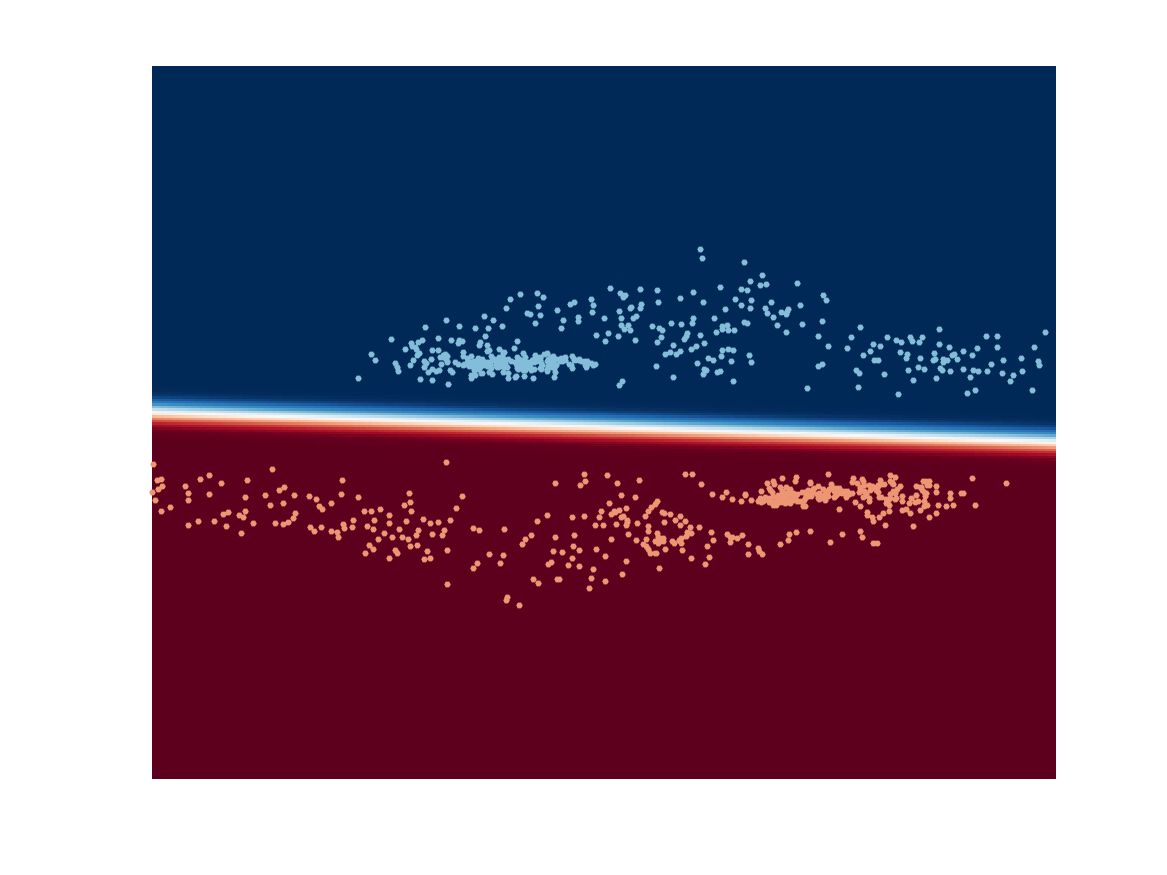}
\PlotRow{2}{Kutta (3)}
{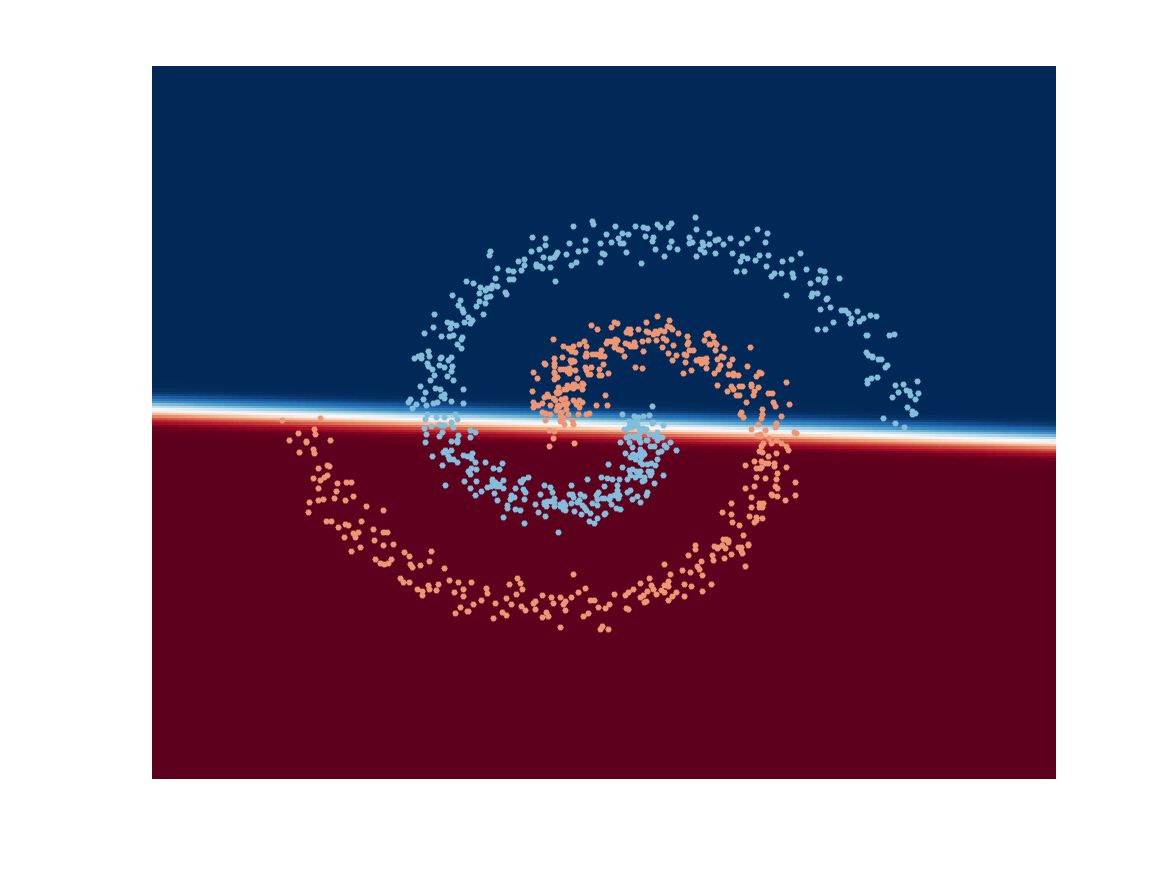}
{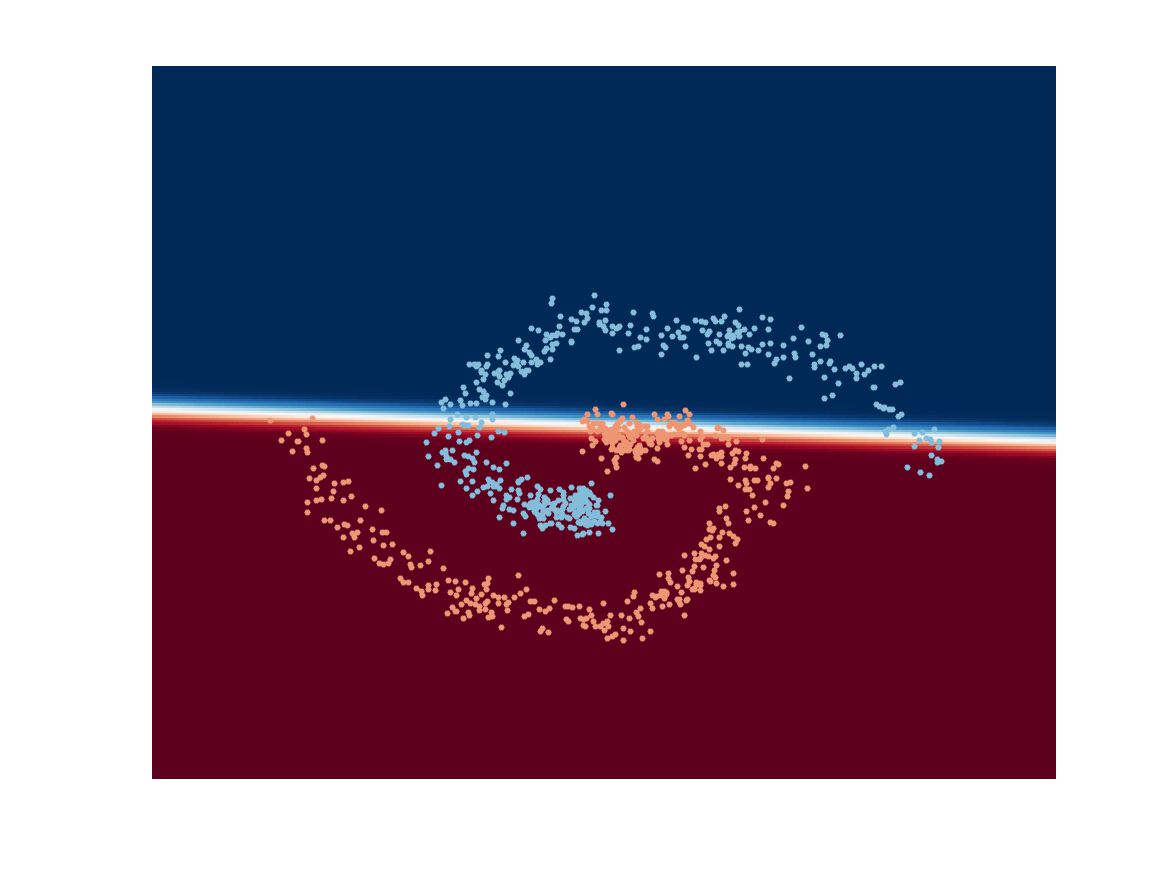}
{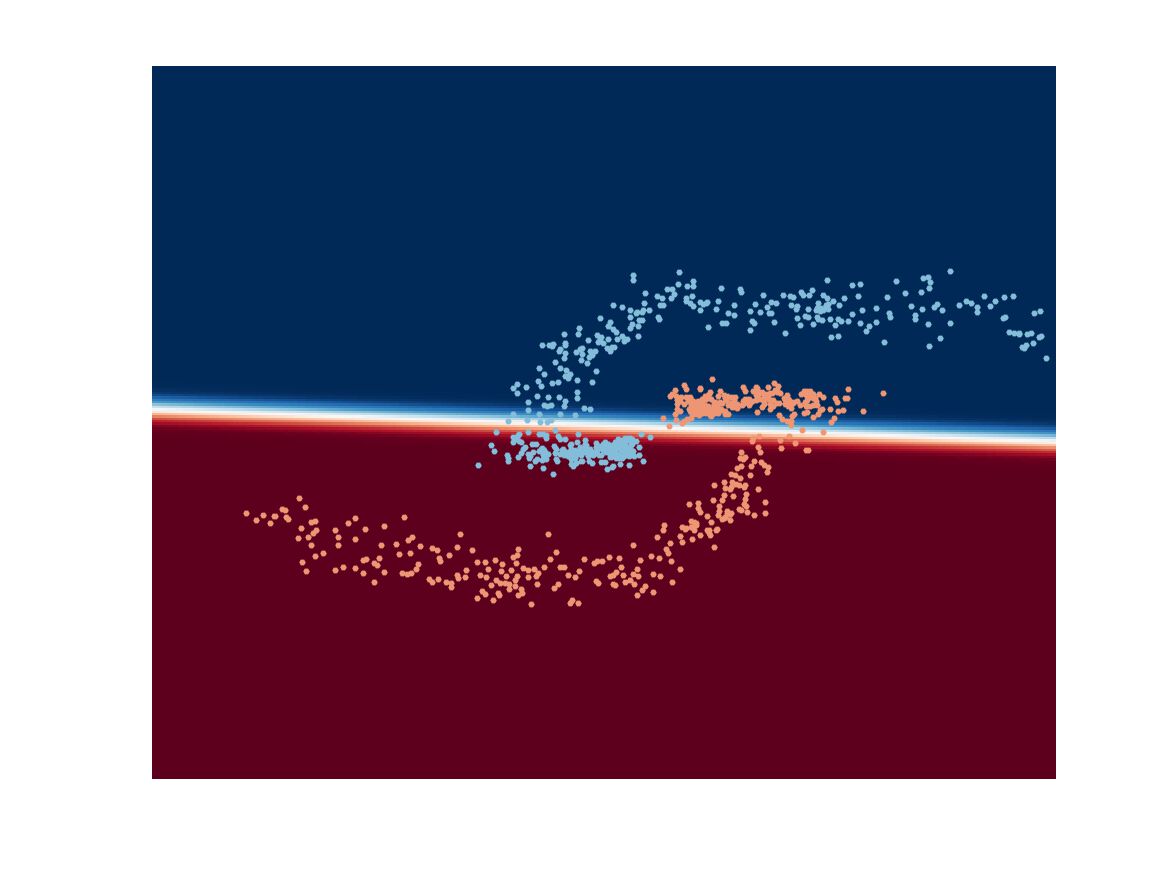}
{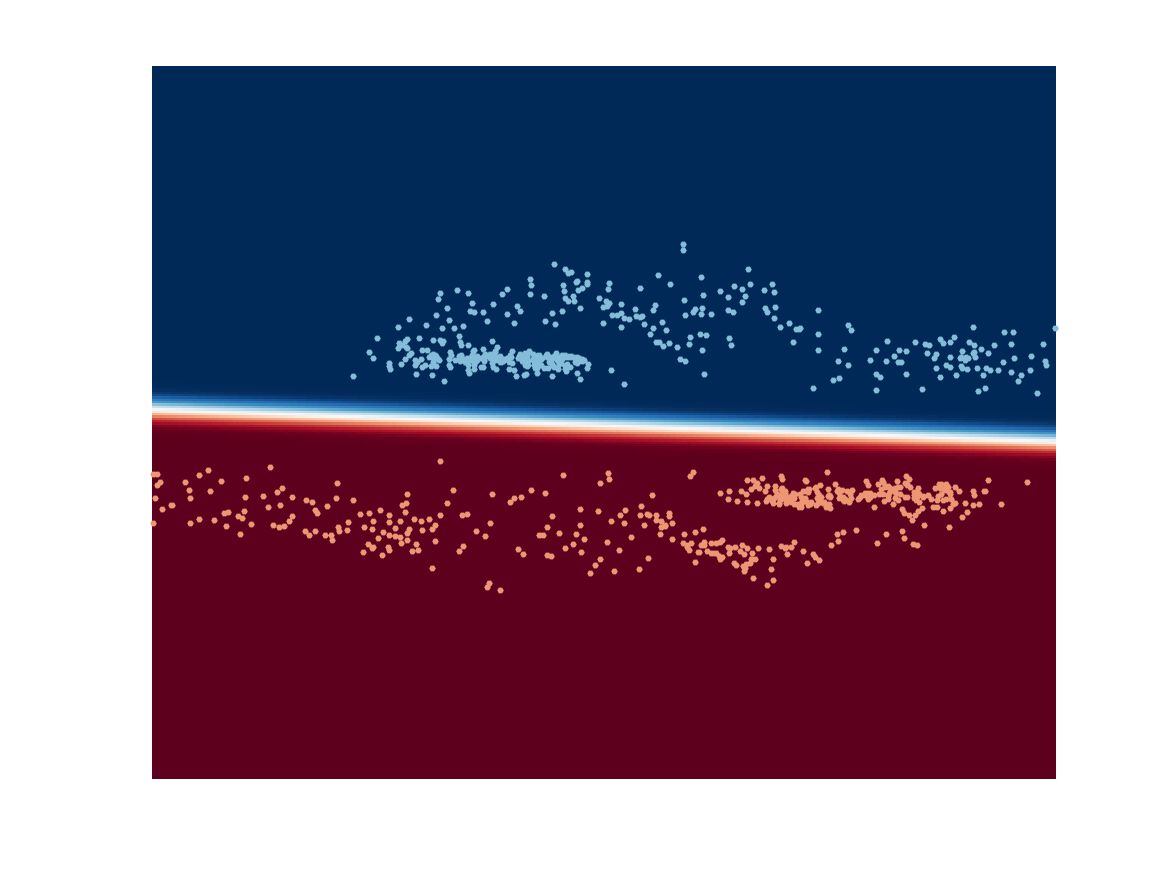}
\PlotRow{3}{Kutta (4)}
{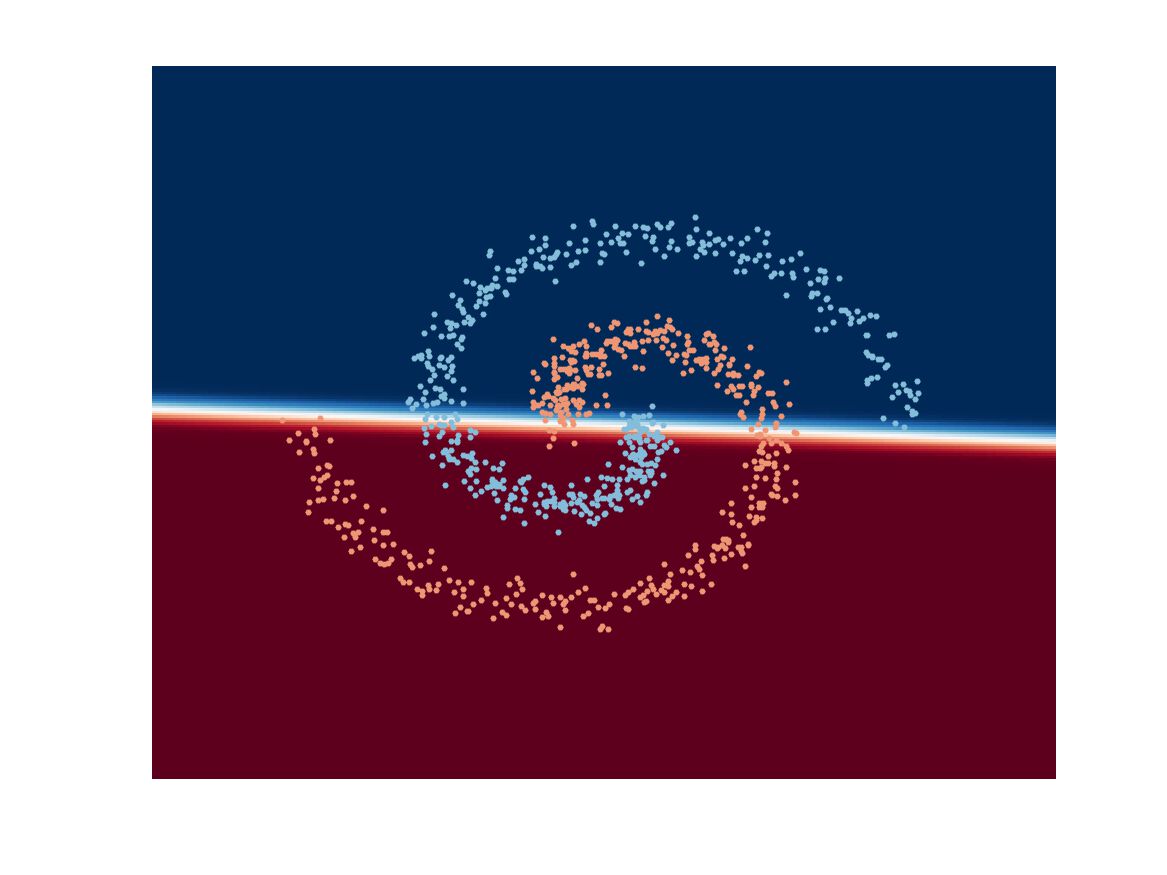}
{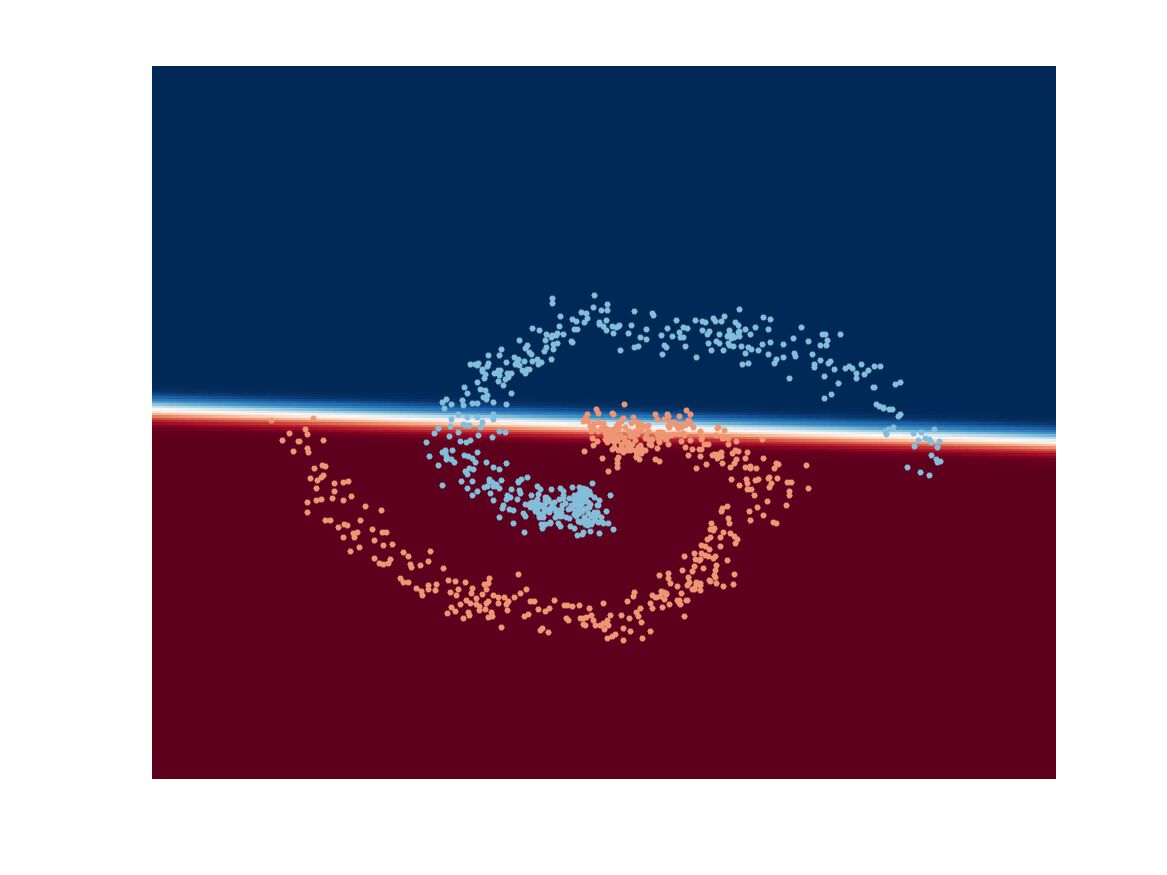}
{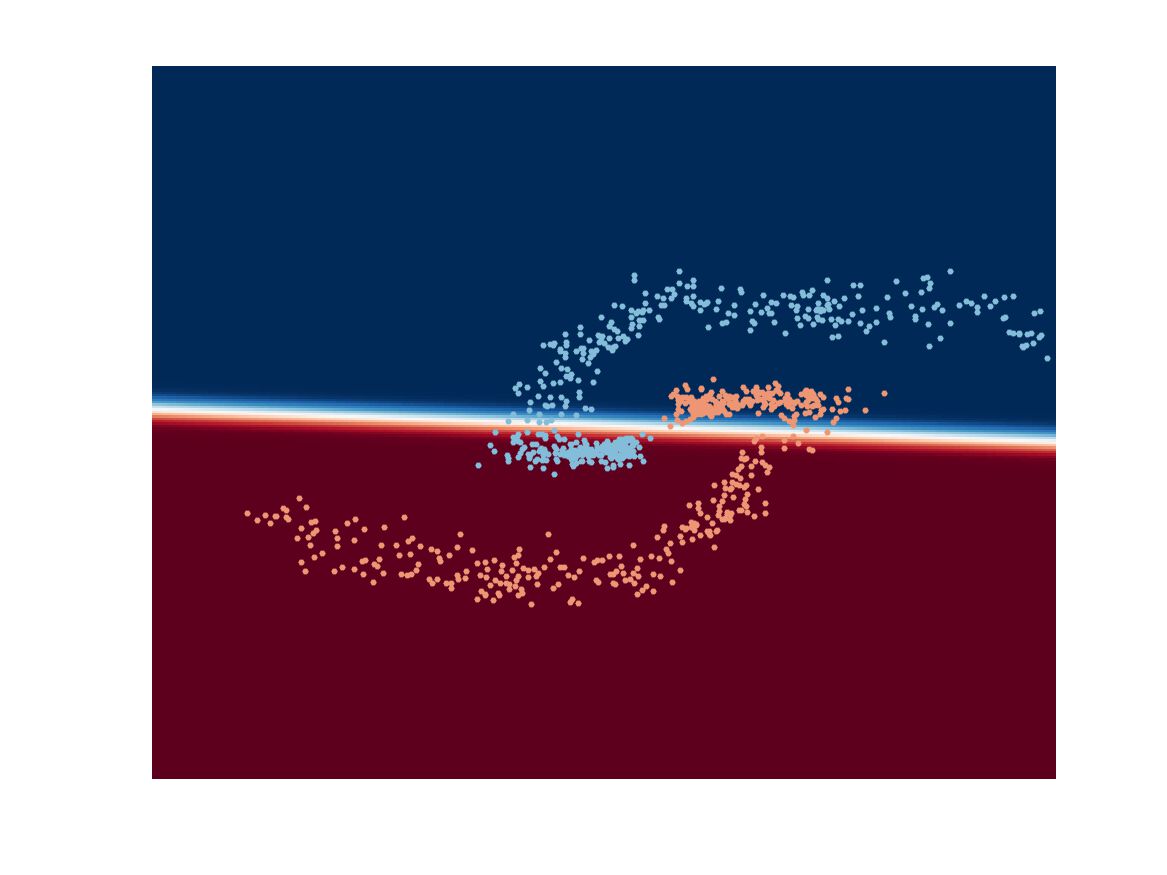}
{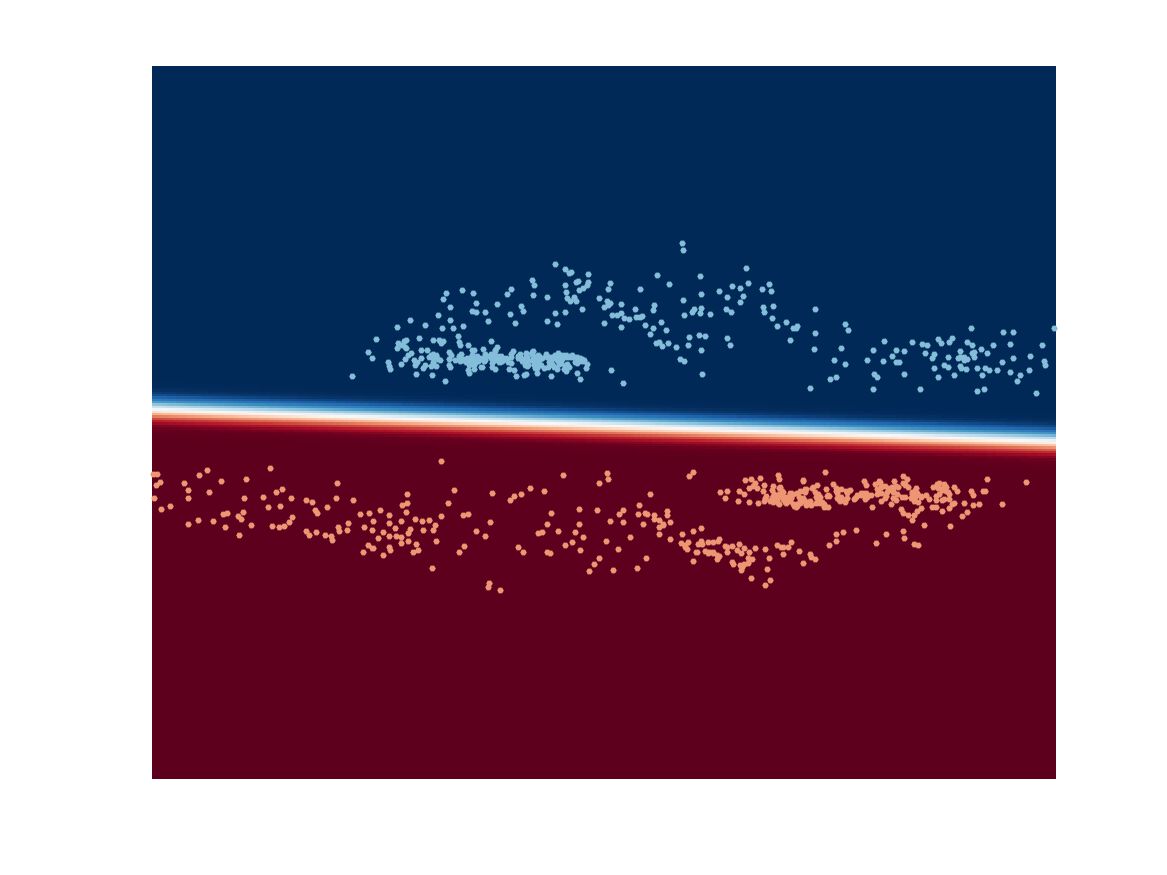}
\PlotLabels{time 0}{}{}{time $T$}
\end{tikzpicture}
\caption{Snap shots of the transition from initial to final state through the network with the data set \DataSpiral.} \label{anim-spiral2d}
\end{figure}

\begin{figure}
\def\RowHeight{3cm}
\renewcommand{\PlotIm}[1]{{\includegraphics[width=\PicWidth, clip, trim=3cm 1.5cm 2cm 1cm]{RKSnaps/#1}}}%

\begin{tikzpicture}
\PlotRow{0}{ResNet/Euler}
{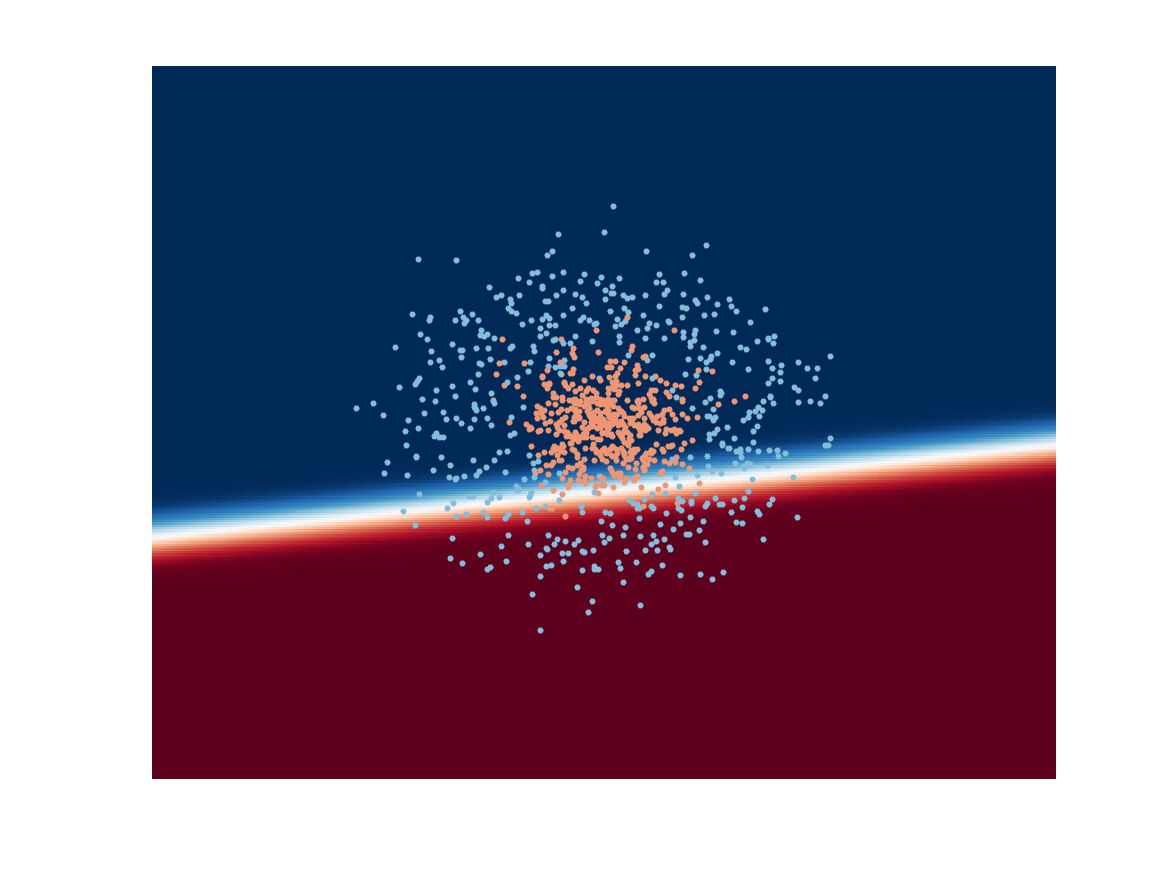}
{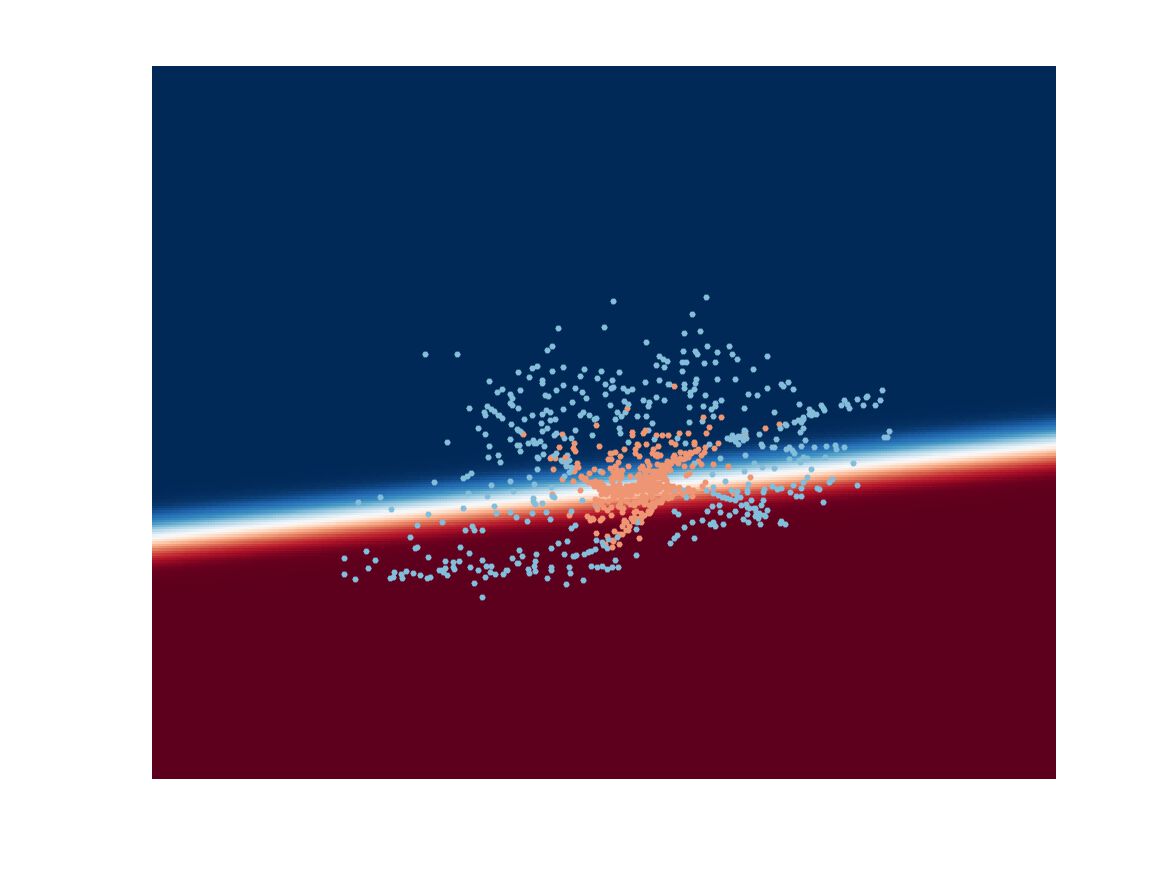}
{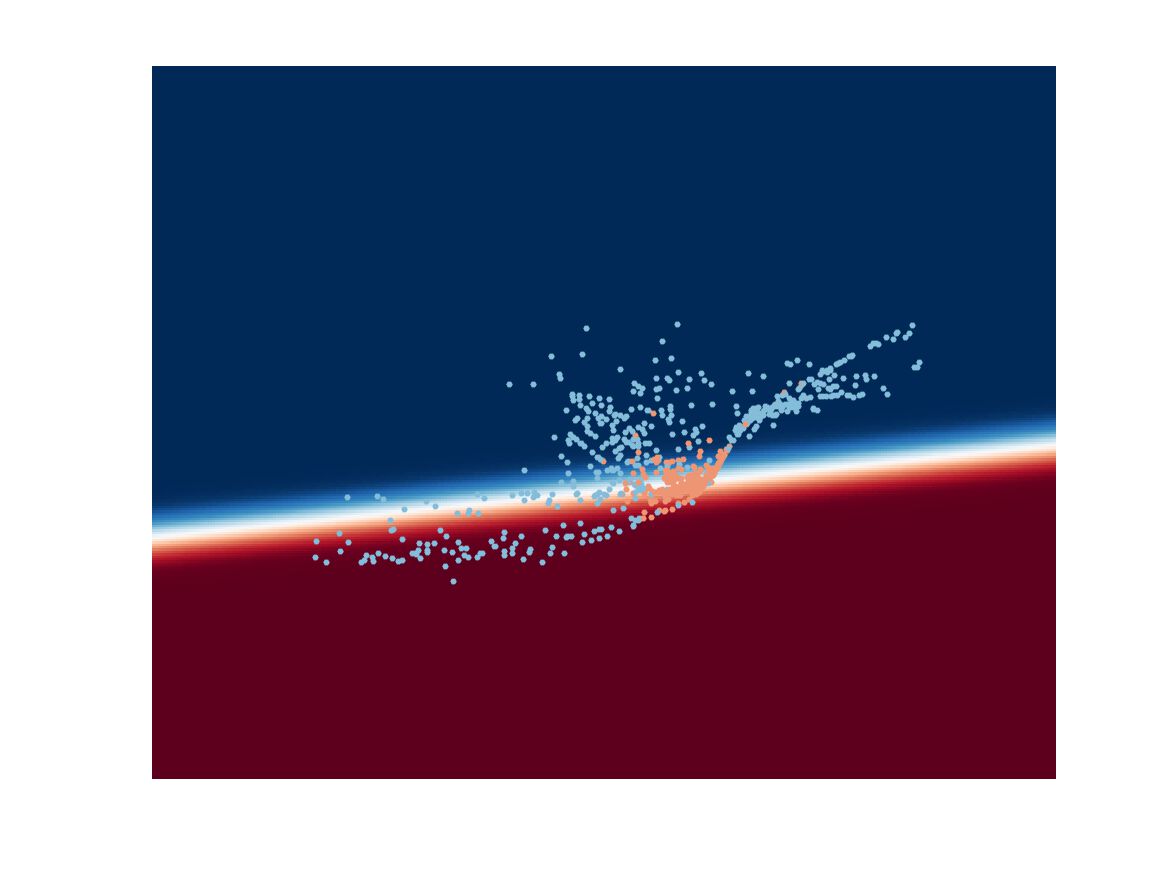}
{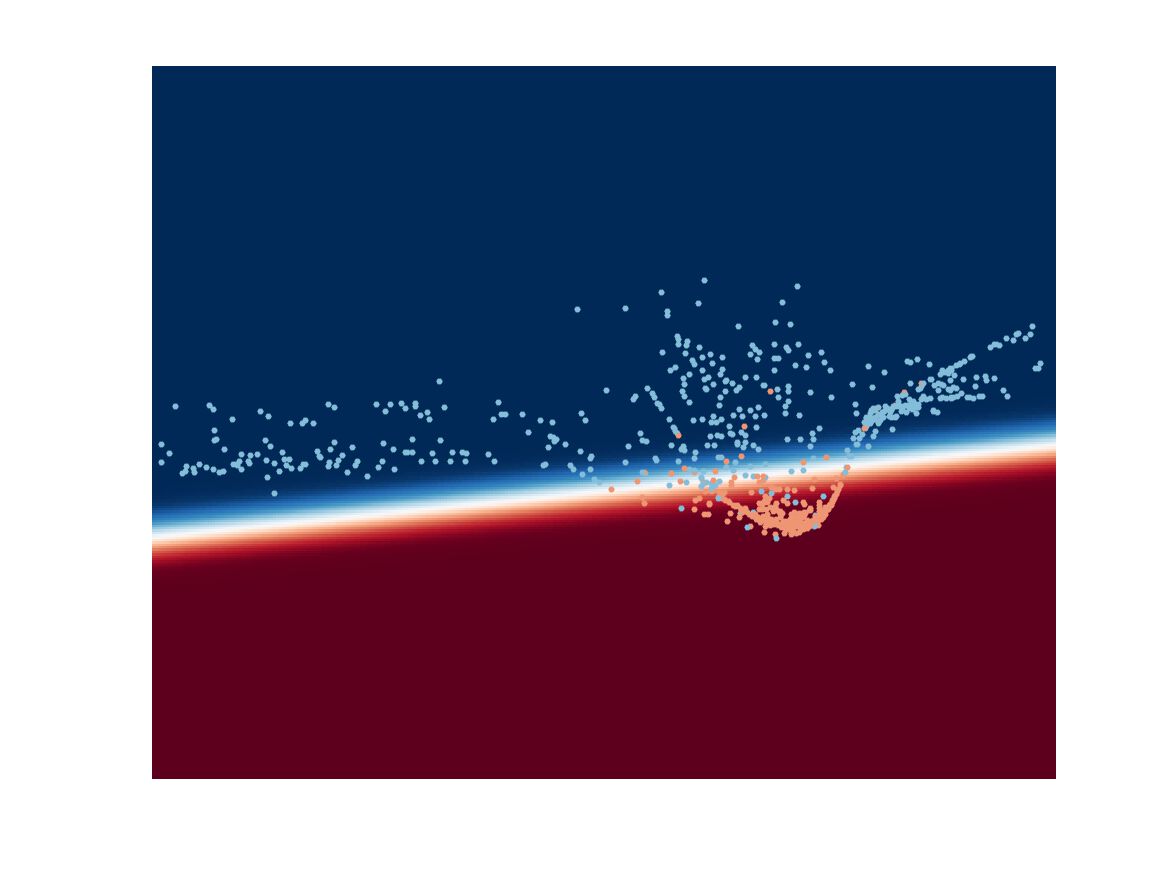}
\PlotRow{1}{Improved Euler}
{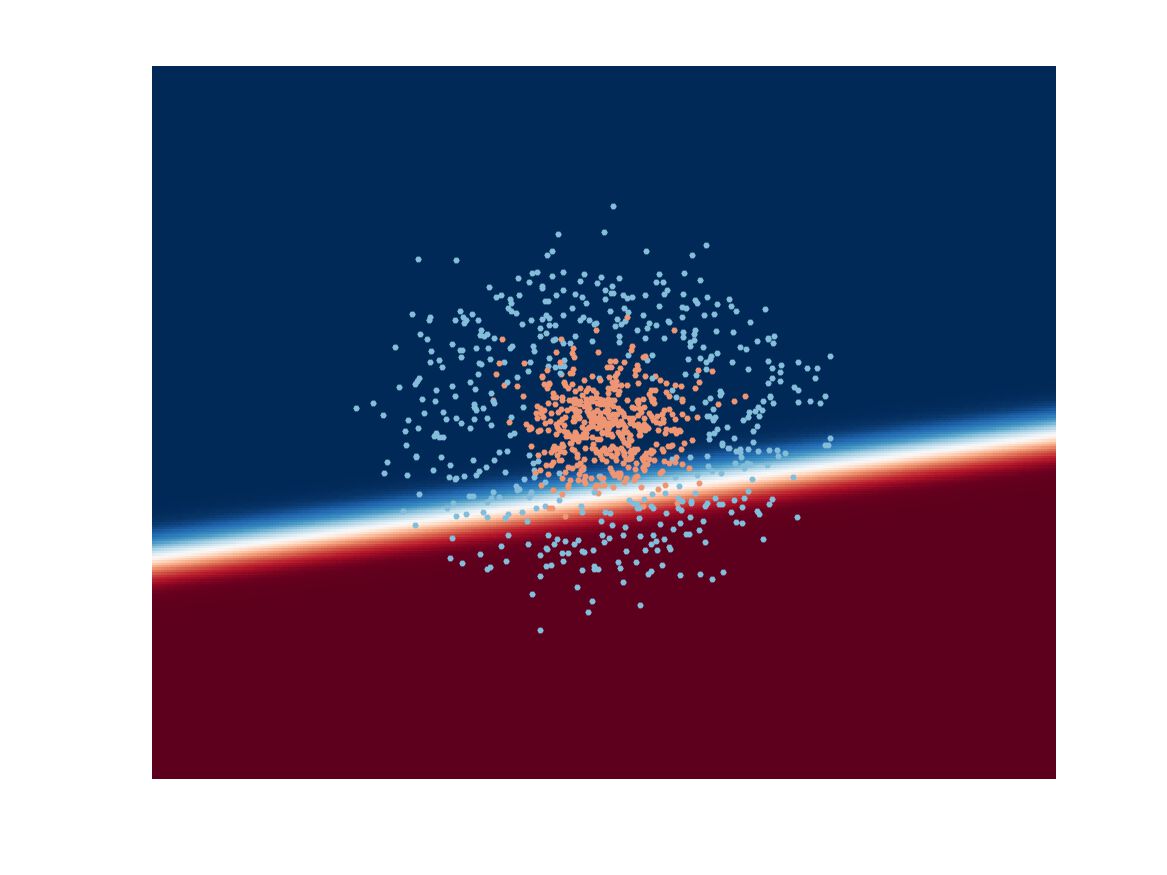}
{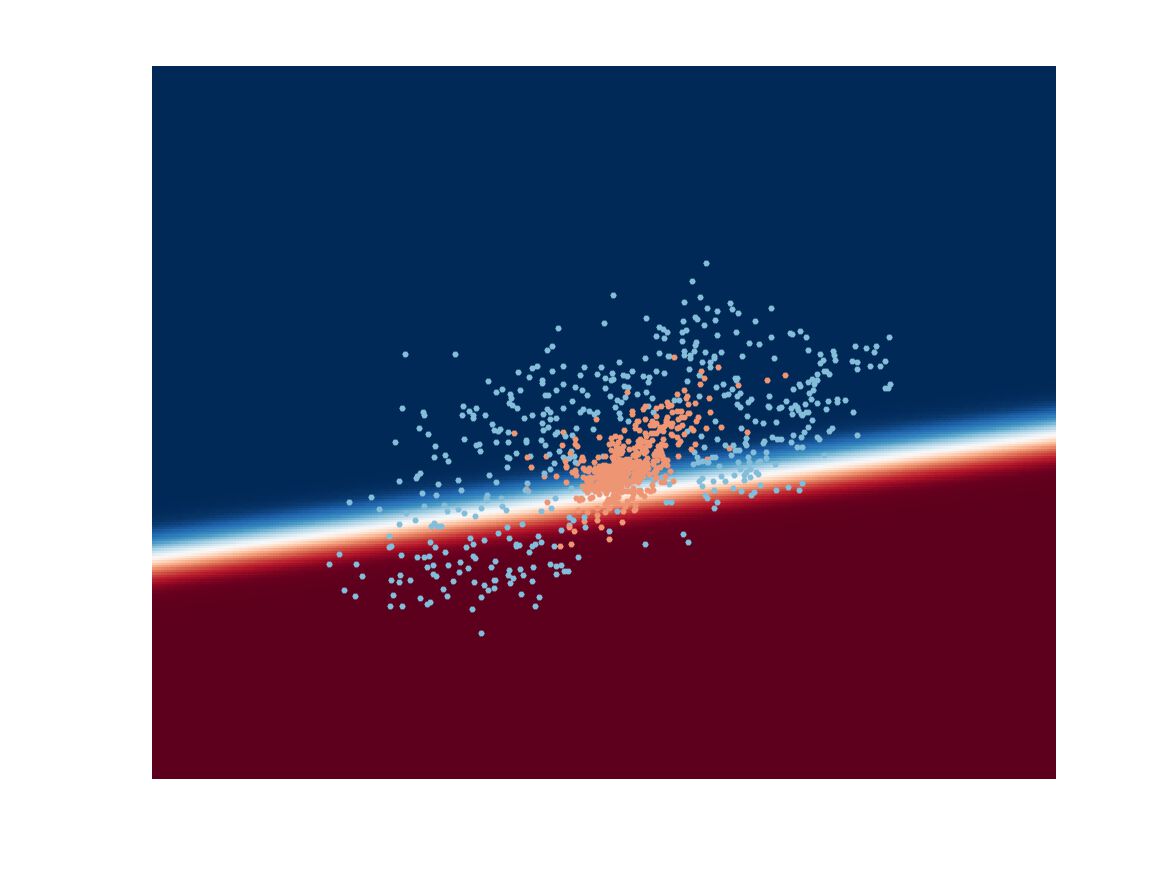}
{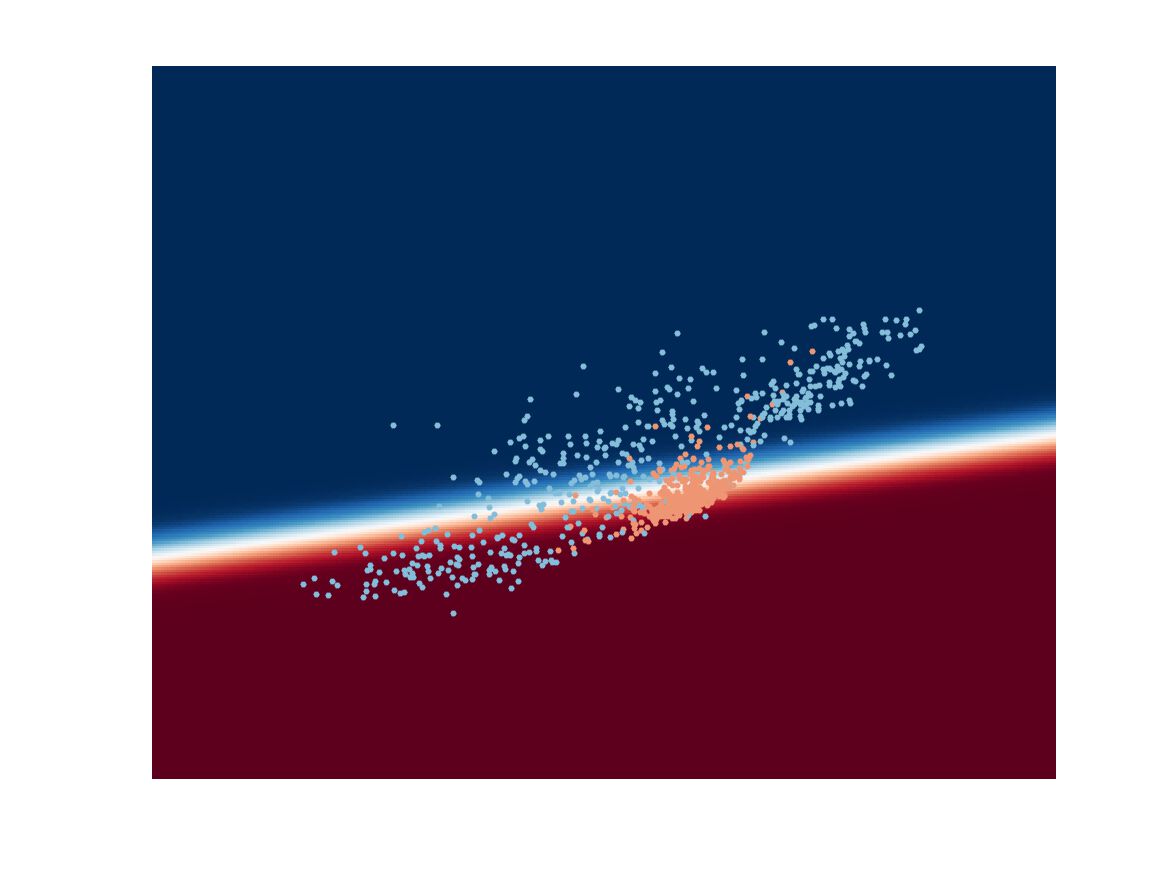}
{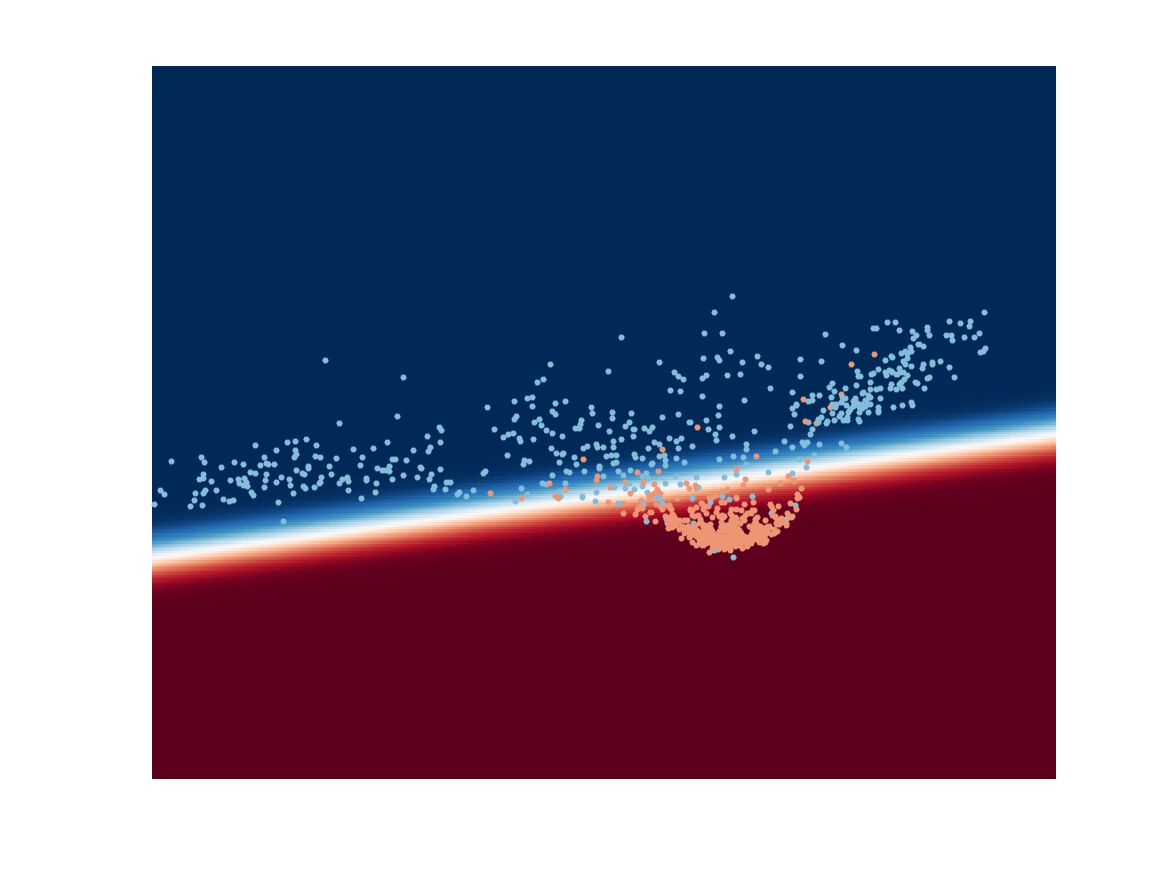}
\PlotRow{2}{Kutta (3)}
{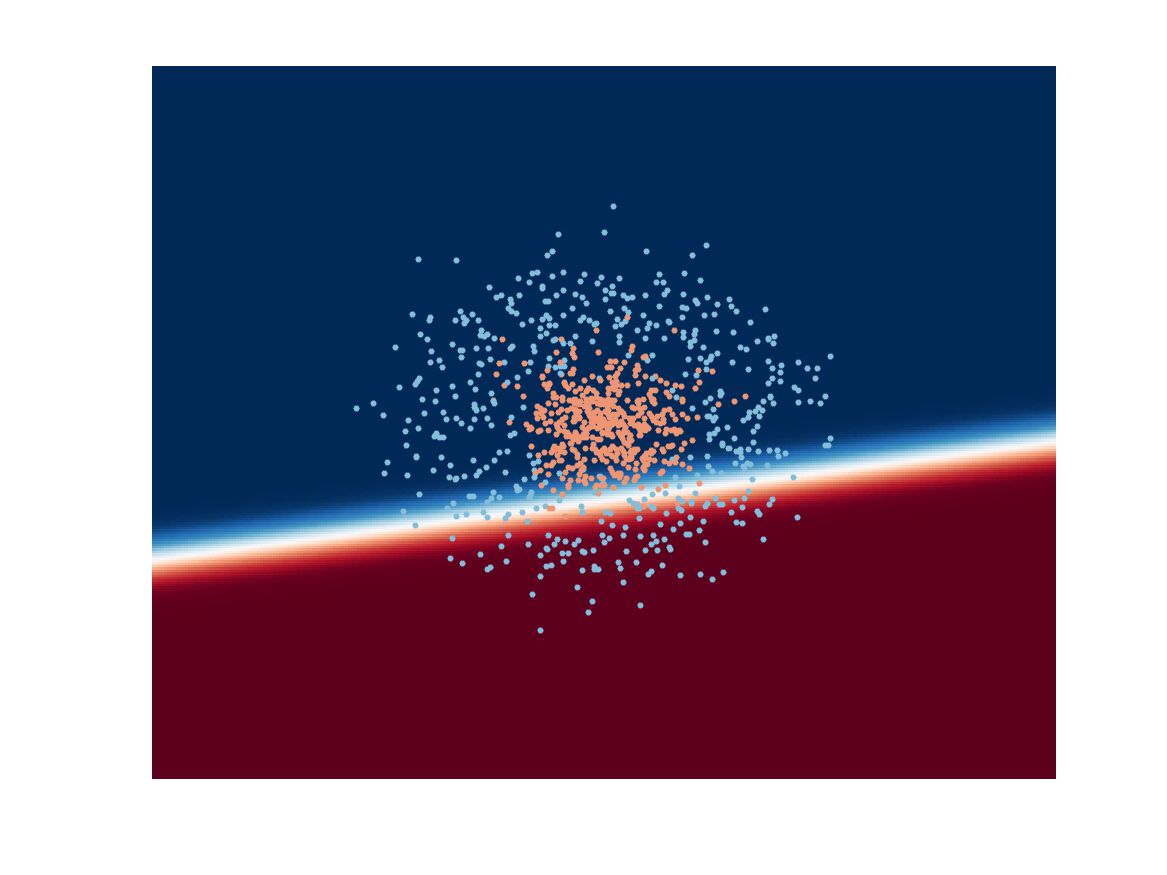}
{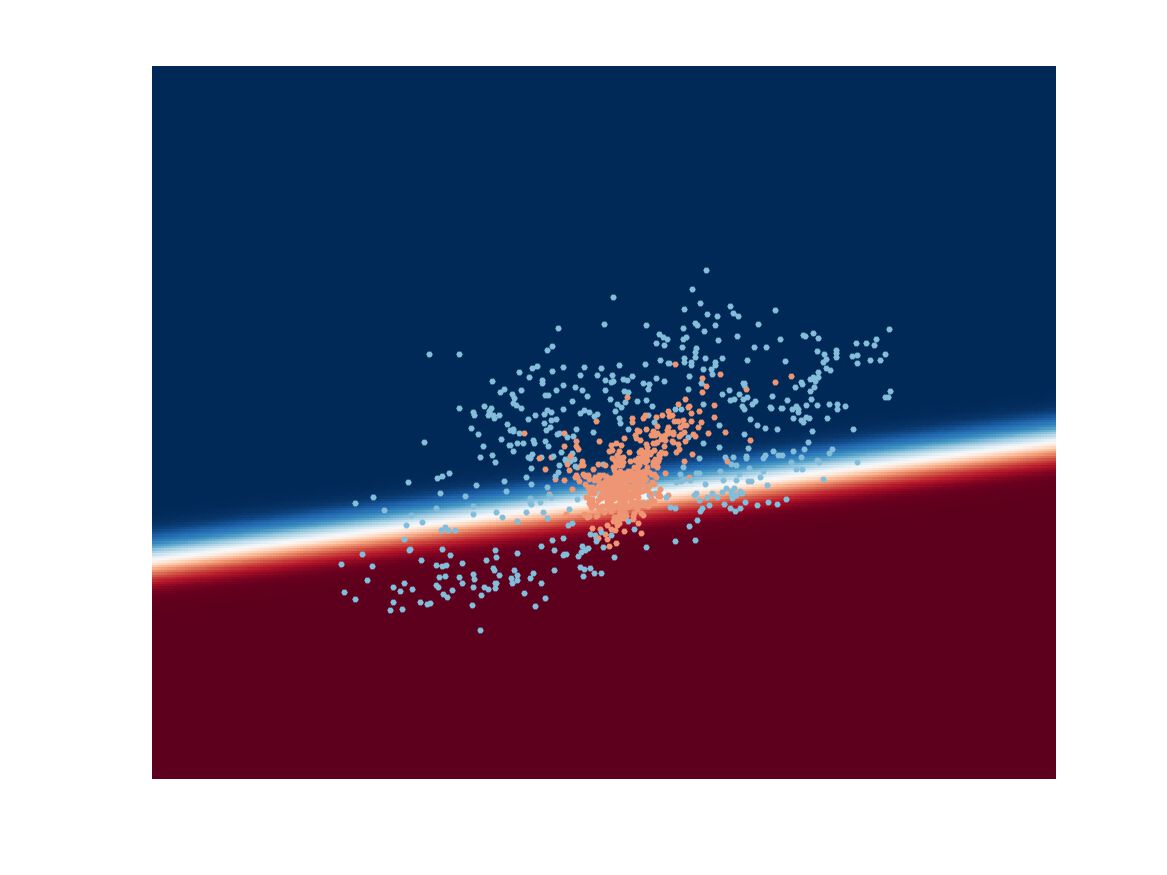}
{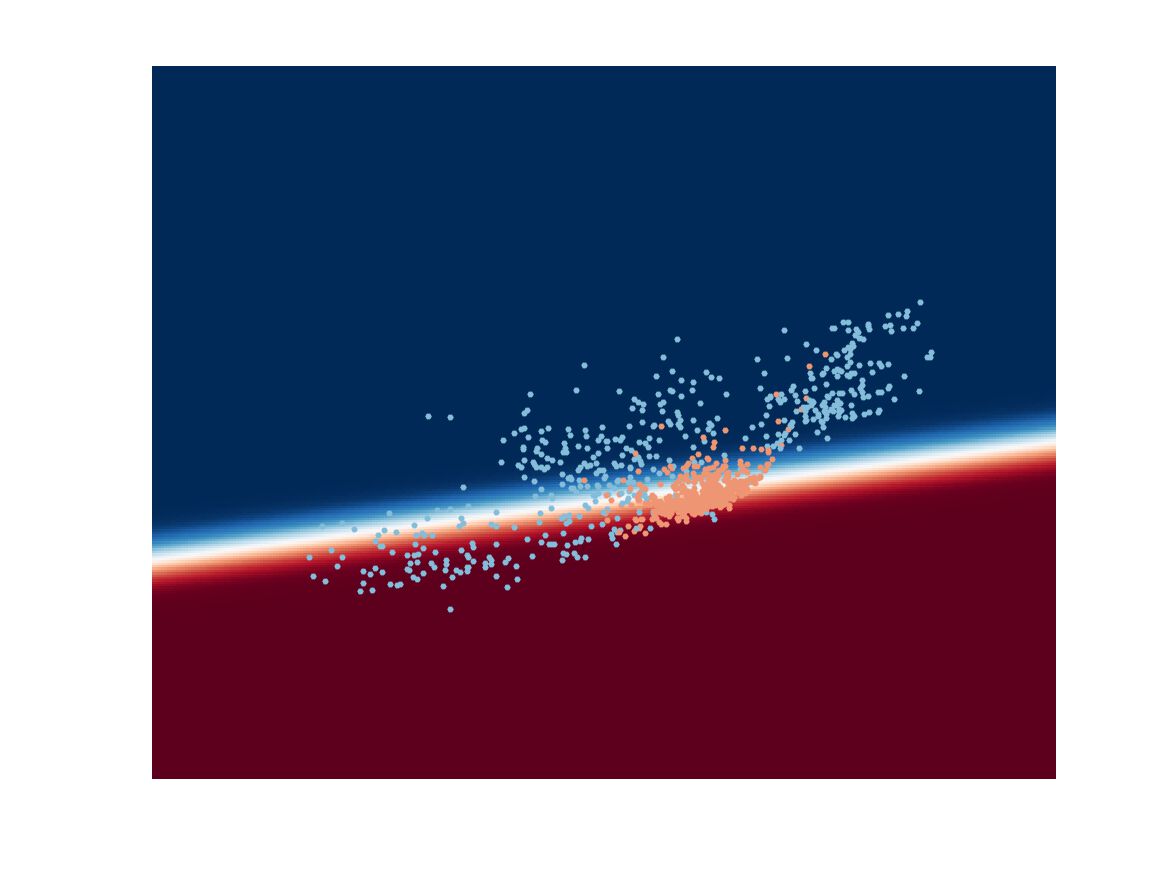}
{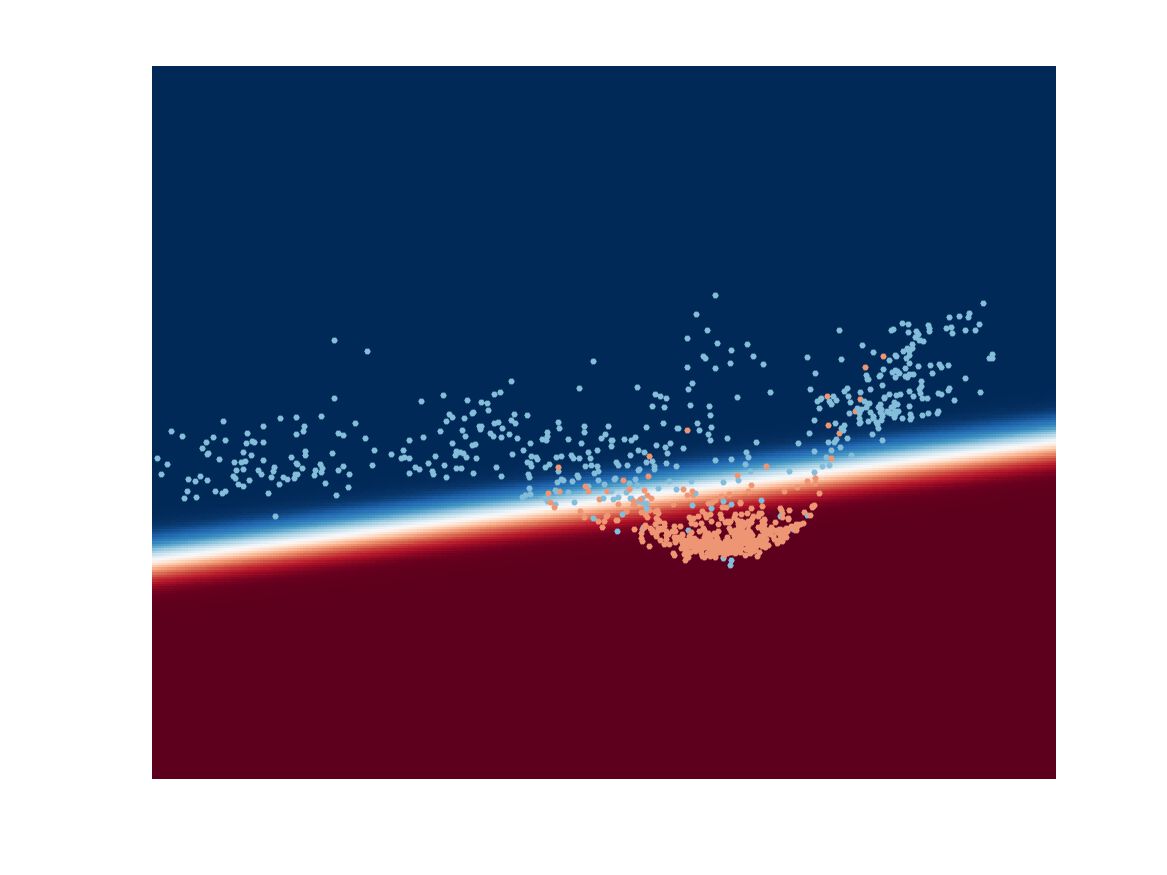}
\PlotRow{3}{Kutta (4)}
{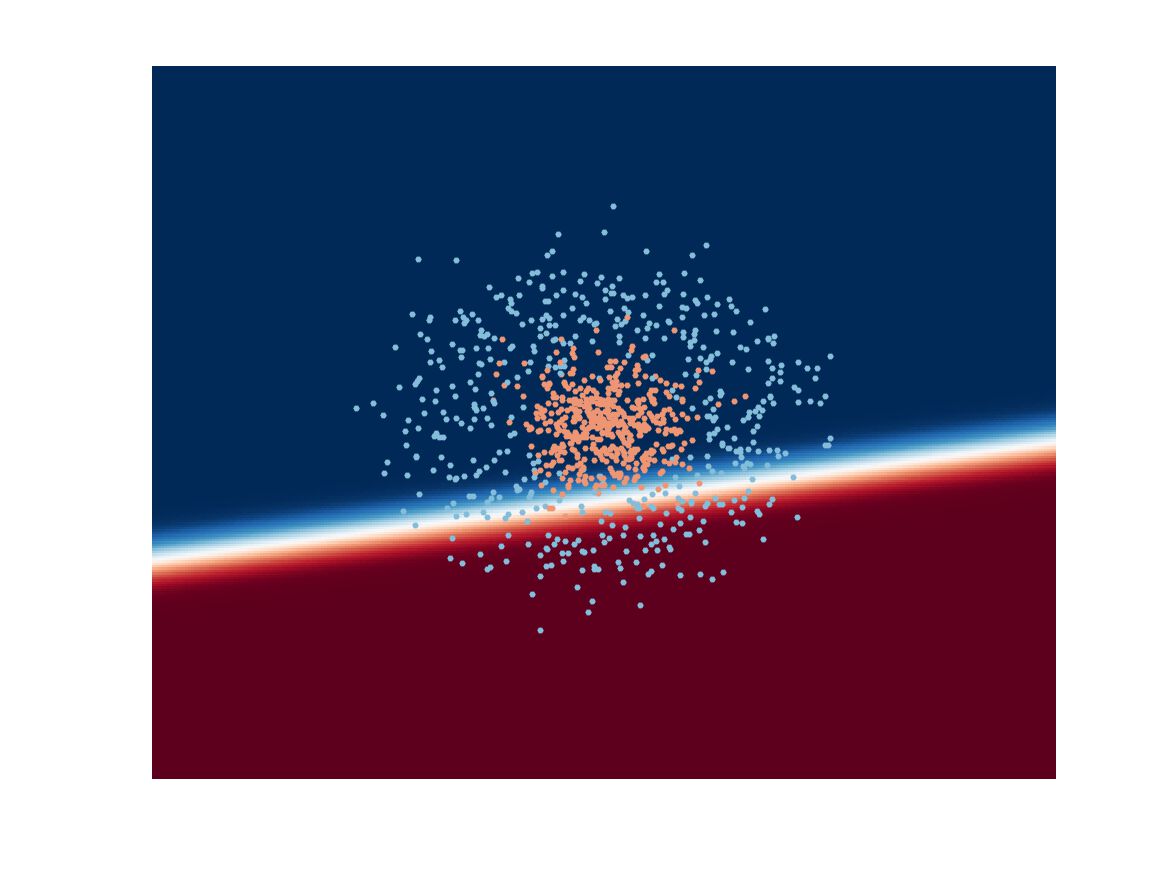}
{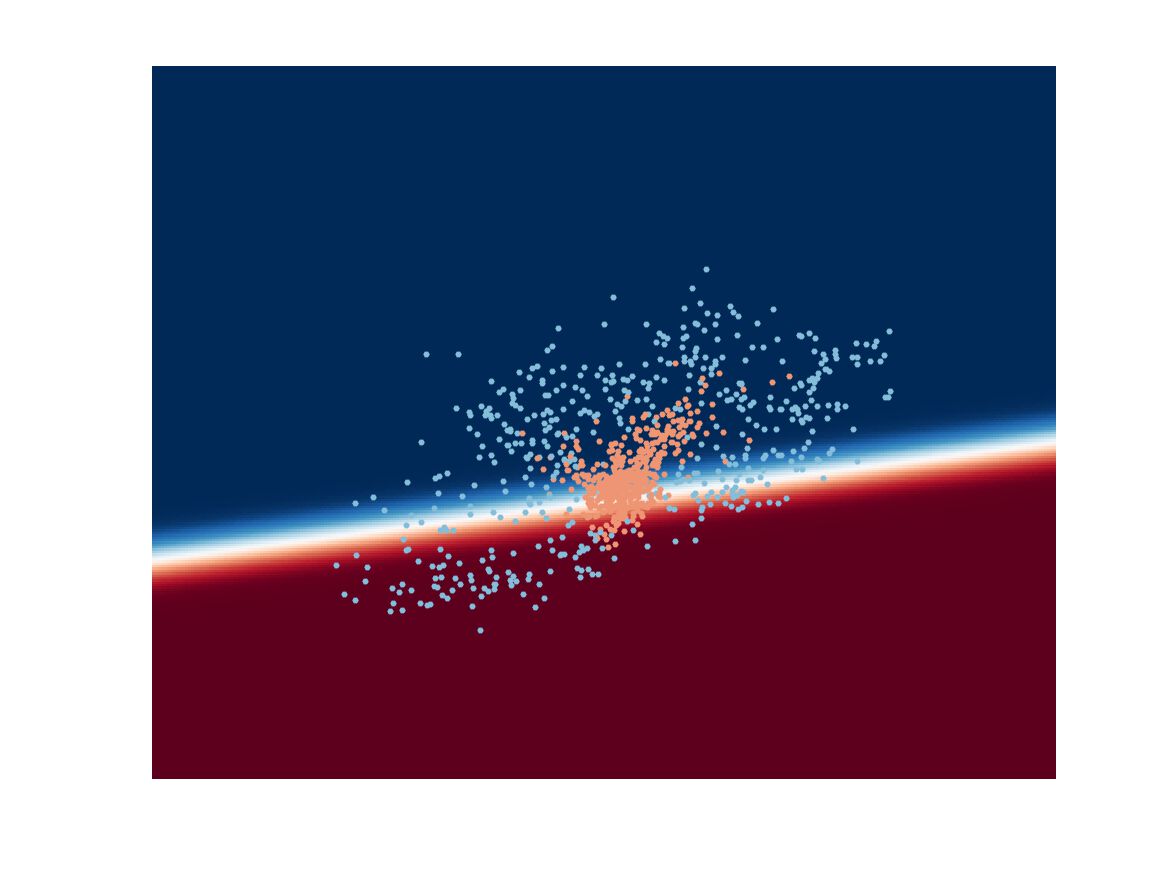}
{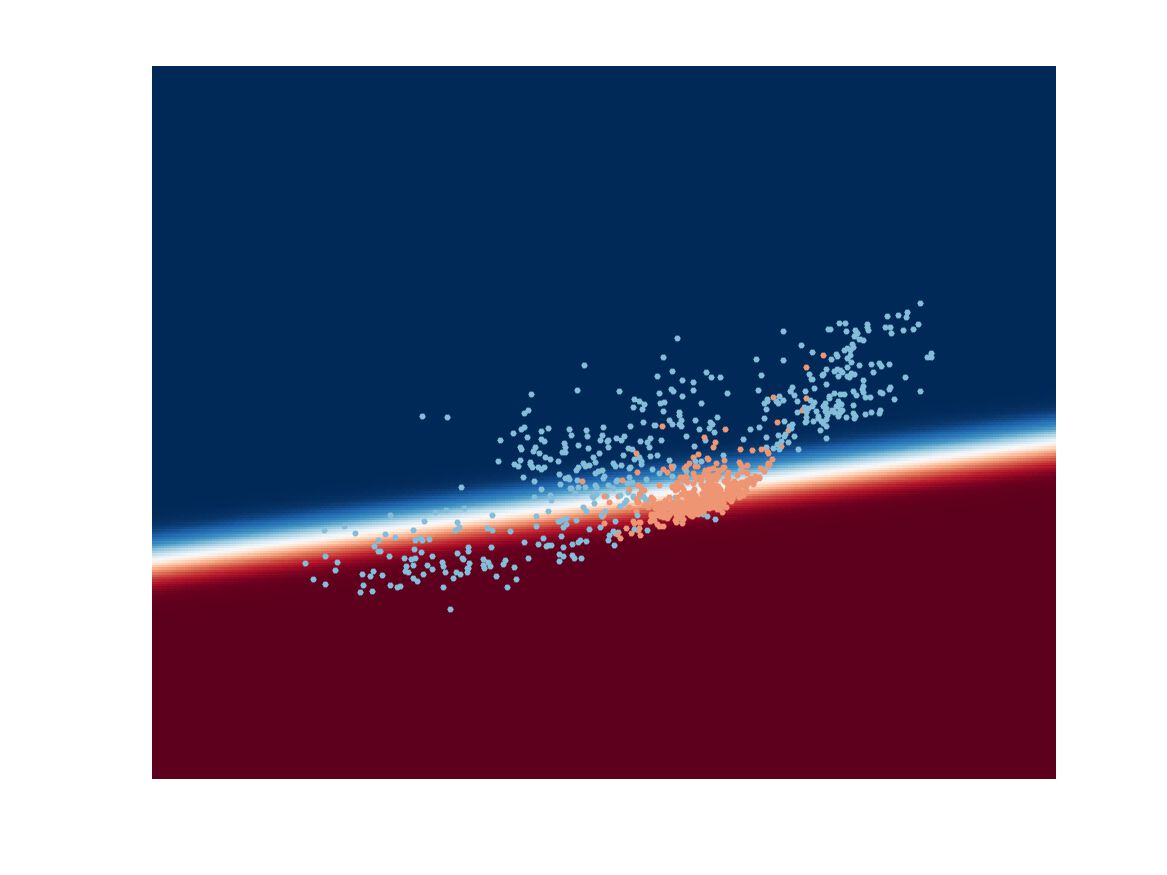}
{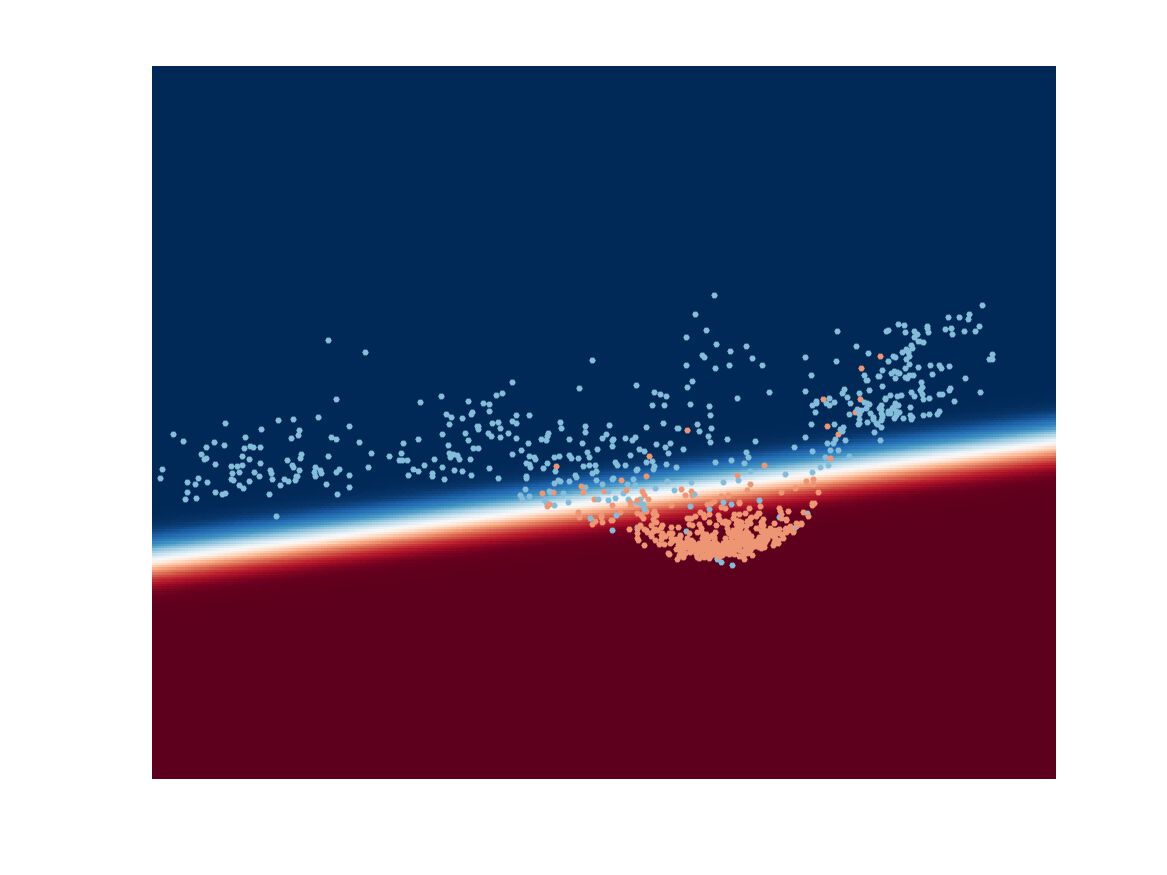}
\PlotLabels{time 0}{}{}{time $T$}
\end{tikzpicture}
\caption{Snap shots of the transition from initial to final state through the network with the data set \DataDonutB.} \label{anim-donut2d}
\end{figure}

\begin{figure}
\def\RowHeight{3cm}
\renewcommand{\PlotIm}[1]{{\includegraphics[width=\PicWidth, clip, trim=3cm 1.5cm 2cm 1cm]{RKSnaps/#1}}}%

\begin{tikzpicture}
\PlotRow{0}{ResNet/Euler}
{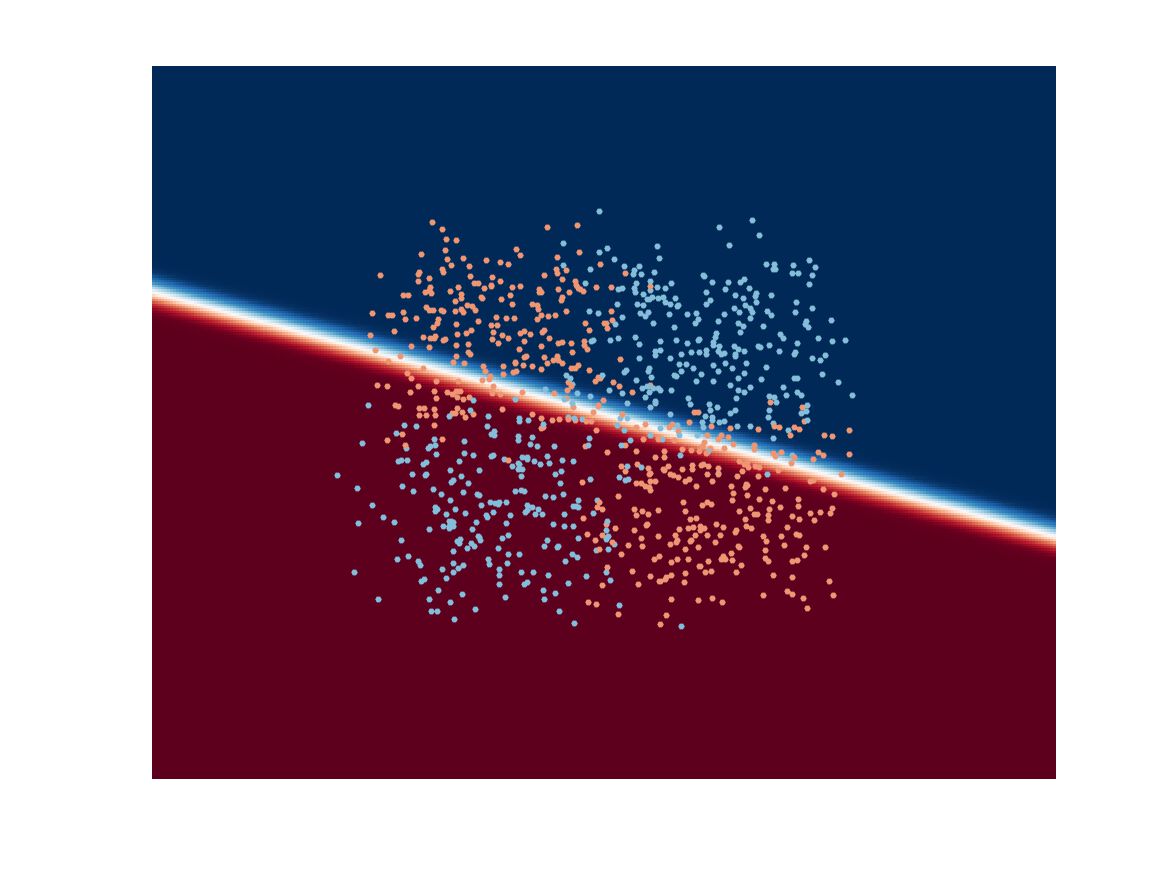}
{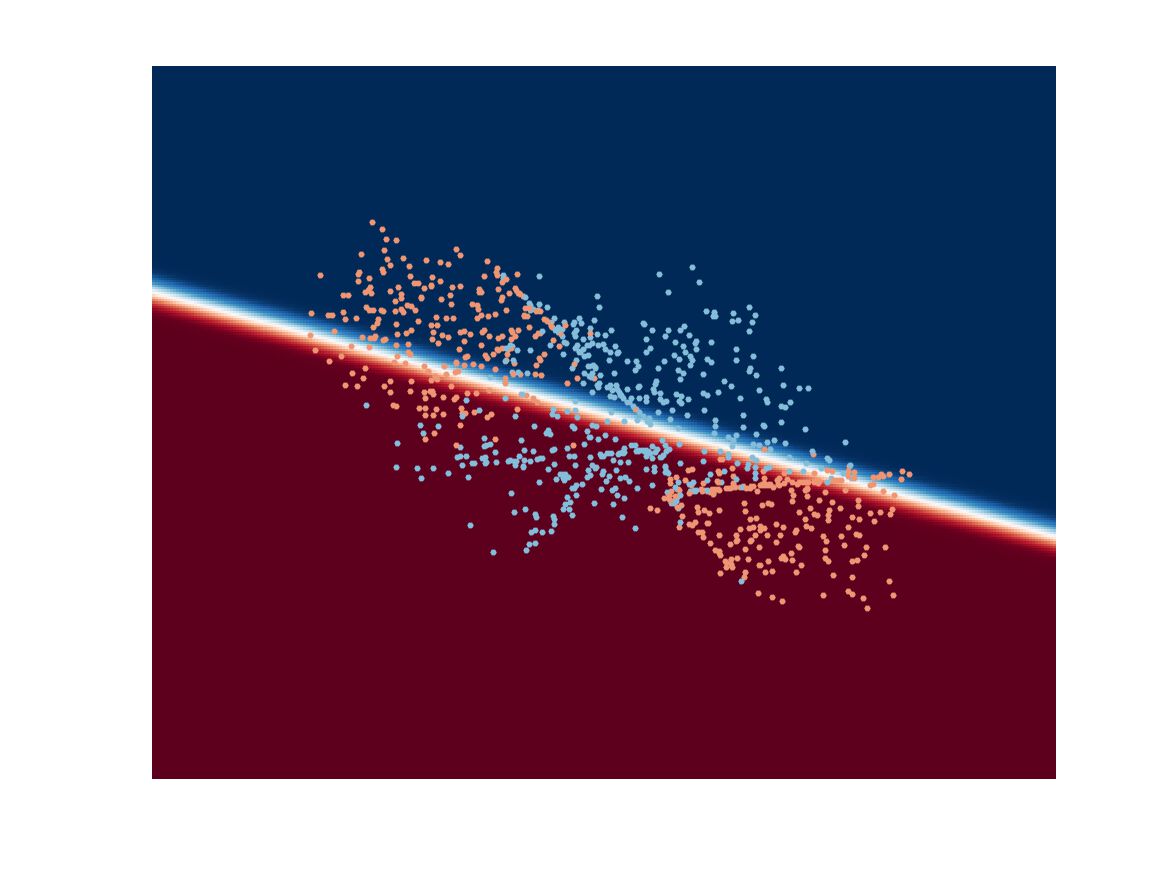}
{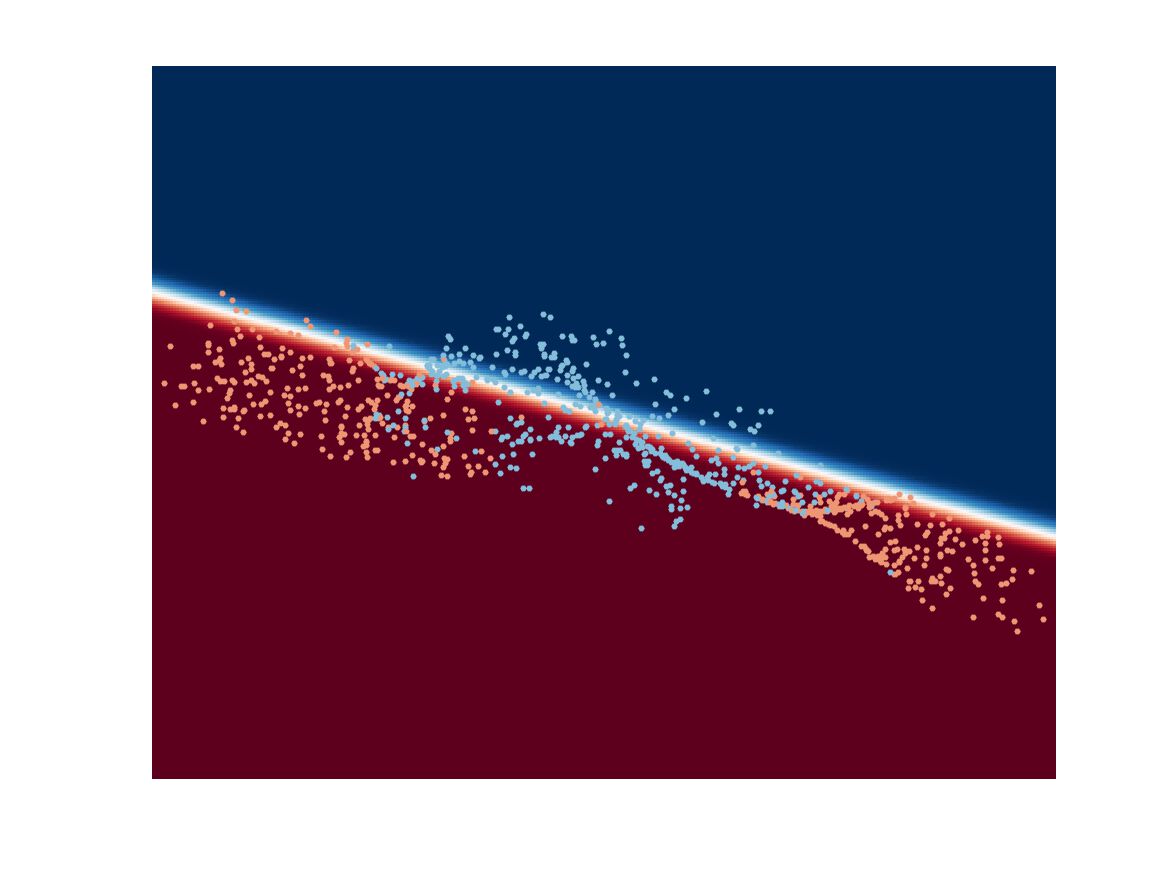}
{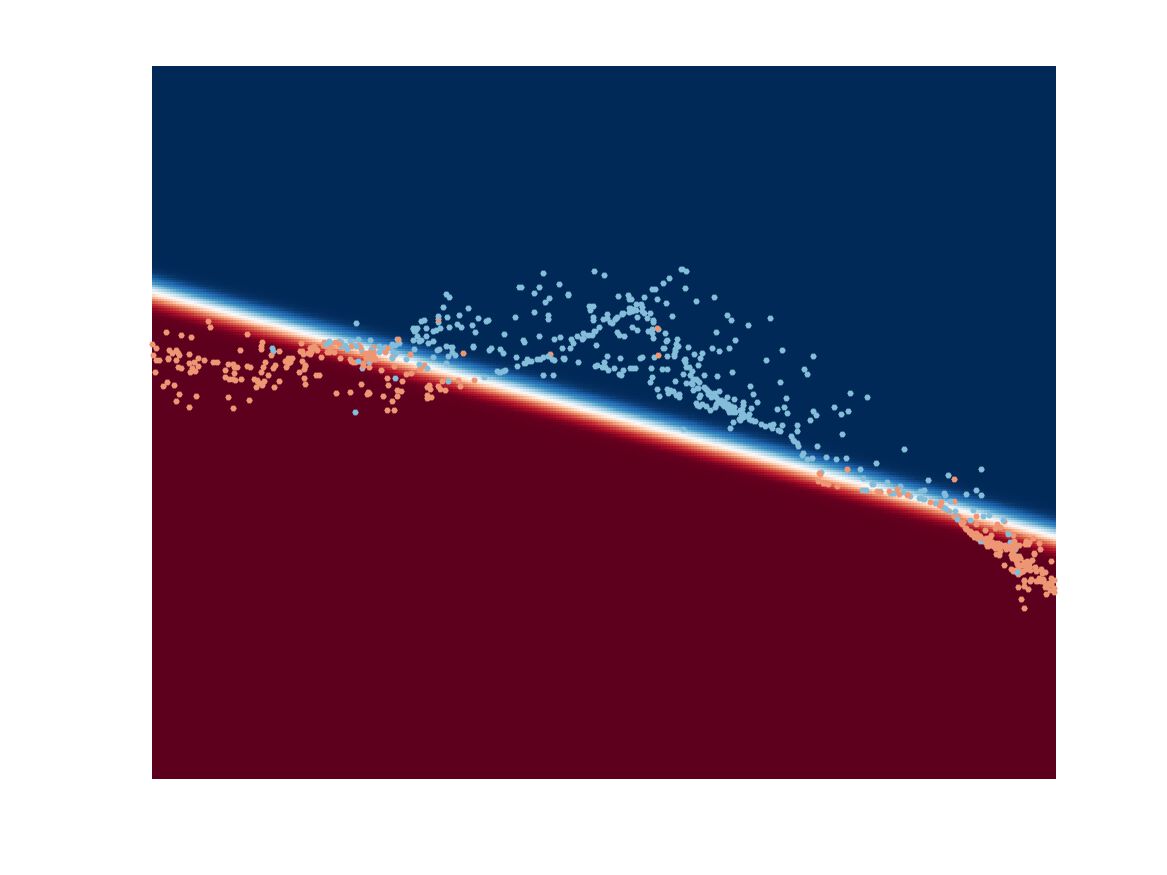}
\PlotRow{1}{Improved Euler}
{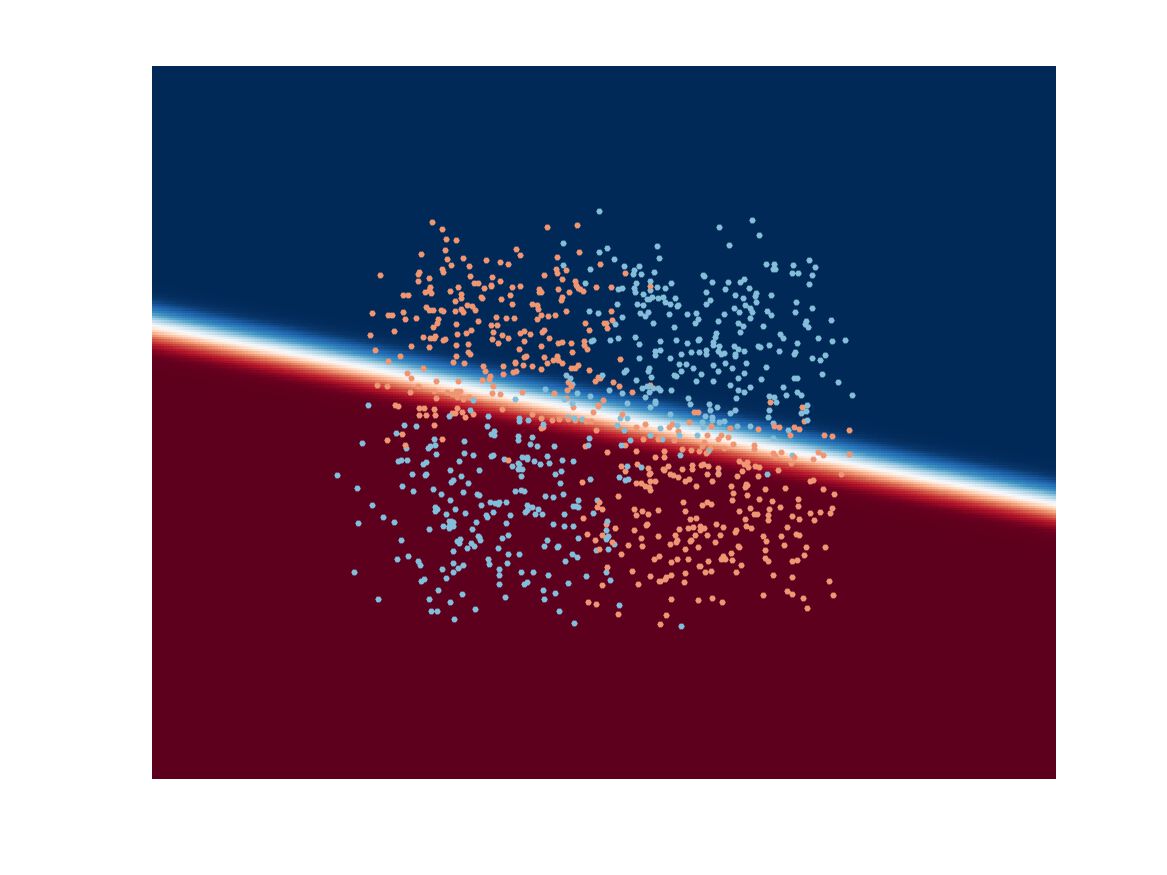}
{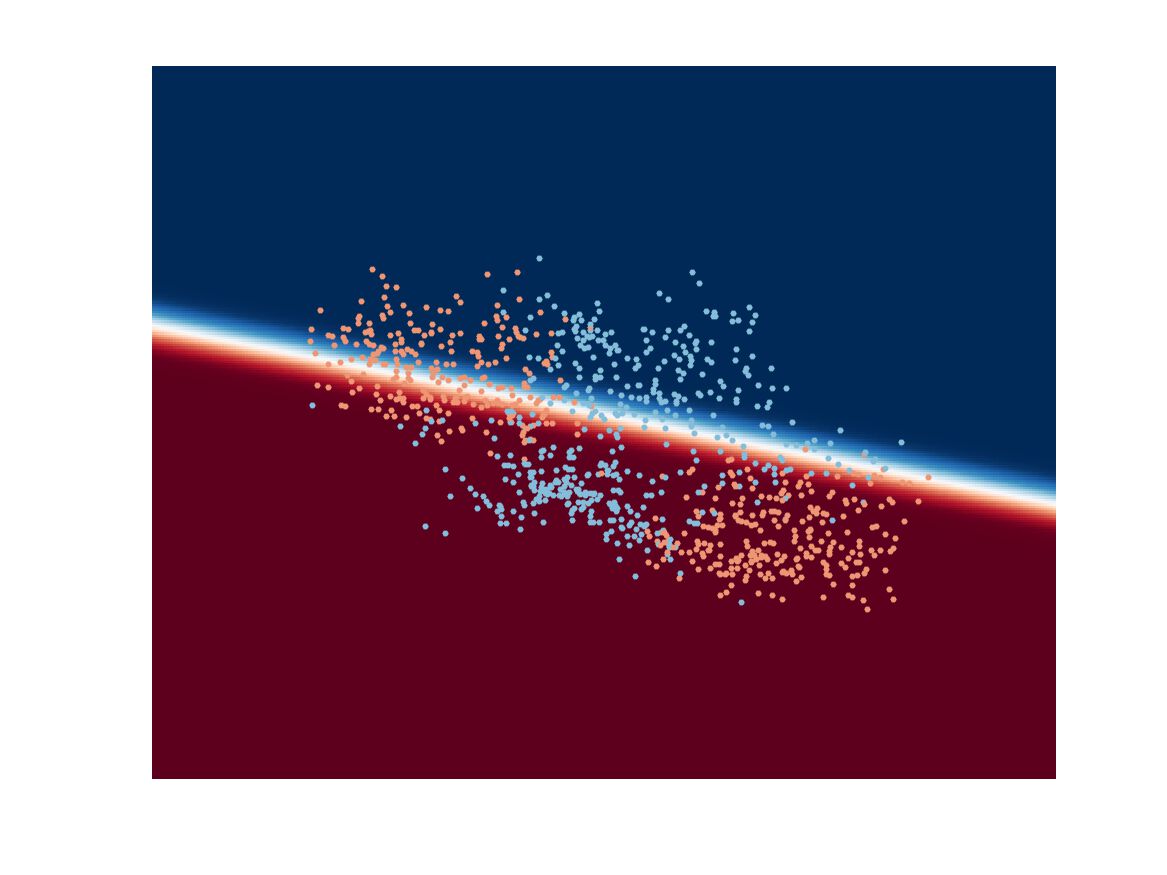}
{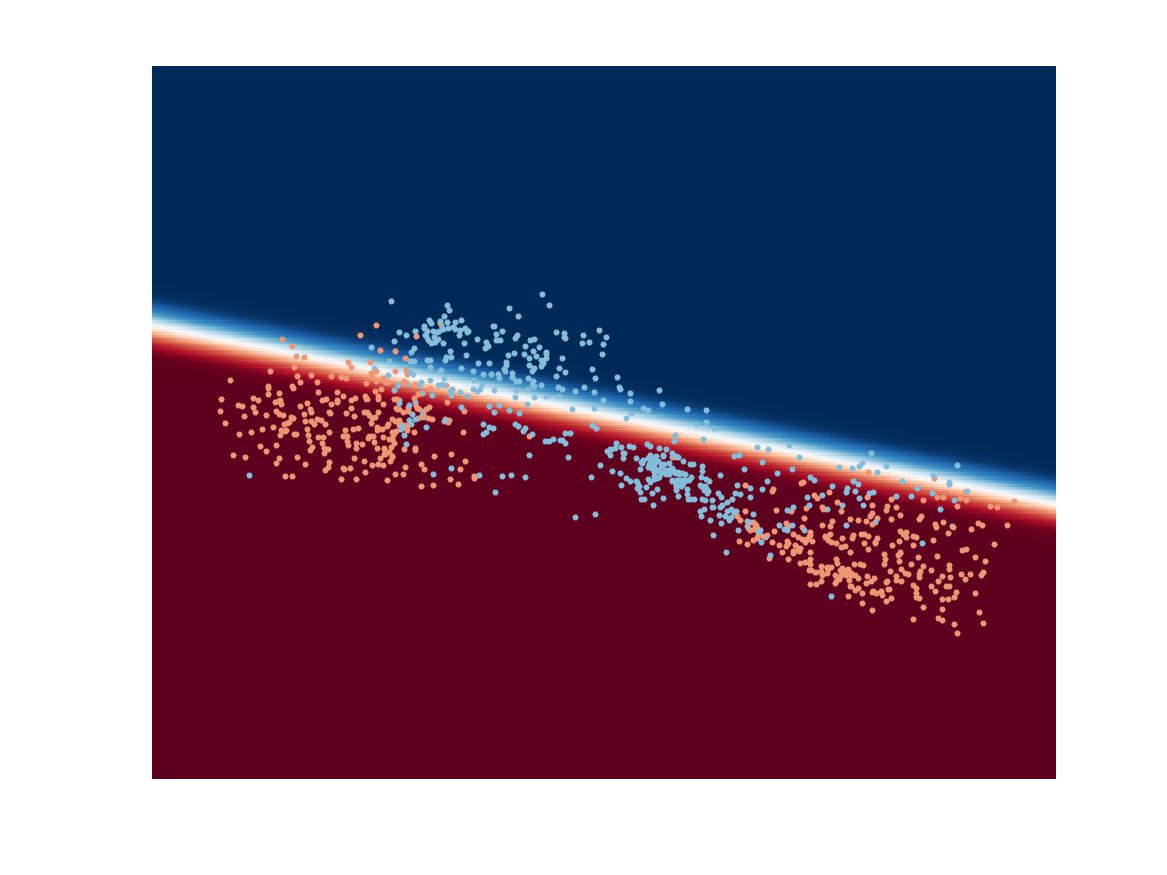}
{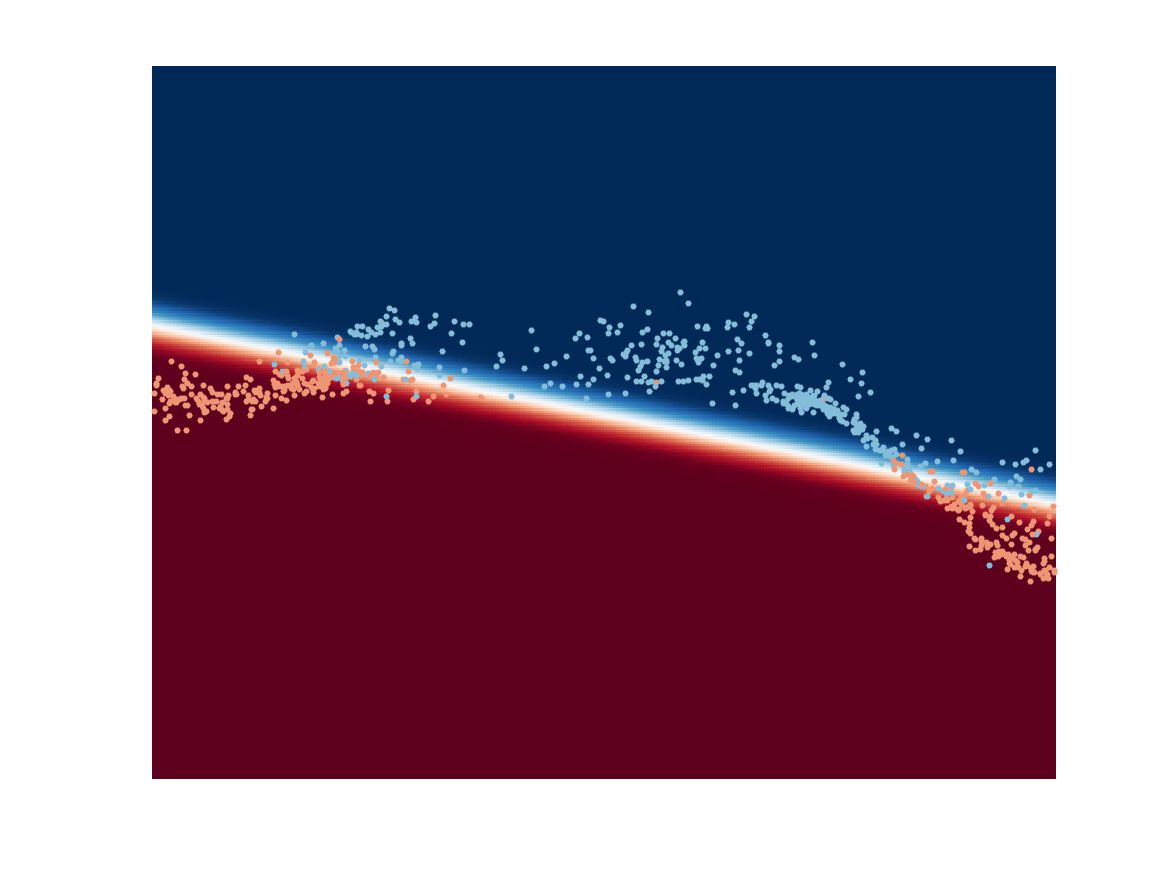}
\PlotRow{2}{Kutta (3)}
{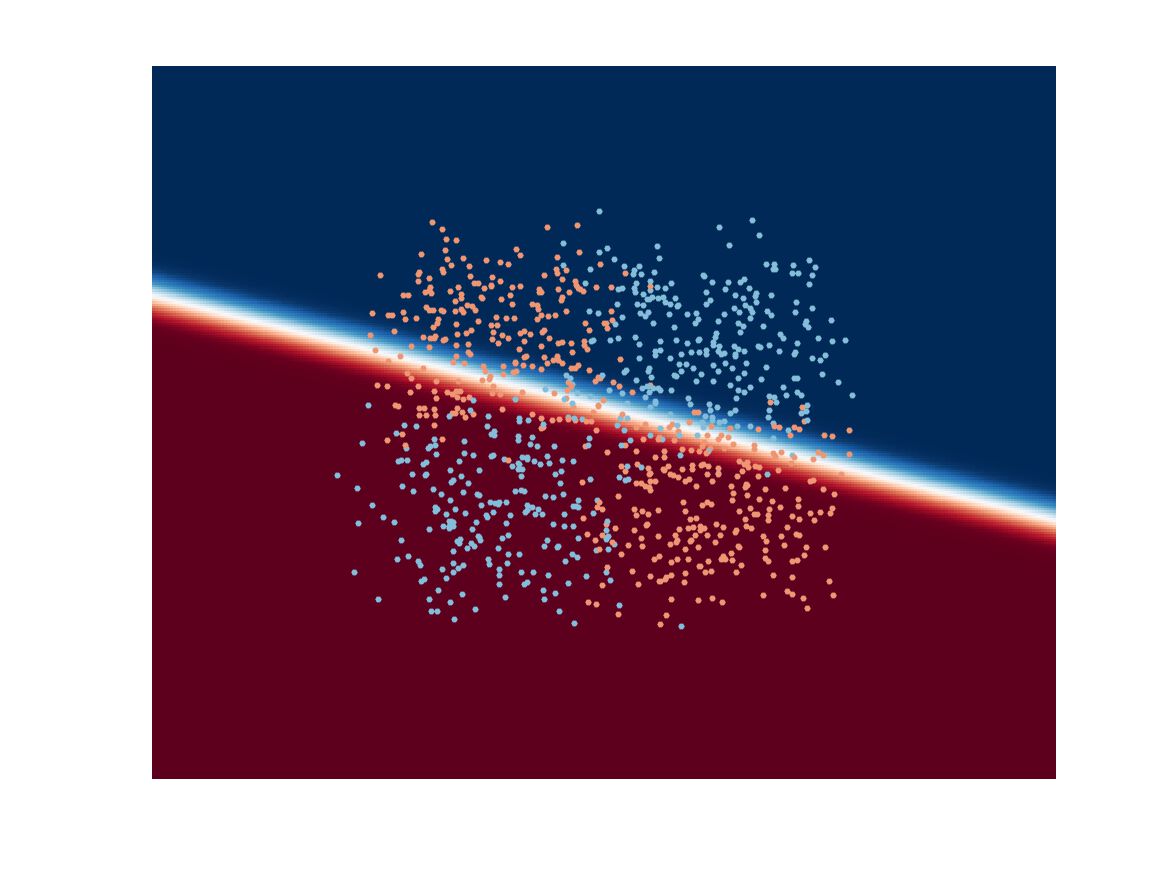}
{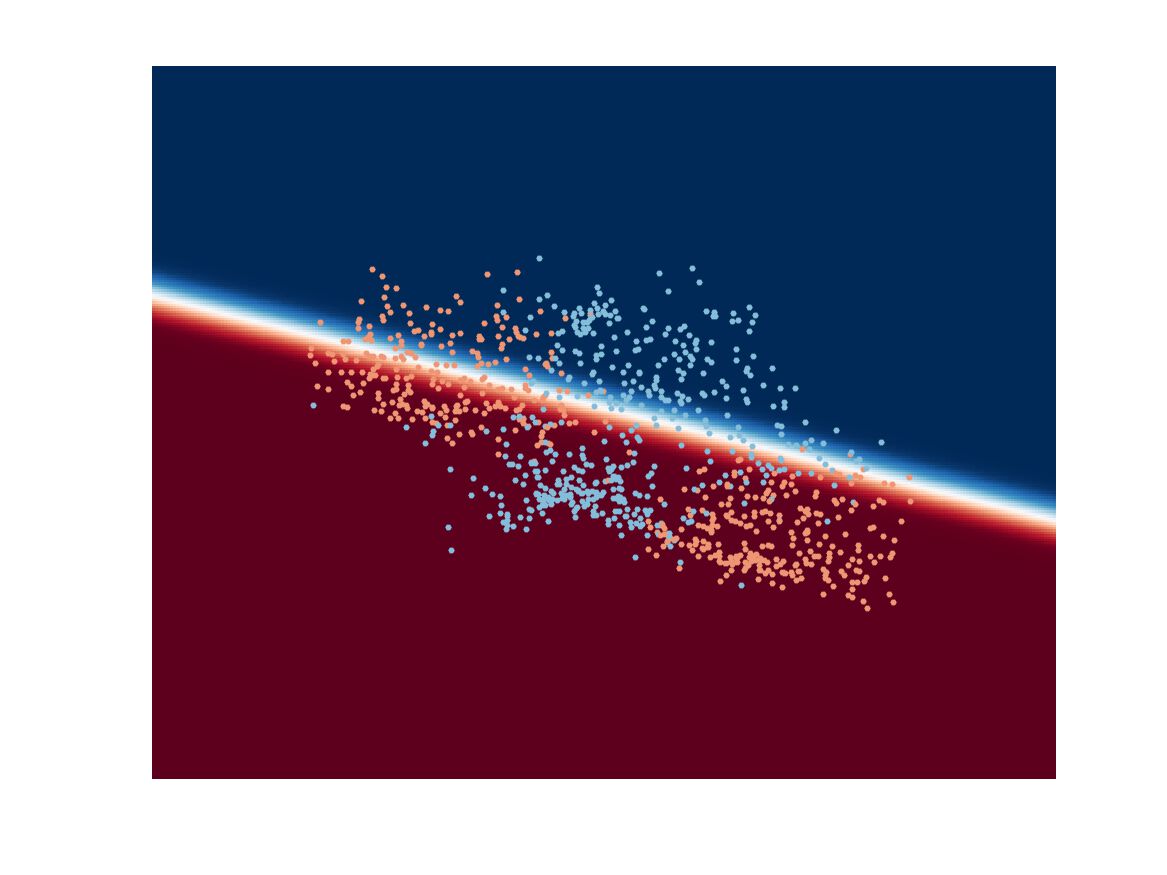}
{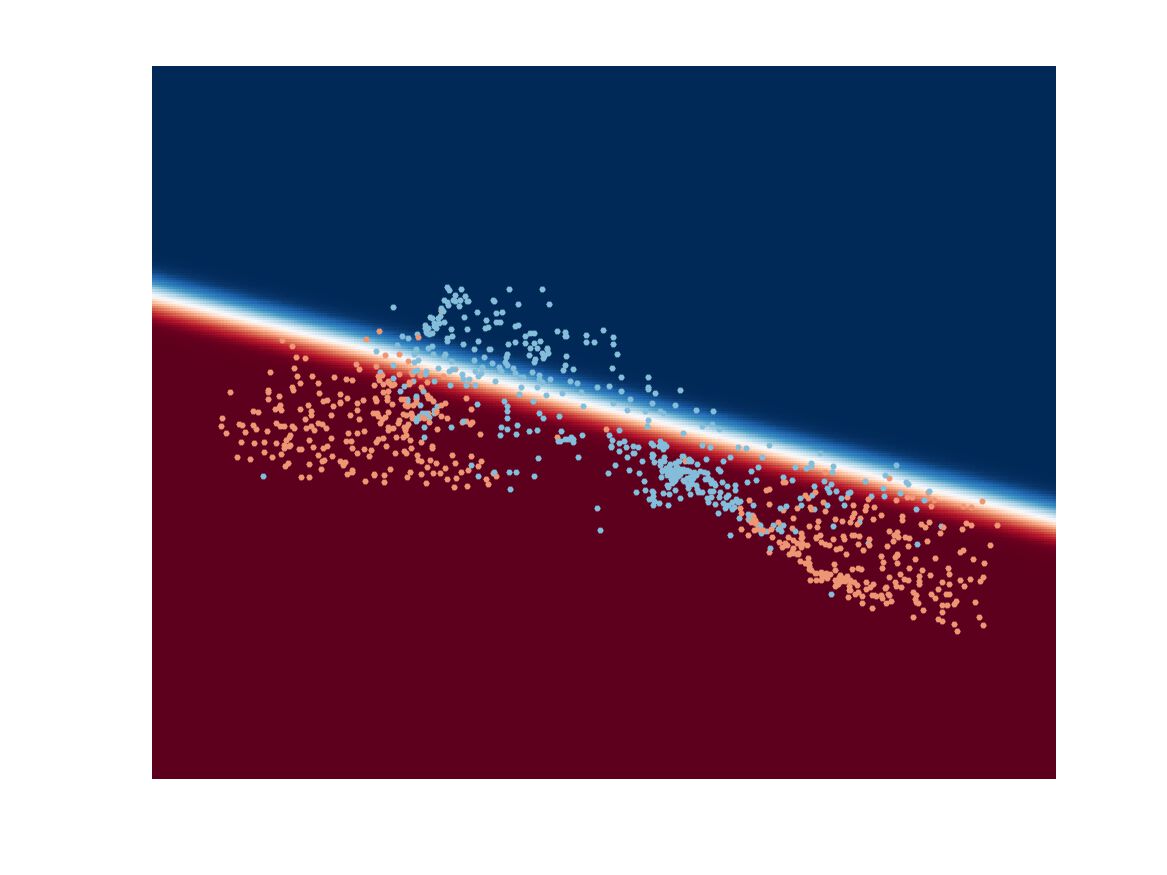}
{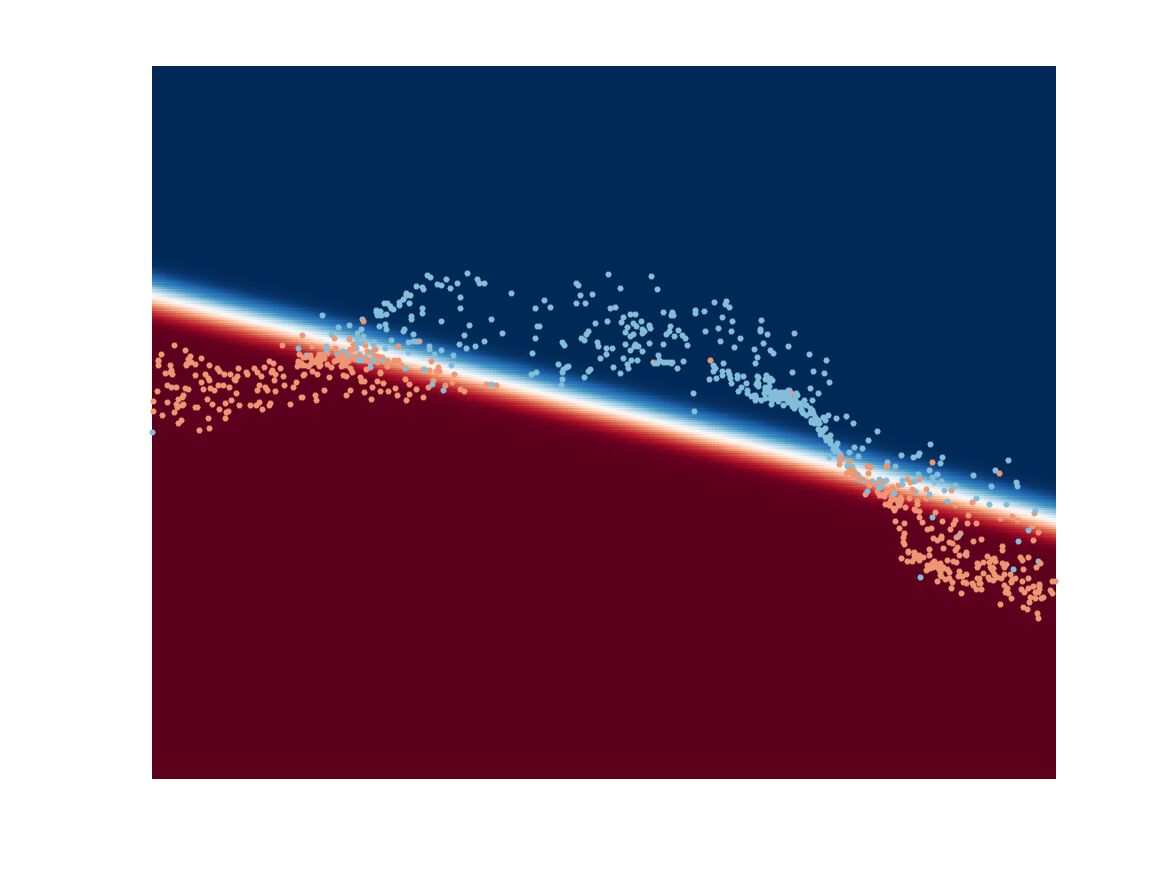}
\PlotRow{3}{Kutta (4)}
{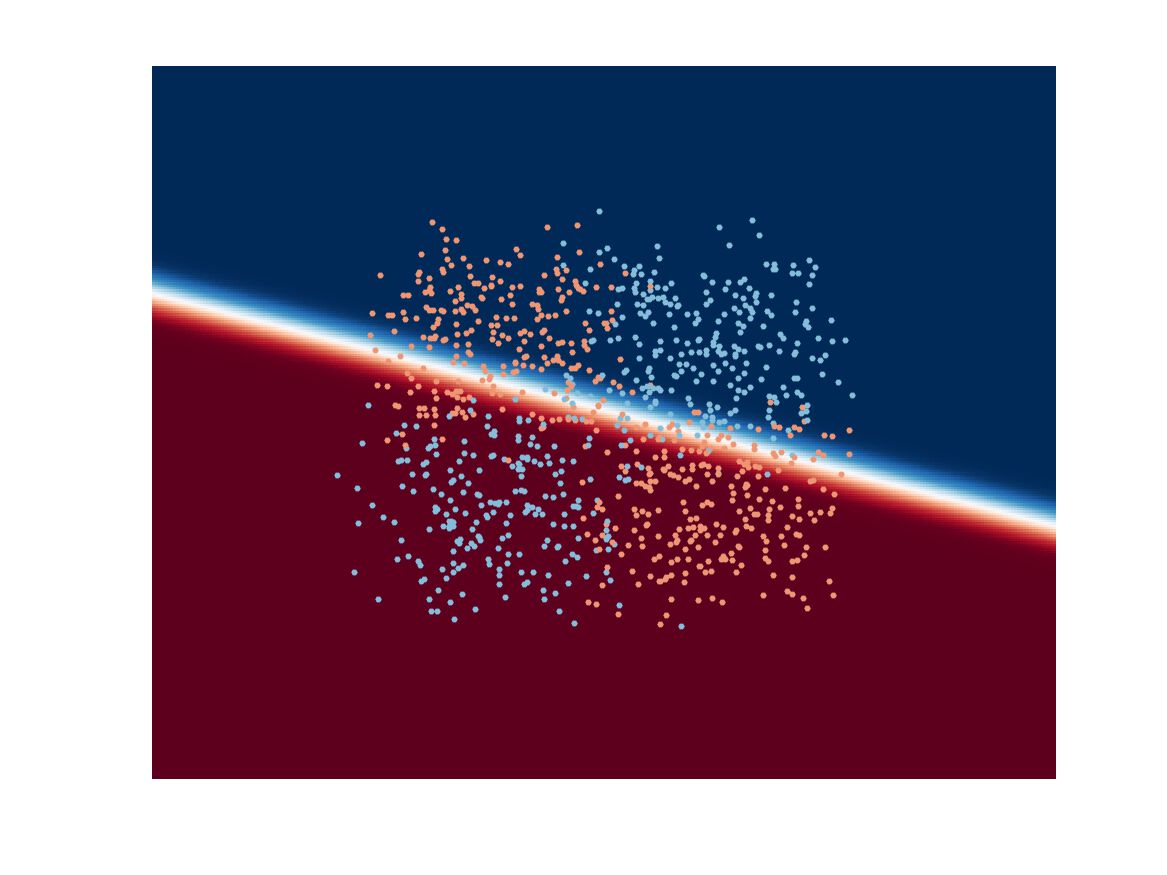}
{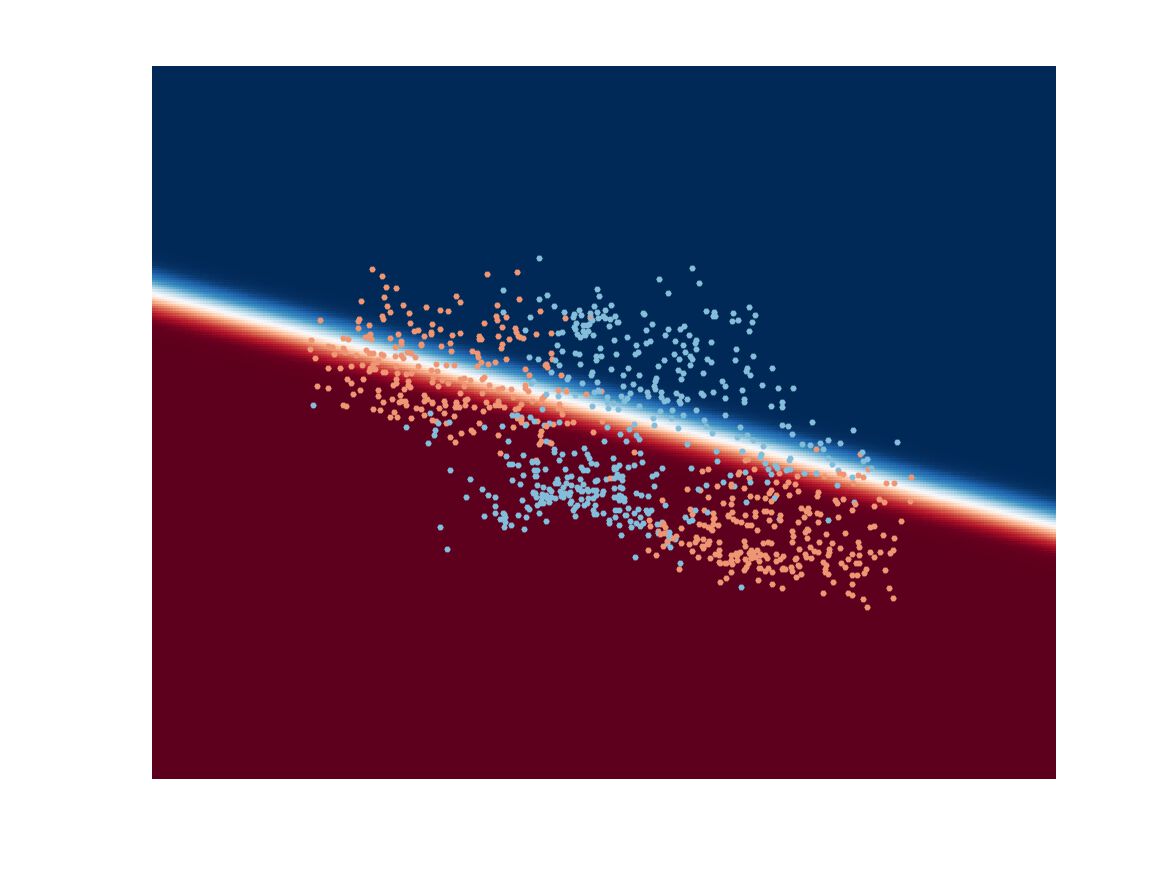}
{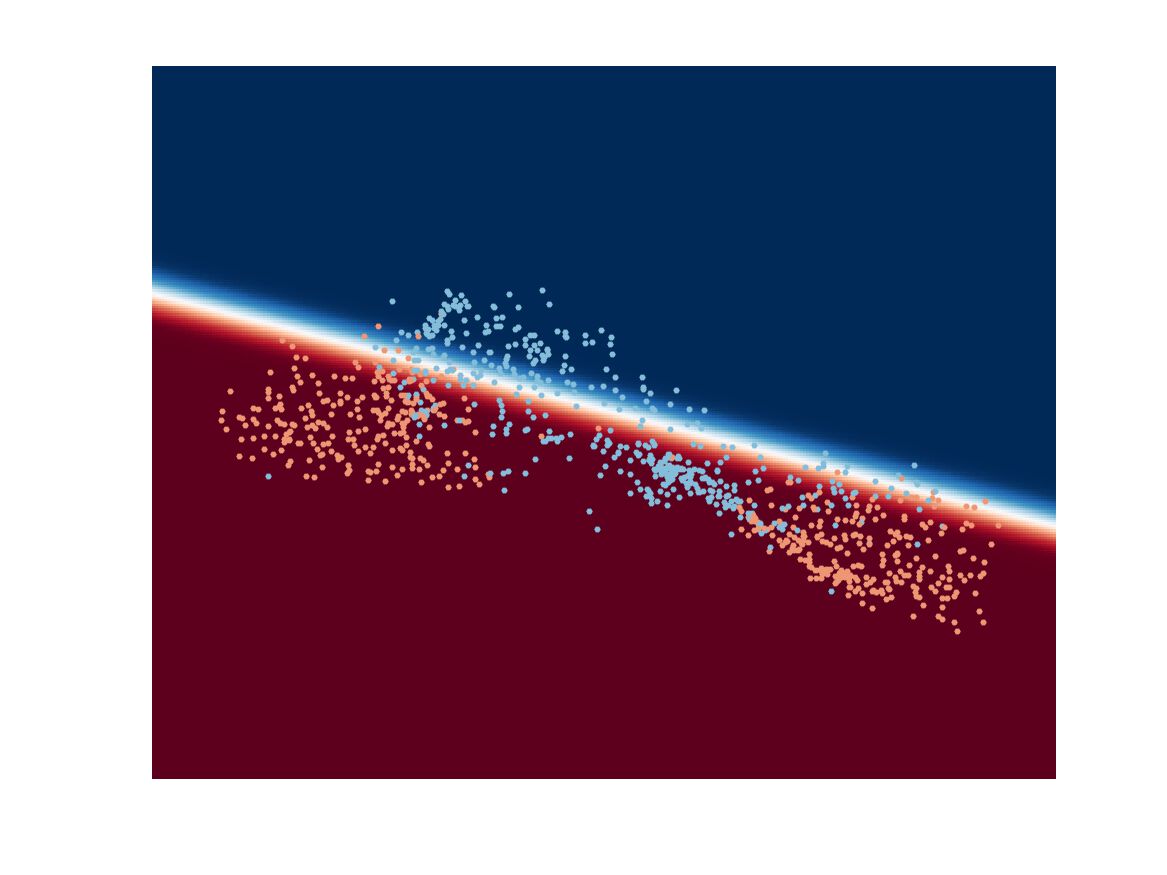}
{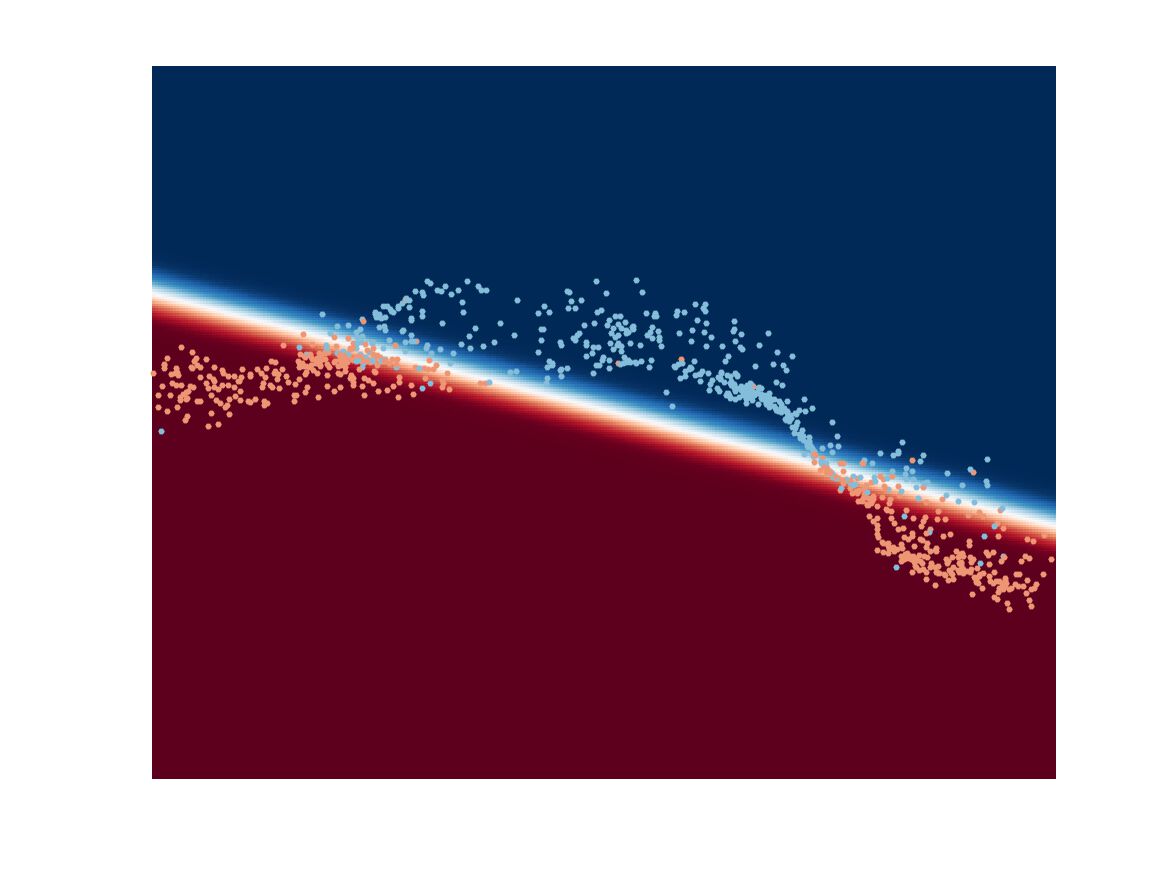}
\PlotLabels{time 0}{}{}{time $T$}
\end{tikzpicture}
\caption{Snap shots of the transition from initial to final state through the network with the data set \DataSquares.} \label{anim-squares2d}
\end{figure}

\begin{figure}
\centering
\includegraphics[width=7cm]{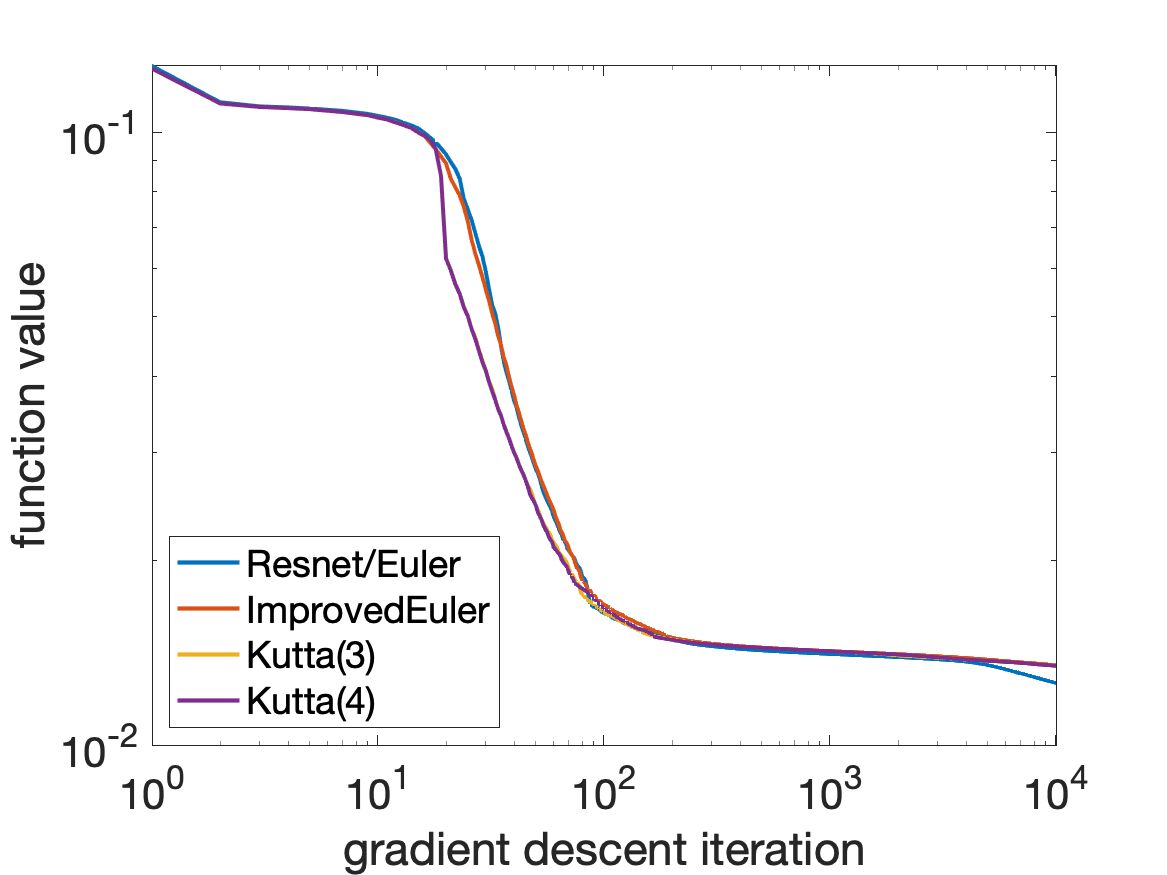}%
\includegraphics[width=7cm]{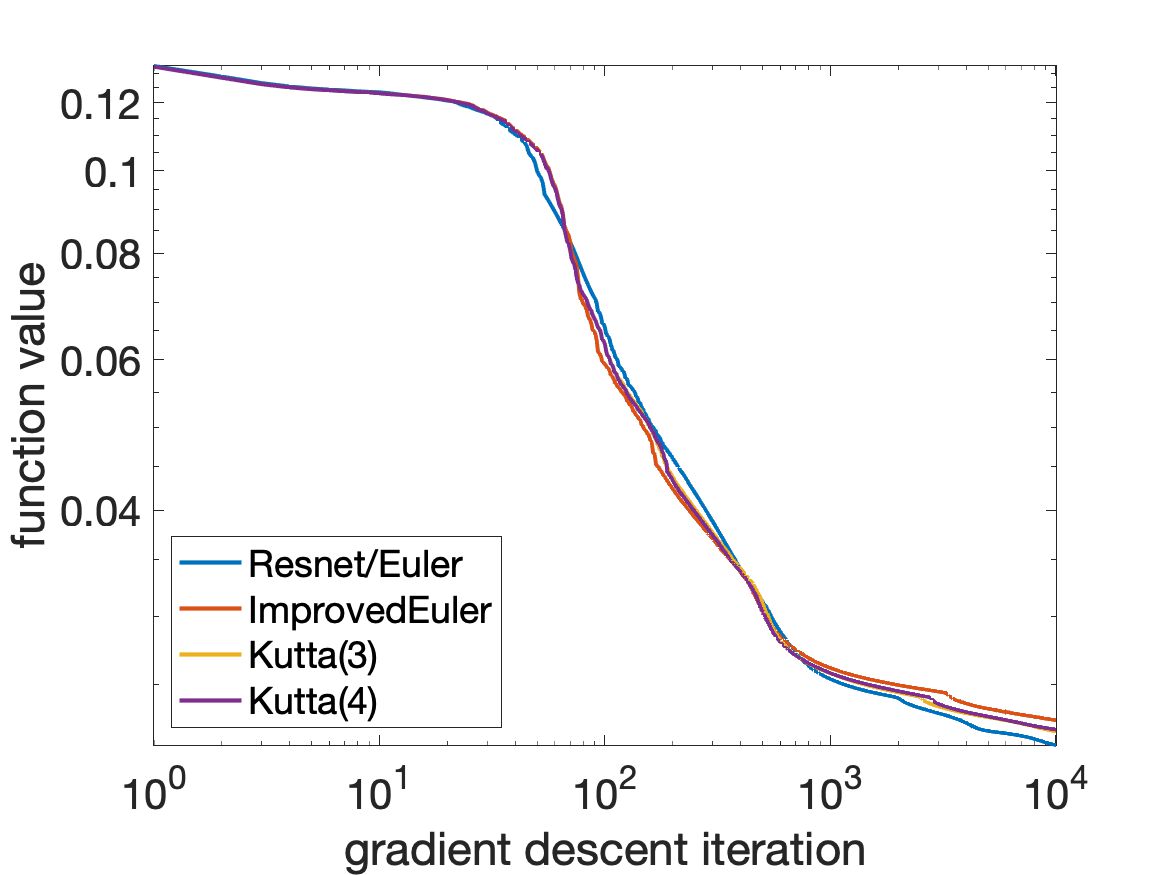}\\
\includegraphics[width=7cm]{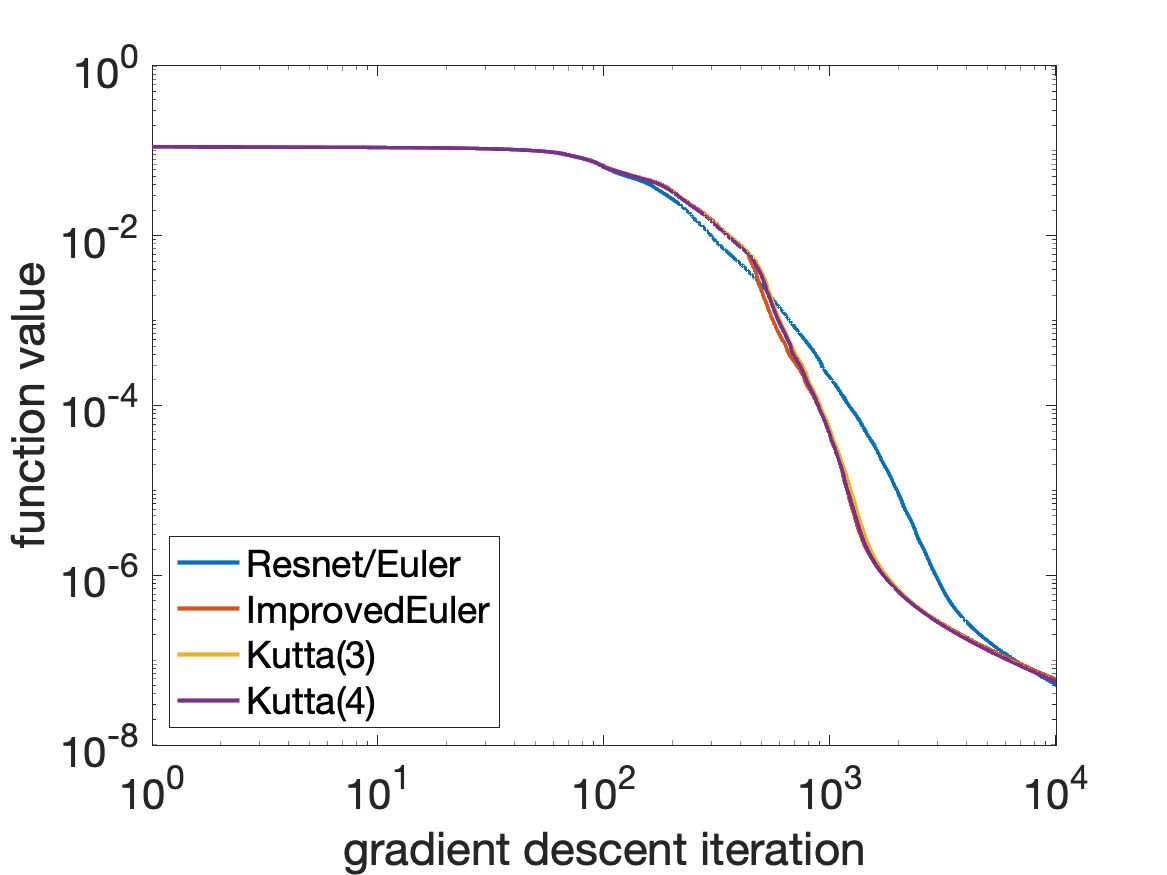}%
\includegraphics[width=7cm]{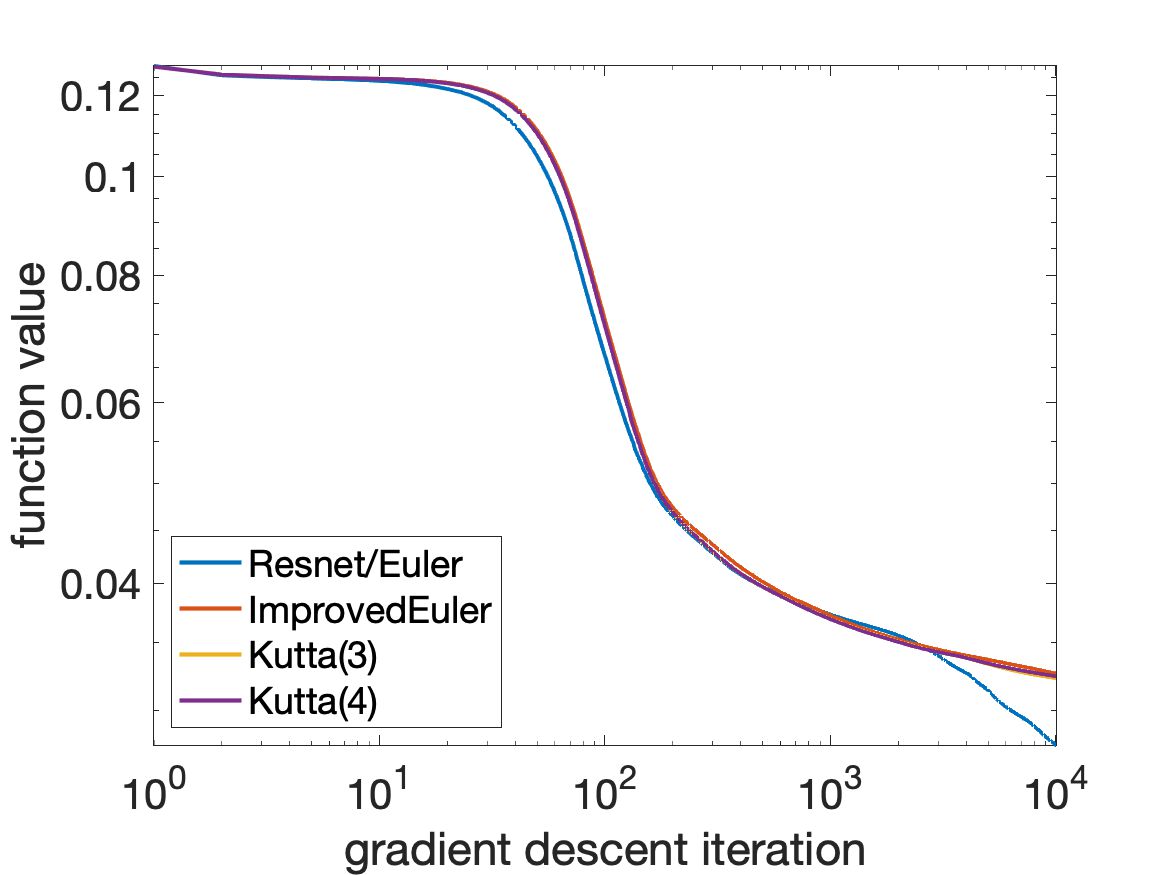}
\caption{Function values over the course of the gradient descent iterations for data sets \DataDonutA, \DataDonutB, \DataSpiral, \DataSquares~(left to right and top to bottom). }\label{FIG:RK:FUNCTION}
\end{figure} 

\subsection{Digit classification with minimal data}
We test four network architectures---three of which are ODE-inspired---on digit classification. The training data is selected from the MNIST data base \cite{LeCun1998mnist} where we restrict ourselves to classifying 0s and 8s. To make this classification more challenging, we train only on 100 images and take another 500 as test data. We refer to this data as \mnist.

There are a couple of observations which can be made from the results shown in Figures \ref{FIG:M:MNIST:QUAN} and \ref{FIG:M:MNIST:QUAL}. First, as can be seen in Figure \ref{FIG:M:MNIST:QUAN}, the results are consistent with the observations made from the toy data in Figure \ref{FIG:M:ACCURACY}: the three ODE-inspired methods seem to perform very well, both on training and test data. Also the trained step sizes show similar profiles as in Figure \ref{FIG:M:TIME}, with ODENet learning negative step sizes and ODENet+Simplex learning very sparse time steps. Second, in Figure \ref{FIG:M:MNIST:QUAL}, we show the transformed test data before the classification. Interestingly, all four methods learn what looks to the human eye as adding noise. Only the ODE-inspired networks retain some of the structure of the input features.

\begin{figure}
\centering
\includegraphics[width=7cm]{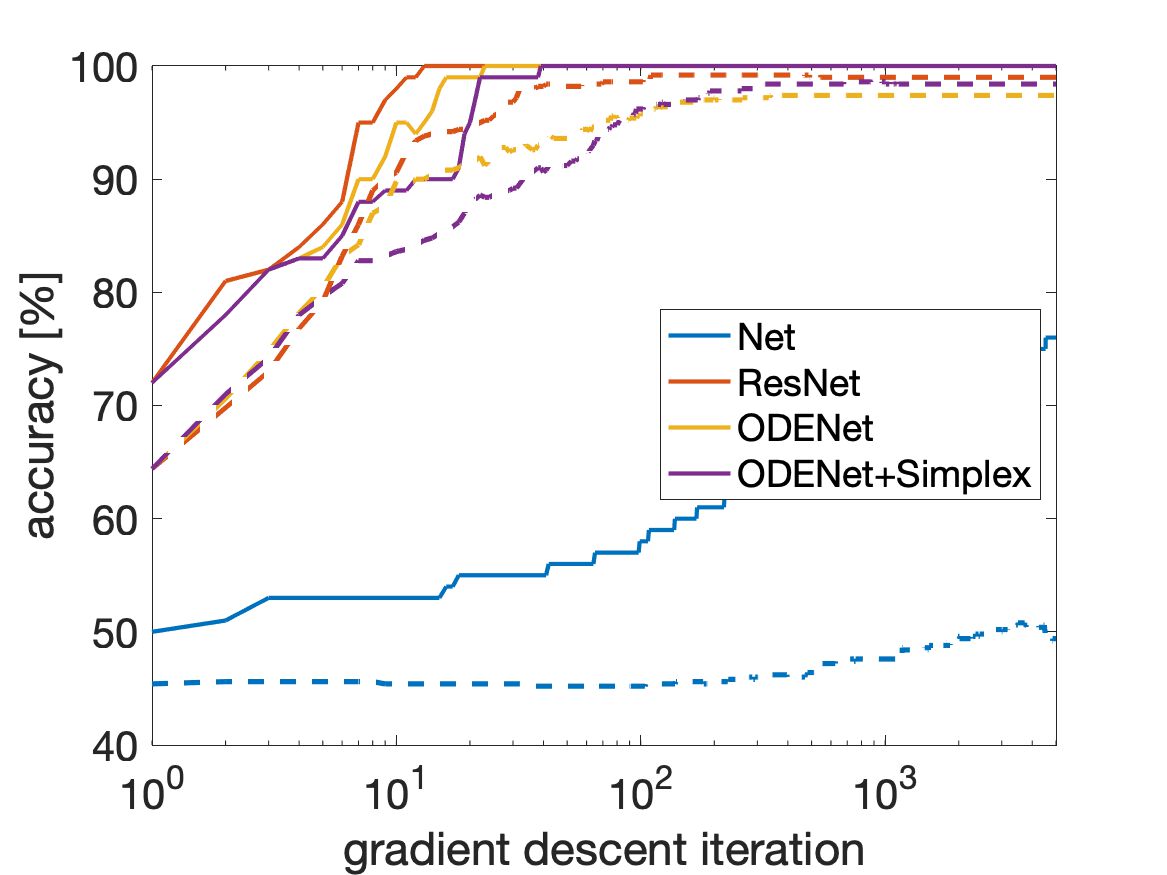}%
\includegraphics[width=7cm]{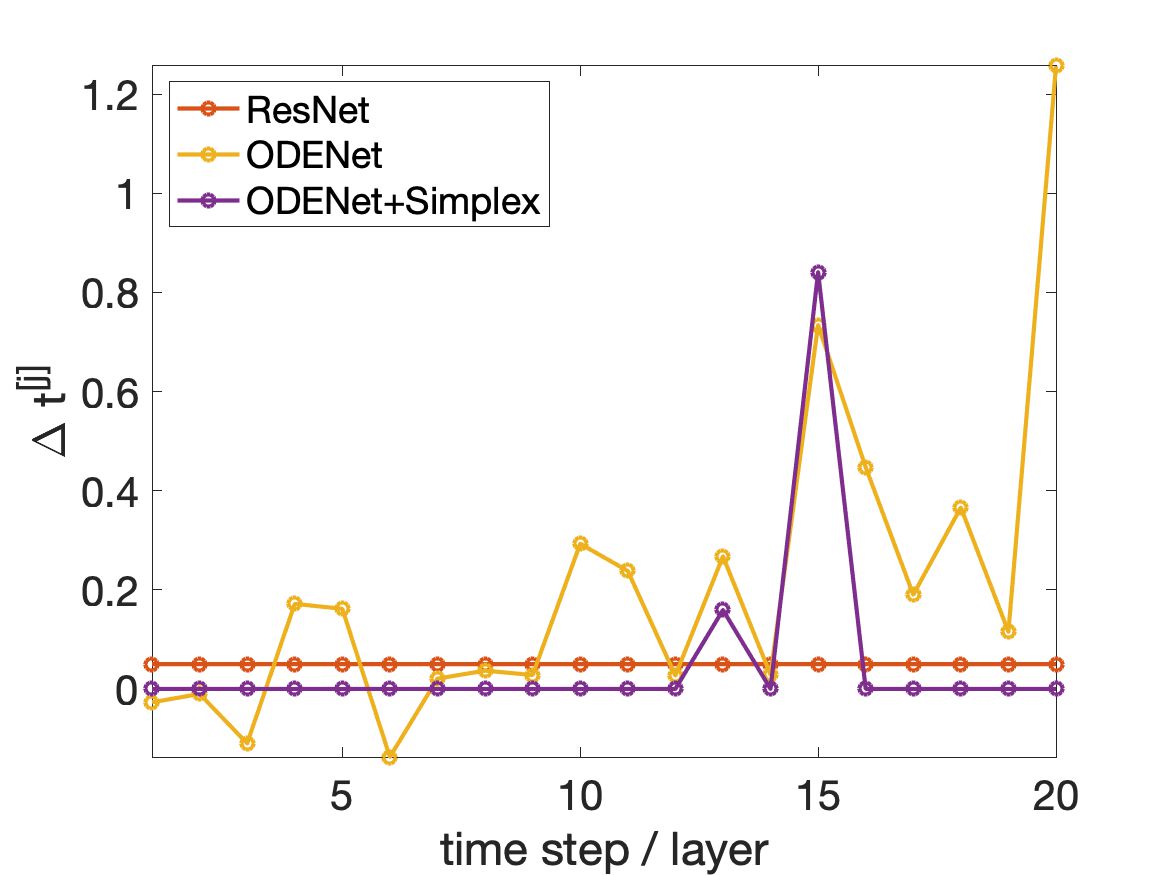}%
\caption{Accuracy (left) and time steps (right) for \mnist~dataset \cite{LeCun1998mnist}.} \label{FIG:M:MNIST:QUAN}
\end{figure}

\def\RowHeight{1.4cm}
\def\CellWidth{1.4cm}

\renewcommand\PlotLabels[4]{
\node at (0*\CellWidth,0.5*\RowHeight + 2mm) {#1};
\node at (1*\CellWidth,0.5*\RowHeight + 2mm) {#2};
\node at (2*\CellWidth,0.5*\RowHeight + 2mm) {#3};
\node at (3*\CellWidth,0.5*\RowHeight + 2mm) {#4};}

\newcommand{\addstrings}[3]{#1#2#3}

\renewcommand\PlotRow[4]{%
\node at (0*\CellWidth,-#1*\RowHeight) {\PlotIm{\addstrings{#3}{1}{#4}}};%
\node at (1*\CellWidth,-#1*\RowHeight) {\PlotIm{\addstrings{#3}{2}{#4}}};%
\node at (2*\CellWidth,-#1*\RowHeight) {\PlotIm{\addstrings{#3}{3}{#4}}};%
\node at (3*\CellWidth,-#1*\RowHeight) {\PlotIm{\addstrings{#3}{4}{#4}}};%
\node at (4*\CellWidth,-#1*\RowHeight) {\PlotIm{\addstrings{#3}{5}{#4}}};%
\node at (5*\CellWidth,-#1*\RowHeight) {\PlotIm{\addstrings{#3}{6}{#4}}};%
\node at (6*\CellWidth,-#1*\RowHeight) {\PlotIm{\addstrings{#3}{7}{#4}}};%
\node at (7*\CellWidth,-#1*\RowHeight) {\PlotIm{\addstrings{#3}{8}{#4}}};%
\node at (8*\CellWidth,-#1*\RowHeight) {\PlotIm{\addstrings{#3}{9}{#4}}};%
\node at (9*\CellWidth,-#1*\RowHeight) {\PlotIm{\addstrings{#3}{10}{#4}}};%
}%

\begin{figure}
\centering
\renewcommand{\PlotIm}[1]{{\includegraphics[width=\CellWidth]{results_mnist100_20layers/#1}}}%
\begin{tikzpicture}
\PlotRow{0}{features}{Net_squared_seed13_testing}{}
\PlotRow{1}{Net}{Net_squared_seed13_testing}{_transformed}
\PlotRow{2}{ResNet}{ResNet_squared_seed13_testing}{_transformed}
\PlotRow{3}{ODENet}{ODENet_squared_seed13_testing}{_transformed}
\PlotRow{4}{ODENet+}{ODENetSimplex_squared_seed13_testing}{_transformed}
\end{tikzpicture}
\caption{Features of testing examples from \mnist~dataset~\cite{LeCun1998mnist} and transformed features by four networks under comparison: Net, ResNet, ODENet, ODENet+Simplex (from top to bottom). All networks have 20 layers.} \label{FIG:M:MNIST:QUAL}
\end{figure}
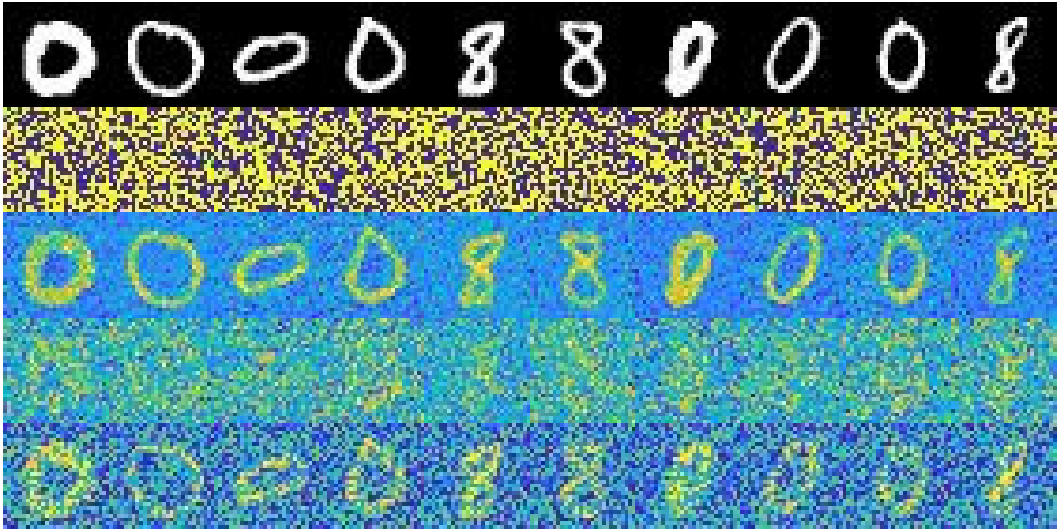

\section{Conclusions and outlook}
In this paper we have investigated the interpretation of deep learning as an optimal control problem. In particular, we have proposed a first-optimise-then-discretise approach for the derivation of ODE-inspired deep neural networks using symplectic partitioned Runge--Kutta methods. The latter discretisation guarantees that also after discretisation the first-order optimality conditions for the optimal control problem are fulfilled. This is in particular interesting under the assumption that the learned ODE discretisation follows some underlying continuous transformation that it approximated. Using partitioned Runge--Kutta methods, we derive several new deep learning algorithms which we compare for their convergence behaviour and the transformation dynamics of the so-learned discrete ODE iterations. Interestingly, while the convergence behaviour for the solution of the optimal control problem shows differences when trained with different partitioned Runge--Kutta methods, the learned transformation given by the discretised ODE with optimised parameters shows similar characteristics. It is probably too strong of a statement to suggest that our experiments therefore support our hypothesis of an underlying continuous optimal transformation as the similar behaviour could be a consequence of other causes. However, the experiments encourage our hypothesis.

The optimal control formulation naturally lends itself to learning more parameters such as the time discretisation which can be constrained to lie in a simplex. As we have seen in Figure \ref{FIG:M:TIME}, the simplex constraint lead to sparse time steps such that the effectively only very few layers were needed to represent the dynamics, thus these networks have faster online classification performance and lower memory footprint. Another advantage of this approach is that one does not need to know precisely in advance how many layers to choose since the training procedure selects this automatically.

An interesting direction for further investigation is to use the optimal control characterisation of deep learning for studying the stability of the problem under perturbations of $Y_0$. Since the optimal control problem is equivalent to a Hamiltonian boundary value problem, we can study the stability of the first by analysing the second. One can derive conditions on $f$ and $\mathcal J$ that ensure existence and stability of the optimal solutions with respect to perturbations on the initial data, or after increasing the number of data points. For the existence of solutions of the optimal control problem and the Pontryagin maximum principle see \cite{bucy1967two, gay2011clebsch, agrachev2013control, sontag2013mathematical}.

The stability of the problem can be analysed in different ways. The first is to investigate how the parameters $u(t):=(K(t), \beta(t))$ change under change (or perturbation) of the initial data and the cost function. The equation for the momenta of the Hamiltonian boundary value problem (adjoint equation) can be used to compute sensitivities of the optimal control problem under perturbation on the initial data~\cite{sanzserna2015sympletic}. In particular, the answer to this is linked to the Hessian of the Hamiltonian with respect to the parameters $u$. If $H_{u,u}$ is invertible, see Section~\ref{sec:Hamiltonian}, and remains invertible under such perturbations, then $u=\varphi(y,p)$ can be solved for in terms of the state $y$ and the co-state $p$. 

The second is to ask how generalisable the learned parameters $u$ are. The parameters $u$, i.e. $K$ and $\beta$ determine the deformation of the data in such a way that the data becomes classifiable with the Euclidean norm at final time $T$. It would be interesting to show that $\varphi$ does not change much under perturbation, and neither do the deformations determined by $u$. 

Another interesting direction for future research is the generalisation of the optimal control problem to feature an inverse scale-space ODE as a constraint, where we do not consider the time derivative of the state variable, but of a subgradient of a corresponding convex functional with the state variable as its argument, see for example~\cite{scherzer2001inverse,burger2005nonlinear,burger2006nonlinear}. Normally these flows are discretised with an explicit or implicit Euler scheme. These discretisations can reproduce various neural network architectures~\cite[Section 9]{benning2018modern}. Hence, applying the existing knowledge of numerical discretisation methods in a deep learning context may lead to a better and more systematic way of developing new architectures with desirable properties.

Other questions revolve around the sensitivity in the classification error. How can we estimate the error in the classification once the parameters are learned? Given $u$ obtained solving the optimal control problem, if we change (or update) the set of features, how big is the error in the classification $\mathcal J(y^{[N]})$?

\subsection*{Acknowledgments:} MB acknowledges support from the Leverhulme Trust Early Career Fellowship ECF-2016-611 'Learning from mistakes: a supervised feedback-loop for imaging applications'. CBS acknowledges support from the Leverhulme Trust project on Breaking the non-convexity barrier, the Philip Leverhulme Prize, the EPSRC grant No. EP/M00483X/1, the EPSRC Centre No. EP/N014588/1, the European Union Horizon 2020 research and innovation programmes under the Marie Skodowska-Curie grant agreement No. 777826 NoMADS and No. 691070 CHiPS, the Cantab Capital Institute for the Mathematics of Information and the Alan Turing Institute. We gratefully acknowledge the support of NVIDIA Corporation with the donation of a Quadro P6000 and a Titan Xp GPU used for this research. EC and BO thank the SPIRIT project (No. 231632) under the Research Council of Norway FRIPRO funding scheme. The authors would like to thank the Isaac Newton Institute for Mathematical Sciences, Cambridge, for support and hospitality during the programmes \emph{Variational methods and effective algorithms for imaging and vision (2017)} and \emph{Geometry, compatibility and structure preservation in computational differential equations (2019)} where work on this paper was undertaken. This work was supported by EPSRC grant No. EP/K032208/1.

\begin{appendices}
\section{Discrete necessary optimality conditions}\label{DNOC}
We prove Proposition \ref{necdiscreteoptimality} for the general symplectic partitioned Runge--Kutta method. 
\begin{proof}[Proof of Proposition \ref{necdiscreteoptimality}] 
We introduce Lagrangian multipliers $p_i^{[j]}$, $p^{[j+1]}$ and 
consider the Lagrangian
{\small
\begin{equation}
\mathcal{L}=\mathcal{L}\left(\{y^{[j]}\}_{j=1}^N,\{y_i^{[j]}\}_{j=1}^N,\{u_i^{[j]}\}_{j=0}^{N-1},\{p^{[j+1]}\}_{j=0}^{N-1},\{p_i^{[j]}\}_{j=0}^{N-1}\right)
\end{equation}
}
$i=1,\dots , s$,
\begin{eqnarray*}
\mathcal{L}&:=&\mathcal{J}(y^{[N]})-\Delta t \sum_{j=0}^{N-1}\langle p^{[j+1]},\frac{y^{[j+1]}-y^{[j]}}{\Delta t }-\sum_{i=1}^s f(y_i^{[j]},u_i^{[j]}) \rangle \\
&-&\Delta t \sum_{j=0}^{N-1}\Delta t \sum_{i=1}^s b_i\langle \ell_i^{[j]},\frac{y_i^{[j]}-y^{[j]}}{\Delta t }-\sum_{m=1}^s a_{i,m}f(y_m^{[j]},u_m^{[j]}) \rangle.
\end{eqnarray*}
An equivalent formulation of \eqref{Doptimisation1} subject to \eqref{RK1}-\eqref{RK3} is
$$\inf_{
\begin{array}{l}
\{u_i^{[j]}\}_{j=0}^{N-1},\\
\{y^{[j]}\}_{j=1}^{N},\{y_i^{[j]}\}_{j=1}^{N},
\end{array}
} \, 
\sup_{
\{p^{[j]}\}_{j=1}^{N},\{\ell_i^{[j]}\}_{j=1}^{N}}\, \mathcal{L}$$
Taking arbitrary and independent variations
$$y^{[j]}+\xi v^{[j]},\quad y_i^{[j]}+\xi v_i^{[j]},\quad u_i^{[j]}+\xi w_i^{[j]}\quad p^{[j+1]}+\xi \gamma^{[j+1]},\quad \ell_i^{[j]}+\xi \gamma_i^{[j]}$$
and imposing $\delta \mathcal{L}=0$ for all variations,
we obtain
\begin{eqnarray*}
0=\delta \mathcal{L}&=&\langle \partial \mathcal{J}(y^{[N]}), v^{[N]}\rangle-\Delta t \sum_{j=0}^{N-1}\langle \gamma^{[j+1]},\frac{y^{[j+1]}-y^{[j]}}{\Delta t}-\sum_{i=1}^s b_i f(y_i^{[j]},u_i^{[j]})\rangle\\
&-&\Delta t \sum_{j=0}^{N-1}\langle p^{[j+1]},\frac{v^{[j+1]}-v^{[j]}}{\Delta t}-\sum_{i=1}^s b_i\partial_y f(y_i^{[j]},u_i^{[j]})v_i^{[j]}\rangle\\
&+&\Delta t \sum_{j=0}^{N-1}\langle p^{[j+1]},\sum_{i=1}^s b_i\partial_u f(y_i^{[j]},u_i^{[j]})w_i^{[j]} \rangle\\
&-&\Delta t \sum_{j=0}^{N-1} \Delta t \sum_{i=1}^s b_i\langle \gamma_i^{[j+1]},\frac{y_i^{[j+1]}-y^{[j]}}{\Delta t }-\sum_{m=1}^s a_{i,m} f(y_m^{[j]},u_m^{[j]})\rangle\\
&-&\Delta t \sum_{j=0}^{N-1}\Delta t \sum_{i=1}^s b_i\langle \,
\ell_i^{[j]}, \frac{v_i^{[j+1]}-v^{[j]}}{\Delta t }+\sum_{m=1}^s a_{i,m}\partial_y f(y_m^{[j]},u_m^{[j]})v_m^{[j]}\,\rangle\\
&-&\Delta t \sum_{j=0}^{N-1}\Delta t \sum_{i=1}^s b_i\langle \,
\ell_i^{[j]}, \sum_{m=1}^s a_{i,m}
\partial_u f(y_m^{[j]},u_m^{[j]})w_m^{[j]} \,\rangle.
\end{eqnarray*}
Because the variations $\gamma^{[j]}$, $\gamma_i^{[j]}$ are arbitrary, we must have
\begin{eqnarray*}
\frac{y^{[j+1]}-y^{[j]}}{\Delta t}&=&\sum_{i=1}^s b_i f(y_i^{[j]}, u_i^{[j]})\\
\frac{y_i^{[j+1]}-y^{[j]}}{\Delta t}&=&\sum_{m=1}^s a_{i,m} f(y_m^{[j]}, u_m^{[j]})
\end{eqnarray*}
corresponding to the forward method, \eqref{PRK1}, \eqref{PRK2}, and we are left with terms depending on  $w_i^{[j]}$ and $v_i^{[j]}$ which we can discuss separately. Collecting all the terms containing the variations $w_i^{[j]}$ we get
\begin{equation}
    \label{weak_gradients}
\sum_{j=0}^{N-1}\left(\sum_{i=1}^s b_i\langle p^{[j+1]}, \partial_u f(y_i^{[j]},u_i^{[j]})w_i^{[j]} \rangle-\Delta t \sum_{i=1}^s b_i\sum_{m=1}^s a_{i,m}\langle \,
\ell_i^{[j]}, \partial_u f(y_m^{[j]},u_m^{[j]})w_m^{[j]} \,\rangle\right).
\end{equation}
 In \eqref{weak_gradients}, renaming the indexes so that $i\rightarrow k$ in the first sum and $m\rightarrow k$ and $w_m^{[j]}\rightarrow w_k^{[j]}$ in the second sum, we get
$$
\sum_{k=1}^s \left(b_k\langle p^{[j+1]}, \partial_u f(y_k^{[j]}, u_k^{[j]})w_k^{[j]} \rangle-\Delta t \sum_{i=1}^s b_i a_{i,k}\langle \,
\ell_i^{[j]}, \partial_u f(y_k^{[j]},u_k^{[j]})w_k^{[j]} \,\rangle \right), 
$$
for $j=0,\dots, N-1$, and
$$
\sum_{k=1}^s \left(\langle b_k \partial_u f(y_k^{[j]},u_k^{[j]})^T p^{[j+1]}-\Delta t \sum_{i=1}^sb_i a_{i,k} \partial_u f(y_k^{[j]},u_k^{[j]})^T\ell_i^{[j]},w_k^{[j]} \rangle \right).
$$
Because each of the variations  $w_k^{[j]}$ is arbitrary for  $k=1,\dots , s$ and $j=0,\dots, N-1$ each of the terms must vanish and we get
$$
 \langle \partial_u f(y_k^{[j]},u_k^{[j]})^Tp^{[j+1]}-\Delta t \sum_{i=1}^s\frac{b_i a_{i,k}}{b_k}\partial_u f(y_k^{[j]},u_k^{[j]})^T\ell_i^{[j]},w_k^{[j]} \rangle  =0,
$$
and finally
$$
\partial_u f(y_k^{[j]},u_k^{[j]})^T\left(p^{[j+1]}-\Delta t \sum_{i=1}^s\frac{b_i a_{i,k}}{b_k}\ell_i^{[j]}\right)=0
$$
corresponding to the discretised constraints, and where we recognise that
$$p_k^{[j]}=p^{[j+1]}-\Delta t \sum_{i=1}^s\frac{b_i a_{i,k}}{b_k}\ell_i^{[j]}.$$
The remaining terms contain the variations $v_i^{[j]}$ and we have
\begin{eqnarray*}
\langle \partial \mathcal{J}(y^{[N]}), v^{[N]}\rangle&-&\Delta t \sum_{j=0}^{N-1}\langle p^{[j+1]},\frac{v^{[j+1]}-v^{[j]}}{\Delta t }-\sum_{i=1}^s b_i\partial_y f(y_i^{[j]},u_i^{[j]})v_i^{[j]}\rangle \\
&-&\Delta t ^2\sum_{j=0}^{N-1}\sum_{i=1}^s b_i\langle \,
\ell_i^{[j]}, \frac{v_i^{[j+1]}-v^{[j]}}{\Delta t }+\sum_{m=1}^s a_{i,m}\partial_y f(y_m^{[j]},u_m^{[j]})v_m^{[j]}\,\rangle =0
\end{eqnarray*}
There are only two terms involving $v_N$, leading to
$$\mathcal{J}(y^{[N]}), v^{[N]}\rangle-\langle p^{[N]},v_N\rangle =0$$
corresponding to the condition $p^{[N]}=\mathcal{J}(y^{[N]})$. We consider separately for each $j$ terms involving $v^{[j]}$ and $V_i^{[j]}$ for $i=1,\dots, s$ and see that
\begin{eqnarray*}
&&\langle p^{[j+1]},v^{[j]}+\Delta t \sum_{i=1}^s b_i\partial_y f(y_i^{[j]},u_i^{[j]})v_i^{[j]}\rangle \\
&-&\Delta t \sum_{i=1}^s b_i\langle \ell_{i}^{[j]},v_i^{[j]}-v^{[j]}+\Delta t \sum_{m=1}^sa_{i,m}\partial_y f(y_m^{[j]},u_m^{[j]})v_m^{[j]}\\
&-&\langle p^{[j]},v^{[j]}\rangle=0
\end{eqnarray*}
which we rearrange into
\begin{eqnarray*}
&&\langle p^{[j+1]}-p^{[j]}+h\sum_{i=1}^sb_i\ell_i^{[j]},v^{[j]}\rangle\\
&&\Delta t \sum_{k=1}^sb_k\langle \partial_y f(y_k^{[j]},u_k^{[k]})^Tp^{[j+1]},v_k^{[j]} \rangle\\
&-&\Delta t \sum_{i=1}^sb_i\left( \langle \ell_i^{[j]},v_i^{[j]}\rangle +\Delta t  \sum_{m=1}^sa_{i,m}\langle \partial_y f(y_m^{[j}],u_m^{[j]})\ell_i^{[j]},v_m^{[j]}\rangle\right)=0 
\end{eqnarray*}
This yields
$$
 p^{[j+1]}=p^{[j]}-\Delta t \sum_{i=1}^sb_i\ell_i^{[j]}
$$
and
\begin{eqnarray*}
&&\Delta t \sum_{k=1}^sb_k\langle \partial_y f(y_k^{[j]},u_k^{[k]})^Tp^{[j+1]},v_k^{[j]} \rangle\\
&-&\Delta t \sum_{i=1}^sb_i\left( \langle \ell_i^{[j]},v_i^{[j]}\rangle +\Delta t  \sum_{m=1}^sa_{i,m}\langle \partial_y f(u_m^{[j}],u_m^{[j]})\ell_i^{[j]},v_m^{[j]}\rangle\right)=0.
\end{eqnarray*}
From the last equation we get
$$0=\partial_y f(y_k^{[j]},u_k^{[k]})^Tp^{[j+1]}-\ell_k^{[j]}-\Delta t \sum_{i=1}^s\frac{b_i a_{i,k}}{b_k}\partial_y f(y_k^{[j]},u_k^{[j]})^T\ell_i^{[j]}.$$
with
$$\ell_k^{[j]}=\partial_y f(y_k^{[j]},u_k^{[j]})^T \left[p^{[j+1]}-\Delta t \sum_{i=1}^s\frac{b_i a_{i,k}}{b_k}\ell_i^{[j]}\right].$$
\end{proof}
\end{appendices}


\providecommand{\href}[2]{#2}
\providecommand{\arxiv}[1]{\href{http://arxiv.org/abs/#1}{arXiv:#1}}
\providecommand{\url}[1]{\texttt{#1}}
\providecommand{\urlprefix}{URL }


@inproceedings{LeCun1998mnist,
author = {LeCun, Y and Bottou, L and Bengio, Y and Haffner, P},
booktitle = {Proceedings of the IEEE},
number = {11},
pages = {2278--2324},
title = {{Gradient-Based Learning Applied to Document Recognition}},
volume = {86},
year = {1998}
}

@inproceedings{He2016resnet,
author = {He, Kaiming and Zhang, Xiangyu and Ren, Shaoqing and Sun, Jian},
booktitle = {IEEE Conference on Computer Vision and Pattern Recognition},
pages = {770--778},
title = {{Deep Residual Learning for Image Recognition}},
year = {2016}
}

@misc{Higham2018deeplearning,
archivePrefix = {arXiv},
eprint = {1801.05894v1},
author = {Higham, Catherine F and Higham, Desmond J},
title = {{Deep Learning: An Introduction for Applied Mathematicians}},
year = {2018}
}

@misc{Han2018,
archivePrefix = {arXiv},
eprint = {1807.01083v1},
author = {Weinan, E and Han, Jiequn and Li, Qianxiao},
title = {{A Mean-Field Optimal Control Formulation of Deep Learning}},
year = {2018}
}

@misc{thorpe2018dlo,
archivePrefix = {arXiv},
eprint = {1810.11741},
author = {Matthew Thorpe and Yves van Gennip},
title = {{Deep Limits of Residual Neural Networks}},
year = {2018},
}

@misc{Li2018,
archivePrefix = {arXiv},
eprint = {1803.01299v2},
author = {Li, Qianxiao and Hao, Shuji},
title = {{An Optimal Control Approach to Deep Learning and Applications to Discrete-Weight Neural Networks}},
year = {2018}
}

@inproceedings{chen2018neural,
  title={Neural ordinary differential equations},
  author={Chen, Tian Qi and Rubanova, Yulia and Bettencourt, Jesse and Duvenaud, David K},
  booktitle={Advances in Neural Information Processing Systems},
  pages={6572--6583},
  year={2018}
}

@article{abarbanel2018machine,
  title={Machine learning: Deepest learning as statistical data assimilation problems},
  author={Abarbanel, Henry DI and Rozdeba, Paul J and Shirman, Sasha},
  journal={Neural computation},
  volume={30},
  number={8},
  pages={2025--2055},
  year={2018},
  publisher={MIT Press}
}

@book{agrachev2013control,
  title={Control theory from the geometric viewpoint},
  author={Agrachev, Andrei A and Sachkov, Yuri},
  volume={87},
  year={2013},
  publisher={Springer Science \& Business Media}
}

@article{blanes05otn,
title={On  the  necessity  of  negative  coefficients for  operator  splitting  schemes  of  order  higher  than  two},
author={S.  Blanes  and  F.  Casas},  
journal={Applied Numerical  Mathematics},  
volume={54},
number={1},
pages={23--37},
year={2005},  
url = {http://dx.doi.org/10.1016/j.apnum.2004.10.005}
}

@article{bucy1967two,
  title={Two-point boundary value problems of linear Hamiltonian systems},
  author={Bucy, RS},
  journal={SIAM Journal on Applied Mathematics},
  volume={15},
  number={6},
  pages={1385--1389},
  year={1967},
  publisher={SIAM}
}

@inproceedings{dauphin2014identifying,
  title={Identifying and attacking the saddle point problem in high-dimensional non-convex optimization},
  author={Dauphin, Yann N and Pascanu, Razvan and Gulcehre, Caglar and Cho, Kyunghyun and Ganguli, Surya and Bengio, Yoshua},
  booktitle={Advances in neural information processing systems},
  pages={2933--2941},
  year={2014}
}

@article{sutskever2013importance,
  title={On the importance of initialization and momentum in deep learning.},
  author={Sutskever, Ilya and Martens, James and Dahl, George E and Hinton, Geoffrey E},
  journal={ICML (3)},
  volume={28},
  number={1139-1147},
  pages={5},
  year={2013}
}

@inproceedings{chang2018reversible,
  title={Reversible architectures for arbitrarily deep residual neural networks},
  author={Chang, Bo and Meng, Lili and Haber, Eldad and Ruthotto, Lars and Begert, David and Holtham, Elliot},
  booktitle={Thirty-Second AAAI Conference on Artificial Intelligence},
  year={2018}
}

@article{weinan2017proposal,
  title={A proposal on machine learning via dynamical systems},
  author={Weinan E},
  journal={Communications in Mathematics and Statistics},
  volume={5},
  number={1},
  pages={1--11},
  year={2017},
  publisher={Springer}
}

@article{gay2011clebsch,
  title={Clebsch optimal control formulation in mechanics},
  author={Gay-Balmaz, Fran{\c{c}}ois and Ratiu, Tudor S},
  journal={J. Geom. Mech},
  volume={3},
  number={1},
  pages={41--79},
  year={2011}
}

@article{haber2017stable,
  title={Stable architectures for deep neural networks},
  author={Haber, Eldad and Ruthotto, Lars},
  journal={Inverse Problems},
  volume={34},
  number={1},
  pages={014004},
  year={2017},
  publisher={IOP Publishing}
}

@article{hager2000runge,
  title={Runge--{K}utta methods in optimal control and the transformed adjoint system},
  author={Hager, William W},
  journal={Numerische Mathematik},
  volume={87},
  number={2},
  pages={247--282},
  year={2000},
  publisher={Springer}
}

@article{li2017maximum,
  title={Maximum principle based algorithms for deep learning},
  author={Li, Qianxiao and Chen, Long and Tai, Cheng and Weinan, E},
  journal={The Journal of Machine Learning Research},
  volume={18},
  number={1},
  pages={5998--6026},
  year={2017},
  publisher={JMLR. org}
}

@InProceedings{pmlr-v80-lu18d,
  title = 	 {Beyond Finite Layer Neural Networks: Bridging Deep Architectures and Numerical Differential Equations},
  author = 	 {Lu, Yiping and Zhong, Aoxiao and Li, Quanzheng and Dong, Bin},
  booktitle = 	 {Proceedings of the 35th International Conference on Machine Learning},
  pages = 	 {3276--3285},
  year = 	 {2018},
  editor = 	 {Dy, Jennifer and Krause, Andreas},
  volume = 	 {80},
  series = 	 {Proceedings of Machine Learning Research},
  address = 	 {Stockholmsmässan, Stockholm Sweden},
  publisher = 	 {PMLR},
  url = 	 {http://proceedings.mlr.press/v80/lu18d.html}
}

@article{dontchev2000tea,
  title={The Euler approximation in state constrained optimal control},
  author={Dontchev A L and Hager, William W},
  journal={Mathematics of Computation},
  volume={70},
  number={233},
  pages={173--203},
  year={2000},
  publisher={Springer}
}

@book{hairer2006geometric,
  title={Geometric numerical integration: structure-preserving algorithms for ordinary differential equations},
  author={Hairer, Ernst and Lubich, Christian and Wanner, Gerhard},
  volume={31},
  year={2006},
  publisher={Springer Science \& Business Media}
}

@article{ross2005roadmap,
title={A roadmap for optimal control: the right way to commute},
author={Ross, I. M.},
journal={Annals of the New York Academy of Sciences}, 
volume={1065},
number={1}, 
pages={210--231},
year={2005}
}

@inproceedings{he2016identity,
author={He, K. and Zhang, X. and Ren, S. and Sun, J.},
title={Identity mappings in deep residual networks},
booktitle={European Conference on Computer Vision},
pages={630--645},
year={2016}
}

@book{Pontryagin,
title={L.S. Pontryagin selected works, The Mathematical theory of optimal processes, Classics of Soviet Mathematics},
author={L.S. Pontryagin},
volume={4},
publisher={CRC Press.},
year = {1987}
}

@article{sanzserna2015sympletic,
author={Sanz-Serna, J. M.},
title={Symplectic {R}unge-{K}utta schemes for adjoint equations automatic differentiation, optimal control and more},
journal={SIAM Review},
volume={58}, 
year={2015},
pages={3--33}
}

@book{sontag2013mathematical,
author={Sontag, Eduardo D.},
title={Mathematical control theory: deterministic finite dimensional systems},
volume={6},
publisher={Springer Science \& Business Media}, year={2013}
}

@article{LeCun:2015aa,
	Author = {LeCun, Y. and Bengio, Y. and Hinton, G.},
	Journal = {Nature},
	Number = {7553},
	Pages = {436--444},
	Title = {Deep learning},
	Volume = {521},
	Year = {2015}
}

@misc{gholami19aua,
archivePrefix = {arXiv},
eprint = {1902.10298},
	Author = {Gholami, A. and Keutzer, K. and Biros, G.},
	Title = {ANODE: Unconditionally Accurate Memory-Efficient Gradients for Neural ODEs.},
	Year = {2019}
}

@misc{Igami:2017aa,
archivePrefix = {arXiv},
eprint = {1710.10967},
	Author = {Igami, M.},
	Title = {Artificial Intelligence as Structural Estimation: Economic Interpretations of {Deep Blue}, {Bonanza}, and {AlphaGo}},
	Year = {2017}
}

@inproceedings{Szegedy2013IntriguingPO,
  title={Intriguing Properties of Neural Networks},
  author={Christian Szegedy and Wojciech Zaremba and Ilya Sutskever and Joan Bruna and Dumitru Erhan and Ian J. Goodfellow and Rob Fergus},
  booktitle={International Conference on Learning Representations},
  year={2014}
}

@inproceedings{moosavi2016deepfool,
  title={Deepfool: a Simple and Accurate Method to Fool Deep Neural Networks},
  author={Moosavi-Dezfooli, Seyed-Mohsen and Fawzi, Alhussein and Frossard, Pascal},
  booktitle={Proceedings of the IEEE Conference on Computer Vision and Pattern Recognition},
  pages={2574--2582},
  year={2016}
}

@misc{explainingharnessing,
archivePrefix = {arXiv},
eprint = {1412.6572},
author = {Goodfellow, Ian and Shlens, Jonathon and Szegedy, Christian},
year = {2014},
title = {Explaining and Harnessing Adversarial Examples}
}

@misc{kurakin2016adversarial,
archivePrefix = {arXiv},
eprint = {1607.02533},
  title={Adversarial Examples in the Physical World},
  author={Kurakin, Alexey and Goodfellow, Ian and Bengio, Samy},
  year={2016}
}

@inproceedings{lecun1988theoretical,
  title={A theoretical framework for back-propagation},
  author={LeCun, Yann},
  booktitle={Proceedings of the 1988 connectionist models summer school},
  editors={Touresky, D and Hinton, G and Sejnowski, T},
  volume={1},
  pages={21--28},
  year={1988},
  organization={CMU, Pittsburgh, Pa: Morgan Kaufmann}
}

@inproceedings{scherzer2001inverse,
  title={Inverse scale space theory for inverse problems},
  author={Scherzer, Otmar and Groetsch, Chuck},
  booktitle={International Conference on Scale-Space Theories in Computer Vision},
  pages={317--325},
  year={2001},
  organization={Springer}
}

@inproceedings{sonoda2017double,
  title={Double continuum limit of deep neural networks},
  author={Sonoda, Sho and Murata, Noboru},
  booktitle={ICML Workshop Principled Approaches to Deep Learning},
  year={2017}
}

@inproceedings{burger2005nonlinear,
  title={Nonlinear inverse scale space methods for image restoration},
  author={Burger, Martin and Osher, Stanley and Xu, Jinjun and Gilboa, Guy},
  booktitle={International Workshop on Variational, Geometric, and Level Set Methods in Computer Vision},
  pages={25--36},
  year={2005},
  organization={Springer}
}

@article{burger2006nonlinear,
  title={Nonlinear inverse scale space methods},
  author={Burger, Martin and Gilboa, Guy and Osher, Stanley and Xu, Jinjun and others},
  journal={Communications in Mathematical Sciences},
  volume={4},
  number={1},
  pages={179--212},
  year={2006},
  publisher={International Press of Boston}
}

@article{benning2018modern,
  title={Modern regularization methods for inverse problems},
  author={Benning, Martin and Burger, Martin},
  journal={Acta Numerica},
  volume={27},
  pages={1--111},
  year={2018},
  publisher={Cambridge University Press}
}

@article{Kutta1901,
author={Kutta, Wilhelm},
title={Beitrag zur näherungsweisen Integration totaler Differentialgleichungen}, 
journal={Z. Math. Phys.},
volume={46}, 
year={1901}, 
pages={435-453}
}

@inproceedings{Duchi2008,
author = {Duchi, John and Shalev-Shwartz, Shai and Singer, Yoram and Chandra, Tushar},
booktitle = {Proceedings of the 25th International Conference on Machine Learning},
pages = {272--279},
title = {{Efficient Projections onto the L1 -ball for Learning in High dimensions}},
year = {2008}
}

@article{plaut1986experiments,
  title={Experiments on Learning by Back Propagation.},
  author={Plaut, David C and others},
  year={1986},
  publisher={ERIC}
}

@inproceedings{tikhonov1963solution,
  title={Solution of incorrectly formulated problems and the regularization method},
  author={Tikhonov, Andrei N},
  booktitle={Dokl. Akad. Nauk.},
  volume={151},
  pages={1035--1038},
  year={1963}
}

@article{phillips1962technique,
  title={A technique for the numerical solution of certain integral equations of the first kind},
  author={Phillips, David L},
  journal={Journal of the ACM (JACM)},
  volume={9},
  number={1},
  pages={84--97},
  year={1962},
  publisher={ACM}
}

@article{santosa1986linear,
  title={Linear inversion of band-limited reflection seismograms},
  author={Santosa, Fadil and Symes, William W},
  journal={SIAM Journal on Scientific and Statistical Computing},
  volume={7},
  number={4},
  pages={1307--1330},
  year={1986},
  publisher={SIAM}
}

@article{tibshirani1996regression,
  title={Regression shrinkage and selection via the lasso},
  author={Tibshirani, Robert},
  journal={Journal of the Royal Statistical Society: Series B (Methodological)},
  volume={58},
  number={1},
  pages={267--288},
  year={1996},
  publisher={Wiley Online Library}
}

@book{Bishop2006ml,
author = {Bishop, Christopher M},
pages = {738},
publisher = {Springer},
title = {{Pattern Recognition and Machine Learning}},
year = {2006}
}
\end{document}